\documentclass{book}
 
\newcommand{\href}[1]{#1} 

\usepackage{ifthen}
\newboolean{PrintVersion}
\setboolean{PrintVersion}{false}

\usepackage{amsmath,amssymb,amsfonts} 
\usepackage[pdftex]{graphicx} 
\usepackage{dsfont}
\usepackage{mathrsfs}
\usepackage{yhmath}
\usepackage{geometry}
\usepackage{url}
\usepackage{stmaryrd}
\usepackage{amsthm}
\newtheorem{theorem}{Theorem}[section]
\newtheorem{lemma}[theorem]{Lemma}
\newtheorem{definition}[theorem]{Definition}
\newtheorem{proposition-definition}[theorem]{Proposition-Definition}
\newtheorem{Theorem-definition}[theorem]{Theorem-Definition}
\newtheorem{theorem-definition}[theorem]{Theorem-Definition}
\newtheorem{corollary}[theorem]{Corollary}
\newtheorem{proposition}[theorem]{Proposition}
\newtheorem{remark}[theorem]{Remark}
\usepackage[pdftex,pagebackref=false]{hyperref} 
\hypersetup{
    plainpages=false,       
    unicode=false,          
    pdftoolbar=true,        
    pdfmenubar=true,        
    pdffitwindow=false,     
    pdfstartview={FitH},    
    pdfnewwindow=true,      
    colorlinks=true,        
    linkcolor=blue,         
    citecolor=green,        
    filecolor=magenta,      
    urlcolor=cyan           
}
\ifthenelse{\boolean{PrintVersion}}{   
\hypersetup{	
    citecolor=black,%
    filecolor=black,%
    linkcolor=black,%
    urlcolor=black}
}{} 

\usepackage[toc,abbreviations]{glossaries-extra} 

\let\origdoublepage\cleardoublepage
\newcommand{\clearemptydoublepage}{%
  \clearpage{\pagestyle{empty}\origdoublepage}}
\let\cleardoublepage\clearemptydoublepage

\newglossary*{symbols}{Glossary}
\newglossaryentry{sto1}
{
name={$\xi$},
type=symbols,
description={}
}
\newglossaryentry{sto2}
{
name={$X_{r,1}$},
type=symbols,
description={}
}
\newglossaryentry{sto3}
{
name={$\widetilde{X}_{r,q,1}$},
type=symbols,
description={}
}
\newglossaryentry{sto4}
{
name={$\widetilde{X}_{r,2}$},
type=symbols,
description={}
}
\newglossaryentry{sto5}
{
name={$\widetilde{X}_{r,\epsilon,2}$},
type=symbols,
description={}
}
\newglossaryentry{sto6}
{
name={$\widetilde{X}_{r,3}$},
type=symbols,
description={}
}
\newglossaryentry{sto7}
{
name={$\widetilde{X}_{r,\epsilon,3}$},
type=symbols,
description={}
}
\newglossaryentry{sto8}
{
name={$\widetilde{\Xi}_{r,1}$},
type=symbols,
description={}
}
\newglossaryentry{sto9}
{
name={$\widetilde{\Xi}_{r,\epsilon,1}$},
type=symbols,
description={}
}
\newglossaryentry{sto10}
{
name={$\widetilde{\Xi}_{r,2}$},
type=symbols,
description={}
}
\newglossaryentry{sto11}
{
name={$\widetilde{\Xi}_{r,\epsilon,2}$},
type=symbols,
description={}
}
\newglossaryentry{sto12}
{
name={$\widetilde{\Xi}_{r,3}$},
type=symbols,
description={}
}
\newglossaryentry{sto13}
{
name={$\widetilde{\Xi}_{r,\epsilon,3}$},
type=symbols,
description={}
}
\newglossaryentry{sto14}
{
name={$\widetilde{Y}_{r,1}$},
type=symbols,
description={}
}
\newglossaryentry{sto15}
{
name={$\widetilde{Y}_{r,2}$},
type=symbols,
description={}
}
\newglossaryentry{sto16}
{
name={$\widetilde{Y}_{r,\epsilon,1}$},
type=symbols,
description={}
}
\newglossaryentry{sto17}
{
name={$\widetilde{Y}_{r,\epsilon,2}$},
type=symbols,
description={}
}
\newglossaryentry{pc1}
{
name={$\mathscr{S}'$},
type=symbols,
description={}
}
\newglossaryentry{pc2}
{
name={$\mathscr{S}$},
type=symbols,
description={}
}
\newglossaryentry{pc3}
{
name={$B_{r,\gamma}^\alpha$},
type=symbols,
description={}
}\newglossaryentry{pc4}
{
name={$\mathscr{C}^\alpha$},
type=symbols,
description={}
}\newglossaryentry{pc5}
{
name={$C_U\mathscr{C}^\alpha$},
type=symbols,
description={}
}\newglossaryentry{pc6}
{
name={$\mathfrak{Y}_U^\alpha$},
type=symbols,
description={}
}\newglossaryentry{f1}
{
name={$\delta_{q}f$},
type=symbols,
description={}
}
\newglossaryentry{f2}
{
name={$f_1\varodot f_2$},
type=symbols,
description={}
}
\newglossaryentry{f3}
{
name={$f_1\varolessthan f_2$},
type=symbols,
description={}
}
\newglossaryentry{f4}
{
name={$f_2\varogreaterthan f_1$},
type=symbols,
description={}
}
\newglossaryentry{f5}
{
name={$\Psi$},
type=symbols,
description={}
}
\newglossaryentry{f6}
{
name={$\eta_{r}$},
type=symbols,
description={}
}
\newglossaryentry{f7}
{
name={$G_rf$},
type=symbols,
description={}
}
\newglossaryentry{f8}
{
name={$V_m$},
type=symbols,
description={}
}

\makenoidxglossaries

\begin{document}

\pagestyle{empty}
\pagenumbering{roman}

\begin{titlepage}
        \begin{center}
        \vspace*{1.0cm}

        \Huge
        {\bf Paracontrolled approach to the stochastic Cahn-Hilliard equation }

        \vspace*{1.0cm}

        \normalsize
        by \\

        \vspace*{1.0cm}

        \Large
        Joe Ghafari \\

        \vspace*{3.0cm}

        \normalsize
        A thesis \\
        presented to the University of Ottawa \\ 
        in fulfillment of the \\
        requirement for the degree of \\
        Doctor of Philosophy \\
        in \\
        Mathematics and Statistics \\

        \vspace*{2.0cm}

        Ottawa, Ontario, Canada, 2026 \\

        \vspace*{1.0cm}

        \copyright\ Joe Ghafari 2026 \\
        \end{center}
\end{titlepage}

\pagestyle{plain}
\setcounter{page}{2}

\cleardoublepage 
\phantomsection    

\cleardoublepage
\phantomsection    

\addcontentsline{toc}{chapter}{Abstract}
\begin{center}\textbf{Abstract}\end{center}

We prove the existence and uniqueness of a local-in-time solution to the stochastic Cahn-Hilliard
equation in space dimensions 1, 2, 3, 4 and 5. Our method relies on the Da Prato-Debussche trick for $d=4$ and paracontrolled distributions for $d=5.$

\cleardoublepage
\phantomsection    

\addcontentsline{toc}{chapter}{Acknowledgements}
\begin{center}\textbf{Acknowledgements}\end{center}

I would also like to thank my supervisors for their guidance and supervision throughout the development of this research. Despite the differences in our areas of expertise, I truly appreciate the effort you made to guide me through this process. It has been a rewarding and challenging journey, and I am particularly grateful for the freedom you granted me, which allowed me to independently pursue the research project I was most passionate about.

Lastly, I would like to thank the Department of Mathematics and Statistics at the University of Ottawa for their financial support, which was essential to the completion of my studies.

\cleardoublepage
\phantomsection    

\addcontentsline{toc}{chapter}{Dedication}
\begin{center}\textbf{Dedication}\end{center}

Family always comes first, and I am deeply grateful for their unwavering presence throughout this journey. To my father, Khalil, thank you for always backing me up and for your constant support. To my mother, Olga, thank you for everything, it is from you that I inherited my intelligence. This achievement is also dedicated to the memory of Fahed and Ivette.
\cleardoublepage
\phantomsection    

\renewcommand\contentsname{Table of Contents}
\tableofcontents
\cleardoublepage
\phantomsection    




\printnoidxglossary[type=symbols]
\cleardoublepage
\phantomsection		

\pagenumbering{arabic}

\chapter{Introduction}

Stochastic partial differential equations (SPDEs) appear naturally in the mathematical description of systems subject to random influences, such as fluctuating interfaces, turbulent fluids, or fields in statistical mechanics and quantum field theory. These equations provide a rigorous framework for modeling dynamics where macroscopic deterministic forces continuously compete with microscopic, random fluctuations. Several of these equations are singular in the sense that their solutions are too irregular to be treated using classical analytic methods. Among the most fundamentally significant of these models is the stochastic Cahn-Hilliard equation, which serves as a cornerstone for describing complex phase separation dynamics in multicomponent mixtures.

The physical origins of this model trace back to the seminal work of Cahn and Hilliard \cite{10.1063/1.1744102}, who developed a continuum theory to describe the free energy of a nonuniform system. Their research primarily focused on interfacial free energy and the spontaneous phase separation, known as spinodal decomposition, that occurs when a homogeneous mixture, such as a binary alloy or a polymer blend, becomes thermodynamically unstable and separates into distinct, coexisting phases. The deterministic Cahn-Hilliard equation captures the gradient flow of this system, describing how regions of differing concentrations evolve, coarsen, and form distinct spatial domains over time to minimize the total free energy.

However, the deterministic model relies on macroscopic averages and inherently neglects the thermal fluctuations present at the microscopic level. Cook \cite{COOK1970297} demonstrated that incorporating a random source term into the spinodal decomposition model is necessary to accurately capture the Brownian motion and thermal fluctuations driving the early stages of phase separation. This integration of thermal noise transforms the deterministic physical model into the stochastic Cahn-Hilliard equation, which is formally expressed as:

$$\partial_th + \Delta^2h = \Delta(h^3 - h) + \xi$$

In this equation, $h$ denotes the concentration difference of the binary mixture at a given time $t$. The biharmonic operator, $\Delta^2$, represents the stabilizing force of surface tension that penalizes sharp gradients and the creation of interfaces between the separating phases. The non-linear term, $\Delta(h^3 - h)$, acts as the chemical potential derived from a double-well free energy landscape, which drives the system to separate into two distinct, stable pure states. Finally, $\xi$ represents the space-time white noise, mathematically formalizing the continuous, highly irregular thermal fluctuations introduced by Cook.

The transition from a deterministic partial differential equation to a stochastic one introduces profound mathematical complexities. The primary analytical obstacle in the study of the stochastic Cahn-Hilliard equation is its severe irregularity. Because the space-time white noise $\xi$ is a distribution rather than a classical function, it forces the solution $h$ into spaces of very low regularity. In standard functional analysis, multiplying distributions is an inherently ill-defined operation, consequently, the non-linear cubic term $h^3$ cannot be evaluated using classical analytic methods. This mathematical barrier is not a localized anomaly but rather mirrors the ultraviolet divergences frequently encountered in quantum field theory. In both physical domains, the models are rendered highly singular, mandating the use of advanced renormalization techniques to interpret the non-linear interactions and construct physically meaningful solutions.

The severity of these singularities, and the corresponding difficulty of mathematically defining the system, is fundamentally dependent on the spatial dimension, $d$. The behavior of the equation at varying spatial dimensions can be analyzed using a rigorous scaling argument. By scaling the equation to examine its dynamics at infinitesimally small scales, the physical model categorizes into three distinct mathematical regimes. If the spatial dimension is $d < 6$, the nonlinear term effectively vanishes at the small scale limit; this classifies the equation as subcritical. At exactly $d = 6$, the model becomes scaling invariant and is considered critical, whereas in dimensions $d > 6$, the small-scale behavior is entirely dominated by the nonlinear term, placing the system in the supercritical regime. Because the stochastic Cahn-Hilliard equation remains subcritical in dimensions strictly less than 6, it is possible to leverage spatial scaling to study the existence and uniqueness of local-in-time solutions up to space dimension 5.

Establishing local-in-time well-posedness is only the first step toward a
global theory. To formulate the global problem, one must first construct a
maximal solution and denote its lifetime by \(\tau_\ast\). Global-in-time
well-posedness then requires proving that
\[
\mathbb{P}(\tau_\ast=\infty)=1.
\]

In the singular regime considered here, a finite explosion time should not
necessarily be interpreted as the pointwise divergence of the concentration
field, since the solution may be distribution-valued. Rather, explosion means
that the solution can no longer be continued in the function or distribution
space in which the local theory is constructed, typically because a norm
governing the continuation criterion becomes unbounded. Accordingly, a proof
of global existence requires a priori estimates that keep this norm finite on
every bounded time interval.

For the equation considered in this thesis, one possible approach is to
exploit the dissipative structure of the Cahn--Hilliard drift and derive
suitable energy estimates for the regular components of the solution. Such
estimates would have to control the nonlinear interactions with the
renormalized stochastic objects and prevent the norm appearing in the
continuation criterion from becoming unbounded in finite time. Establishing
these bounds would rule out finite-time explosion and extend the local
solution globally.

\chapter{Stochastic Cahn-Hilliard equation}
\author{Joe Ghafari}

\section{Introduction}
\subsection{Stochastic Cahn-Hilliard equation}

In recent years, singular stochastic partial differential equations have been studied intensively. Da Prato and Debussche \cite{MR2016604} first proved the existence of a strong solution to the $\Phi^4_2$ equation. Hairer \cite{MR3071506} proved local-in-time well-posedness of the Kardar-Parisi-Zhang equation using rough path theory. Hairer \cite{MR3274562} employed his new theory of regularity structures to study the $\Phi^4_3$ equation. Catellier and Chouk \cite{MR3846835} proved the existence and uniqueness of a local-in-time solution to the $\Phi^4_3$ equation using paracontrolled distributions, recently developed by Gubinelli, Imkeller, and Perkowski \cite{MR3406823}. Gubinelli and Perkowski \cite{MR3592748} studied the Kardar-Parisi-Zhang equation using paracontrolled calculus.

Regularity structures and paracontrolled calculus were developed from ideas in rough paths theory and provide a pathwise description of solutions to singular stochastic partial differential equations. Compared to regularity structures, paracontrolled calculus relies on Littlewood-Paley decomposition, Besov spaces, and paraproducts.

In this thesis, we are interested in the stochastic Cahn-Hilliard equation $$\partial_th+\Delta^2h=\Delta\left(h^3-h\right)+\xi.$$ It is possible to study the existence and uniqueness of local-in-time solutions to the stochastic Cahn-Hilliard equation up to space dimension 5. Heuristically, this can be seen using a scaling argument: if $d \in \mathbb{N}^*$ is the space dimension, $u:\mathbb{R_+}\times\mathbb{R}^d$ such that $$\partial_ru+\Delta^2u=\Delta(u^3-u)+\xi,$$ $f_{\lambda}(r,x):=\lambda^{\frac{d}{2}-2}u(\lambda^4 r,\lambda x)$ and $\xi_\lambda(r,x)=\lambda^{\frac{d}{2}+2}\xi(\lambda^4r,\lambda x),$ then $\xi_\lambda$ is again a space-time white noise and $f_\lambda$ verifies the equation: $$\partial_rf_\lambda+\Delta^2f_\lambda=\lambda^{6-d}\Delta f_\lambda^3-\lambda^2\Delta f_\lambda+\xi_\lambda,$$ and since we are interested in small scales we take $\lambda\to 0.$ Therefore we have three different cases:
\begin{itemize}
\item if $d<6$ the nonlinear term vanishes when $\lambda\to 0,$ and the equation is called subcritical (as defined by Hairer \cite{MR3274562})
\item if $d=6$ the equation is scaling invariant and called critical. 
\item if $d>6$ the small scale is dominated by the nonlinear term and the equation is called supercritical.
\end{itemize}
Hence the equation is locally subcritical in dimensions $d\leq 5$, which suggests that these dimensions can be treated using singular SPDE techniques.

The $4$-dimensional stochastic Cahn-Hilliard equation involves a very irregular space-time white noise preventing the solution from being a smooth function and therefore the nonlinear term appearing in the equation is a priori ill-defined and can only be made well-defined with the help of a trick developed by Da Prato and Debussche \cite{MR1941997,MR2016604}. This trick identifies the most singular part of the solution to the $4$-dimensional stochastic Cahn-Hilliard equation and reinterprets the singular nonlinear term by renormalization techniques. 

The Da Prato and Debussche trick isn't enough to treat $d=5,$ since the noise is highly irregular, requiring commutator estimates to define the solution.

The purpose of this chapter is to prove the existence and uniqueness of a local-in-time solution to the stochastic Cahn-Hilliard equation in spatial dimensions $1,2,3,4$ and $5$ using paracontrolled calculus. To the best of our knowledge, Theorem \ref{theoremd=4} and Theorem \ref{theoremd=5} are the first results
in the literature providing details on the existence and uniqueness of a local-in-time solution to stochastic Cahn-Hilliard equation for $d=4,5$ using paracontrolled calculus. We note that regularity structures could be applied to obtain similar results, there now exists a black-box theory providing a local-in-time renormalization and existence-uniqueness theory for large classes of subcritical equations (See \cite{MR3274562,MR3935036,MR4210726,chandra2016analytic}).

\subsection{Procedure}
Let us describe the procedure for treating the equation in the full subcritical regime. The presentation is inspired by \cite{MR3719541}.

The treatment for $d\leq 5$ would be as follows:
\begin{itemize}
\item for $1 \leq d\leq 3,$ the equation is deterministic and has been treated before by Da Prato and Debussche \cite{MR1359472}. For the convenience of the reader, we have included the study of the local solution for this
case using a fixed point theorem in a suitable Besov space.

\item for $d=4$ we treat this equation using the Da Prato-Debussche trick, which was first applied successfully in \cite{MR2016604} to treat the 2-dimensional stochastic quantization equation.
\item for $d=5$ we will use paracontrolled distributions. 
\end{itemize}

$1\leq d\leq 3$ could be treated using classical methods, there is nothing new to add here. Let us explain $d \in \{4,5\}$ 

The main problem in finding a solution for our equation in $d=4,5$ is that the solution will be a distribution, therefore we are not sure how to define the non-linear term $f^3,$ given that the product of two distributions is not necessarily a distribution. To solve this problem, we will introduce the stochastic convolution (see Theorem \ref{theorem1}): $$(\partial_r+\Delta^2)X_1=\xi$$ and we denote by $X_{\epsilon,1}$ a mollified version of $X_1.$ We can prove, using probabilistic arguments, that there exists $c_{\epsilon,1}$ such that $X_{\epsilon,1},X_{\epsilon,2}:=X_{\epsilon,1}^2-c_{\epsilon,1}$ and $X_{\epsilon,3}:=X_{\epsilon,1}^3-3c_{\epsilon,1}X_{\epsilon,1}$ converge in probability, when $\epsilon \to0,$ to $X_1,X_2,X_3,$ respectively (see Theorem \ref{theorem1}, Theorem \ref{theorem2}, Theorem \ref{theorem3}).

We will also rely on the following rules:
\begin{itemize}
\item Every term has regularity. This regularity $\gamma$ is the index of the Besov space $\mathscr{C}^\gamma$ (see Definition \ref{b}) in which this term lives. For example, $X_1$ has a regularity $\left(2-\frac{d}{2}\right)^-$, which means that its regularity is $2-\frac{d}{2}-\epsilon$ for arbitrary $\epsilon>0.$ The regularity of $\xi$ is $\left(-2-\frac{d}{2}\right)^-.$
\item If $\alpha_1<0<\alpha_2$ and $u_1$ and $u_2$ have regularities $\alpha_1$ and $\alpha_2,$ respectively such that $\alpha_1+\alpha_2>0$, then the product $u_1u_2$ is well-defined and has regularity $\alpha_1$ (see Corollary \ref{product}) 
\item Convolution with the biharmonic heat kernel $\partial_r+\Delta^2$ increases the regularity by 4 (see Lemma \ref{lemma})
\item Certain products of stochastic distributions that are not classically well-defined can nevertheless be constructed as renormalized random distribution: the product of stochastic terms of regularity $\alpha_1$ and $\alpha_2$ has regularity $\min(\alpha_1,\alpha_2,\alpha_1+\alpha_2)$ (see Section \ref{section3}).
\item If $u$ has regularity $\theta,$ then $\Delta u$ has regularity $\theta-2$ (see Theorem \ref{derivative}).
\end{itemize}
\subsubsection{$d=4$}
We will briefly explain here the method for $d=4$ which goes back to Da Prato and Debussche \cite{MR1359472}. In this case, $X_1$ has regularity $0^-,$ therefore we can't define the product $f^3$ ($f$ has regularity $0^-$), but as indicated before, we can define $X_2$ and $X_3$ and they have regularity $0^-.$ The idea is to remove the term with worst regularity, more precisely, if we write $f=u+X_1,$ then $u$ solves: 
\begin{equation}
\label{eq2}
(\partial_r+\Delta^2)u=\Delta(u^3+3X_1u^2+3X_2u+X_3-u-X_1).
\end{equation}
In this case, $u$ is expected to have regularity $2^-.$ Hence, all the products in equation ($\ref{eq2}$) are well-defined. We can prove that equation (\ref{eq2}) has a local mild solution and the solution of equation (\ref{eq2}) is defined by $f:=u+X_1$.
\subsubsection{$d=5$}
\label{sec}
In this case, $X_1,X_2,$ and $X_3$ have regularities $(-\frac{1}{2})^-,(-1)^{-},$ and $(-\frac{3}{2})^-,$ respectively. We expect $u$ to have regularity $(\frac{1}{2})^-,$ which is not sufficient to define the products in equation (\ref{eq2}). Let $Y_{2}$ be such that $(\partial_r+\Delta^2)Y_2=\Delta X_3$ and we decompose: $u=h+Y_2.$ $h$ solves $$(\partial_r+\Delta^2)h=\Delta(h^3+3X_2(h+Y_2)+Q(h)),$$ where $$Q(h):=\Theta_1+\Theta_2h+\Theta_3h^2,$$ $$\Theta_1:=Y_2^3+3X_1Y^2_2-X_1-Y_2,$$ $$\Theta_2:=3Y_2^2+6X_1Y_2-1,$$ $$\Theta_3:=3X_1+3Y_2.$$
$\Theta_1,\Theta_2,\Theta_3$ formally have expected regularity $(-\frac{1}{2})^-$ and $h$ is expected to have regularity $1^-,$ therefore the product $h X_2$ is ill-defined.

To solve this problem, we decompose $(h+Y_2)X_2$ into paraproducts (see Corollary \ref{product}):
$$(h+Y_2)X_2=(h+Y_2)\varolessthan X_2+(h+Y_2)\varodot X_2+(h+Y_2)\varogreaterthan X_2$$ and decompose $h=v+w$ solving 
\begin{equation}
\label{eq3}
(\partial_r+\Delta^2)v=3\Delta((v+w+Y_2)\varolessthan X_2),
\end{equation}
\begin{equation}
\label{eq4}
(\partial_r+\Delta^2)w=\Delta((v+w)^3+3(v+w+Y_2)\varodot X_2 +3(v+w+Y_2)\varogreaterthan X_2+Q(v+w))
\end{equation}
In this case, we expect $v$ to have regularity $1^-$ and $w$ to have regularity $(\frac{3}{2})^-.$ All terms appearing in equations (\ref{eq3}) and (\ref{eq4}) are well-defined (see Theorem \ref{estimate}), except for the resonant product $(v+w+Y_2)\varodot X_2$ (the sum of regularities should be strictly positive). Since $w$ has regularity $(\frac{3}{2})^-,$ it follows that $w\varodot X_2$ is well-defined. Furthermore, we note that it's possible to define $\Xi_3:=Y_2 \varodot X_2$ (see Theorem \ref{theorem6}).

It remains to make sense of $v \varodot X_2.$ For this, we introduce $Y_1$ such that $$(\partial_r+\Delta^2)Y_1=\Delta X_2,$$ that is $$Y_1(r)=\int_0^r\mathrm{e^{-(r-q)\Delta^2}}(\Delta X_2(q))\,\mathrm{d}q$$

We write equation (\ref{eq3}) in the mild form (with initial condition $v_0=0$):

$$v(r)=3\int_0^r\mathrm{e}^{-(r-q)\Delta^2}(\Delta((v(q)+w(q)+Y_2(q))\varolessthan X_2(q)))\,\mathrm{d}q$$ and introduce $$\Lambda(v,w)(r):=3\int_0^r\mathrm{e}^{-(r-q)\Delta^2}(\Delta((v(q)+w(q)+Y_2(q))\varolessthan X_2(q)))\,\mathrm{d}q-3(v(r)+w(r)+Y_2(r))\varolessthan Y_1(r),$$ which has better regularity than $v$ (this follows from our commutator estimates, see Proposition \ref{co2}).
Therefore $$v \varodot X_2 =3((v+w+Y_2)\varolessthan Y_1)\varodot X_2+\Lambda(v,w)\varodot X_2.$$
We note that $\Lambda(v,w)\varodot X_2$ is well-defined. It remains to make sense of $((v+w+Y_2)\varolessthan Y_1)\varodot X_2.$ For this we consider the operator: 
$$\Psi(g_1,g_2,g_3)=(g_1 \varolessthan g_2)\varodot g_3-g_1(g_2\varodot g_3).$$ Again using our commutator estimates (see Proposition \ref{co1}), $\Psi$ is well-defined, hence $\Psi(v+w+Y_2,Y_1,X_2)$ makes sense. 

Supposing that we can define $\Xi_2:=Y_1\varodot X_2$ we define: 
$$v\varodot X_2:=3\Psi(v+w+Y_2,Y_1,X_2)+3(v+w+Y_2)\Xi_2+\Lambda(v,w)\varodot X_2.$$

Lastly, we are interested in proving the existence and uniqueness of:
\begin{equation}
\label{eq5}
(\partial_r+\Delta^2)v=3\Delta((v+w+Y_2)\varolessthan X_2),
\end{equation}
\begin{equation}
\label{eq6}
(\partial_r+\Delta^2)w=\Delta((v+w)^3+9\Psi(v+w+Y_2,Y_1,X_2)+3\Lambda(v,w)\varodot X_2+3 w\varodot X_2+3(v+w+Y_2)\varogreaterthan X_2+P(v+w)),
\end{equation}
where $$P(v+w):=\Theta_4+\Theta_5(v+w)+\Theta_6(v+w)^2,$$ $$\Theta_4:=\Theta_1+9Y_2\Xi_2+3\Xi_3,$$ $$\Theta_5:=\Theta_2+9\Xi_2,$$ $$\Theta_6:=\Theta_3.$$

We need to make sense of the terms appearing in $\Theta_1$ and $\Theta_2.$ To do so, we postulate that we can define $\Xi_1:=Y_2\varodot X_1$ (see Theorem \ref{theorem4}). Consequently, we let $$\Theta_4:=Y_2^3+3X_1\varolessthan Y_2^2+3X_1\varodot(Y_2\varodot Y_2)+6Y_2\Xi_1+6\Psi(Y_2,Y_2,X_1)+3X_1\varogreaterthan Y_2^2+9Y_2\Xi_2+3\Xi_3-X_1-Y_2,$$ $$\Theta_5:=3Y_2^2+6Y_2 \varolessthan X_1+6Y_2 \varogreaterthan X_1+6\Xi_1+9\Xi_2-1.$$

The paracontrolled solution of the 5-dimensional stochastic Cahn-Hilliard equation is defined by $f:=X_1+Y_2+v+w,$ where $(v,w)$ solves equations (\ref{eq5}) and (\ref{eq6}).

The rest of this chapter is organized as follows. In Section \ref{section2} we introduce Besov spaces and their properties. In Section \ref{section3} we provide a detailed analysis of the stochastic objects involved in our equation. In Section \ref{sec4} we treat $1 \leq d\leq 3$.  In Section \ref{sec5} we treat $d=4$. Finally,  in Section \ref{sec6} we treat $d=5$. 
\section{Besov spaces and paraproducts}
\label{section2}
In this section we recall the definitions and the basic results of paracontrolled calculus (see \cite{MR2768550,MR3406823,MR3445609} for the proofs).

Unless otherwise stated, throughout we fix $d\in\mathbb{N}^*$ and we denote by $\mathbb{T}^d$ the $d$-dimensional torus.
\subsection{Tempered distributions}
\begin{definition}
Let $\left(q,m\right) \in \left[1,+\infty\right[\times \mathbb{N}^*.$
\begin{enumerate}
\item We denote by $L^2\left(\left(\mathbb{R}_+\times \mathbb{T}^d\right)^m\right)$ the set of functions $g:\left(\mathbb{R}_+\times \mathbb{T}^d\right)^m\longrightarrow\mathbb{R}$ such that $g$ is a $\left(\mathcal{B}\left(\left(\mathbb{R}_+\times \mathbb{T}^d\right)^m\right),\mathcal{B}\left(\mathbb{R}\right)\right)$-measurable function and $$\int_{\left(\mathbb{R}_+\times \mathbb{T}^d\right)^m}\left|g\left(x\right)\right|^2\,\mathrm{d}x<+\infty.$$
\item We denote by $L^q\left(\mathbb{T}^d\right)$ the set of functions $g:\mathbb{T}^d\longrightarrow\mathbb{R}$ such that $g$ is a $\left(\mathcal{B}\left(\mathbb{T}^d\right),\mathcal{B}\left(\mathbb{R}\right)\right)$-measurable function and $$\int_{\mathbb{T}^d}\left|g\left(x\right)\right|^q\,\mathrm{d}x<+\infty.$$
We also define
\begin{align*} 
\left\Vert\cdot\right\Vert_{L^q\left(\mathbb{T}^d\right)}: L^q\left(\mathbb{T}^d\right) &\longrightarrow \mathbb{R}_+ \\ h & \longmapsto \left\Vert h\right\Vert_{L^q\left(\mathbb{T}^d\right)}=\left(\int_{\mathbb{T}^d}\left|h\left(x\right)\right|^q\,\mathrm{d}x\right)^{\frac{1}{q}}
\end{align*}
\item We denote by $L^\infty\left(\mathbb{T}^d\right)$ the set of functions $g:\mathbb{T}^d\longrightarrow\mathbb{R}$ such that $g$ is a $\left(\mathcal{B}\left(\mathbb{T}^d\right),\mathcal{B}\left(\mathbb{R}\right)\right)$-measurable function and there exists $\mu \in \mathbb{R}_+$ such that $\left|g\left(x\right)\right|\leq \mu$ $\mathrm{d}x$-almost everywhere on $\mathbb{T}^d.$
\\
We also define
\begin{align*} 
\left\Vert\cdot\right\Vert_{L^\infty\left(\mathbb{T}^d\right)}: L^\infty\left(\mathbb{T}^d\right) &\longrightarrow \mathbb{R}_+ \\ h & \longmapsto \left\Vert h\right\Vert_{L^\infty\left(\mathbb{T}^d\right)}=\mathrm{ess\,sup}_{x \in \mathbb{T}^d}|h\left(x\right)|
\end{align*}
\end{enumerate}
\end{definition}
\begin{definition}
We write $C^{\infty}\left(\mathbb{T}^d,\mathbb{R}\right)$ for the set of functions $h:\mathbb{T}^d\longrightarrow\mathbb{R}$ such that $h$ is infinitely differentiable on $\mathbb{T}^d$ and we denote by $C_c^{\infty}\left(\mathbb{R}^d,\mathbb{R}\right)$ the set of functions $g:\mathbb{R}^d\longrightarrow\mathbb{R}$ such that $g$ is infinitely differentiable on $\mathbb{R}^d$ and $\overline{\left\{x \in \mathbb{R}^d,g\left(x\right)\neq 0\right\}}$ is a compact set in $\mathbb{R}^d.$
\end{definition}
For every $q=(q_1,...,q_d) \in \mathbb{N}^d,$ we define $|q|:=\sum_{r=1}^dq_r.$

\begin{definition}
For all $j \in \mathbb{N},$ let
\begin{align*} 
\left\Vert\cdot\right\Vert_{j,\mathscr{S}}: C^{\infty}\left(\mathbb{T}^d,\mathbb{R}\right)&\longrightarrow \mathbb{R}_+ \\ \varphi & \longmapsto \left\Vert \varphi\right\Vert_{j,\mathscr{S}}=\max_{\substack{q \in \mathbb{N}^d\\ \left|q\right|\leq j}}\sup_{x \in \mathbb{T}^d}\left|\partial^q\varphi\left(x\right)\right|
\end{align*}
\begin{enumerate}
\item The Schwartz space is the vector space \gls{pc2}$:=C^\infty\left(\mathbb{T}^d,\mathbb{R}\right)$ equipped with the topology generated by the family of semi-norms $\left(\left\Vert\cdot\right\Vert_{n,\mathscr{S}}\right)_{n \in \mathbb{N}}.$
\item The space of tempered distributions is the vector space
\\
\gls{pc1}$:=\left\{f \in \mathbb{R}^{\mathscr{S}},\exists (\mu,r) \in \mathbb{R}_+\times \mathbb{N},\forall \varphi \in \mathscr{S},\left|f\left(\varphi\right)\right|\leq \mu\left\Vert\varphi\right\Vert_{r,\mathscr{S}}\right\}$
\\
$\cap\left\{f \in \mathbb{R}^{\mathscr{S}},\forall \left(\alpha,\varphi_1,\varphi_2\right) \in \mathbb{R}\times\mathscr{S}^2,f\left(\alpha\varphi_1+\varphi_2\right)=\alpha f\left(\varphi_1\right)+f\left(\varphi_2\right)\right\}$ equipped with the weak$^*$ topology $\sigma\left(\mathscr{S}',\mathscr{S}\right).$
\item Let $\left(f_n\right)_{n \in \mathbb{N}}$ be a sequence in $\mathscr{S}'.$ We say that $\left(f_n\right)_{n \in \mathbb{N}}$ converges in $\mathscr{S}'$ if and only if there exists $f \in \mathscr{S}'$ such that for all $\varphi \in \mathscr{S},\lim_{m\to+\infty}f_m\left(\varphi\right)=f\left(\varphi\right).$
\end{enumerate}
\end{definition}
For all $q \in \mathbb{Z}^d,$ let
\begin{align*}
\psi_{q,1}:\mathbb{T}^d &\longrightarrow \mathbb{R}\\
x&\longmapsto \psi_{q,1}\left(x\right)=\cos\left(2\pi\left\langle x,q\right\rangle\right)
\end{align*}
and 
\begin{align*}
\psi_{q,2}:\mathbb{T}^d&\longrightarrow\mathbb{R}\\
x&\longmapsto\psi_{q,2}\left(x\right)=\sin\left(2\pi \left\langle x,q\right\rangle\right)
\end{align*}
\begin{definition}
Let $f \in \mathscr{S}'$ and $h:\mathbb{Z}^d\longrightarrow\mathbb{C}$ be a function such that $$\exists \left(\mu,r\right) \in \mathbb{R}_+\times \mathbb{N},\forall x \in \mathbb{Z}^d,\left|h\left(x\right)\right|\leq \mu\left(1+\left|x\right|\right)^r.$$
\begin{enumerate}
\item The Fourier transform of $f$ is given by
\begin{align*}
\mathscr{F}f: \mathbb{Z}^d &\longrightarrow \mathbb{C} \\ 
m & \longmapsto \left(\mathscr{F}f\right)\left(m\right)=f\left(\psi_{m,1}\right)-\mathrm{i}f\left(\psi_{m,2}\right)
\end{align*}
\item The inverse Fourier transform of $h$ is given by
\begin{align*} 
\mathscr{F}^{-1}h: \mathscr{S} &\longrightarrow \mathbb{C} \\ \varphi & \longmapsto \left(\mathscr{F}^{-1}h\right)\left(\varphi\right)=\sum_{n \in \mathbb{Z}^d}h\left(n\right)\int_{\mathbb{T}^d}\varphi\left(x\right)\mathrm{e}^{2\pi \mathrm{i}\left\langle x,n\right\rangle}\,\mathrm{d}x
\end{align*}
\end{enumerate}
\end{definition}
\subsection{Besov spaces}
\begin{proposition-definition}[Proposition 2.10, \cite{MR2768550}]
There exist two non-negative radial functions $\rho_{-1}$ and $\rho_0$ such that $\left(\rho_{-1},\rho_0\right) \in \left(C_c^\infty\left(\mathbb{R}^d,\mathbb{R}\right)\right)^2,\overline{\left\{y \in \mathbb{R}^d,\rho_{-1}\left(y\right)\neq 0\right\}}\subset\left\{y \in \mathbb{R}^d,\left|y\right|\leq \frac{4}{3}\right\},\overline{\left\{y \in \mathbb{R}^d,\rho_0\left(y\right)\neq 0\right\}}\subset\left\{y \in \mathbb{R}^d,\frac{3}{4}\leq \left|y\right|\leq \frac{8}{3}\right\},$ and $$\forall x \in \mathbb{R}^d,\rho_{-1}\left(x\right)+\sum_{q \in \mathbb{N}}\rho_{0}\left(2^{-q}x\right)=1.$$ $\left(\rho_{-1},\rho_0\right)$ is called a dyadic partition of unity.
\end{proposition-definition}
For every $q \in \mathbb{N}^*,$ we define
\begin{align*} 
\rho_{q}: \mathbb{R}^d &\longrightarrow \mathbb{R} \\ x & \longmapsto \rho_{q}\left(x\right)=\rho_0\left(2^{-q}x\right)
\end{align*}

For all $j \in \mathbb{N}\cup\left\{-1\right\},$ let
\begin{align*} 
K_{j}: \mathbb{T}^d &\longrightarrow \mathbb{R} \\ 
x & \longmapsto K_{j}\left(x\right)=\sum_{m \in \mathbb{Z}^d}\rho_j\left(m\right)\cos\left(2\pi \left\langle x,m\right\rangle\right)
\end{align*}
\begin{definition}
\label{volu}
Let $f \in \mathscr{S}',\phi \in\mathscr{S}.$ We define the convolution of $f$ with $\phi$ to be the function $\mathbb{T}^d\longrightarrow\mathbb{R}$ defined by: $$(\phi*f)(x)=f(\phi(\cdot-x)),x \in \mathbb{T}^d.$$
\end{definition}
\begin{definition}
Let $f \in \mathscr{S}',\phi \in\mathscr{S}.$ We define the convolution of $f$ with $\phi$ to be the distribution $\mathscr{S}\longrightarrow\mathbb{R}$ defined by: $$(\phi*f)(\varphi)=f(\phi*\varphi),\varphi \in \mathscr{S},$$ where $\phi*\varphi$ is the usual convolution of two functions.
\end{definition}
\begin{remark}
We note that for $\phi \in \mathscr{S},f \in \mathscr{S}',\phi*f \in \mathscr{S}\cap\mathscr{S}',$ depending on the context. The two definition are related, we have: $$(\phi*f)(\varphi)=\int_{\mathbb{T}^d}\varphi(x)(\phi*f)(x)\,\mathrm{d}x.$$
We refer to sections 34 and 35 of \cite{lecture} for more details on convolution of distributions on the torus.
\end{remark} 
\begin{definition}
\label{def}
For every $\left(f,q\right) \in \mathscr{S}'\times \left(\mathbb{N}\cup\left\{-1\right\}\right),$ we define the Littlewood-Paley blocks of $f$ by \gls{f1}$:=K_q*f.$
\end{definition}
\begin{definition}
\label{b}
Let $\left(\alpha,\gamma,r\right)\in \mathbb{R}\times \left[1,+\infty\right[^2,$
\begin{align*} 
\left\Vert\cdot\right\Vert_{\mathscr{C}^\alpha}: \mathscr{S}' &\longrightarrow \overline{\mathbb{R}}_+ \\ f & \longmapsto \left\Vert f\right\Vert_{\mathscr{C}^\alpha}=\sup_{q \in \mathbb{N}}2^{\alpha\left(q-1\right)}\Vert \delta_{q-1}f\Vert_{L^{\infty}\left(\mathbb{T}^d\right)}
\end{align*}
and
\begin{align*} 
\left\Vert\cdot\right\Vert_{B^{\alpha}_{r,\gamma}}: \mathscr{S}' &\longrightarrow \overline{\mathbb{R}}_+ \\ f & \longmapsto \left\Vert f\right\Vert_{B^{\alpha}_{r,\gamma}}=\left(\sum_{q \in \mathbb{N}}2^{\alpha\gamma\left(q-1\right)}\left\Vert \delta_{q-1} f\right\Vert_{L^r\left(\mathbb{T}^d\right)}^\gamma\right)^{\frac{1}{\gamma}}
\end{align*}
We define $$\text{\gls{pc3}}:=\left\{f \in \mathscr{S}',\left\Vert f\right\Vert_{B_{r,\gamma}^\alpha}<+\infty\right\}$$ and $$\text{\gls{pc4}}:=\left\{f \in \mathscr{S}',\lim_{q\to+\infty}2^{\alpha q}\left\Vert\delta_q f\right\Vert_{L^{\infty}\left(\mathbb{T}^d\right)}=0\right\}.$$
The Banach spaces $\left(B^\alpha_{r,\gamma},\left\Vert\cdot\right\Vert_{B^\alpha_{r,\gamma}}\right)$ and $\left(\mathscr{C}^\alpha,\left\Vert\cdot\right\Vert_{\mathscr{C}^\alpha}\right)$ are called Besov spaces.
\end{definition}
\begin{definition}
Let $\left(\alpha,U\right) \in \mathbb{R}\times \mathbb{R}_+.$ We denote by \gls{pc5} the set of functions $h:\left[0,U\right]\longrightarrow\mathscr{C}^\alpha$ such that $h$ is continuous on $\left[0,U\right].$
We also define
\begin{align*} 
\left\Vert\cdot\right\Vert_{C_U\mathscr{C}^\alpha}: C_U\mathscr{C}^\alpha &\longrightarrow \mathbb{R}_+ \\ f & \longmapsto \left\Vert f\right\Vert_{C_U\mathscr{C}^\alpha}=\sup_{r \in \left[0,U\right]}\left\Vert f\left(r\right)\right\Vert_{\mathscr{C}^\alpha}
\end{align*}
\end{definition}
We note that for every $\left(\alpha,U\right) \in \mathbb{R}\times \mathbb{R}_+,\left(C_U\mathscr{C}^\alpha,\left\Vert\cdot\right\Vert_{C_U\mathscr{C}^\alpha}\right)$ is a Banach space.

The following Bernstein inequalities are very useful when dealing with functions with compactly supported Fourier transform.
\begin{theorem}[Bernstein inequalities, Lemma 2.1, \cite{MR2768550}]
\label{Bernstein}
Let $\left(n,\gamma_1,\gamma_2,\gamma_3\right) \in \mathbb{N}\times\mathbb{R}_+^*\times\mathbb{R}_+^*\times\mathbb{R}_+^*$ such that $\gamma_2<\gamma_3.$
\begin{enumerate}
\item There exists $\mu \in \mathbb{R}_+^*$ such that for any $\left(r,\beta\right) \in \left[1,+\infty\right[^2$ and all $\varphi \in \mathscr{S},$ if \\ $\left\{x \in \mathbb{Z}^d,\int_{\mathbb{T}^d}\varphi\left(y\right)\mathrm{e}^{-2\pi \mathrm{i}\left\langle x,y\right\rangle}\,\mathrm{d}y\neq 0\right\}\subset \left\{x \in \mathbb{R}^d,\left|x\right|\leq \beta\gamma_1\right\},$ then $$\max_{\substack{q\in\mathbb{N}^d \\ \left|q\right|=n}}\left\Vert \partial^q \varphi\right\Vert_{L^\infty\left(\mathbb{T}^d\right)}\leq \mu\beta^{n+\frac{d}{r}}\left\Vert \varphi\right\Vert_{L^r\left(\mathbb{T}^d\right)}.$$
\item There exists $\left(\mu_1,\mu_2\right) \in \left(\mathbb{R}_+^*\right)^2$ such that for any $\beta \in \left[1,+\infty\right[$ and all $\varphi \in \mathscr{S},$ if \\ $\left\{x \in \mathbb{Z}^d,\int_{\mathbb{T}^d}\varphi\left(y\right)\mathrm{e}^{-2\pi \mathrm{i}\left\langle x,y\right\rangle}\,\mathrm{d}y\neq 0\right\}\subset\left\{x \in \mathbb{R}^d,\beta\gamma_2\leq \left|x\right|\leq\beta\gamma_3\right\},$ then $$\mu_1\beta^{n}\left\Vert \varphi\right\Vert_{L^\infty\left(\mathbb{T}^d\right)}\leq \max_{\substack{q\in\mathbb{N}^d\\\left|q\right|=n}}\left\Vert \partial^{q}\varphi\right\Vert_{L^\infty\left(\mathbb{T}^d\right)}\leq \mu_2\beta^{n}\left\Vert \varphi\right\Vert_{L^{\infty}\left(\mathbb{T}^d\right)}.$$
\end{enumerate}
\end{theorem}
The following two theorems are simple applications of Theorem \ref{Bernstein}.
\begin{theorem}[Besov embedding, Theorem 36.18, \cite{lecture}]
\label{embedding}
There exists $\mu \in \mathbb{R}_+^*$ such that $$\forall\left(f,\alpha,q\right) \in \mathscr{S}'\times \mathbb{R}\times \left[1,+\infty\right[,\left\Vert f\right\Vert_{\mathscr{C}^{\alpha-\frac{d}{q}}}\leq \mu\left\Vert f\right\Vert_{B^\alpha_{q,q}}.$$
\end{theorem}
\begin{theorem}[Theorem 21.23, \cite{lecture}]
\label{derivative}
For every $q \in \mathbb{N}^d,$ there exists $\mu \in \mathbb{R}_+^*$ such that $$\forall \left(f,\alpha\right) \in \mathscr{S}'\times \mathbb{R},\left\Vert\partial^q f\right\Vert_{\mathscr{C}^{\alpha-\left|q\right|}}\leq \mu\left\Vert f\right\Vert_{\mathscr{C}^\alpha}.$$
\end{theorem}
\subsection{Paraproducts and resonant products}
The main difficulty in solving singular stochastic partial differential equations is to multiply distributions. Paraproducts are bilinear operations useful to decompose the multiplication into simpler problems. 
\begin{definition}
Let $\left(f_1,f_2\right) \in \left(\mathscr{S}'\right)^2.$
\begin{enumerate}
\item For all $m \in \mathbb{N}^*,$ let
\begin{align*}
g_m:\mathscr{S}&\longrightarrow\mathbb{R}\\
\varphi&\longmapsto g_m\left(\varphi\right)=\sum_{n=1}^{m}\int_{\mathbb{T}^d}\varphi\left(x\right)\left(\delta_{n-2}f_1\right)\left(x\right)\left(\delta_{m}f_2\right)\left(x\right)\,\mathrm{d}x
\end{align*}
If $\left(\sum_{j=1}^qg_j\right)_{q \in \mathbb{N}^*}$ converges in $\mathscr{S}',$ then the limit is denoted by \gls{f3}$=$\gls{f4}$=\sum_{q \in \mathbb{N}^*}g_q$ and it is called the paraproduct of $f_2$ by $f_1.$
\item For every $m \in \mathbb{N},$ let
\begin{align*}
h_m:\mathscr{S}&\longrightarrow\mathbb{R}\\
\varphi&\longmapsto h_m\left(\varphi\right)=\sum_{n=0}^{\min\left(m+1,2\right)}\int_{\mathbb{T}^d}\varphi\left(x\right)\left(\delta_{n+\max\left(m-2,-1\right)}f_1\right)\left(x\right)\left(\delta_{m-1}f_2\right)\left(x\right)\,\mathrm{d}x
\end{align*}
If $\left(\sum_{j=0}^qh_j\right)_{q \in \mathbb{N}}$ converges in $\mathscr{S}',$ then the limit is denoted by \gls{f2}$=\sum_{q \in \mathbb{N}}h_q$ and it is called the resonant product of $f_1$ and $f_2.$
\item For all $\left(m_1,m_2\right) \in \mathbb{N}^2,$ let
\begin{align*}
\chi_{m_1,m_2}:\mathscr{S}&\longrightarrow\mathbb{R}\\
\varphi&\longmapsto \chi_{m_1,m_2}\left(\varphi\right)=\int_{\mathbb{T}^d}\varphi\left(x\right)\left(\delta_{m_1-1}f_1\right)\left(x\right)\left(\delta_{m_2-1}f_2\right)\left(x\right)\,\mathrm{d}x
\end{align*}
If $\left(\sum_{j_1=0}^q\sum_{j_2=0}^q\chi_{j_1,j_2}\right)_{q \in \mathbb{N}}$ converges in $\mathscr{S}',$ then the limit is denoted by $f_1f_2=\sum_{\left(q_1,q_2\right) \in \mathbb{N}^2}\chi_{q_1,q_2}$ and it is called the product of $f_1$ and $f_2.$
\end{enumerate}
\end{definition}
Bony \cite{MR0631751} noticed that paraproducts are always well-defined distributions. The only problem in constructing the product of distributions is the resonant term. The following estimates give the basic result about these bilinear operations.
\begin{theorem}[Lemma 2.1, \cite{MR3406823}]
\label{estimate}
\begin{enumerate}
\item For all $(\alpha,\beta) \in \mathbb{R}^*\times\mathscr{R},$ there exists $\mu \in \mathbb{R}_+^*$ such that for every $\left(f,h\right) \in \mathscr{C}^\alpha\times \mathscr{C}^\beta,f\varolessthan h \in \mathscr{C}^{\min\left(\alpha,0\right)+\beta}$ and $$\left\Vert f\varolessthan h\right\Vert_{\mathscr{C}^{\min\left(\alpha,0\right)+\beta}}\leq \mu\left\Vert f\right\Vert_{\mathscr{C}^\alpha}\left\Vert h\right\Vert_{\mathscr{C}^\beta}.$$
\item For every $\left(\alpha,\beta\right) \in \mathbb{R}^2,$ if $\alpha+\beta>0,$ then there exists $\mu \in \mathbb{R}_+^*$ such that for all $\left(f,h\right) \in \mathscr{C}^\alpha\times \mathscr{C}^\beta,f\varodot h \in \mathscr{C}^{\alpha+\beta}$ and $$\left\Vert f\varodot h\right\Vert_{\mathscr{C}^{\alpha+\beta}}\leq \mu\left\Vert f\right\Vert_{\mathscr{C}^\alpha}\left\Vert h\right\Vert_{\mathscr{C}^\beta}.$$
\end{enumerate}
\end{theorem}
We deduce from Theorem \ref{estimate} the following simple corollary.
\begin{corollary}[Theorem 27.11, \cite{lecture}]
\label{product}
For all $\left(\alpha,\beta\right) \in \left(\mathbb{R}^*\right)^2,$ if $\alpha+\beta>0,$ then there exists $\mu \in \mathbb{R}_+^*$ such that for every $\left(f,h\right) \in \mathscr{C}^\alpha\times \mathscr{C}^\beta,fh=f\varolessthan h+f\varodot h+f\varogreaterthan h \in \mathscr{C}^{\min\left(\alpha,\beta\right)}$ and $$\left\Vert fh\right\Vert_{\mathscr{C}^{\min\left(\alpha,\beta\right)}}\leq \mu\left\Vert f\right\Vert_{\mathscr{C}^\alpha}\left\Vert h\right\Vert_{\mathscr{C}^\beta}.$$
\end{corollary}
For every $(\phi,f) \in \mathscr{S}^2,$ let $J_{\phi}(f)=\int_{\mathbb{T}^d}\phi(x)f(x)\,\mathrm{d}x.$ 

The following proposition provides commutator estimates for paraproducts and resonant products. 
\begin{proposition}[Lemma 2.4, \cite{MR3406823}]
\label{co1}
For every $(\alpha,\beta,\gamma) \in ]0,1[\times \mathbb{R}^2,$ if $-\alpha<\beta+\gamma<0,$ then there exists \gls{f5}$:\mathscr{C}^\alpha\times\mathscr{C}^\beta\times\mathscr{C}^\gamma\longrightarrow\mathscr{C}^{\alpha+\beta+\gamma}$ such that $$\exists \mu \in \mathbb{R}_+^*,\forall(f_1,f_2,f_3)\in\mathscr{C}^\alpha\times\mathscr{C}^\beta\times\mathscr{C}^\gamma,\Vert\Psi(f_1,f_2,f_3)\Vert_{\mathscr{C}^{\alpha+\beta+\gamma}}\leq \mu\Vert f_1\Vert_{\mathscr{C}^\alpha}\Vert f_2\Vert_{\mathscr{C}^\beta}\Vert f_3\Vert_{\mathscr{C}^\gamma}$$ and $$\forall(g_1,g_2,g_3)\in\mathscr{S}^3,\Psi(J_{g_1},J_{g_2},J_{g_3})=(J_{g_1}\varolessthan J_{g_2})\varodot J_{g_3}-J_{g_1}(J_{g_2}\varodot J_{g_3}).$$ 
\end{proposition}
\subsection{Schauder estimates}
For all $\left(r,f\right)\in\mathbb{R}_+^*\times \mathscr{S}',$ let $G_0f:=f,$
\begin{align*} 
\text{\gls{f6}}: \mathbb{T}^d &\longrightarrow \mathbb{R} \\ 
x & \longmapsto \eta_{r}\left(x\right)=\sum_{q \in \mathbb{Z}^d}\cos\left(2\pi \left\langle x,q\right\rangle\right)\mathrm{e}^{-r\left|2\pi  q\right|^4}
\end{align*}
and \gls{f7}$:=\eta_r*f.$

We study here the regularizing effect of the semi-group $\left(G_q\right)_{q \in \mathbb{R}_+}.$

We begin with the following fundamental lemma.
\begin{lemma}[Lemma A.5, \cite{MR3406823}]
\label{lemma}
\begin{enumerate}
\item For every $\beta \in \mathbb{R}_+,$ there exists $\mu \in \mathbb{R}_+^*$ such that for all $\alpha \in \mathbb{R},$ $$\forall \left(r,f\right) \in \mathbb{R}_+^*\times \mathscr{C}^\alpha,\left\Vert G_rf\right\Vert_{\mathscr{C}^{\alpha+\beta}}\leq \mu\max\left(r^{-\frac{\beta}{4}},1\right)\left\Vert f\right\Vert_{\mathscr{C}^\alpha}.$$
\item There exists $\mu \in \mathbb{R}_+^*$ such that for any $\gamma \in \mathbb{R}$ and all $\left(U,f\right) \in \mathbb{R}_+\times\mathscr{C}^\gamma,$ the function $\begin{aligned}[t]
\left[0,U\right]&\longrightarrow \mathscr{C}^\gamma\\
m &\longmapsto G_mf
\end{aligned}$ is continuous on $\left[0,U\right]$ and $\sup_{r\in\left[0,U\right]}\left\Vert G_rf\right\Vert_{\mathscr{C}^\gamma}\leq \mu\left\Vert f\right\Vert_{\mathscr{C}^\gamma}.$
\item For every $\beta \in \left[0,4\right],$ there exists $\mu \in \mathbb{R}_+^*$ such that for all $\alpha \in \mathbb{R},$ $$\forall \left(r,f\right) \in \mathbb{R}_+^*\times \mathscr{C}^\alpha,\left\Vert G_rf-f\right\Vert_{\mathscr{C}^{\alpha-\beta}}\leq \mu\min\left(r^{\frac{\beta}{4}},1\right)\left\Vert f\right\Vert_{\mathscr{C}^\alpha}.$$
\end{enumerate}
\end{lemma}
We derive next Schauder estimates for the semi-group $\left(G_r\right)_{r \in \mathbb{R}_+}.$
\begin{theorem}[Schauder estimates]
\label{Schauder}
For every $\beta \in \left[0,4\right[,$ there exists $\mu \in \mathbb{R}_+^*$ such that for all $\left(\alpha,U\right) \in \mathbb{R}\times \mathbb{R}_+,$ $$\forall g \in C_{U}\mathscr{C}^\alpha,\sup_{r \in \left[0,U\right]}\left\Vert \int_0^rG_{r-q}\left(\Delta g\left(q\right)\right)\,\mathrm{d}q\right\Vert_{\mathscr{C}^{\alpha+\beta-2}}\leq \mu\max\left(U,U^{1-\frac{\beta}{4}}\right)\left\Vert g\right\Vert_{C_U\mathscr{C}^\alpha}.$$
\end{theorem}
\begin{proof}
Applying Lemma \ref{lemma} and Theorem \ref{derivative} we deduce that for every $\beta \in \left[0,4\right[,$ such that for all $\left(\alpha,U\right)\in\mathbb{R}\times \mathbb{R}_+,$
\begin{align*}
\forall g \in C_U\mathscr{C}^\alpha,\sup_{r \in \left[0,U\right]}\left\Vert \int_0^rG_{r-q}\left(\Delta g\left(q\right)\right)\,\mathrm{d}q\right\Vert_{\mathscr{C}^{\alpha+\beta-2}}&\leq\sup_{r \in \left[0,U\right]}\int_{0}^r\left\Vert G_{r-q}\left(g\left(q\right)\right)\right\Vert_{\mathscr{C}^{\alpha+\beta}}\,\mathrm{d}q\\
&\lesssim \sup_{r \in \left[0,U\right]}\int_{0}^r\max\left(\left|r-q\right|^{-\frac{\beta}{4}},1\right)\left\Vert g\right\Vert_{C_U\mathscr{C}^\alpha}\,\mathrm{d}q\\
&\lesssim \max\left(U,U^{1-\frac{\beta}{4}}\right)\left\Vert g\right\Vert_{C_U\mathscr{C}^\alpha}.
\end{align*}
\end{proof}
The following proposition provides commutator estimates for paraproducts and $(G_r)_{r \in \mathbb{R}_+}$. 
\begin{proposition}[Lemma A.1, \cite{MR3846835}]
\label{co2}
For every $\left(\alpha,\beta,\gamma\right) \in \mathbb{R}^2\times\mathbb{R}_+,$ there exists $\mu \in \mathbb{R}_+^*$ such that $$\forall \left(r,f,g\right) \in \mathbb{R}_+^*\times \mathscr{C}^\alpha\times\mathscr{C}^\beta,\left\Vert G_r\left(f\varolessthan g\right)-f\varolessthan\left(G_rg\right)\right\Vert_{\mathscr{C}^{\alpha+\beta+\gamma}}\leq \mu r^{-\frac{\gamma}{4}}\left\Vert f\right\Vert_{\mathscr{C}^\alpha}\left\Vert g\right\Vert_{\mathscr{C}^\beta}.$$
\end{proposition}
\section{Construction of the stochastic objects}
\label{section3}
In this section we provide the necessary probabilistic tools to solve the stochastic Cahn-Hilliard equation.

For the rest of this chapter we fix a complete probability space $\left(\Omega,\mathcal{H},\mathds{P}\right)$ and we denote by $L^2\left(\Omega\right)$ the set of random variables $W:\Omega\longrightarrow\mathbb{R}$ such that $\mathds{E}\left[W^2\right]<+\infty.$

We start by recalling the definition of a space-time white noise.
\begin{definition}
A space-time white noise is a family $($\gls{sto1}$_\varphi)_{\varphi \in L^2\left(\mathbb{R}_+\times \mathbb{T}^d\right)}$ of real random variables such that for every $h \in L^2\left(\mathbb{R}_+\times \mathbb{T}^d\right),\xi_h$ is normally distributed, $\mathds{E}\left[\xi_h\right]=0,$ and $\mathds{E}\left[\xi_h^2\right]=\int_{\mathbb{R}_+\times \mathbb{T}^d}\left|h\left(q,x\right)\right|^2\,\mathrm{d}q\,\mathrm{d}x.$
\end{definition}
For all $r \in \mathbb{R}_+,$ let $$\mathcal{A}_r:=\left\{\varphi \in L^2\left(\mathbb{R}_+\times \mathbb{T}^d\right),\forall \left(q,x\right) \in \right]r,+\infty\left[\times \mathbb{T}^d,\varphi\left(q,x\right)=0\right\}$$ and $$\mathcal{H}_r:=\bigcap_{q \in \left]r,+\infty\right[}\sigma\left(\left\{M \in \mathcal{H},\mathds{P}\left(M\right)=0\right\}\cup \sigma\left(\bigcup_{\varphi \in \mathcal{A}_q}\sigma\left(\xi_\varphi\right)\right)\right).$$

Unless otherwise stated, all stochastic processes throughout are defined on the filtered probability space $\left(\Omega,\mathcal{H},\left(\mathcal{H}_r\right)_{r \in \mathbb{R}_+},\mathds{P}\right).$

For all $r \in \mathbb{R}_+,$ we define the stochastic convolution by:
\begin{align*}
\text{\gls{sto2}}:L^2\left(\mathbb{T}^d\right)&\longrightarrow L^2\left(\Omega\right)\\
\varphi&\longmapsto X_{r,1}\left(\varphi\right)=\int_{\left[0,r\right]\times \mathbb{T}^d}\left(\eta_{r-q}*\varphi\right)\left(x\right)\,\xi\left(\mathrm{d}q,\mathrm{d}x\right)
\end{align*}

Let $\zeta \in C_c^\infty(\mathbb{R}^d,\mathbb{R})$ such that $\zeta(0)=1$ and $$\forall x \in \mathbb{R}^d,\zeta(-x)=\zeta(x).$$

For all $\epsilon \in \mathbb{R}_+^*,$ 
\begin{align*}
\zeta_\epsilon:\mathbb{T}^d&\longrightarrow\mathbb{R}\\
x&\longmapsto\zeta_\epsilon\left(x\right)=\sum_{q \in \mathbb{Z}^d}\zeta\left(\epsilon q\right)\cos\left(2\pi \left\langle x,q\right\rangle\right)
\end{align*}
and 
\begin{align*}
\psi_{\epsilon,3}:\mathbb{R}_+&\longrightarrow\mathbb{R}\\
r&\longmapsto\psi_{\epsilon,3}\left(r\right)=r+\sum_{q \in \mathbb{Z}^d-\left\{0\right\}}\frac{1}{2\left|2\pi q\right|^4}\left|\zeta\left(\epsilon q\right)\right|^2\left(1-\mathrm{e}^{-2r\left|2\pi q\right|^4}\right)
\end{align*}

We provide next the optimal regularity properties of the stochastic objects involved in our equation.
\begin{theorem}
\label{theorem1}
Let $\theta \in \left]-\infty,2-\frac{d}{2}\right[.$
\begin{enumerate}
\item
There exists a stochastic process $\left(\widetilde{X}_{r,1}\right)_{r \in \mathbb{R}_+}$ with state space $\left(\mathscr{C}^\theta,\mathcal{B}\left(\mathscr{C}^\theta\right)\right)$ such that the sample paths of $\left(\widetilde{X}_{q,1}\right)_{q \in \mathbb{R}_+}$ are continuous and for $\mathds{P}$-almost every $\omega \in \Omega,\left(\widetilde{X}_{r,1}\left(\omega\right)\right)\left(\varphi\right)=\left(X_{r,1}\left(\varphi\right)\right)\left(\omega\right).$
\item For all $\left(r,q\right) \in \mathbb{R}_+\times \mathbb{R}_+^*,$ let \gls{sto3}$:=\zeta_q*\widetilde{X}_{r,1}.$ For every $\left(U,\varsigma\right) \in \mathbb{R}_+\times \mathbb{R}_+^*,$ $$\lim_{\epsilon\downarrow 0}\mathds{P}\left(\sup_{r \in \left[0,U\right]}\left\Vert \widetilde{X}_{r,\epsilon,1}-\widetilde{X}_{r,1}\right\Vert_{\mathscr{C}^\theta}>\varsigma\right)=0.$$
\end{enumerate}
\end{theorem}
\begin{proof}
See Appendix \ref{ppe1}
\end{proof}
For all $\left(
r,\epsilon\right)
\in \mathbb{R}_+\times \mathbb{R}_+^*,$ let
\begin{align*}
X_{r,\epsilon,2}:L^2\left(\mathbb{T}^d\right)&\longrightarrow L^2\left(\Omega\right)\\
\varphi&\longmapsto X_{r,\epsilon,2}\left(\varphi\right)=\int_{\left(\left[0,r\right]\times \mathbb{T}^d\right)^2}\left(\int_{\mathbb{T}^d}\varphi\left(x\right)\prod_{m=1}^2\left(\zeta_\epsilon*\eta_{r-q_m}\right)\left(x-y_m\right)\,\mathrm{d}x\right)\,\xi\left(\mathrm{d}q_1,\mathrm{d}y_1\right)\,\xi\left(\mathrm{d}q_2,\mathrm{d}y_2\right)
\end{align*}
For every $(r,\epsilon) \in\mathbb{R}_+\times \mathbb{R}_+^*,$ let \gls{sto5}$:\Omega\longrightarrow \mathscr{S}'$ be such that for every $\omega \in \Omega,\varphi \in \mathscr{S},$
$$\left(\widetilde{X}_{r,\epsilon,2}\left(\varphi\right)\right)\left(\omega\right)=\int_{\mathbb{T}^d}\varphi\left(x\right)\left(\left(\left(\widetilde{X}_{r,\epsilon,1}\left(\omega\right)\right)\left(x\right)\right)^2-\psi_{\epsilon,3}\left(r\right)\right)\,\mathrm{d}x.$$

\begin{theorem}
\label{theorem2}
Let $d\in \left\{4,5,6,7\right\}$ and $\theta \in \left]-\infty,4-d\right[.$
There exists a stochastic process $\left(\text{\gls{sto4}}\right)_{r \in \mathbb{R}_+}$ with state space $\left(\mathscr{C}^\theta,\mathcal{B}\left(\mathscr{C}^\theta\right)\right)$ such that the sample paths of $\left(\widetilde{X}_{q,2}\right)_{q \in \mathbb{R}_+}$ are continuous and for all $\left(U,\varsigma\right) \in \mathbb{R}_+\times \mathbb{R}_+^*,$ $$\lim_{\epsilon\downarrow 0}\mathds{P}\left(\sup_{r \in \left[0,U\right]}\left\Vert\widetilde{X}_{r,\epsilon,2}-\widetilde{X}_{r,2}\right\Vert_{\mathscr{C}^\theta}>\varsigma\right)=0.$$
\end{theorem}
\begin{proof}
See Appendix \ref{ppe2}
\end{proof}
For every $(r,\epsilon) \in\mathbb{R}_+\times \mathbb{R}_+^*,$ let \gls{sto7}$:\Omega\longrightarrow \mathscr{S}'$ be such that for every $\omega \in \Omega,\varphi \in \mathscr{S},$  

$$\left(\widetilde{X}_{r,\epsilon,3}\left(\varphi\right)\right)\left(\omega\right)=\int_{\mathbb{T}^d}\varphi\left(x\right)\left(\left(\left(\widetilde{X}_{r,\epsilon,1}\left(\omega\right)\right)\left(x\right)\right)^3-3\psi_{\epsilon,3}\left(r\right)\left(\widetilde{X}_{r,\epsilon,1}\left(\omega\right)\right)\left(x\right)\right)\,\mathrm{d}x.$$

\begin{theorem}
\label{theorem3}
Let $d\in \left\{4,5\right\}$ and $\theta \in \left]-\infty,6-\frac{3d}{2}\right[.$
There exists a stochastic process $\left(\text{\gls{sto6}}\right)_{r \in \mathbb{R}_+}$ with state space $\left(\mathscr{C}^\theta,\mathcal{B}\left(\mathscr{C}^\theta\right)\right)$ such that the sample paths of $\left(\widetilde{X}_{q,3}\right)_{q \in \mathbb{R}_+}$ are continuous and for every $\left(U,\varsigma\right) \in \mathbb{R}_+\times \mathbb{R}_+^*,$ $$\lim_{\epsilon\downarrow 0}\mathds{P}\left(\sup_{r \in \left[0,U\right]}\left\Vert\widetilde{X}_{r,\epsilon,3}-\widetilde{X}_{r,3}\right\Vert_{\mathscr{C}^\theta}>\varsigma\right)=0.$$
\end{theorem}
\begin{proof}
See Appendix \ref{ppe3}
\end{proof}
For every $(
r,\epsilon) 
\in\mathbb{R}_+\times 
\mathbb{R}_+^*,$ let $$\text{\gls{sto14}}:=\int_0^rG_{r-q}(\Delta\widetilde{X}_{q,2})\,\mathrm{d}q,$$ 
$$\text{\gls{sto16}}:=\int_0^rG_{r-q}(\Delta\widetilde{X}_{q,\epsilon,2})\,\mathrm{d}q,$$
$$\text{\gls{sto15}}:=\int_0^rG_{r-q}(\Delta\widetilde{X}_{q,3})\,\mathrm{d}q,$$
$$\text{\gls{sto17}}:=\int_0^rG_{r-q}(\Delta\widetilde{X}_{q,\epsilon,3})\,\mathrm{d}q,$$

For every $(r,\epsilon) \in\mathbb{R}_+\times \mathbb{R}_+^*,$ let \gls{sto9}$:\Omega\longrightarrow \mathscr{S}'$ be such that for every $\omega \in \Omega,\varphi \in \mathscr{S},$   
$$\left(\widetilde{\Xi}_{r,\epsilon,1}\left(\omega\right)\right)\left(\varphi\right)=(\widetilde{Y}_{r,\epsilon,2}\left(\omega\right)\varodot\widetilde{X}_{r,\epsilon,1}\left(\omega\right))(\varphi).$$
\begin{theorem}
\label{theorem4}
Let $d=5$ and $\theta \in \mathbb{R}_-^*.$
There exists a stochastic process $\left(\text{\gls{sto8}}\right)_{r \in \mathbb{R}_+}$ with state space $\left(\mathscr{C}^\theta,\mathcal{B}\left(\mathscr{C}^\theta\right)\right)$ such that the sample paths of $\left(\widetilde{\Xi}_{q,1}\right)_{q \in \mathbb{R}_+}$ are continuous and for every $\left(U,\varsigma\right) \in \mathbb{R}_+\times \mathbb{R}_+^*,$ $$\lim_{\epsilon\downarrow 0}\mathds{P}\left(\sup_{r \in \left[0,U\right]}\left\Vert\widetilde{\Xi}_{r,\epsilon,1}-\widetilde{\Xi}_{r,1}\right\Vert_{\mathscr{C}^\theta}>\varsigma\right)=0.$$
\end{theorem}
\begin{proof}
See Appendix \ref{ppe4}
\end{proof}

For every $(r,\epsilon) \in\mathbb{R}_+\times \mathbb{R}_+^*,$ let

$$\psi_{\epsilon,4}\left(r\right)=\sum_{\left(m_1,m_2\right) \in \left(\mathbb{Z}^d\right)^2}\left|2\pi\left(m_1+m_2\right)\right|^2\left|\zeta\left(\epsilon m_1\right)\right|^2\left|\zeta\left(\epsilon m_2\right)\right|^2\int_{0}^r\mathrm{e}^{-\left|2\pi\left(m_1+m_2\right)\right|^4\left(r-q\right)}$$$$\left(\int_{\left[0,q\right]^2}\mathrm{e}^{-\left|2\pi m_1\right|^4\left(r+q-2v_1\right)-\left|2\pi m_2\right|^4 \left(r+q-2v_2\right)}\,\mathrm{d}v_1\,\mathrm{d}v_2\right)\,\mathrm{d}q$$.

For every $(r,\epsilon) \in\mathbb{R}_+\times \mathbb{R}_+^*,$ let \gls{sto11}$:\Omega\longrightarrow \mathscr{S}'$ 
be such that for every $\omega \in \Omega,\varphi \in \mathscr{S},$ 
$$\left(\widetilde{\Xi}_{r,\epsilon,2}\left(\omega\right)\right)\left(\varphi\right)=(\widetilde{Y}_{r,\epsilon,1}\left(\omega\right)\varodot\widetilde{X}_{r,\epsilon,2}\left(\omega\right))(\varphi)-2\psi_{\epsilon,4}\left(r\right)\int_{\mathbb{T}^d}\varphi\left(x\right)\,\mathrm{d}x$$
\begin{theorem}
\label{theorem5}
Let $d=5$ and $\theta \in \mathbb{R}_-^*.$
There exists a stochastic process $\left(\text{\gls{sto10}}\right)_{r \in \mathbb{R}_+}$ with state space $\left(\mathscr{C}^\theta,\mathcal{B}\left(\mathscr{C}^\theta\right)\right)$ such that the sample paths of $\left(\widetilde{\Xi}_{q,2}\right)_{q \in \mathbb{R}_+}$ are continuous and for every $\left(U,\varsigma\right) \in \mathbb{R}_+\times \mathbb{R}_+^*,$ $$\lim_{\epsilon\downarrow 0}\mathds{P}\left(\sup_{r \in \left[0,U\right]}\left\Vert\widetilde{\Xi}_{r,\epsilon,2}-\widetilde{\Xi}_{r,2}\right\Vert_{\mathscr{C}^\theta}>\varsigma\right)=0.$$
\end{theorem}
\begin{proof}
See Appendix \ref{ppe5}
\end{proof}
For every $(
r,\epsilon) 
\in\mathbb{R}_+\times \mathbb{R}_+^*,
$ let \gls{sto13}$:\Omega\longrightarrow \mathscr{S}'$ 
be such that for every $\omega \in \Omega,\varphi \in \mathscr{S},$ 
$$\left(\widetilde{\Xi}_{r,\epsilon,3}\left(\omega\right)\right)\left(\varphi\right)=(\widetilde{Y}_{r,\epsilon,2}\left(\omega\right)\varodot\widetilde{X}_{r,\epsilon,2}\left(\omega\right))(\varphi)-6\psi_{\epsilon,4}\left(r\right)\left(\widetilde{X}_{r,\epsilon,1}\left(\omega\right)\right)\left(\varphi\right)$$
\begin{theorem}
\label{theorem6}
Let $d=5$ and $\theta \in \left]-\infty,-\frac{1}{2}\right[.$
There exists a stochastic process $\left(\text{\gls{sto12}}\right)_{r \in \mathbb{R}_+}$ with state space $\left(\mathscr{C}^\theta,\mathcal{B}\left(\mathscr{C}^\theta\right)\right)$ such that the sample paths of $\left(\widetilde{\Xi}_{q,3}\right)_{q \in \mathbb{R}_+}$ are continuous and for every $\left(U,\varsigma\right) \in \mathbb{R}_+\times \mathbb{R}_+^*,$ $$\lim_{\epsilon\downarrow 0}\mathds{P}\left(\sup_{r \in \left[0,U\right]}\left\Vert\widetilde{\Xi}_{r,\epsilon,3}-\widetilde{\Xi}_{r,3}\right\Vert_{\mathscr{C}^\theta}>\varsigma\right)=0.$$
\end{theorem}
\begin{proof}
See Appendix \ref{ppe6}
\end{proof}
\section{$d \in\{1,2,3\} $}
\label{sec4}

In this section we consider $d \in \{1,2,3\}.$   

We begin with the following proposition.
\begin{proposition}
\label{proposition1}
If $d\in \left\{1,2,3\right\},\alpha \in \left]0,2-\frac{d}{2}\right[,$ and $\left(g_1,g_2\right) \in \mathscr{C}^\alpha\times \left(\bigcap_{U \in \mathbb{R}_+}C_U\mathscr{C}^\alpha\right),$ then there exists $\beta \in \left]0,+\infty\right]$ such that $$\exists ! f \in \bigcap_{U \in \left[0,\beta\right[}C_U\mathscr{C}^\alpha,\forall r \in \left[0,\beta\right[,f\left(r\right)=G_rg_1+\int_0^rG_{r-q}\left(\Delta\left(\left(f\left(q\right)+g_2\left(q\right)\right)^3-f\left(q\right)-g_2\left(q\right)\right)\right)\,\mathrm{d}q$$ and $$\beta \in \mathbb{R}_+^*\implies \limsup_{q\uparrow\beta}\left\Vert f\left(q\right)\right\Vert_{\mathscr{C}^\alpha}=+\infty.$$
\end{proposition}
\begin{proof}
For all $\left(U,h\right) \in \mathbb{R}_+\times \mathscr{C}^{\alpha},$ consider the fixed point map $\Upsilon_{U,h,g_2}:C_U\mathscr{C}^\alpha\longrightarrow C_U\mathscr{C}^\alpha$ such that for every $\varphi \in C_U\mathscr{C}^\alpha,$
\begin{align*}
\Upsilon_{U,h,g_2}\left(\varphi\right):\left[0,U\right]&\longrightarrow \mathscr{C}^\alpha\\
r&\longmapsto \left(\Upsilon_{U,h,g_2}\left(\varphi\right)\right)\left(r\right)=G_rh+\int_0^rG_{r-q}\left(\Delta\left(\left(\varphi\left(q\right)+g_2\left(q\right)\right)^3-\varphi\left(q\right)-g_2\left(q\right)\right)\right)\,\mathrm{d}q
\end{align*}
We fix $h \in \mathscr{C}^\alpha.$ We will prove that $\Upsilon_{U,h,g_2}$ is a contraction on a suitable Banach space and for suitable $U \in [0,1[$.

It follows from Theorem \ref{Schauder} that for any $U \in \mathbb{R}_+$ and all $\varphi_1\in C_U\mathscr{C}^\alpha,$ 
\begin{align*}
&\sup_{r\in[0,U]}\Vert\int_0^rG_{r-q}\left(\Delta\left(\left(\varphi_1\left(q\right)+g_2\left(q\right)\right)^3-\varphi_1\left(q\right)-g_2\left(q\right)\right)\right)\,\mathrm{d}q\Vert_{\alpha}\\
&\lesssim \max(U,U^{\frac{1}{2}})\sup_{q \in[0,U]}\Vert \left(\varphi_1\left(q\right)+g_2\left(q\right)\right)^3-\varphi_1\left(q\right)-g_2\left(q\right)\Vert_{\mathscr{C}^{\alpha}}
\end{align*}
We deduce from Corollary \ref{product} that
\begin{align*}
&\Vert\left(\varphi_1\left(q\right)+g_2\left(q\right)\right)^3-\varphi_1\left(q\right)-g_2\left(q\right)\Vert_{\mathscr{C}^{\alpha}}\\
&\lesssim \Vert\left(\varphi_1\left(q\right)+g_2\left(q\right)\right)^3\Vert_{\mathscr{C}^\alpha}+\Vert\varphi_1\left(q\right)\Vert_{\mathscr{C}^\alpha}+\Vert g_2\left(q\right)\Vert_{\mathscr{C}^{\alpha}}\\
&\lesssim \Vert\varphi_1\left(q\right)+g_2\left(q\right)\Vert^3_{\mathscr{C}^\alpha}+\Vert\varphi_1\left(q\right)\Vert_{\mathscr{C}^\alpha}+\Vert g_2\left(q\right)\Vert_{\mathscr{C}^{\alpha}}\\&\lesssim \left\Vert \varphi_1\right\Vert_{C_U\mathscr{C}^\alpha}^3+\left\Vert g_2\right\Vert^3_{\mathscr{C}^\alpha}+\left\Vert \varphi_1\right\Vert_{C_U\mathscr{C}^\alpha}+\left\Vert g_2\right\Vert_{\mathscr{C}^\alpha}
\end{align*}
Therefore there exists $\mu \in\mathbb{R}_+,$ such that $$\left\Vert\Upsilon_{U,h,g_2}\left(\varphi_1\right)\right\Vert_{C_U\mathscr{C}^\alpha}\leq \mu\left\Vert h\right\Vert_{\mathscr{C}^\alpha}+\mu\max\left(U,U^{\frac{1}{2}}\right)\left(\left\Vert \varphi_1\right\Vert_{C_U\mathscr{C}^\alpha}^3+\left\Vert g_2\right\Vert_{C_U\mathscr{C}^\alpha}^3+\left\Vert\varphi_1\right\Vert_{C_U\mathscr{C}^\alpha}+\left\Vert g_2\right\Vert_{C_U\mathscr{C}^\alpha}\right)$$
We repeat this argument for $\Upsilon_{U,h,g_2}\left(\varphi_2\right)-\Upsilon_{U,h,g_2}\left(\varphi_1\right)$ to deduce for $\left(\varphi_1,\varphi_2\right) \in \left(C_U\mathscr{C}^\alpha\right)^2$ that $$\left\Vert\Upsilon_{U,h,g_2}\left(\varphi_2\right)-\Upsilon_{U,h,g_2}\left(\varphi_1\right)\right\Vert_{C_U\mathscr{C}^\alpha}\leq \mu_2\max\left(U,U^{\frac{1}{2}}\right)\left(\left\Vert\varphi_1\right\Vert_{C_U\mathscr{C}^\alpha}^2+\left\Vert\varphi_2\right\Vert_{C_U\mathscr{C}^\alpha}^2+\left\Vert g_2\right\Vert_{C_U\mathscr{C}^\alpha}^2+1\right)\left\Vert\varphi_2-\varphi_1\right\Vert_{C_U\mathscr{C}^\alpha}.$$
Let $\theta \in \left]0,1\right]$ be such that $$8\theta^{\frac{1}{2}}\left(\mu+1\right)\left(\mu\left\Vert h\right\Vert_{\mathscr{C}^\alpha}+1\right)^3\left(\left\Vert g_2\right\Vert_{C_1\mathscr{C}^\alpha}+1\right)^3= 1.$$ For every $\left(\varphi_1,\varphi_2\right) \in \left(C_\theta\mathscr{C}^\alpha\right)^2,$ if $$\max\left(\left\Vert\varphi_1\right\Vert_{C_\theta\mathscr{C}^\alpha},\left\Vert\varphi_2\right\Vert_{C_\theta\mathscr{C}^\alpha}\right)\leq \mu\left\Vert h\right\Vert_{\mathscr{C}^\alpha}+1,$$ then $$\left\Vert\Upsilon_{\theta,h,g_2}\left(\varphi_1\right)\right\Vert_{C_\theta\mathscr{C}^\alpha}\leq \mu\left\Vert h\right\Vert_{\mathscr{C}^\alpha}+1$$ and $$\left\Vert\Upsilon_{\theta,h,g_2}\left(\varphi_2\right)-\Upsilon_{\theta,h,g_2}\left(\varphi_1\right)\right\Vert_{C_\theta\mathscr{C}^\alpha}\leq \frac{1}{2}\left\Vert\varphi_2-\varphi_1\right\Vert_{C_\theta\mathscr{C}^\alpha}.$$
Applying the Banach fixed point theorem iteratively we deduce that there exist a sequence $\left(m_q\right)_{q\in\mathbb{N}}$ in $\left]0,1\right]$ and a unique sequence $\left(\chi_r\right)_{r\in\mathbb{N}}$ in $\bigcup_{q\in\mathbb{N}}C_{m_q}\mathscr{C}^\alpha$ such that for all $j \in \mathbb{N},\chi_j \in C_{m_j}\mathscr{C}^\alpha,\Upsilon_{m_0,g_1,g_2}\left(\chi_0\right)=\chi_0,\Upsilon_{m_{j+1},\chi_{j}\left(m_{j}\right),g_2\left(\sum_{r=0}^{j}m_r+\cdot\right)}\left(\chi_{j+1}\right)=\chi_{j+1},$ and $$8m_{j+1}^{\frac{1}{2}}\left(\mu+1\right)\left(\mu\left\Vert \chi_j\left(m_j\right)\right\Vert_{\mathscr{C}^\alpha}+1\right)^3\left(\left\Vert g_2\left(\sum_{r=0}^{j}m_r+\cdot\right)\right\Vert_{C_1\mathscr{C}^\alpha}+1\right)^3= 1.$$

Let $\beta:=\sum_{q}m_q \in ]0,+\infty].$ 

Concatenating the solutions obtained before, we deduce that there exists a unique solution $f$ to our problem, defined on $[0,\beta[.$

If $\beta<+\infty,$ then $\lim_{q\to+\infty}m_q=0,$ yielding that $$\limsup_{r\uparrow\beta}\Vert f(r)\Vert_{\mathscr{C}^\alpha}\geq\limsup_{q\to+\infty}\Vert f(\sum_{r=0}^qm_r)\Vert_{\mathscr{C}^\alpha}\geq \limsup_{q\to+\infty}\Vert\chi_q(m_q)\Vert_{\mathscr{C}^\alpha}=+\infty.$$
\end{proof}
We deduce from Proposition \ref{proposition1} the following theorem.
\begin{theorem}
If $d\in \left\{1,2,3\right\},\alpha \in \left]0,2-\frac{d}{2}\right[,$ and $g \in \mathscr{C}^\alpha,$ then there exist a stopping time $\tau:\Omega\longrightarrow \left]0,+\infty\right]$ and a unique stochastic process $\left(f_r\right)_{r \in \left[0,\tau\right[}$ with state space $\left(\mathscr{C}^\alpha,\mathcal{B}\left(\mathscr{C}^\alpha\right)\right)$ such that for every $\omega \in \Omega,$ the function $\begin{aligned}[t]
\left[0,\tau\left(\omega\right)\right[&\longrightarrow \mathscr{C}^\alpha\\
r &\longmapsto f_r\left(\omega\right)
\end{aligned}$ is continuous on $\left[0,\tau\left(\omega\right)\right[$ and $$\forall U \in \left[0,\tau\left(\omega\right)\right[,f_U\left(\omega\right)=G_Ug+\widetilde{X}_{U,1}\left(\omega\right)+\int_0^UG_{U-q}\left(\Delta\left(\left(f_q\left(\omega\right)\right)^3-f_q\left(\omega\right)\right)\right)\,\mathrm{d}q,$$ and $$\tau\left(\omega\right)<+\infty\implies\limsup_{q\uparrow\tau\left(\omega\right)}\left\Vert f_q\left(\omega\right)\right\Vert_{\mathscr{C}^\alpha}=+\infty.$$
\end{theorem}
\begin{proof}
It is sufficient to apply Proposition \ref{proposition1} with $g_1=g$ and $g_2=\widetilde{X}_1$ to deduce that there exist a stopping time $\tau:\Omega\longrightarrow \left]0,+\infty\right]$ and a unique stochastic process $\left(h_r\right)_{r \in \left[0,\tau\right[}$ with state space $\left(\mathscr{C}^\alpha,\mathcal{B}\left(\mathscr{C}^\alpha\right)\right)$ such that for every $\omega \in \Omega,$ the function $\begin{aligned}[t]
\left[0,\tau\left(\omega\right)\right[&\longrightarrow \mathscr{C}^\alpha\\
r &\longmapsto h_r\left(\omega\right)
\end{aligned}$ is continuous on $\left[0,\tau\left(\omega\right)\right[$ and $$\forall U \in \left[0,\tau\left(\omega\right)\right[,h_U\left(\omega\right)=G_Ug+\int_0^UG_{U-q}\left(\Delta\left(\left(h_q+\widetilde{X}_{q,1}\left(\omega\right)\right)^3-h_q\left(\omega\right)-\widetilde{X}_{q,1}\left(\omega\right)\right)\right)\,\mathrm{d}q$$ and $$\tau\left(\omega\right)<+\infty\implies\limsup_{q\uparrow\tau\left(\omega\right)}\left\Vert h_q\left(\omega\right)\right\Vert_{\mathscr{C}^\alpha}=+\infty.$$ We conclude the proof by letting $f_r=h_r+\widetilde{X}_{r,1}$ 
\end{proof}
\section{$d=4$}
\label{sec5}
In this section we treat $d=4.$

We study next the mild solution to a deterministic partial differential equation.
\begin{proposition}
\label{proposition2}
If $\left(d,\alpha\right) \in \left\{4\right\}\times \left]0,2\right[,\epsilon \in \left]0,\min\left(\alpha,2-\alpha\right)\right[,\left(g_1,g_2,g_3,g_4\right) \in \mathscr{C}^\alpha\times \left(\bigcap_{U \in \mathbb{R}_+}C_U\mathscr{C}^{-\epsilon}\right)^3$ 
then there exists $\beta_1 \in \left]0,+\infty\right]$ such that $$\exists !f_1 \in \bigcap_{U \in \left[0,\beta_1\right[}C_U\mathscr{C}^\alpha,\forall r \in \left[0,\beta_1\right[,f_1\left(r\right)=G_rg_1+$$$$\int_0^rG_{r-q}\left(\Delta\left(\left(f_1\left(q\right)\right)^3+3\left(f_1\left(q\right)\right)^2g_2\left(q\right)+3f_1\left(q\right)g_3\left(q\right)+g_4\left(q\right)-f_1\left(q\right)-g_2\left(q\right)\right)\right)\,\mathrm{d}q,$$ $$\beta_1 \in \mathbb{R}_+^*\implies \limsup_{q\uparrow\beta_1}\left\Vert f_1\left(q\right)\right\Vert_{\mathscr{C}^\alpha}=+\infty,$$ 
\end{proposition}
\begin{proof}
For all $\left(U,h_1\right) \in \mathbb{R}_+\times \mathscr{C}^\alpha,$ consider the fixed point map $\Upsilon_{U,h_1,g_2,g_3,g_4}:C_U\mathscr{C}^\alpha\longrightarrow C_U\mathscr{C}^\alpha$ such that for every $\varphi \in C_U\mathscr{C}^\alpha,$
\begin{align*}
\Upsilon_{U,h_1,g_2,g_3,g_4}\left(\varphi\right):\left[0,U\right]&\longrightarrow \mathscr{C}^\alpha\\
r&\longmapsto \left(\Upsilon_{U,h_1,g_2,g_3,g_4}\left(\varphi\right)\right)\left(r\right)=G_rh_1
\end{align*}
$$+\int_0^rG_{r-q}\left(\Delta\left(\left(\varphi\left(q\right)\right)^3+3\left(\varphi\left(q\right)\right)^2g_2\left(q\right)+3\varphi\left(q\right)g_3\left(q\right)+g_4\left(q\right)-\varphi\left(q\right)-g_2\left(q\right)\right)\right)\,\mathrm{d}q$$
Let $h_1
\in \mathscr{C}^\alpha$, 
$\gamma_1:=\frac{1}{4}\left(\alpha+\epsilon+2\right)$ 

We will prove that $\Upsilon_{U,h_1,g_2,g_3,g_4}$ is a contraction on a suitable Banach space and for suitable $U \in [0,1[$.

It follows from Theorem \ref{Schauder} that for any $U \in \mathbb{R}_+$ and all $\varphi_1 \in C_U\mathscr{C}^\alpha,$ 
\begin{align*}
&\sup_{r\in[0,U]}\Vert\int_0^rG_{r-q}\left(\Delta\left(\left(\varphi_1\left(q\right)\right)^3+3\left(\varphi_1\left(q\right)\right)^2g_2\left(q\right)+3\varphi_1\left(q\right)g_3\left(q\right)+g_4\left(q\right)-\varphi_1\left(q\right)-g_2\left(q\right)\right)\right)\,\mathrm{d}q\Vert_{\alpha}\\&=\sup_{r\in[0,U]}\Vert\int_0^rG_{r-q}\left(\Delta\left(\left(\varphi_1\left(q\right)\right)^3+3\left(\varphi_1\left(q\right)\right)^2g_2\left(q\right)+3\varphi_1\left(q\right)g_3\left(q\right)+g_4\left(q\right)-\varphi_1\left(q\right)-g_2\left(q\right)\right)\right)\,\mathrm{d}q\Vert_{4\gamma_1-\epsilon-2}\\
&\lesssim \max(U,U^{1-\gamma_1})\sup_{q \in[0,U]}\Vert \left(\varphi_1\left(q\right)\right)^3+3\left(\varphi_1\left(q\right)\right)^2g_2\left(q\right)+3\varphi_1\left(q\right)g_3\left(q\right)+g_4\left(q\right)-\varphi_1\left(q\right)-g_2\left(q\right)\Vert_{\mathscr{C}^{-\epsilon}}
\end{align*}
We deduce from Corollary \ref{product} that
\begin{align*}
&\Vert\left(\varphi_1\left(q\right)\right)^3+3\left(\varphi_1\left(q\right)\right)^2g_2\left(q\right)+3\varphi_1\left(q\right)g_3\left(q\right)+g_4\left(q\right)-\varphi_1\left(q\right)-g_2\left(q\right)\Vert_{\mathscr{C}^{-\epsilon}}\\
&\lesssim \Vert\left(\varphi_1\left(q\right)\right)^3\Vert_{\mathscr{C^\alpha}}+3\Vert\left(\varphi_1\left(q\right)\right)^2\Vert_{\mathscr{C}^\alpha}\Vert g_2\left(q\right)\Vert_{\mathscr{C}^{-\epsilon}}+3\Vert\varphi_1\left(q\right)\Vert_{\mathscr{C}^\alpha}\Vert g_3\left(q\right)\Vert_{\mathscr{C}^{-\epsilon}}+\Vert g_4\left(q\right)\Vert_{\mathscr{C}^{-\epsilon}}+\Vert\varphi_1\left(q\right)\Vert_{\mathscr{C}^\alpha}\\&+\Vert g_2\left(q\right)\Vert_{\mathscr{C}^{-\epsilon}} \\
&\lesssim \left\Vert \varphi_1\right\Vert_{C_U\mathscr{C}^\alpha}^3+\left\Vert g_2\right\Vert_{C_U\mathscr{C}^{-\epsilon}}\left(\left\Vert \varphi_1\right\Vert_{C_U\mathscr{C}^\alpha}^2+1\right)+\left\Vert \varphi_1\right\Vert_{C_U\mathscr{C}^\alpha}\left(\left\Vert g_3\right\Vert_{C_U\mathscr{C}^{-\epsilon}}+1\right)+\left\Vert g_4\right\Vert_{C_U\mathscr{C}^{-\epsilon}}
\end{align*}
Therefore there exists $\mu_1 \in\mathbb{R}_+,$ such that $$\left\Vert \Upsilon_{U,h_1,g_2,g_3,g_4}\left(\varphi_1\right)\right\Vert_{C_U\mathscr{C}^\alpha}\leq \mu_1\left\Vert h_1\right\Vert_{\mathscr{C}^\alpha}$$$$+\mu_1\max\left(U,U^{1-\gamma_1}\right)\left(\left\Vert \varphi_1\right\Vert_{C_U\mathscr{C}^\alpha}^3+\left\Vert g_2\right\Vert_{C_U\mathscr{C}^{-\epsilon}}\left(\left\Vert \varphi_1\right\Vert_{C_U\mathscr{C}^\alpha}^2+1\right)+\left\Vert \varphi_1\right\Vert_{C_U\mathscr{C}^\alpha}\left(\left\Vert g_3\right\Vert_{C_U\mathscr{C}^{-\epsilon}}+1\right)+\left\Vert g_4\right\Vert_{C_U\mathscr{C}^{-\epsilon}}\right)$$ 
We repeat this argument for $\Upsilon_{U,h_1,g_2,g_3,g_4}\left(\varphi_2\right)-\Upsilon_{U,h_1,g_2,g_3,g_4}\left(\varphi_1\right)$ to deduce for $\left(\varphi_1,\varphi_2\right) \in \left(C_U\mathscr{C}^\alpha\right)^2$ that  
$$\left\Vert\Upsilon_{U,h_1,g_2,g_3,g_4}\left(\varphi_2\right)-\Upsilon_{U,h_1,g_2,g_3,g_4}\left(\varphi_1\right)\right\Vert_{C_U\mathscr{C}^\alpha}\leq \mu_2\max\left(U,U^{1-\gamma_1}\right)\left\Vert \varphi_2-\varphi_1\right\Vert_{C_U\mathscr{C}^\alpha}$$$$\left(\left\Vert \varphi_1\right\Vert_{C_U\mathscr{C}^\alpha}^2+\left\Vert \varphi_2\right\Vert_{C_U\mathscr{C}^\alpha}^2+\left\Vert g_2\right\Vert_{C_U\mathscr{C}^{-\epsilon}}\left(\left\Vert \varphi_1\right\Vert_{C_U\mathscr{C}^\alpha}+\left\Vert \varphi_2\right\Vert_{C_U\mathscr{C}^\alpha}\right)+\left\Vert g_3\right\Vert_{C_U\mathscr{C}^{-\epsilon}}+1
\right).$$
Let $\theta \in \left]0,1\right]$ be such that $$2\theta^{1-\gamma_1}\left(\mu_1+1\right)\left(\mu_1\left\Vert h_1\right\Vert_{\mathscr{C}^\alpha}+1\right)^3\left(2\left\Vert g_2\right\Vert_{C_1\mathscr{C}^{-\epsilon}}+\left\Vert g_3\right\Vert_{C_1\mathscr{C}^{-\epsilon}}+\left\Vert g_4\right\Vert_{C_1\mathscr{C}^{-\epsilon}}+3\right)=1.$$ For every $\left(\varphi_1,\varphi_2\right) \in \left(C_\theta\mathscr{C}^\alpha\right)^2,$ if $$\max\left(\left\Vert\varphi_1\right\Vert_{C_\theta\mathscr{C}^\alpha},\left\Vert\varphi_2\right\Vert_{C_\theta\mathscr{C}^\alpha}\right)\leq\mu_1\left\Vert h_1\right\Vert_{\mathscr{C}^\alpha}+1,$$ then $$\left\Vert\Upsilon_{\theta,h_1,g_2,g_3,g_4}\left(\varphi_1\right)\right\Vert_{C_\theta\mathscr{C}^\alpha}\leq \mu_1\left\Vert h_1\right\Vert_{\mathscr{C}^\alpha}+1$$ and $$\left\Vert\Upsilon_{\theta,h_1,g_2,g_3,g_4}\left(\varphi_2\right)-\Upsilon_{\theta,h_1,g_2,g_3,g_4}\left(\varphi_1\right)\right\Vert_{C_\theta\mathscr{C}^\alpha}\leq \frac{1}{2}\left\Vert\varphi_2-\varphi_1\right\Vert_{C_\theta\mathscr{C}^\alpha}.$$
Applying the Banach fixed point theorem iteratively we deduce that there exist a sequence $\left(
m_{q,1}
\right)_{q\in\mathbb{N}}$ in $\left]0,1\right]
$ and a unique sequence $
\chi_{r,1}
$ in $\bigcup_{q\in\mathbb{N}}
C_{m_{q,1}}\mathscr{C}^\alpha
$ such that for all $j \in \mathbb{N},\chi_{j,1} \in C_{m_{j,1}}\mathscr{C}^\alpha,
\Upsilon_{m_{0,1},g_1,g_2,g_3,g_4}\left(\chi_{0,1}\right)=\chi_{0,1},
$
$\Upsilon_{m_{j+1,1},\chi_{j,1}\left(m_{j,1}\right),g_2\left(\sum_{r=0}^{j}m_{r,1}+\cdot\right),g_3\left(\sum_{r=0}^{j}m_{r,1}+\cdot\right),g_4\left(\sum_{r=0}^{j}m_{r,1}+\cdot\right)}\left(\chi_{j+1,1}\right)=\chi_{j+1,1},
$ $$2m_{j+1,1}^{1-\gamma_1}\left(\mu_1+1\right)\left(\mu_1\left\Vert \chi_{j,1}\left(m_{j,1}\right)\right\Vert_{\mathscr{C}^\alpha}+1\right)^3
$$$$\left(
2\left\Vert g_2\left(\sum_{r=0}^{j}m_{r,1}+\cdot\right)\right\Vert_{C_1\mathscr{C}^{-\epsilon}}
+\left\Vert g_3\left(\sum_{r=0}^{j}m_{r,1}+\cdot\right)\right\Vert_{C_1\mathscr{C}^{-\epsilon}}
+\left\Vert g_4\left(\sum_{r=0}^{j}m_{r,1}+\cdot\right)\right\Vert_{C_1\mathscr{C}^{-\epsilon}}
+3\right)= 1,$$ 
Let $\beta_1:=\sum_{q}m_{q,1} \in ]0,+\infty].$ 

Concatenating the solutions obtained before, we deduce that there exists a unique solution $f_1$ to our problem, defined on $[0,\beta_1[.$

If $\beta_1<+\infty,$ then $\lim_{q\to+\infty}m_{q,1}=0,$ yielding that $$\limsup_{r\uparrow\beta_1}\Vert f_1(r)\Vert_{\mathscr{C}^\alpha}\geq\limsup_{q\to+\infty}\Vert f_1(\sum_{r=0}^qm_{r,1})\Vert_{\mathscr{C}^\alpha}\geq \limsup_{q\to+\infty}\Vert\chi_{q,1}(m_{q,1})\Vert_{\mathscr{C}^\alpha}=+\infty.$$
\end{proof}

We end this section by stating and proving existence and uniqueness of a local-in-time solution to the $4$-dimensional stochastic Cahn-Hilliard equation.
\begin{theorem-definition}[Da Prato-Debussche]
\label{theoremd=4}
If $\left(d,\alpha\right) \in \left\{4\right\}\times \left]0,2\right[$ and $g \in \mathscr{C}^\alpha,$ then 
there exist a stopping time $\tau:\Omega\longrightarrow \left]0,+\infty\right]$ and a unique stochastic process $\left(h_r\right)_{r \in \left[0,\tau\right[}$ with state space $\left(\mathscr{C}^\alpha,\mathcal{B}\left(\mathscr{C}^\alpha\right)\right)$ such that for every $\omega \in \Omega,$ the function $\begin{aligned}[t]
\left[0,\tau\left(\omega\right)\right[&\longrightarrow \mathscr{C}^\alpha\\
r &\longmapsto h_r\left(\omega\right)
\end{aligned}$ is continuous on $\left[0,\tau\left(\omega\right)\right[,$ $$\forall U \in \left[0,\tau\left(\omega\right)\right[,h_U\left(\omega\right)=G_Ug+\int_0^UG_{U-q}(\Delta(\left(h_q\left(\omega\right)\right)^3+3\left(h_q\left(\omega\right)\right)^2\widetilde{X}_{q,1}\left(\omega\right)+3h_q\left(\omega\right)\widetilde{X}_{q,2}\left(\omega\right)+\widetilde{X}_{q,3}\left(\omega\right)$$$$-h_q\left(\omega\right)-\widetilde{X}_{q,1}\left(\omega\right)))\,\mathrm{d}q,$$ and $$\tau\left(\omega\right)<+\infty\implies\limsup_{q\uparrow\tau\left(\omega\right)}\left\Vert h_q\left(\omega\right)\right\Vert_{\mathscr{C}^\alpha}=+\infty.$$
The maximal local mild solution of the 4-dimensional stochastic
Cahn-Hilliard equation is defined by the processes $(h_r+\widetilde{X}_{r,1})_{r\in[0,\tau[}$
\end{theorem-definition}
\begin{proof}
Considering $g_1=g,g_2=X_{1},g_3=X_2,g_4=X_3,$ it follows from Proposition \ref{proposition2} that there exist a stopping time $\tau:\Omega\longrightarrow \left]0,+\infty\right]$ and a unique stochastic process $\left(h_r\right)_{r \in \left[0,\tau\right[}$ with state space $\left(\mathscr{C}^\alpha,\mathcal{B}\left(\mathscr{C}^\alpha\right)\right)$ such that for every $\omega \in \Omega,$ the function $\begin{aligned}[t]
\left[0,\tau\left(\omega\right)\right[&\longrightarrow \mathscr{C}^\alpha\\
r &\longmapsto h_r\left(\omega\right)
\end{aligned}$ is continuous on $\left[0,\tau\left(\omega\right)\right[,$ $$\forall U \in \left[0,\tau\left(\omega\right)\right[,h_U\left(\omega\right)=G_Ug+\int_0^UG_{U-q}(\Delta(\left(h_q\left(\omega\right)\right)^3+3\left(h_q\left(\omega\right)\right)^2\widetilde{X}_{q,1}\left(\omega\right)+3h_q\left(\omega\right)\widetilde{X}_{q,2}\left(\omega\right)+\widetilde{X}_{q,3}\left(\omega\right)$$$$-h_q\left(\omega\right)-\widetilde{X}_{q,1}\left(\omega\right)))\,\mathrm{d}q,$$ and $$\tau\left(\omega\right)<+\infty\implies\limsup_{q\uparrow\tau\left(\omega\right)}\left\Vert h_q\left(\omega\right)\right\Vert_{\mathscr{C}^\alpha}=+\infty.$$ 
\end{proof}
\section{$d=5$}
\label{sec6}

In this section we treat $d=5$. 

For simplicity of notation, we drop the tilde when referring to the stochastic objects.

In Section \ref{sec}, we derived the following system of equations (interpreted in the mild sense):
\begin{equation*}
(\partial_r+\Delta^2)v=3\Delta((v+w+Y_2)\varolessthan X_2),
\end{equation*}
\begin{equation*}
(\partial_r+\Delta^2)w=\Delta((v+w)^3+9\Psi(v+w+Y_2,Y_1,X_2)+3\Lambda(v,w)\varodot X_2+3 w\varodot X_2+3(v+w+Y_2)\varogreaterthan X_2+P(v+w)),
\end{equation*}
where $$\Lambda(v,w)(r):=3\int_0^r\mathrm{e}^{-(r-q)\Delta^2}(\Delta((v(q)+w(q)+Y_2(q))\varolessthan X_2(q)))\,\mathrm{d}q-3(v(r)+w(r)+Y_2(r))\varolessthan Y_1(r),$$ $$P(v+w):=\Theta_4+\Theta_5(v+w)+\Theta_6(v+w)^2,$$ $$\Theta_4:=Y_2^3+3X_1\varolessthan Y_2^2+3X_1\varodot(Y_2\varodot Y_2)+6Y_2\Xi_1+6\Psi(Y_2,Y_2,X_1)+3X_1\varogreaterthan Y_2^2+9Y_2\Xi_2+3\Xi_3-X_1-Y_2,$$ $$\Theta_5:=3Y_2^2+6Y_2 \varolessthan X_1+6Y_2 \varogreaterthan X_1+6\Xi_1+9\Xi_2-1$$ $$\Theta_6:=3X_1+3Y_2.$$

The objective here is to prove the existence and uniqueness of a local mild solution to this system of equations. The space $(C_U\mathscr{C}^\alpha,\Vert\cdot\Vert)$ is not enough for the construction of the solution. Therefore we introduce the following space.  
\begin{definition}
Let $\left(\alpha,U\right) \in \mathbb{R}\times \mathbb{R}_+.$ We define $$\text{\gls{pc6}}:=\left\{g \in C_U\mathscr{C}^{\max\left(\alpha,0\right)},\exists\mu\in\mathbb{R}_+,\forall\left(r,q\right)\in\left[0,U\right]^2,\left\Vert g\left(r\right)-g\left(q\right)\right\Vert_{\mathscr{C}^0}\leq\mu\left|r-q\right|^{\frac{1}{16}}\right\}$$
and
\begin{align*} 
\left\Vert\cdot\right\Vert_{\mathfrak{Y}_U^\alpha}: \mathfrak{Y}_U^\alpha &\longrightarrow \mathbb{R}_+ \\ f & \longmapsto \left\Vert f\right\Vert_{\mathfrak{Y}_U^\alpha}=\max\left(\left\Vert f\right\Vert_{C_U\mathscr{C}^{\max\left(\alpha,0\right)}},\sup_{\substack{\left(r,q\right)\in\left[0,U\right]^2\\ r\neq q}}\frac{\left\Vert f\left(r\right)-f\left(q\right)\right\Vert_{\mathscr{C}^0}}{\left|r-q\right|^{\frac{1}{16}}}\right)
\end{align*}
\end{definition}
\begin{remark}
\begin{enumerate}
\item We note that for every $\left(\alpha,U\right) \in \mathbb{R}\times \mathbb{R}_+,\left(\mathfrak{Y}_U^\alpha,\left\Vert\cdot\right\Vert_{\mathfrak{Y}_U^\alpha}\right)$ is a Banach space.
\item By construction, $\Theta_4,\Theta_5,\Theta_6$ are continuous stochastic processes with state space $(\mathscr{C}^{-1/2-\epsilon},\mathcal{B}(\mathscr{C}^{-1/2-\epsilon}))$ for $\epsilon>0.$ This follows directly from Theorem \ref{estimate} and Proposition \ref{co1}.
\end{enumerate}
\end{remark}

We begin with the following proposition.

\begin{proposition}
\label{proposition4}
For $1/4\leq\beta<2,\alpha \in [1/4-\beta,+\infty[$, there exists $\mu\in\mathbb{R}_+^*$ such that for every $U\in[0,1]$, $(r_1,r_2)\in[0,U]^2,g\in\mathscr{C}^\beta,f\in C_U\mathscr{C}^{\alpha},$ $$\Vert G_{r_2}g-G_{r_1}g\Vert_{\mathscr{C}^0}\leq \mu |r_2-r_1|^{1/16}\Vert g\Vert_{\mathscr{C}^\beta}$$ and $$\Vert \int_0^{r_2}G_{r_2-q}(\Delta(f(q)))\,dq-\int_0^{r_1}G_{r_1-q}(\Delta(f(q)))\,dq\Vert_{\mathscr{C}^0}\leq \mu U^{1-(\beta+2)/4}|r_2-r_1|^{1/16}\Vert f\Vert_{C_U\mathscr{C}^{\alpha}}$$\end{proposition}\begin{proof}
The first inequality follows directly from Lemma \ref{lemma}. \\Let us prove the second inequality. \\ Writing for $r_1\leq r_2$\begin{align*}
\int_0^{r_2}G_{r_2-q}(\Delta(f(q)))\,dq-\int_0^{r_1}G_{r_1-q}(\Delta(f(q)))\,dq&=\int_0^{r_1}G_{r_1-q}(\Delta(G_{r_2-r_1}(f(q))-f(q)))\,dq\\&+\int_{r_1}^{r_2}G_{r_2-q}(\Delta (f(q)))\,dq
\end{align*}  and observing that \begin{align*}\int_0^{r_1}\Vert G_{r_1-q}(\Delta(G_{r_2-r_1}(f(q))-f(q)))\Vert_{\mathscr{C}^{\alpha+\beta-1/4}}\,dq&\stackrel{\text{Theorem \ref{Schauder}}}{\lesssim} U^{1-(\beta+2)/4}\sup_{q\in[0,U]}\Vert G_{r_2-r_1}(f(q))-f(q)\Vert_{\mathscr{C}^{\alpha-1/4}}\\&\stackrel{\text{Lemma \ref{lemma}}}{\lesssim}U^{1-(\beta+2)/4}|r_2-r_1|^{1/16}\Vert f\Vert_{C_U\mathscr{C}^{\alpha}}\end{align*} and \begin{align*}\int_{r_1}^{r_2}\Vert G_{r_2-q}(\Delta (f(q)))\Vert_{\mathscr{C}^{\alpha+\beta-1/4}}\,dq&\stackrel{\text{Theorem \ref{derivative}}}{\lesssim}\int_{r_1}^{r_2}\Vert G_{r_2-q}(f(q))\Vert_{\mathscr{C}^{\alpha+\beta+2-1/4}}\,dq\\&\stackrel{\text{Lemma \ref{lemma}}}{\lesssim}\int_{r_1}^{r_2}\max(|r_2-q|^{-(\beta+2-1/4)/4},1)\Vert f\Vert_{C_U\mathscr{C}^{\alpha}}\,dq\\&\lesssim U^{1-(\beta+2)/4}|r_2-r_1|^{1/16}\Vert f\Vert_{C_U\mathscr{C}^{\alpha}}\int_{0}^{1}|q|^{-(\beta+2-1/4)/4}\,dq\\&\lesssim U^{1-(\beta+2)/4}|r_2-r_1|^{1/16}\Vert f\Vert_{C_U\mathscr{C}^{\alpha}},
\end{align*}it follows that \begin{align*}
\Vert \int_0^{r_2}G_{r_2-q}(\Delta(f(q)))\,dq-\int_0^{r_1}G_{r_1-q}(\Delta(f(q)))\,dq\Vert_{\mathscr{C}^0}&\lesssim\Vert \int_0^{r_2}G_{r_2-q}(\Delta(f(q)))\,dq-\int_0^{r_1}G_{r_1-q}(\Delta(f(q)))\,dq\Vert_{\mathscr{C}^{\alpha+\beta-1/4}}\\&\lesssim \int_0^{r_1}\Vert G_{r_1-q}(\Delta(G_{r_2-r_1}(f(q))-f(q)))\Vert_{\mathscr{C}^{\alpha+\beta-1/4}}\,dq\\&+\int_{r_1}^{r_2}\Vert G_{r_2-q}(\Delta (f(q)))\Vert_{\mathscr{C}^{\alpha+\beta-1/4}}\,dq\\&\lesssim U^{1-(\beta+2)/4}|r_2-r_1|^{1/16}\Vert f\Vert_{C_U\mathscr{C}^{\alpha}}\end{align*}
\end{proof}\begin{remark}
For $\epsilon>0$ sufficiently small, it follows from Proposition \ref{proposition4} that $Y_2 \in\mathfrak{Y}_1^{1/2-\epsilon}.$ 
\end{remark}\begin{proposition}
\label{proposition3}
For $\epsilon>0$ sufficiently small, $\alpha\in]1,3/2[,$ there exists $\mu\in\mathbb{R}_+^*$ such that for every $(v,w)\in\mathfrak{Y}_1^{\alpha-1/2}\times\mathfrak{Y}_1^\alpha,$ $$\Vert\Lambda (v,w)\Vert_{C_1\mathscr{C}^{1+2\epsilon}}\leq \mu (\Vert v\Vert_{\mathfrak{Y}_1^{\alpha-1/2}}+\Vert w\Vert_{\mathfrak{Y}^{\alpha}_1}+\Vert Y_2\Vert_{\mathfrak{Y}_1^{1/2-\epsilon}})\Vert X_2\Vert_{C_1\mathscr{C}^{-1-\epsilon}}$$\end{proposition}
\begin{proof}
We recall the following identity, for $f,g\in\mathscr{S}',$ $$\Delta(f\varolessthan g)=(\Delta f)\varolessthan g+2\sum_i (\partial_i f)\varolessthan(\partial_i g)+f\varolessthan(\Delta g)$$ and write 
\begin{align*}
\frac{1}{3}(\Lambda(v,w))(r)&=\int_{0}^r\Delta\left(G_{r-q}(\left(v(q)+w\left(q\right)+Y_2(q)\right)\varolessthan X_2(q))-(v(q)+w\left(q\right)+Y_2(q))\varolessthan\left(G_{r-q}\left(X_2\left(q\right)\right)\right)\right)\,\mathrm{d}q\\
& \phantom{=}\ + \int_{0}^r(\Delta(v(q)+w(q)+Y_2(q)))\varolessthan\left(G_{r-q}\left(X_2\left(q\right)\right)\right)\,\mathrm{d}q\\& \phantom{=}\ + 2\sum_i\int_{0}^r(\partial_i(v(q)+w(q)+Y_2(q)))\varolessthan\left(\partial_i G_{r-q}\left(X_2\left(q\right)\right)\right)\,\mathrm{d}q\\
& \phantom{=}\ +\int_0^r\left(v\left(q\right)+w(q)+Y_2(q)-v\left(r\right)-w(r)-Y_2(r)\right)\varolessthan\left(G_{r-q}\left(\Delta\left(X_2\left(q\right)\right)\right)\right)\,\mathrm{d}q.
\end{align*}
The following estimates hold for $(v,w)\in\mathfrak{Y}_1^{\alpha-1/2}\times\mathfrak{Y}_1^\alpha,r\in[0,1],$
\begin{align*}
&\left\Vert \int_{0}^r\Delta\left(G_{r-q}(\left(v(q)+w\left(q\right)+Y_2(q)\right)\varolessthan X_2(q))-(v(q)+w\left(q\right)+Y_2(q))\varolessthan\left(G_{r-q}\left(X_2\left(q\right)\right)\right)\right)\,\mathrm{d}q\right\Vert_{\mathscr{C}^{2\epsilon+1}}\\&\leq\int_{0}^r\left\Vert \Delta\left(G_{r-q}(\left(v(q)+w\left(q\right)+Y_2(q)\right)\varolessthan X_2(q))-(v(q)+w\left(q\right)+Y_2(q))\varolessthan\left(G_{r-q}\left(X_2\left(q\right)\right)\right)\right)\right\Vert_{\mathscr{C}^{2\epsilon+1}}\,\mathrm{d}q\\
&\stackrel{\text{Theorem \ref{derivative}}}{\lesssim} \int_{0}^r\left\Vert G_{r-q}(\left(v(q)+w\left(q\right)+Y_2(q)\right)\varolessthan X_2(q))-(v(q)+w\left(q\right)+Y_2(q))\varolessthan\left(G_{r-q}\left(X_2\left(q\right)\right)\right)\right\Vert_{\mathscr{C}^{2\epsilon+3}}\,\mathrm{d}q\\
&\stackrel{\text{Proposition \ref{co2}}}{\lesssim} \int_{0}^r|r-q|^{-(3+2\epsilon-1/2+\epsilon+1+\epsilon)/4}\Vert v(q)+w\left(q\right)+Y_2(q)\Vert_{\mathscr{C}^{1/2-\epsilon}}\Vert X_2\left(q\right)\Vert_{\mathscr{C}^{-1-\epsilon}}\,\mathrm{d}q\\
&\lesssim (\Vert v\Vert_{\mathfrak{Y}_1^{\alpha-1/2}}+\Vert w\Vert_{\mathfrak{Y}^{\alpha}_1}+\Vert Y_2\Vert_{\mathfrak{Y}_1^{1/2-\epsilon}})\Vert X_2\Vert_{C_1\mathscr{C}^{-1-\epsilon}},
\end{align*}\begin{align*}
&\left\Vert \int_{0}^r(\Delta(v(q)+w(q)+Y_2(q)))\varolessthan\left(G_{r-q}\left(X_2\left(q\right)\right)\right)\,\mathrm{d}q\right\Vert_{\mathscr{C}^{2\epsilon+1}}\\&\leq\int_{0}^r\left\Vert (\Delta(v(q)+w(q)+Y_2(q)))\varolessthan\left(G_{r-q}\left(X_2\left(q\right)\right)\right)\right\Vert_{\mathscr{C}^{2\epsilon+1}}\,\mathrm{d}q\\
&\stackrel{\text{Theorem \ref{estimate}}}{\lesssim} \int_{0}^r\Vert \Delta(v(q)+w\left(q\right)+Y_2(q))\Vert_{\mathscr{C}^{-3/2-\epsilon}}\Vert G_{r-q}\left(X_2\left(q\right)\right)\Vert_{\mathscr{C}^{1+2\epsilon+3/2+\epsilon}}\,\mathrm{d}q\\
&\stackrel{\text{Theorem \ref{derivative}}}{\lesssim} \int_{0}^r\Vert v(q)+w\left(q\right)+Y_2(q)\Vert_{\mathscr{C}^{1/2-\epsilon}}\Vert G_{r-q}\left(X_2\left(q\right)\right)\Vert_{\mathscr{C}^{1+2\epsilon+3/2+\epsilon}}\,\mathrm{d}q\\
&\stackrel{\text{Lemma \ref{lemma}}}{\lesssim} \int_{0}^r|r-q|^{-(1+2\epsilon+3/2+\epsilon+1+\epsilon)/4}\Vert v(q)+w\left(q\right)+Y_2(q)\Vert_{\mathscr{C}^{1/2-\epsilon}}\Vert X_2\left(q\right)\Vert_{\mathscr{C}^{-1-\epsilon}}\,\mathrm{d}q\\&\lesssim (\Vert v\Vert_{\mathfrak{Y}_1^{\alpha-1/2}}+\Vert w\Vert_{\mathfrak{Y}^{\alpha}_1}+\Vert Y_2\Vert_{\mathfrak{Y}_1^{1/2-\epsilon}})\Vert X_2\Vert_{C_1\mathscr{C}^{-1-\epsilon}},
\end{align*}\begin{align*}
&\left\Vert \sum_i\int_{0}^r(\partial_i(v(q)+w(q)+Y_2(q)))\varolessthan\left(\partial_i G_{r-q}\left(X_2\left(q\right)\right)\right)\,\mathrm{d}q\right\Vert_{\mathscr{C}^{2\epsilon+1}}\\&\leq\sum_i\int_{0}^r\left\Vert (\partial_i(v(q)+w(q)+Y_2(q)))\varolessthan\left(\partial_i G_{r-q}\left(X_2\left(q\right)\right)\right)\right\Vert_{\mathscr{C}^{2\epsilon+1}}\,\mathrm{d}q\\
&\stackrel{\text{Theorem \ref{estimate}}}{\lesssim} \sum_i\int_{0}^r\Vert \partial_i(v(q)+w(q)+Y_2(q))\Vert_{\mathscr{C}^{-1/2-\epsilon}}\Vert \partial_iG_{r-q}\left(X_2\left(q\right)\right)\Vert_{\mathscr{C}^{2\epsilon+3/2+\epsilon}}\,\mathrm{d}q\\
&\stackrel{\text{Theorem \ref{derivative}}}{\lesssim} \int_{0}^r\Vert v(q)+w\left(q\right)+Y_2(q)\Vert_{\mathscr{C}^{1/2-\epsilon}}\Vert G_{r-q}\left(X_2\left(q\right)\right)\Vert_{\mathscr{C}^{2\epsilon+5/2+\epsilon}}\,\mathrm{d}q\\
&\stackrel{\text{Lemma \ref{lemma}}}{\lesssim} \int_{0}^r|r-q|^{-(2\epsilon+5/2+\epsilon+1+\epsilon)/4}\Vert v(q)+w\left(q\right)+Y_2(q)\Vert_{\mathscr{C}^{1/2-\epsilon}}\Vert X_2\left(q\right)\Vert_{\mathscr{C}^{-1-\epsilon}}\,\mathrm{d}q\\&\lesssim (\Vert v\Vert_{\mathfrak{Y}_1^{\alpha-1/2}}+\Vert w\Vert_{\mathfrak{Y}^{\alpha}_1}+\Vert Y_2\Vert_{\mathfrak{Y}_1^{1/2-\epsilon}})\Vert X_2\Vert_{C_1\mathscr{C}^{-1-\epsilon}},
\end{align*}and\begin{align*}
&\left\Vert \int_0^r\left(v\left(r\right)+w(r)+Y_2(r)-v\left(q\right)-w(q)-Y_2(q)\right)\varolessthan\left(G_{r-q}\left(\Delta\left(X_2\left(q\right)\right)\right)\right)\,\mathrm{d}q\right\Vert_{\mathscr{C}^{2\epsilon+1}}\\&\leq\int_{0}^r\left\Vert \left(v\left(r\right)+w(r)+Y_2(r)-v\left(q\right)-w(q)-Y_2(q)\right)\varolessthan\left(G_{r-q}\left(\Delta\left(X_2\left(q\right)\right)\right)\right)\right\Vert_{\mathscr{C}^{2\epsilon+1}}\,\mathrm{d}q\\
&\stackrel{\text{Theorem \ref{estimate}}}{\lesssim} \int_{0}^r\Vert v\left(r\right)+w(r)+Y_2(r)-v\left(q\right)-w(q)-Y_2(q)\Vert_{\mathscr{C}^{0}}\Vert G_{r-q}\left(\Delta\left(X_2\left(q\right)\right)\right)\Vert_{\mathscr{C}^{1+2\epsilon}}\,\mathrm{d}q\\
&\stackrel{\text{Theorem \ref{derivative}}}{\lesssim} (\Vert v\Vert_{\mathfrak{Y}_1^{\alpha-1/2}}+\Vert w\Vert_{\mathfrak{Y}^{\alpha}_1}+\Vert Y_2\Vert_{\mathfrak{Y}_1^{1/2-\epsilon}})\int_{0}^r|r-q|^{1/16}\Vert G_{r-q}\left(X_2\left(q\right)\right)\Vert_{\mathscr{C}^{3+2\epsilon}}\,\mathrm{d}q\\
&\stackrel{\text{Lemma \ref{lemma}}}{\lesssim} (\Vert v\Vert_{\mathfrak{Y}_1^{\alpha-1/2}}+\Vert w\Vert_{\mathfrak{Y}^{\alpha}_1}+\Vert Y_2\Vert_{\mathfrak{Y}_1^{1/2-\epsilon}})\int_{0}^r|r-q|^{1/16-(3+2\epsilon+1+\epsilon)/4}\Vert X_2\left(q\right)\Vert_{\mathscr{C}^{-1-\epsilon}}\,\mathrm{d}q\\&\lesssim (\Vert v\Vert_{\mathfrak{Y}_1^{\alpha-1/2}}+\Vert w\Vert_{\mathfrak{Y}^{\alpha}_1}+\Vert Y_2\Vert_{\mathfrak{Y}_1^{1/2-\epsilon}})\Vert X_2\Vert_{C_1\mathscr{C}^{-1-\epsilon}}.
\end{align*} We deduce from these estimates that there exists $\mu\in\mathbb{R}_+^*$ such that for every $(v,w)\in\mathfrak{Y}_1^{\alpha-1/2}\times\mathfrak{Y}_1^\alpha,$ $$\Vert\Lambda (v,w)\Vert_{C_1\mathscr{C}^{1+2\epsilon}}\leq \mu (\Vert v\Vert_{\mathfrak{Y}_1^{\alpha-1/2}}+\Vert w\Vert_{\mathfrak{Y}^{\alpha}_1}+\Vert Y_2\Vert_{\mathfrak{Y}_1^{1/2-\epsilon}})\Vert X_2\Vert_{C_1\mathscr{C}^{-1-\epsilon}}.$$
\end{proof}\begin{theorem-definition}[\cite{MR3406823,MR3719541}]
\label{theoremd=5}If $\left(d,\alpha\right) \in \left\{5\right\}\times \left]1,3/2\right[$ and $g \in \mathscr{C}^\alpha,$ then 
there exist a stopping time $\tau:\Omega\longrightarrow \left]0,1\right]$ and a unique stochastic process $\left(v,w\right)\in \mathfrak{Y}_\tau^{\alpha-1/2}\times\mathfrak{Y}_\tau^{\alpha}$ such that $$\forall U \in \left[0,\tau\right],v_U=3\int_0^UG_{U-q}(\Delta((v_q+w_q+Y_{q,2})\varolessthan X_{q,2}))\,\mathrm{d}q,$$ $$\forall U \in \left[0,\tau\right],w_U=G_Ug+\int_0^UG_{U-q}(\Delta((v_q+w_q)^3+9\Psi(v_q+w_q+Y_{q,2},Y_{q,1},X_{q,2})+3(\Lambda(v,w))(q)\varodot X_{q,2}+3 w_q\varodot X_{q,2}$$$$+3(v_q+w_q+Y_{q,2})\varogreaterthan X_{q,2}+(P(v+w))(q)))\,\mathrm{d}q,$$
The local mild solution of the 5-dimensional stochastic
Cahn-Hilliard equation is defined by the processes $(v_r+w_r+\widetilde{X}_{r,1}+\widetilde{Y}_{r,2})_{r\in[0,\tau]}$\end{theorem-definition}
\begin{proof}
For all $U \in [0,1],$ consider the fixed point map $(\Upsilon_{U},\widetilde{\Upsilon}_U):\mathfrak{Y}_U^{\alpha-1/2}\times\mathfrak{Y}_U^\alpha\longrightarrow \mathfrak{Y}_U^{\alpha-1/2}\times\mathfrak{Y}_U^\alpha$ such that for every $(v,w) \in \mathfrak{Y}_U^{\alpha-1/2}\times\mathfrak{Y}_U^\alpha,$
\begin{align*}
\Upsilon_{U}\left(v,w\right):\left[0,U\right]&\longrightarrow \mathscr{C}^{\alpha-1/2}\\
r&\longmapsto \left(\Upsilon_{U}\left(v,w\right)\right)\left(r\right)=3\int_0^rG_{r-q}(\Delta((v_q+w_q+Y_{q,2})\varolessthan X_{q,2}))\,\mathrm{d}q,
\end{align*} \begin{align*}
\widetilde{\Upsilon}_{U}\left(v,w\right):\left[0,U\right]&\longrightarrow \mathscr{C}^{\alpha}\\
r&\longmapsto \left(\widetilde{\Upsilon}_{U}\left(v,w\right)\right)\left(r\right)=G_rg+\int_0^rG_{r-q}(\Delta((v_q+w_q)^3+9\Psi(v_q+w_q+Y_{q,2},Y_{q,1},X_{q,2})
\end{align*}$$+3(\Lambda(v,w))(q)\varodot X_{q,2}+3 w_q\varodot X_{q,2}+3(v_q+w_q+Y_{q,2})\varogreaterthan X_{q,2}+(P(v+w))(q)))\,\mathrm{d}q$$
Let $\epsilon>0$ sufficiently small, $\gamma_1:=(\alpha+\epsilon+5/2)/4.$

We will prove that $(\Upsilon_{U},\widetilde{\Upsilon}_U)$ is a contraction on a suitable Banach space and for suitable $U \in [0,1[$.

It follows from Proposition \ref{proposition4} that for any $U \in [0,1]$ and all $v \in\mathfrak{Y}_U^{\alpha-1/2},w\in\mathfrak{Y}_U^\alpha,$ $$\sup_{\substack{\left(r,q\right)\in\left[0,U\right]^2\\ r\neq q}}\frac{\left\Vert (\Upsilon_{U}\left(v,w\right))\left(r\right)-(\Upsilon_{U}\left(v,w\right))\left(q\right)\right\Vert_{\mathscr{C}^0}}{\left|r-q\right|^{\frac{1}{16}}}\lesssim U^{1-\gamma_1}\sup_{q \in[0,U]}\Vert (v(q)+w(q)+Y_{q,2})\varolessthan X_{q,2}\Vert_{\mathscr{C}^{-1-\epsilon}}$$ and $$\sup_{\substack{\left(r,q\right)\in\left[0,U\right]^2\\ r\neq q}}\frac{\left\Vert (\widetilde{\Upsilon}_{U}\left(v,w\right))\left(r\right)-(\widetilde{\Upsilon}_{U}\left(v,w\right))\left(q\right)\right\Vert_{\mathscr{C}^0}}{\left|r-q\right|^{\frac{1}{16}}}\lesssim\Vert g\Vert_{\mathscr{C}^\alpha}$$$$+U^{1-\gamma_1}\sup_{q\in[0,U]}\Vert (v(q)+w(q))^3+9\Psi(v(q)+w(q)+Y_{q,2},Y_{q,1},X_{q,2})+3(\Lambda(v,w))(q)\varodot X_{q,2}+3 w(q)\varodot X_{q,2})$$$$+3(v(q)+w(q)+Y_{q,2})\varogreaterthan X_{q,2}+(P(v+w))(q)\Vert_{\mathscr{C}^{-1/2-\epsilon}}$$  
It follows from Theorem \ref{Schauder} that for any $U \in [0,1]$ and all $v \in\mathfrak{Y}_U^{\alpha-1/2},w\in\mathfrak{Y}_U^\alpha,$ 

$$\Vert\Upsilon_{U}\left(v,w\right)\Vert_{C_U\mathscr{C}^{\alpha-1/2}}\lesssim U^{1-\gamma_1}\sup_{q \in[0,U]}\Vert (v(q)+w(q)+Y_{q,2})\varolessthan X_{q,2}\Vert_{\mathscr{C}^{-1-\epsilon}}$$
and
$$\Vert\widetilde{\Upsilon}_{U}\left(v,w\right)\Vert_{C_U\mathscr{C}^{\alpha}}\lesssim\Vert g\Vert_{\mathscr{C}^\alpha}+U^{1-\gamma_1}\sup_{q\in[0,U]}\Vert (v(q)+w(q))^3+9\Psi(v(q)+w(q)+Y_{q,2},Y_{q,1},X_{q,2})$$$$+3(\Lambda(v,w))(q)\varodot X_{q,2}+3 w(q)\varodot X_{q,2})+3(v(q)+w(q)+Y_{q,2})\varogreaterthan X_{q,2}+(P(v+w))(q)\Vert_{\mathscr{C}^{-1/2-\epsilon}}$$
We deduce from Theorem \ref{estimate} that
$$\Vert (v(q)+w(q)+Y_{q,2})\varolessthan X_{q,2}\Vert_{\mathscr{C}^{-1-\epsilon}}\lesssim (\Vert v\Vert_{C_U\mathscr{C}^{\alpha-1/2}}+\Vert w\Vert_{C_U\mathscr{C}^{\alpha}}+\Vert Y_{2}\Vert_{C_U\mathscr{C}^{1/2-\epsilon}})\Vert X_2\Vert_{C_U\mathscr{C}^{-1-\epsilon}},$$ \begin{align*}
\Vert(\Lambda(v,w))(q)\varodot X_{q,2}\Vert_{\mathscr{C}^{-1/2-\epsilon}}&\lesssim\Vert(\Lambda(v,w))(q)\varodot X_{q,2}\Vert_{\mathscr{C}^{\epsilon}}\\&\lesssim\Vert(\Lambda(v,w))(q)\Vert_{\mathscr{C}^{2\epsilon+1}}\Vert X_{q,2}\Vert_{\mathscr{C}^{-1-\epsilon}}\\&\stackrel{\text{Proposition \ref{proposition3}}}{\lesssim} (\Vert v\Vert_{\mathfrak{Y}_1^{\alpha-1/2}}+\Vert w\Vert_{\mathfrak{Y}^{\alpha}_1}+\Vert Y_2\Vert_{\mathfrak{Y}_1^{1/2-\epsilon}})\Vert X_2\Vert_{C_1\mathscr{C}^{-1-\epsilon}}^2\end{align*} and \begin{align*}
\Vert (v(q)+w(q)+Y_{q,2})\varogreaterthan X_{q,2}\Vert_{\mathscr{C}^{-1/2-\epsilon}}&=\Vert X_{q,2}\varolessthan (v(q)+w(q)+Y_{q,2})\Vert_{\mathscr{C}^{-1/2-\epsilon}}\\&\lesssim (\Vert v\Vert_{C_U\mathscr{C}^{\alpha-1/2}}+\Vert w\Vert_{C_U\mathscr{C}^{\alpha}}+\Vert Y_{2}\Vert_{C_U\mathscr{C}^{1/2-\epsilon}})\Vert X_2\Vert_{C_U\mathscr{C}^{-1-\epsilon}},\end{align*} and $$\Vert w(q)\varodot X_{q,2}\Vert_{\mathscr{C}^{-1/2-\epsilon}}\lesssim\Vert X_2\Vert_{C_U\mathscr{C}^{-\epsilon-1}}\Vert w\Vert_{C_U\mathscr{C}^\alpha}.$$ It follows from Corollary \ref{product} that $$\Vert (v(q)+w(q))^3\Vert_{\mathscr{C}^{-1/2-\epsilon}}\lesssim \Vert (v(q)+w(q))^3\Vert_{\mathscr{C}^{\alpha-1/2}}\lesssim\Vert v(q)+w(q)\Vert_{\mathscr{C}^{\alpha-1/2}}^3\lesssim \Vert v\Vert^3_{C_U\mathscr{C}^{\alpha-1/2}}+\Vert w\Vert^3_{C_U\mathscr{C}^{\alpha}},$$ and $$\Vert (P(v+w))(q)\Vert_{\mathscr{C}^{-\epsilon-1/2}}\lesssim \Vert\Theta_6\Vert_{C_U\mathscr{C}^{-1/2-\epsilon}}(\Vert v\Vert_{C_{U}\mathscr{C}^{\alpha-1/2}}^2+\Vert w\Vert_{C_{U}\mathscr{C}^{\alpha}}^2)$$$$+\Vert \Theta_5\Vert_{C_{U}\mathscr{C}^{-1/2-\epsilon}}(\Vert v\Vert_{C_{U}\mathscr{C}^{\alpha-1/2}}+\Vert w\Vert_{C_{U}\mathscr{C}^{\alpha}})+\Vert \Theta_4\Vert_{C_{U}\mathscr{C}^{-1/2-\epsilon}}$$
We deduce from Proposition \ref{co1} that \begin{align*} \Vert\Psi(v(q)+w(q)+Y_{q,2},Y_{q,1},X_{q,2})\Vert_{\mathscr{C}^{-1-\epsilon}}&\lesssim \Vert v(q)+w(q)+Y_{q,2}\Vert_{\mathscr{C}^{1/2-\epsilon}}\Vert Y_{q,1}\Vert_{\mathscr{C}^{1-\epsilon}}\Vert X_{q,2}\Vert_{\mathscr{C}^{-1-\epsilon}}\\&\lesssim\Vert Y_{1}\Vert_{C_U\mathscr{C}^{1-\epsilon}}\Vert X_{2}\Vert_{C_U\mathscr{C}^{-1-\epsilon}}(\Vert v\Vert_{C_{U}\mathscr{C}^{\alpha-1/2}}+\Vert w\Vert_{C_{U}\mathscr{C}^{\alpha}}+\Vert Y_2\Vert_{C_U\mathscr{C}^{1/2-\epsilon}})\end{align*} Therefore there exists $\mu_1 \in\mathbb{R}_+,$ such that for $(v,w)\in\mathfrak{Y}_U^{\alpha-1/2}\times\mathfrak{Y}_U^{\alpha},$ $$\Vert\Upsilon_{U}\left(v,w\right)\Vert_{C_U\mathscr{C}^{\alpha-1/2}}\leq\mu_1 U^{1-\gamma_1}(\Vert v\Vert_{C_U\mathscr{C}^{\alpha-1/2}}+\Vert w\Vert_{C_U\mathscr{C}^{\alpha}}+\Vert Y_{2}\Vert_{C_U\mathscr{C}^{1/2-\epsilon}})\Vert X_2\Vert_{C_U\mathscr{C}^{-1-\epsilon}}$$ and $$\Vert\widetilde{\Upsilon}_{U}\left(v,w\right)\Vert_{C_U\mathscr{C}^{\alpha}}\leq \mu_1\Vert g\Vert_{\mathscr{C}^\alpha}+\mu_1U^{1-\gamma_1}(\Vert v\Vert^3_{C_U\mathscr{C}^{\alpha-1/2}}+\Vert w\Vert^3_{C_U\mathscr{C}^\alpha}+\Vert Y_{1}\Vert_{C_U\mathscr{C}^{1-\epsilon}}\Vert X_{2}\Vert_{C_U\mathscr{C}^{-1-\epsilon}}(\Vert v\Vert_{C_{U}\mathscr{C}^{\alpha-1/2}}$$$$+\Vert w\Vert_{C_{U}\mathscr{C}^{\alpha}}+\Vert Y_2\Vert_{C_U\mathscr{C}^{1/2-\epsilon}})+(\Vert v\Vert_{\mathfrak{Y}_1^{\alpha-1/2}}+\Vert w\Vert_{\mathfrak{Y}^{\alpha}_1}+\Vert Y_2\Vert_{\mathfrak{Y}_1^{1/2-\epsilon}})\Vert X_2\Vert_{C_1\mathscr{C}^{-1-\epsilon}}^2+\Vert X_2\Vert_{C_U\mathscr{C}^{-\epsilon-1}}\Vert w\Vert_{C_U\mathscr{C}^\alpha}+(\Vert v\Vert_{C_U\mathscr{C}^{\alpha-1/2}}$$$$+\Vert w\Vert_{C_U\mathscr{C}^{\alpha}}+\Vert Y_{2}\Vert_{C_U\mathscr{C}^{1/2-\epsilon}})\Vert X_2\Vert_{C_U\mathscr{C}^{-1-\epsilon}}+\Vert\Theta_6\Vert_{C_U\mathscr{C}^{-1/2-\epsilon}}(\Vert v\Vert_{C_{U}\mathscr{C}^{\alpha-1/2}}^2+\Vert w\Vert_{C_{U}\mathscr{C}^{\alpha}}^2)$$$$+\Vert \Theta_5\Vert_{C_{U}\mathscr{C}^{-1/2-\epsilon}}(\Vert v\Vert_{C_{U}\mathscr{C}^{\alpha-1/2}}+\Vert w\Vert_{C_{U}\mathscr{C}^{\alpha}})+\Vert \Theta_4\Vert_{C_{U}\mathscr{C}^{-1/2-\epsilon}})$$
We repeat this argument for $\Upsilon_{U}\left(v_1,w_1\right)-\Upsilon_{U}\left(v_2,w_2\right)$ and $\widetilde{\Upsilon}_{U}\left(v_1,w_1\right)-\widetilde{\Upsilon}_{U}\left(v_2,w_2\right)$ to deduce for $\left(v_1,w_1,v_2,w_2\right) \in \mathfrak{Y}_U^{\alpha-1/2}\times\mathfrak{Y}_U^{\alpha}\times\mathfrak{Y}_U^{\alpha-1/2}\times\mathfrak{Y}_U^{\alpha}$ that  
$$\left\Vert\Upsilon_{U}\left(v_1,w_1\right)-\Upsilon_{U}\left(v_2,w_2\right)\right\Vert_{\mathfrak{Y}_U^{\alpha-1/2}}\leq \mu_2U^{1-\gamma_1}(\left\Vert v_1-v_2\right\Vert_{\mathfrak{Y}_U^{\alpha-1/2}}+\Vert w_1-w_2\Vert_{\mathfrak{Y}_U^{\alpha}})$$ and $$\Vert\widetilde{\Upsilon}_{U}\left(v_1,w_1\right)-\widetilde{\Upsilon}_{U}\left(v_2,w_2\right)\Vert_{\mathfrak{Y}_U^{\alpha}}\leq\mu_2U^{1-\gamma_1}\max(\left\Vert v_1-v_2\right\Vert_{\mathfrak{Y}_U^{\alpha-1/2}},\Vert w_1-w_2\Vert_{\mathfrak{Y}_U^{\alpha}})$$$$(\Vert v_1\Vert^2_{C_U\mathscr{C}^{\alpha-1/2}}+\Vert v_2\Vert^2_{C_U\mathscr{C}^{\alpha-1/2}}+\Vert w_1\Vert^2_{C_U\mathscr{C}^\alpha}+\Vert w_2\Vert^2_{C_U\mathscr{C}^{\alpha}}+\Vert Y_{1}\Vert_{C_U\mathscr{C}^{1-\epsilon}}\Vert X_{2}\Vert_{C_U\mathscr{C}^{-1-\epsilon}}+\Vert X_2\Vert_{C_1\mathscr{C}^{-1-\epsilon}}^2+\Vert X_2\Vert_{C_U\mathscr{C}^{-1-\epsilon}}$$$$+\Vert\Theta_6\Vert_{C_U\mathscr{C}^{-1/2-\epsilon}}(\Vert v_1\Vert_{C_{U}\mathscr{C}^{\alpha-1/2}}+\Vert v_2\Vert_{C_{U}\mathscr{C}^{\alpha-1/2}}+\Vert w_1\Vert_{C_{U}\mathscr{C}^{\alpha}}+\Vert w_2\Vert_{C_{U}\mathscr{C}^{\alpha}})+\Vert \Theta_5\Vert_{C_{U}\mathscr{C}^{-1/2-\epsilon}}.$$
We can choose $U \in \left]0,1\right]$ such that for every $\left(v_1,v_2,w_1,w_2\right) \in \mathfrak{Y}_U^{\alpha-1/2}\times\mathfrak{Y}_U^{\alpha-1/2}\times\mathfrak{Y}_U^\alpha\times\mathfrak{Y}_U^\alpha,$ if 
$$
\max\left(\left\Vert v_1\right\Vert_{\mathfrak{Y}_U^{\alpha-1/2}},\left\Vert v_2\right\Vert_{\mathfrak{Y}_U^{\alpha-1/2}},\Vert w_1\Vert_{\mathfrak{Y}_U^\alpha},\Vert w_2\Vert_{\mathfrak{Y}_U^\alpha}\right)\leq\mu\left\Vert g\right\Vert_{\mathscr{C}^\alpha}+1,
$$ 
then 
$$
\max\left(\left\Vert\Upsilon_{U}\left(v_1,w_1\right)\right\Vert_{\mathfrak{Y}_U^{\alpha-1/2}},\Vert \widetilde{\Upsilon}_U(v_1,w_1)\Vert_{\mathfrak{Y}_U^\alpha}\right)\leq\mu \left\Vert g\right\Vert_{\mathscr{C}^\alpha}+1,
$$ 
$$
\left\Vert\Upsilon_{U}\left(v_2,w_2\right)-\Upsilon_{U}\left(v_1,w_1\right)\right\Vert_{\mathfrak{Y}_U^{\alpha-1/2}}\leq \frac{1}{2}\max\left(\left\Vert v_2-v_1\right\Vert_{\mathfrak{Y}_U^{\alpha-1/2}},\Vert w_2-w_1\Vert_{\mathfrak{Y}^\alpha_U}\right),
$$ 
and 
$$
\left\Vert\widetilde{\Upsilon}_{U}\left(v_2,w_2\right)-\widetilde{\Upsilon}_{U}\left(v_1,w_1\right)\right\Vert_{\mathfrak{Y}_U^\alpha}\leq \frac{1}{2}\max\left(\left\Vert v_2-v_1\right\Vert_{\mathfrak{Y}_U^{\alpha-1/2}},\Vert w_2-w_1\Vert_{\mathfrak{Y}^\alpha_U}\right).
$$
Applying the Banach fixed point theorem we deduce that there exists a unique $(v,w)\in\mathfrak{Y}_U^{\alpha-1/2}\times\mathfrak{Y}_U^\alpha$ such that $(\Upsilon_U(v,w),\widetilde{\Upsilon}_U(v,w))=(v,w).$  
\end{proof}

\chapter{Future Research Directions}
\label{chapter4}

In this thesis we proved local well-posedness of the stochastic Cahn-Hilliard equation in the full subcritical regime. My future research will extend this program along several lines. Three immediate directions related to the stochastic Cahn-Hilliard equation are: 
\begin{enumerate}
\item Piecewise linear approximation of the stochastic Cahn-Hilliard.

In this paper, a Wong-Zakai (see \cite{MR3417505,MR4056954}) type of approximation of the stochastic Cahn-Hilliard equation is considered. The proof relies on paracontrolled distributions. 
\item Lattice approximation of the stochastic Cahn-Hilliard equation:
We investigate robust discretization methods inspired by paracontrolled calculus. When the equation is well-posed, and we approximate it by replacing the differentiation operator by discretization schemes, it is speculated that the approximating equation should converge to the solution of the original equation as the discretization approaches zero. More precisely, I aim to prove that the solutions of discrete lattice systems converge to those of the original equation.
\item Global-in-time well-posedness of the stochastic Cahn-Hilliard equation:
The objective here is to rule out finite time blow-up of solutions, more precisely that the explosion time is actually infinite, and therefore the solution exists globally in time. I will study the problem using two completely different arguments: 
\begin{enumerate}
\item PDE construction relying on apriori bounds and energy estimates (see \cite{MR3719541,MR3693966,MR3865664}). 
\item employing invariant measures to rule out finite time blow-up of solutions.
\end{enumerate} 
\end{enumerate}



\bibliographystyle{plainurl}
\cleardoublepage 
\phantomsection  
\renewcommand*{\bibname}{References}

\addcontentsline{toc}{chapter}{\textbf{References}}

\bibliography{uw-ethesis.bib}

@book {MR3445609,
    AUTHOR = {Gubinelli, Massimiliano and Perkowski, Nicolas},
     TITLE = {Lectures on singular stochastic {PDE}s},
 PUBLISHER = {Sociedade Brasileira de Matem\'atica, Rio de Janeiro},
      YEAR = {2015},
}

@article {MR3846835,
    AUTHOR = {Catellier, R\'emi and Chouk, Khalil},
     TITLE = {Paracontrolled distributions and the 3-dimensional stochastic
              quantization equation},
   JOURNAL = {Ann. Probab.},
  FJOURNAL = {The Annals of Probability},
    VOLUME = {46},
      YEAR = {2018},
    NUMBER = {5},
     PAGES = {2621--2679},
}

@article {MR3406823,
    AUTHOR = {Gubinelli, Massimiliano and Imkeller, Peter and Perkowski,
              Nicolas},
     TITLE = {Paracontrolled distributions and singular {PDE}s},
   JOURNAL = {Forum Math. Pi},
  FJOURNAL = {Forum of Mathematics. Pi},
    VOLUME = {3},
      YEAR = {2015},
     PAGES = {e6, 75},
}

@article {MR3592748,
    AUTHOR = {Gubinelli, Massimiliano and Perkowski, Nicolas},
     TITLE = {K{PZ} reloaded},
   JOURNAL = {Comm. Math. Phys.},
  FJOURNAL = {Communications in Mathematical Physics},
    VOLUME = {349},
      YEAR = {2017},
    NUMBER = {1},
     PAGES = {165--269},
}

@article {MR3071506,
    AUTHOR = {Hairer, Martin},
     TITLE = {Solving the {KPZ} equation},
   JOURNAL = {Ann. of Math. (2)},
  FJOURNAL = {Annals of Mathematics. Second Series},
    VOLUME = {178},
      YEAR = {2013},
    NUMBER = {2},
     PAGES = {559--664},
}

@article {MR3274562,
    AUTHOR = {Hairer, Martin},
     TITLE = {A theory of regularity structures},
   JOURNAL = {Invent. Math.},
  FJOURNAL = {Inventiones Mathematicae},
    VOLUME = {198},
      YEAR = {2014},
    NUMBER = {2},
     PAGES = {269--504},
}

@article {MR2016604,
    AUTHOR = {Da Prato, Giuseppe and Debussche, Arnaud},
     TITLE = {Strong solutions to the stochastic quantization equations},
   JOURNAL = {Ann. Probab.},
  FJOURNAL = {The Annals of Probability},
    VOLUME = {31},
      YEAR = {2003},
    NUMBER = {4},
     PAGES = {1900--1916},
}

@book {MR2768550,
    AUTHOR = {Bahouri, Hajer and Chemin, Jean-Yves and Danchin, Rapha\"el},
     TITLE = {Fourier analysis and nonlinear partial differential equations},
 PUBLISHER = {Springer, Heidelberg},
      YEAR = {2011},
}

@book {MR2200233,
    AUTHOR = {Nualart, David},
     TITLE = {The {M}alliavin calculus and related topics},
   EDITION = {Second},
 PUBLISHER = {Springer-Verlag, Berlin},
      YEAR = {2006},
}

@inproceedings {MR3746744,
    AUTHOR = {Mourrat, Jean-Christophe and Weber, Hendrik and Xu, Weijun},
    EDITOR = {Gon{\c{c}}alves, Patr{\'i}cia and Soares, Ana Jacinta}, 
     TITLE = {Construction of {$\Phi^4_3$} diagrams for pedestrians},
 BOOKTITLE = {From particle systems to partial differential equations},
    VOLUME = {209},
     PAGES = {1--46},
 PUBLISHER = {Springer, Cham},
      YEAR = {2017},
}

@article {MR1359472,
    AUTHOR = {Da Prato, Giuseppe and Debussche, Arnaud},
     TITLE = {Stochastic {C}ahn-{H}illiard equation},
   JOURNAL = {Nonlinear Anal.},
  FJOURNAL = {Nonlinear Analysis. Theory, Methods \& Applications. An
              International Multidisciplinary Journal},
    VOLUME = {26},
      YEAR = {1996},
    NUMBER = {2},
     PAGES = {241--263},
}

@article {MR0631751,
    AUTHOR = {Bony, Jean-Michel},
     TITLE = {Calcul symbolique et propagation des singularit\'es pour les
              \'equations aux d\'eriv\'ees partielles non lin\'eaires},
   JOURNAL = {Ann. Sci. \'Ecole Norm. Sup. (4)},
  FJOURNAL = {Annales Scientifiques de l'\'Ecole Normale Sup\'erieure.
              Quatri\`eme S\'erie},
    VOLUME = {14},
      YEAR = {1981},
    NUMBER = {2},
     PAGES = {209--246},
}

@article {MR3693966,
    AUTHOR = {Mourrat, Jean-Christophe and Weber, Hendrik},
     TITLE = {Global well-posedness of the dynamic {$\Phi^4$} model in the
              plane},
   JOURNAL = {Ann. Probab.},
  FJOURNAL = {The Annals of Probability},
    VOLUME = {45},
      YEAR = {2017},
    NUMBER = {4},
     PAGES = {2398--2476},
}

@article {MR4191230,
    AUTHOR = {Inahama, Yuzuru and Sawano, Yoshihiro},
     TITLE = {Paracontrolled quasi-geostrophic equation with space-time
              white noise},
   JOURNAL = {Dissertationes Math.},
  FJOURNAL = {Dissertationes Mathematicae},
    VOLUME = {558},
      YEAR = {2020},
     PAGES = {1--81},
}

@article {MR3719541,
    AUTHOR = {Mourrat, Jean-Christophe and Weber, Hendrik},
     TITLE = {The dynamic {$\Phi^4_3$} model comes down from infinity},
   JOURNAL = {Comm. Math. Phys.},
  FJOURNAL = {Communications in Mathematical Physics},
    VOLUME = {356},
      YEAR = {2017},
    NUMBER = {3},
     PAGES = {673--753},
}

@article {MR3742401,
    AUTHOR = {Hoshino, Masato and Inahama, Yuzuru and Naganuma, Nobuaki},
     TITLE = {Stochastic complex {G}inzburg-{L}andau equation with
              space-time white noise},
   JOURNAL = {Electron. J. Probab.},
  FJOURNAL = {Electronic Journal of Probability},
    VOLUME = {22},
      YEAR = {2017},
    NUMBER = {104}, 
     PAGES = {1--68},
}

@article {MR4556829,
    AUTHOR = {Yamazaki, Kazuo},
     TITLE = {Three-dimensional magnetohydrodynamics system forced by
              space-time white noise},
   JOURNAL = {Electron. J. Probab.},
  FJOURNAL = {Electronic Journal of Probability},
    VOLUME = {28},
      YEAR = {2023},
    NUMBER = {39},
     PAGES = {1--66},
}

@article {MR4281449,
    AUTHOR = {Yamazaki, Kazuo},
     TITLE = {Strong {F}eller property of the magnetohydrodynamics system
              forced by space-time white noise},
   JOURNAL = {Nonlinearity},
  FJOURNAL = {Nonlinearity},
    VOLUME = {34},
      YEAR = {2021},
    NUMBER = {6},
     PAGES = {4373--4463},
}

@article {MR3373412,
    AUTHOR = {Zhu, Rongchan and Zhu, Xiangchan},
     TITLE = {Three-dimensional {N}avier-{S}tokes equations driven by
              space-time white noise},
   JOURNAL = {J. Differential Equations},
  FJOURNAL = {Journal of Differential Equations},
    VOLUME = {259},
      YEAR = {2015},
    NUMBER = {9},
     PAGES = {4443--4508},
}

@article {MR3825880,
    AUTHOR = {Tsatsoulis, Pavlos and Weber, Hendrik},
     TITLE = {Spectral gap for the stochastic quantization equation on the
              2-dimensional torus},
   JOURNAL = {Ann. Inst. Henri Poincar\'e{} Probab. Stat.},
  FJOURNAL = {Annales de l'Institut Henri Poincar\'e{} Probabilit\'es et
              Statistiques},
    VOLUME = {54},
      YEAR = {2018},
    NUMBER = {3},
     PAGES = {1204--1249},
}

@article {MR4908981,
    AUTHOR = {Martini, Adrian and Mayorcas, Avi},
     TITLE = {An additive-noise approximation to
              {K}eller-{S}egel-{D}ean-{K}awasaki dynamics: local
              well-posedness of paracontrolled solutions},
   JOURNAL = {Stoch. Partial Differ. Equ. Anal. Comput.},
  FJOURNAL = {Stochastics and Partial Differential Equations. Analysis and
              Computations},
    VOLUME = {13},
      YEAR = {2025},
    NUMBER = {2},
     PAGES = {956--1033}
}

@article {MR4721025,
    AUTHOR = {Gubinelli, Massimiliano and Koch, Herbert and Oh, Tadahiro},
     TITLE = {Paracontrolled approach to the three-dimensional stochastic
              nonlinear wave equation with quadratic nonlinearity},
   JOURNAL = {J. Eur. Math. Soc.},
  FJOURNAL = {Journal of the European Mathematical Society},
    VOLUME = {26},
      YEAR = {2024},
    NUMBER = {3},
     PAGES = {817--874}
}

@article {MR3841850,
    AUTHOR = {Gubinelli, Massimiliano and Koch, Herbert and Oh, Tadahiro},
     TITLE = {Renormalization of the two-dimensional stochastic nonlinear
              wave equations},
   JOURNAL = {Trans. Amer. Math. Soc.},
  FJOURNAL = {Transactions of the American Mathematical Society},
    VOLUME = {370},
      YEAR = {2018},
    NUMBER = {10},
     PAGES = {7335--7359}
}

@article {MR1941997,
    AUTHOR = {Da Prato, Giuseppe and Debussche, Arnaud},
     TITLE = {Two-dimensional {N}avier-{S}tokes equations driven by a
              space-time white noise},
   JOURNAL = {J. Funct. Anal.},
  FJOURNAL = {Journal of Functional Analysis},
    VOLUME = {196},
      YEAR = {2002},
    NUMBER = {1},
     PAGES = {180--210}
}

@article {MR3508062,
    AUTHOR = {E, Weinan and Jentzen, Arnulf and Shen, Hao},
     TITLE = {Renormalized powers of {O}rnstein-{U}hlenbeck processes and
              well-posedness of stochastic {G}inzburg-{L}andau equations},
   JOURNAL = {Nonlinear Anal.},
  FJOURNAL = {Nonlinear Analysis. Theory, Methods \& Applications. An
              International Multidisciplinary Journal},
    VOLUME = {142},
      YEAR = {2016},
     PAGES = {152--193}
}

@article{COOK1970297,
author = {Cook, Harry E. },
title = {Brownian motion in spinodal decomposition},
journal = {Acta Metall.},
volume = {18},
number = {3},
pages = {297-306},
year = {1970}
}

@article{10.1063/1.1744102,
    author = {Cahn, John W. and Hilliard, John E.},
    title = {Free Energy of a Nonuniform System. I. Interfacial Free Energy},
    journal = {J. Chem. Phys.},
    volume = {28},
    number = {2},
    pages = {258-267},
    year = {1958}
}

@book {MR4174393,
    AUTHOR = {Friz, Peter K. and Hairer, Martin},
     TITLE = {A course on rough paths},
    SERIES = {Universitext},
   EDITION = {Second},
      NOTE = {With an introduction to regularity structures},
 PUBLISHER = {Springer, Cham},
      YEAR = {2020}
}

@article {MR3865664,
    AUTHOR = {Hoshino, Masato},
     TITLE = {Global well-posedness of complex {G}inzburg-{L}andau equation
              with a space-time white noise},
   JOURNAL = {Ann. Inst. Henri Poincar\'e{} Probab. Stat.},
  FJOURNAL = {Annales de l'Institut Henri Poincar\'e{} Probabilit\'es et
              Statistiques},
    VOLUME = {54},
      YEAR = {2018},
    NUMBER = {4},
     PAGES = {1969--2001}
}

@article {MR3417505,
    AUTHOR = {Hairer, Martin and Pardoux, \'Etienne},
     TITLE = {A {W}ong-{Z}akai theorem for stochastic {PDE}s},
   JOURNAL = {J. Math. Soc. Japan},
  FJOURNAL = {Journal of the Mathematical Society of Japan},
    VOLUME = {67},
      YEAR = {2015},
    NUMBER = {4},
     PAGES = {1551--1604}
}

@article {MR4056954,
    AUTHOR = {Zhu, Rongchan and Zhu, Xiangchan},
     TITLE = {Piecewise linear approximation for the dynamical {$\Phi_3^4$}
              model},
   JOURNAL = {Sci. China Math.},
  FJOURNAL = {Science China. Mathematics},
    VOLUME = {63},
      YEAR = {2020},
    NUMBER = {2},
     PAGES = {381--410},
}

@article {MR267632,
    AUTHOR = {Garsia, A. M. and Rodemich, E. and Rumsey, Jr., H.},
     TITLE = {A real variable lemma and the continuity of paths of some
              {G}aussian processes},
   JOURNAL = {Indiana Univ. Math. J.},
  FJOURNAL = {Indiana University Mathematics Journal},
    VOLUME = {20},
      YEAR = {1970/71},
     PAGES = {565--578}
}

@book{lecture,
    AUTHOR = {Willem van Zuijlen},
     TITLE = {Theory of function and distribution spaces},
     URL = {https://e.pcloud.link/publink/show?code=kZOtmwZWFBPgBt97fQ05fCNeKzDHS9GbsxX}
}

@article {MR3935036,
    AUTHOR = {Bruned, Y. and Hairer, M. and Zambotti, L.},
     TITLE = {Algebraic renormalisation of regularity structures},
   JOURNAL = {Invent. Math.},
  FJOURNAL = {Inventiones Mathematicae},
    VOLUME = {215},
      YEAR = {2019},
    NUMBER = {3},
     PAGES = {1039--1156},
}

@article {MR4210726,
    AUTHOR = {Bruned, Y. and Chandra, A. and Chevyrev, I. and Hairer, M.},
     TITLE = {Renormalising {SPDE}s in regularity structures},
   JOURNAL = {J. Eur. Math. Soc. (JEMS)},
  FJOURNAL = {Journal of the European Mathematical Society (JEMS)},
    VOLUME = {23},
      YEAR = {2021},
    NUMBER = {3},
     PAGES = {869--947},
}

@article{chandra2016analytic,
  title={An analytic {BPHZ} theorem for regularity structures},
  author={Chandra, Ajay and Hairer, Martin},
  year={2016},
}

\nocite{*}


\appendix
\chapter*{APPENDICES}
\addcontentsline{toc}{chapter}{APPENDICES}
\chapter{}
For all $m \in \mathbb{N}^*,$ let 
\begin{align*}
\text{\gls{f8}}:L^2\left(\left(\mathbb{R}_+\times\mathbb{T}^d\right)^m\right)&\longrightarrow L^2\left(\Omega\right)\\
\varphi&\longmapsto V_m\left(\varphi\right)=\int_{\left(\mathbb{R}_+\times \mathbb{T}^d\right)^m}\varphi\left(x_1,\dots,x_m\right)\,\xi\left(\mathrm{d}x_1\right)\,\dots\,\xi\left(\mathrm{d}x_m\right)
\end{align*}
with $V_0(\varphi)=\varphi.$\\
The following estimate is needed to control high moments of Wiener-It\^o integrals.
\begin{theorem}[Nelson estimate, Proposition 4, \cite{MR3746744}]
\label{nelson}
If $\left(q,r\right)\in\mathbb{N}^*\times\mathbb{R}_+^*,$ then $$\forall \varphi \in L^2\left(\left(\mathbb{R}_+\times \mathbb{T}^d\right)^q\right),\left(\mathds{E}\left[\left|V_q\left(\varphi\right)\right|^r\right]\right)^{\frac{2}{r}}\leq \left(\max\left(r-1,1\right)\right)^q\mathds{E}\left[\left|V_q\left(\varphi\right)\right|^2\right]\leq\left(\max\left(r-1,1\right)\right)^q \Vert\varphi\Vert^2_{L^2(\left(\mathbb{R}_+\times \mathbb{T}^d\right)^q)}.$$
\end{theorem}
We will also need the following lemma.  
\begin{lemma}[Proposition 1.1.3, \cite{MR2200233}]
\label{l3}
Let $p,q\in\mathbb{N}^*,f \in L^2((\mathbb{R}_+\times \mathbb{T}^d)^p),g \in L^2((\mathbb{R}_+\times \mathbb{T}^d)^q)$ be symmetric functions. Therefore $$V_p(f)V_q(g)=\sum_{r=0}^{\min(p,q)}r!\binom{p}{r}\binom{q}{r}V_{p+q-2r}(f\otimes_rg),$$ where $\otimes_r$ is defined by: $$(f\otimes_r g)((y_1,x_1),...,(y_{p+q-2r},x_{p+q-2r}))$$$$:=\int_{(\mathbb{R}_+\times\mathbb{T}^d)^r}f((y_1,x_{1}),...,(y_{p-r},x_{p-r}),u)g(u,(y_{p-r+1},x_{p-r+1}),...,(y_{p+q-2r},x_{p+q-2r}))\,\mathrm{d}u,$$ for $r=0,(f\otimes_0g)(y,x)=f(y)g(x).$
\end{lemma}

The following simple observation is needed to compute the regularity of the stochastic objects. 
\begin{lemma}
\label{l33}
If $0\leq x\leq \min(y,u),$ then $$\forall \gamma \in ]0,1[,x=x^{\gamma}x^{1-\gamma}\leq y^{\gamma} u^{1-\gamma}$$ 
\end{lemma}

The following version of Kolmogorov continuity theorem is needed for the construction of the stochastic objects in Besov spaces.
\begin{theorem}
\label{kolmogorov}
Let $\left(f_r\right)_{r \in \mathbb{R}_+}$ be a stochastic process of state space $\left(\mathscr{S}',\mathcal{B}\left(\mathscr{S}'\right)\right)$ and $\left(\alpha,\beta,\theta\right) \in \mathbb{R}^2\times \mathbb{R}_+^*$ such that $\beta<\alpha$ and $\theta>\max\left(\frac{2d}{\alpha-\beta},1\right).$ If for all $U \in \mathbb{R}_+,$ there exists $\left(\mu,\gamma\right) \in \left(\mathbb{R}_+^*\right)^2$ such that for every $\left(r_1,r_2,j,x\right) \in \left[0,U\right]^2\times \mathbb{N}\times \mathbb{T}^d,r_1\neq r_2,$ $$\left(\mathds{E}\left[\left|\left(\delta_{j-1}f_0\right)\left(x\right)\right|^\theta\right]\right)^{\frac{1}{\theta}}+\left|r_2-r_1\right|^{-\gamma-\frac{1}{\theta}}\left(\mathds{E}\left[\left|\left(\delta_{j-1}f_{r_2}\right)\left(x\right)-\left(\delta_{j-1}f_{r_1}\right)\left(x\right)\right|^\theta\right]\right)^{\frac{1}{\theta}}\leq \mu2^{-\alpha\left(j-1\right)},$$ then there exists a stochastic process $\left(g_r\right)_{r \in \mathbb{R}_+}$ of state space $\left(\mathscr{C}^\beta,\mathcal{B}\left(\mathscr{C}^\beta\right)\right)$ such that the sample paths of $\left(g_q\right)_{q \in \mathbb{R}_+}$ are continuous and for all $r \in \mathbb{R}_+$ and for $\mathds{P}$-almost every $\omega \in \Omega,g_r\left(\omega\right)=f_r\left(\omega\right).$
\end{theorem}
\begin{proof}
Let $U \in \mathbb{R}_+$ and $m:=\frac{\alpha+\beta}{2}.$

There exists $\left(\mu_1,\mu_2,\mu_3,\mu_4,\gamma\right)\in \left(\mathbb{R}_+^*\right)^5$ such that for every $\left(r_1,r_2\right) \in \left[0,U\right]^2,$
\begin{align*}
\mathds{E}\left[\left\Vert f_{r_2}-f_{r_1}\right\Vert^\theta_{\mathscr{C}^\beta}\right]&\leq \mu_1\mathds{E}\left[\left\Vert f_{r_2}-f_{r_1}\right\Vert^\theta_{\mathscr{C}^{m-\frac{d}{\theta}}}\right]\\
&\leq \mu_2\mathds{E}\left[\left\Vert f_{r_2}-f_{r_1}\right\Vert^\theta_{B^{m}_{\theta,\theta}}\right]\\
&\leq \mu_2\sum_{j \in \mathbb{N}}2^{m\theta\left(j-1\right)}\int_{\mathbb{T}^d}\mathds{E}\left[\left|\left(\delta_{j-1}f_{r_2}\right)\left(x\right)-\left(\delta_{j-1}f_{r_1}\right)\left(x\right)\right|^\theta\right]\,\mathrm{d}x\\
&\leq \mu_3\left|r_2-r_1\right|^{\gamma \theta+1}\sum_{j \in \mathbb{N}}2^{\theta\left(j-1\right)\left(m-\alpha\right)}\\
&\leq \mu_4\left|r_2-r_1\right|^{\gamma \theta+1}.
\end{align*}

Applying the Kolmogorov continuity theorem completes the proof.
\end{proof}

\chapter{}
The following lemma gives a bound on discrete convolutions.
\begin{lemma}[Lemma 5, \cite{MR3746744}]
\label{lemma2}
Let $\left(\alpha,\beta\right) \in \mathbb{R}^2.$ If $\max\left(\alpha,\beta\right)<d$ and $\alpha+\beta>d,$ then $$\exists \mu \in \mathbb{R}_+^*,\forall q\in \mathbb{Z}^d,\sum_{m \in \mathbb{Z}^d}\left(1+\left|q-m\right|^4\right)^{-\frac{\alpha}{4}}\left(1+\left|m\right|^{4}\right)^{-\frac{\beta}{4}}\leq \mu\left(1+\left|q\right|^4\right)^{\frac{d-\alpha-\beta}{4}}.$$
\end{lemma}
\begin{proof}
Let $$J_1:=\left\{\left(m_1,m_2\right) \in \left(\mathbb{Z}^d\right)^2,\left|m_1\right|\geq 2\left|m_1+m_2\right|\right\},$$ $$J_2:=\left\{\left(m_1,m_2\right) \in \left(\mathbb{Z}^d\right)^2,\left|m_1\right|\leq \frac{1}{2}\left|m_1+m_2\right|\right\},$$ $$J_3:=\left\{\left(m_1,m_2\right) \in\left(\mathbb{Z}^d\right)^2,\left|m_2\right|\leq\frac{1}{2}|m_1+m_2|\right\},$$ and $$J_4:=\left(\mathbb{Z}^d\right)^2-\left(J_1\cup J_2 \cup J_3\right).$$

There exists $\left(\mu_1,\mu_2\right) \in \left(\mathbb{R}_+^*\right)^2$ such that for all $q \in \mathbb{Z}^d,$
\begin{align*}
\sum_{m\in\mathbb{Z}^d}\left(1+\left|q-m\right|^4\right)^{-\frac{\alpha}{4}}\left(1+\left|m\right|^4\right)^{-\frac{\beta}{4}}&\leq\sum_{m\in \mathbb{Z}^d}\left(1+\left|q-m\right|^4\right)^{-\frac{\alpha}{4}}\left(1+\left|m\right|^4\right)^{-\frac{\beta}{4}}\mathds{1}_{J_1}\left(m,q-m\right)\\
& \phantom{\leq}\ +\sum_{m\in \mathbb{Z}^d}\left(1+\left|q-m\right|^4\right)^{-\frac{\alpha}{4}}\left(1+\left|m\right|^4\right)^{-\frac{\beta}{4}}\mathds{1}_{J_2}\left(m,q-m\right)\\
& \phantom{\leq}\ +\sum_{m\in \mathbb{Z}^d}\left(1+\left|q-m\right|^4\right)^{-\frac{\alpha}{4}}\left(1+\left|m\right|^4\right)^{-\frac{\beta}{4}}\mathds{1}_{J_3}\left(m,q-m\right)\\
& \phantom{\leq}\ +\sum_{m\in \mathbb{Z}^d}\left(1+\left|q-m\right|^4\right)^{-\frac{\alpha}{4}}\left(1+\left|m\right|^4\right)^{-\frac{\beta}{4}}\mathds{1}_{J_4}\left(m,q-m\right)\\
&\leq \mu_1\sum_{m\in \mathbb{Z}^d}\left(1+\left|m\right|^4\right)^{-\frac{\alpha+\beta}{4}}\mathds{1}_{J_1}\left(m,q-m\right)\\
& \phantom{\leq}\ +\mu_1\left(1+\left|q\right|^4\right)^{-\frac{\alpha}{4}}\sum_{m \in \mathbb{Z}^d}\left(1+\left|m\right|^4\right)^{-\frac{\beta}{4}}\mathds{1}_{J_2}\left(m,q-m\right)\\
& \phantom{\leq}\ +\mu_1\left(1+\left|q\right|^4\right)^{-\frac{\beta}{4}}\sum_{m \in \mathbb{Z}^d}\left(1+\left|q-m\right|^4\right)^{-\frac{\alpha}{4}}\mathds{1}_{J_3}\left(m,q-m\right)\\
& \phantom{\leq}\ +\mu_1\left(1+\left|q\right|^4\right)^{-\frac{\alpha}{4}}\sum_{m \in \mathbb{Z}^d}\left(1+\left|m\right|^4\right)^{-\frac{\beta}{4}}\mathds{1}_{J_4}\left(m,q-m\right)\\
&\leq \mu_2\left(1+\left|q\right|^4\right)^{\frac{d-\alpha-\beta}{4}}.
\end{align*}
\end{proof}

\chapter{}
\label{ppe1}
\begin{proof}[Proof of Theorem \ref{theorem1}]
\begin{enumerate}
\item 
We have
\[
\forall r \in \mathbb{R}_+,\qquad
\sum_{m \in \mathbb{Z}^d}
\frac{1}{(1+|m|)^{d+1}}
\mathds{E}\left[
\left|X_{r,1}\left(\psi_{m,1}\right)\right|
+
\left|X_{r,1}\left(\psi_{m,2}\right)\right|
\right]
<+\infty.
\]

Therefore, there exists a stochastic process
$\left(\wideparen{X}_{q,1}\right)_{q \in \mathbb{R}_+}$
of state space
$\left(\mathscr{S}',\mathcal{B}\left(\mathscr{S}'\right)\right)$
such that for every
$\left(r,\varphi\right)\in\mathbb{R}_+\times\mathscr{S}$
and for $\mathds{P}$-almost every $\omega\in\Omega$,
\[
\left(\wideparen{X}_{r,1}(\omega)\right)(\varphi)
=
\left(X_{r,1}(\varphi)\right)(\omega).
\]

We fix $(\beta,U,\gamma)
\in
\mathbb{R}_+^*\times\mathbb{R}_+\times]0,1[.$

We recall, from Definition~\ref{def}, that
\[
(\delta_{j-1}\wideparen{X}_{r,1}(\omega))(x)
=
(\wideparen{X}_{r,1}(\omega))(K_{j-1}(\cdot-x))
=
(X_{r,1}(K_{j-1}(\cdot-x)))(\omega)
\]
for $\mathds{P}$-almost every $\omega\in\Omega$.

It follows from Theorem~\ref{nelson} that there exists
$(\mu_1,\mu_2,\mu_3)\in(\mathbb{R}_+^*)^3$
such that for all
$(r_1,r_2,j,x)
\in
[0,U]^2\times\mathbb{N}\times\mathbb{T}^d,$
if $r_1\leq r_2$, then
\begin{align*}
&\left(
\mathds{E}\left[
\left|
\left(\delta_{j-1}\wideparen{X}_{r_2,1}\right)(x)
-
\left(\delta_{j-1}\wideparen{X}_{r_1,1}\right)(x)
\right|^\beta
\right]
\right)^{\frac{2}{\beta}}\\
&\stackrel{\text{Theorem \ref{nelson}}}{\leq}
\mu_1
\mathds{E}\left[
\left|
\left(\delta_{j-1}\wideparen{X}_{r_2,1}\right)(x)
-
\left(\delta_{j-1}\wideparen{X}_{r_1,1}\right)(x)
\right|^2
\right]\\
&=
\mu_1
\int_{[0,r_1]\times\mathbb{T}^d}
\left|
\int_{\mathbb{T}^d}
K_{j-1}(x-y)
\left(
\eta_{r_2-\alpha}(y-v)
-
\eta_{r_1-\alpha}(y-v)
\right)
\,\mathrm{d}y
\right|^2
\,\mathrm{d}\alpha\,\mathrm{d}v\\
&\phantom{=}\ 
+
\mu_1
\int_{[r_1,r_2]\times\mathbb{T}^d}
\left|
\int_{\mathbb{T}^d}
K_{j-1}(x-y)
\eta_{r_2-\alpha}(y-v)
\,\mathrm{d}y
\right|^2
\,\mathrm{d}\alpha\,\mathrm{d}v\\
&=
\mu_1
\int_{[0,r_1]\times\mathbb{T}^d}
\left|
(K_{j-1}*\eta_{r_2-\alpha})(y)
-
(K_{j-1}*\eta_{r_1-\alpha})(y)
\right|^2
\,\mathrm{d}\alpha\,\mathrm{d}y\\
&\phantom{=}\ 
+
\mu_1
\int_{[r_1,r_2]\times\mathbb{T}^d}
\left|
(K_{j-1}*\eta_{r_2-\alpha})(y)
\right|^2
\,\mathrm{d}\alpha\,\mathrm{d}y\\
&\stackrel{\text{Parseval}}{=}
\mu_1
\sum_{m \in \mathbb{Z}^d}
\int_{0}^{r_1}
\left|\rho_{j-1}(m)\right|^2
\left(
\mathrm{e}^{-|2\pi m|^4(r_2-\alpha)}
-
\mathrm{e}^{-|2\pi m|^4(r_1-\alpha)}
\right)^2
\,\mathrm{d}\alpha\\
&\phantom{\leq}\ 
+
\mu_1
\sum_{m \in \mathbb{Z}^d}
\int_{0}^{r_2-r_1}
\left|\rho_{j-1}(m)\right|^2
\mathrm{e}^{-2\alpha|2\pi m|^4}
\,\mathrm{d}\alpha.
\end{align*}

We now make explicit the estimate of the two time integrals. Set
\[
h:=r_2-r_1
\qquad\text{and}\qquad
\lambda_m:=|2\pi m|^4.
\]
For the first integral,
\begin{align*}
&\int_0^{r_1}
\left(
\mathrm{e}^{-\lambda_m(r_2-\alpha)}
-
\mathrm{e}^{-\lambda_m(r_1-\alpha)}
\right)^2
\,\mathrm{d}\alpha\\
&=
\left(1-\mathrm{e}^{-\lambda_m h}\right)^2
\int_0^{r_1}
\mathrm{e}^{-2\lambda_m(r_1-\alpha)}
\,\mathrm{d}\alpha.
\end{align*}
Using
\[
1-\mathrm{e}^{-x}
\lesssim x^{\gamma/2},
\qquad x\geq0,
\]
and
\[
\int_0^{r_1}
\mathrm{e}^{-2\lambda_m(r_1-\alpha)}
\,\mathrm{d}\alpha
\lesssim
(1+\lambda_m)^{-1},
\]
we obtain
\begin{align*}
&\int_0^{r_1}
\left(
\mathrm{e}^{-\lambda_m(r_2-\alpha)}
-
\mathrm{e}^{-\lambda_m(r_1-\alpha)}
\right)^2
\,\mathrm{d}\alpha\\
&\lesssim
h^\gamma
(1+\lambda_m)^{\gamma-1}\\
&\lesssim
|r_2-r_1|^\gamma
\left(1+|m|^4\right)^{\gamma-1}.
\end{align*}

For the second integral, we use
\[
\int_0^h
\mathrm{e}^{-2\lambda_m\alpha}
\,\mathrm{d}\alpha
\lesssim
\min\{h,(1+\lambda_m)^{-1}\}.
\]
Since for $a,b>0$ and $\gamma\in(0,1)$,
\[
\min\{a,b\}\leq a^\gamma b^{1-\gamma},
\]
it follows that
\begin{align*}
\int_0^h
\mathrm{e}^{-2\lambda_m\alpha}
\,\mathrm{d}\alpha
&\lesssim
h^\gamma
(1+\lambda_m)^{\gamma-1}\\
&\lesssim
|r_2-r_1|^\gamma
\left(1+|m|^4\right)^{\gamma-1}.
\end{align*}
Hence,
\begin{align*}
&\stackrel{\text{Lemma \ref{l33}}}{\leq}
\mu_2
|r_2-r_1|^\gamma
\sum_{m \in \mathbb{Z}^d}
\left|\rho_{j-1}(m)\right|^2
\left(1+|m|^4\right)^{\gamma-1}.
\end{align*}

It remains to estimate the dyadic sum. For $j\geq1$, by the support properties of
$\rho_{j-1},|m|\asymp 2^{j-1}$ whenever
$\rho_{j-1}(m)\neq0.$
Moreover, the number of lattice points in the corresponding dyadic annulus is
of order $2^{d(j-1)}.$ For $j=0$, the support of $\rho_{-1}$ is bounded,
and the corresponding estimate follows directly. Therefore,
\begin{align*}
\sum_{m \in \mathbb{Z}^d}
\left|\rho_{j-1}(m)\right|^2
\left(1+|m|^4\right)^{\gamma-1}
&\lesssim
2^{d(j-1)}
2^{4(\gamma-1)(j-1)}\\
&=
2^{(j-1)(d+4\gamma-4)}.
\end{align*}
Consequently,
\begin{align*}
&\left(
\mathds{E}\left[
\left|
\left(\delta_{j-1}\wideparen{X}_{r_2,1}\right)(x)
-
\left(\delta_{j-1}\wideparen{X}_{r_1,1}\right)(x)
\right|^\beta
\right]
\right)^{\frac{2}{\beta}}\\
&\leq
\mu_3
2^{(j-1)(d+4\gamma-4)}
|r_2-r_1|^\gamma.
\end{align*}
Here, the first equality in the computation follows from the definition of a
space-time white noise.

We conclude the proof by applying Theorem~\ref{kolmogorov}.
\item Let $\left(\beta,\gamma\right) \in \mathbb{R}_+^*\times\left]0,1\right[.$

There exists
$\left(\mu_1,\mu_2,\mu_3\right)\in\left(\mathbb{R}_+^*\right)^3$
such that for all
$\left(r_1,r_2,\epsilon,j,x\right)
\in
\left(\mathbb{R}_+\right)^2
\times\mathbb{R}_+^*
\times\mathbb{N}
\times\mathbb{T}^d,$
if $r_1\leq r_2$, then
\begin{align*}
&\left(\mathds{E}\left[\left|
\left(\delta_{j-1}\widetilde{X}_{r_2,\epsilon,1}\right)\left(x\right)
-\left(\delta_{j-1}\widetilde{X}_{r_2,1}\right)\left(x\right)
-\left(\delta_{j-1}\widetilde{X}_{r_1,\epsilon,1}\right)\left(x\right)
+\left(\delta_{j-1}\widetilde{X}_{r_1,1}\right)\left(x\right)
\right|^{\beta}\right]\right)^{\frac{2}{\beta}}\\
&\stackrel{\text{Theorem \ref{nelson}}}{\leq}
\mu_1\mathds{E}\left[\left|
\left(\delta_{j-1}\widetilde{X}_{r_2,\epsilon,1}\right)\left(x\right)
-\left(\delta_{j-1}\widetilde{X}_{r_2,1}\right)\left(x\right)
-\left(\delta_{j-1}\widetilde{X}_{r_1,\epsilon,1}\right)\left(x\right)
+\left(\delta_{j-1}\widetilde{X}_{r_1,1}\right)\left(x\right)
\right|^2\right]\\
&\leq
\mu_1
\sum_{m \in \mathbb{Z}^d}
\left|\rho_{j-1}\left(m\right)\right|^2
\left|\zeta\left(\epsilon m\right)-1\right|^2
\int_{0}^{r_1}
\left(
\mathrm{e}^{-\left|2\pi m\right|^4\left(r_2-\alpha\right)}
-
\mathrm{e}^{-\left|2\pi m\right|^4\left(r_1-\alpha\right)}
\right)^2
\,\mathrm{d}\alpha\\
&\phantom{\leq}\ 
+\mu_1
\sum_{m \in \mathbb{Z}^d}
\left|\rho_{j-1}\left(m\right)\right|^2
\left|\zeta\left(\epsilon m\right)-1\right|^2
\int_{0}^{r_2-r_1}
\mathrm{e}^{-2\alpha\left|2\pi m\right|^4}
\,\mathrm{d}\alpha.
\end{align*}

We now make the estimate of the two terms explicit. Set
\[
h:=r_2-r_1,
\qquad
\lambda_m:=|2\pi m|^4.
\]
Since $\zeta$ is smooth and $\zeta(0)=1$, for every $\gamma\in(0,1)$,
\[
|\zeta(\epsilon m)-1|
\lesssim
\min\{1,\epsilon |m|\}^{\gamma/2}.
\]
Consequently,
\begin{align*}
|\zeta(\epsilon m)-1|^2
&\lesssim
\epsilon^\gamma |m|^\gamma\\
&\lesssim
\epsilon^\gamma
\left(1+|m|^4\right)^{\gamma/2}.
\end{align*}

For the first time integral, we have
\begin{align*}
&\int_0^{r_1}
\left(
\mathrm{e}^{-\lambda_m(r_2-\alpha)}
-
\mathrm{e}^{-\lambda_m(r_1-\alpha)}
\right)^2
\,\mathrm{d}\alpha\\
&=
\left(1-\mathrm{e}^{-\lambda_m h}\right)^2
\int_0^{r_1}
\mathrm{e}^{-2\lambda_m(r_1-\alpha)}
\,\mathrm{d}\alpha.
\end{align*}
Using
\[
1-\mathrm{e}^{-x}\lesssim x^{\gamma/4},
\qquad x\geq0,
\]
we obtain
\[
\left(1-\mathrm{e}^{-\lambda_m h}\right)^2
\lesssim
h^{\gamma/2}\lambda_m^{\gamma/2}.
\]
Moreover,
\[
\int_0^{r_1}
\mathrm{e}^{-2\lambda_m(r_1-\alpha)}
\,\mathrm{d}\alpha
\lesssim
\left(1+\lambda_m\right)^{-1}.
\]
Thus,
\begin{align*}
&\int_0^{r_1}
\left(
\mathrm{e}^{-\lambda_m(r_2-\alpha)}
-
\mathrm{e}^{-\lambda_m(r_1-\alpha)}
\right)^2
\,\mathrm{d}\alpha\\
&\lesssim
h^{\gamma/2}
\left(1+\lambda_m\right)^{\gamma/2-1}\\
&\lesssim
|r_2-r_1|^{\gamma/2}
\left(1+|m|^4\right)^{\gamma/2-1}.
\end{align*}

For the second time integral,
\[
\int_0^h
\mathrm{e}^{-2\alpha\lambda_m}
\,\mathrm{d}\alpha
\lesssim
\min\left\{
h,\left(1+\lambda_m\right)^{-1}
\right\}.
\]
Using the interpolation inequality
\[
\min\{a,b\}
\leq
a^{\gamma/2}b^{1-\gamma/2},
\qquad a,b\geq0,
\]
we deduce that
\begin{align*}
\int_0^h
\mathrm{e}^{-2\alpha\lambda_m}
\,\mathrm{d}\alpha
&\lesssim
h^{\gamma/2}
\left(1+\lambda_m\right)^{\gamma/2-1}\\
&\lesssim
|r_2-r_1|^{\gamma/2}
\left(1+|m|^4\right)^{\gamma/2-1}.
\end{align*}

Combining the preceding estimates, we obtain
\begin{align*}
&|\zeta(\epsilon m)-1|^2
\int_0^{r_1}
\left(
\mathrm{e}^{-\lambda_m(r_2-\alpha)}
-
\mathrm{e}^{-\lambda_m(r_1-\alpha)}
\right)^2
\,\mathrm{d}\alpha\\
&\lesssim
\epsilon^\gamma
|r_2-r_1|^{\gamma/2}
\left(1+|m|^4\right)^{\gamma/2}
\left(1+|m|^4\right)^{\gamma/2-1}\\
&=
\epsilon^\gamma
|r_2-r_1|^{\gamma/2}
\left(1+|m|^4\right)^{\gamma-1},
\end{align*}
and similarly
\begin{align*}
&|\zeta(\epsilon m)-1|^2
\int_0^{r_2-r_1}
\mathrm{e}^{-2\alpha\lambda_m}
\,\mathrm{d}\alpha\\
&\lesssim
\epsilon^\gamma
|r_2-r_1|^{\gamma/2}
\left(1+|m|^4\right)^{\gamma-1}.
\end{align*}
Therefore,
\begin{align*}
&\stackrel{\text{Lemma \ref{l33}}}{\leq}
\mu_2
\epsilon^\gamma
\left|r_2-r_1\right|^{\frac{\gamma}{2}}
\sum_{m \in \mathbb{Z}^d}
\left|\rho_{j-1}\left(m\right)\right|^2
\left(1+\left|m\right|^{4}\right)^{\gamma-1}.
\end{align*}

It remains to estimate the dyadic sum. For $j\geq1$, by the support properties of
$\rho_{j-1},
|m|\asymp 2^{j-1}$ whenever
$\rho_{j-1}(m)\neq0.$
Moreover, the number of lattice points in the corresponding dyadic annulus
is of order $2^{d(j-1)}.$ For $j=0$, the support of $\rho_{-1}$ is bounded,
and the corresponding estimate follows directly. Hence
\begin{align*}
&\sum_{m \in \mathbb{Z}^d}
\left|\rho_{j-1}\left(m\right)\right|^2
\left(1+\left|m\right|^4\right)^{\gamma-1}\\
&\lesssim
2^{d(j-1)}
2^{4(\gamma-1)(j-1)}\\
&=
2^{(j-1)(d+4\gamma-4)}.
\end{align*}
Consequently,
\begin{align*}
&\left(\mathds{E}\left[\left|
\left(\delta_{j-1}\widetilde{X}_{r_2,\epsilon,1}\right)\left(x\right)
-\left(\delta_{j-1}\widetilde{X}_{r_2,1}\right)\left(x\right)
-\left(\delta_{j-1}\widetilde{X}_{r_1,\epsilon,1}\right)\left(x\right)
+\left(\delta_{j-1}\widetilde{X}_{r_1,1}\right)\left(x\right)
\right|^{\beta}\right]\right)^{\frac{2}{\beta}}\\
&\leq
\mu_3
\epsilon^\gamma
2^{\left(j-1\right)\left(d+4\gamma-4\right)}
\left|r_2-r_1\right|^{\frac{\gamma}{2}}.
\end{align*}

We deduce from the Garsia--Rodemich--Rumsey lemma \cite{MR267632} that for
every $U\in\mathbb{R}_+$,
\[
\lim_{\epsilon\downarrow0}
\mathds{E}\left[
\sup_{r\in[0,U]}
\left\|
\widetilde{X}_{r,\epsilon,1}
-
\widetilde{X}_{r,1}
\right\|_{\mathscr{C}^\theta}
\right]
=0.
\]
\end{enumerate}
\end{proof}
\begin{remark}
For the proof of Theorem \ref{theorem1} we used the following fact: If the sequence $(u_k)_{k \in\mathbb{Z}^d}$ grows
at most polynomially in $k$, then there exist $f \in \mathscr{S}'$  having Fourier transform $(u_k)_{k}$ (we refer to Chapter 2 of \cite{MR3445609} for more details).
\end{remark}
\chapter{}
\label{ppe2}
\begin{proof}[Proof of Theorem \ref{theorem2}]

We note that for every $\left(r,\epsilon,\varphi\right) \in \mathbb{R}_+\times \mathbb{R}_+^*\times \mathscr{S}$ and for $\mathds{P}$-almost every $\omega \in \Omega,(\widetilde{X}_{r,\epsilon,2}(\omega))(\varphi)=\left(X_{r,\epsilon,2}\left(\varphi\right)\right)\left(\omega\right).$

We fix $\left(\beta,U,\gamma\right) \in \mathbb{R}_+^*\times\mathbb{R}_+\times \left]0,1-\frac{d}{8}\right[.$

It follows from Theorem \ref{nelson} and Lemma \ref{lemma2} that there exists $\left(\mu_1,\mu_2,\mu_3,\mu_4,\mu_5\right) \in \left(\mathbb{R}_+^*\right)^5$ such that for all $\left(r_1,r_2,\epsilon_1,\epsilon_2,j,x\right) \in\left[0,U\right]^2\times(\mathbb{R}_+^*)^2\times \mathbb{N}\times \mathbb{T}^d,$ if $r_1\leq r_2,$ then

\begin{align*}
&\left(\mathds{E}\left[\left|\left(\delta_{j-1}\widetilde{X}_{r_2,\epsilon_2,2}\right)\left(x\right)-\left(\delta_{j-1}\widetilde{X}_{r_2,\epsilon_1,2}\right)\left(x\right)-\left(\delta_{j-1}\widetilde{X}_{r_1,\epsilon_2,2}\right)\left(x\right)+\left(\delta_{j-1}\widetilde{X}_{r_1,\epsilon_1,2}\right)\left(x\right)\right|^\beta\right]\right)^{\frac{2}{\beta}}\\
&\stackrel{\text{Theorem \ref{nelson}}}{\leq}\mu_1\mathds{E}\left[\left|\left(\delta_{j-1}\widetilde{X}_{r_2,\epsilon_2,2}\right)\left(x\right)-\left(\delta_{j-1}\widetilde{X}_{r_2,\epsilon_1,2}\right)\left(x\right)-\left(\delta_{j-1}\widetilde{X}_{r_1,\epsilon_2,2}\right)\left(x\right)+\left(\delta_{j-1}\widetilde{X}_{r_1,\epsilon_1,2}\right)\left(x\right)\right|^2\right]\\
&\stackrel{\text{Theorem \ref{nelson}}}{\leq}\mu_2\int_{\left[0,r_1\right]^2}\left(\int_{\left(\mathbb{T}^d\right)^2}\left|\int_{\mathbb{T}^d}K_{j-1}\left(x-y\right)\left(\prod_{q=1}^2(\zeta_{\epsilon_2}*\eta_{r_2-\alpha_q})\left(y-v_q\right)-\prod_{q=1}^2(\zeta_{\epsilon_1}*\eta_{r_2-\alpha_q})\left(y-v_q\right)\right.\right.\right.\\
&\left.\left.\left.
\hspace{4cm}
-\prod_{m=1}^2(\zeta_{\epsilon_2}*\eta_{r_1-\alpha_m})\left(y-v_m\right)+\prod_{q=1}^2(\zeta_{\epsilon_1}*\eta_{r_1-\alpha_q})\left(y-v_q\right)\right)\,\mathrm{d}y\right|^2\mathrm{d}v_1\,\mathrm{d}v_2\right)\mathrm{d}\alpha_1\,\mathrm{d}\alpha_2\\
& \phantom{=}\ +\mu_2\int_{\left[0,r_2\right]^2-\left[0,r_1\right]^2}\left(\int_{\left(\mathbb{T}^d\right)^2}\left|\int_{\mathbb{T}^d}K_{j-1}\left(x-y\right)\left(\prod_{q=1}^2(\zeta_{\epsilon_2}*\eta_{r_2-\alpha_q})\left(y-v_q\right)-\prod_{q=1}^2(\zeta_{\epsilon_1}*\eta_{r_2-\alpha_q})\left(y-v_q\right)\right)\,\mathrm{d}y\right|^2\mathrm{d}v_1\,\mathrm{d}v_2\right)\\
&\hspace{11cm}\mathrm{d}\alpha_1\,\mathrm{d}\alpha_2\\
&\stackrel{\text{Fubini}}{=}\mu_2\int_{\left[0,r_1\right]^2}\left(\int_{\left(\mathbb{T}^d\right)^4}K_{j-1}\left(y_1\right)K_{j-1}\left(y_2\right)\prod_{m=1}^2\left(\prod_{q=1}^2(\zeta_{\epsilon_2}*\eta_{r_2-\alpha_q})\left(y_m-v_q\right)-\prod_{q=1}^2(\zeta_{\epsilon_1}*\eta_{r_2-\alpha_q})\left(y_m-v_q\right)\right.\right.\\
&\left.\left.
\hspace{4cm}
-\prod_{q=1}^2(\zeta_{\epsilon_2}*\eta_{r_1-\alpha_q})\left(y_m-v_q\right)+\prod_{q=1}^2(\zeta_{\epsilon_1}*\eta_{r_1-\alpha_q})\left(y_m-v_q\right)\right)\mathrm{d}v_1\,\mathrm{d}v_2\,\mathrm{d}y_1\,\mathrm{d}y_2\right)\mathrm{d}\alpha_1\,\mathrm{d}\alpha_2\\
& \phantom{=}\ +\mu_2\int_{\left[0,r_2\right]^2-\left[0,r_1\right]^2}\left(\int_{\left(\mathbb{T}^d\right)^4}K_{j-1}\left(y_1\right)K_{j-1}(y_2)\prod_{m=1}^2\left(\prod_{q=1}^2(\zeta_{\epsilon_2}*\eta_{r_2-\alpha_q})\left(y_m-v_q\right)-\prod_{q=1}^2(\zeta_{\epsilon_1}*\eta_{r_2-\alpha_q})\left(y_m-v_q\right)\right)\right.\\
&\left.
\hspace{7cm}\mathrm{d}y_1\,\mathrm{d}y_2\,\mathrm{d}v_1\,\mathrm{d}v_2\right)\mathrm{d}\alpha_1\,\mathrm{d}\alpha_2\\
&= \mu_2\int_{\left[0,r_1\right]^2}\left(\int_{\left(\mathbb{T}^d\right)^2}K_{j-1}\left(y_1\right)K_{j-1}\left(y_2\right)\left(\prod_{m=1}^2(\zeta_{\epsilon_2}*\zeta_{\epsilon_2}*\eta_{2\left(r_2-\alpha_m\right)})\left(y_1-y_2\right)\right.\right.\\
&\left.\left.
-2\prod_{q=1}^2(\zeta_{\epsilon_1}*\zeta_{\epsilon_2}*\eta_{2(r_2-\alpha_q)})\left(y_1-y_2\right)
+\prod_{n=1}^2(\zeta_{\epsilon_1}*\zeta_{\epsilon_1}*\eta_{2\left(r_2-\alpha_n\right)})\left(y_1-y_2\right)\right.\right.\\
&\left.\left.
-2\prod_{m=1}^2(\zeta_{\epsilon_2}*\zeta_{\epsilon_2}*\eta_{r_1+r_2-2\alpha_m})\left(y_1-y_2\right)
+4\prod_{q=1}^2(\zeta_{\epsilon_1}*\zeta_{\epsilon_2}*\eta_{r_1+r_2-2\alpha_q})\left(y_1-y_2\right)\right.\right.\\
&\left.\left.
-2\prod_{n=1}^2(\zeta_{\epsilon_1}*\zeta_{\epsilon_1}*\eta_{r_1+r_2-2\alpha_n})\left(y_1-y_2\right)
+\prod_{m=1}^2(\zeta_{\epsilon_2}*\zeta_{\epsilon_2}*\eta_{2\left(r_1-\alpha_m\right)})\left(y_1-y_2\right)\right.\right.\\
&\left.\left.
-2\prod_{q=1}^2(\zeta_{\epsilon_1}*\zeta_{\epsilon_2}*\eta_{2(r_1-\alpha_q)})\left(y_1-y_2\right)
+\prod_{n=1}^2(\zeta_{\epsilon_1}*\zeta_{\epsilon_1}*\eta_{2\left(r_1-\alpha_n\right)})\left(y_1-y_2\right)\right)
\,\mathrm{d}y_1\,\mathrm{d}y_2\right)\mathrm{d}\alpha_1\,\mathrm{d}\alpha_2\\
& \phantom{=}\ +\mu_2\int_{\left[0,r_2\right]^2-\left[0,r_1\right]^2}\left(\int_{\left(\mathbb{T}^d\right)^2}K_{j-1}\left(y_1\right)K_{j-1}\left(y_2\right)\left(\prod_{m=1}^2(\zeta_{\epsilon_2}*\zeta_{\epsilon_2}*\eta_{2\left(r_2-\alpha_m\right)})\left(y_1-y_2\right)\right.\right.\\
&\left.\left.
-2\prod_{q=1}^2(\zeta_{\epsilon_1}*\zeta_{\epsilon_2}*\eta_{2(r_2-\alpha_q)})\left(y_1-y_2\right)
+\prod_{n=1}^2(\zeta_{\epsilon_1}*\zeta_{\epsilon_1}*\eta_{2\left(r_2-\alpha_n\right)})\left(y_1-y_2\right)\right)\mathrm{d}y_1\,\mathrm{d}y_2\right)\mathrm{d}\alpha_1\,\mathrm{d}\alpha_2.
\end{align*}

Using the semi-group property
\[
\eta_{q_1}*\eta_{q_2}=\eta_{q_1+q_2},
\]
before applying Parseval, we obtain
\begin{align*}
&\stackrel{\text{Parseval}}{=}\mu_2
\sum_{\left(m_1,m_2\right)\in\left(\mathbb{Z}^d\right)^2}
\left|\rho_{j-1}\left(m_1+m_2\right)\right|^2\\
&\quad\times
\left|
\zeta\left(\epsilon_2m_1\right)\zeta\left(\epsilon_2m_2\right)
-\zeta\left(\epsilon_1m_1\right)\zeta\left(\epsilon_1m_2\right)
\right|^2\\
&\quad\times
\int_{\left[0,r_1\right]^2}
\left(
\prod_{q=1}^2
\mathrm{e}^{-\left|2\pi m_q\right|^4(r_2-\alpha_q)}
-
\prod_{q=1}^2
\mathrm{e}^{-\left|2\pi m_q\right|^4(r_1-\alpha_q)}
\right)^2
\,\mathrm{d}\alpha_1\,\mathrm{d}\alpha_2\\
&\phantom{=}\ 
+\mu_2
\sum_{\left(m_1,m_2\right)\in\left(\mathbb{Z}^d\right)^2}
\left|\rho_{j-1}\left(m_1+m_2\right)\right|^2\\
&\quad\times
\left|
\zeta\left(\epsilon_2m_1\right)\zeta\left(\epsilon_2m_2\right)
-\zeta\left(\epsilon_1m_1\right)\zeta\left(\epsilon_1m_2\right)
\right|^2\\
&\quad\times
\int_{\left[0,r_2\right]^2-\left[0,r_1\right]^2}
\prod_{q=1}^2
\mathrm{e}^{-2\left|2\pi m_q\right|^4(r_2-\alpha_q)}
\,\mathrm{d}\alpha_1\,\mathrm{d}\alpha_2.
\end{align*}

We now estimate the two terms obtained after Parseval. Set
\[
h:=r_2-r_1
\]
and
\[
\lambda_q:=|2\pi m_q|^4,
\qquad q=1,2.
\]

We first estimate the mollification increment. Since $\zeta$ is smooth,
\begin{align*}
&\left|
\zeta(\epsilon_2m_1)\zeta(\epsilon_2m_2)
-\zeta(\epsilon_1m_1)\zeta(\epsilon_1m_2)
\right|\\
&\leq
\left|\zeta(\epsilon_2m_1)-\zeta(\epsilon_1m_1)\right|
\left|\zeta(\epsilon_2m_2)\right|\\
&\quad+
\left|\zeta(\epsilon_1m_1)\right|
\left|\zeta(\epsilon_2m_2)-\zeta(\epsilon_1m_2)\right|.
\end{align*}
Since $\zeta$ and its first derivatives are bounded,
\[
\left|\zeta(\epsilon_2m)-\zeta(\epsilon_1m)\right|
\lesssim
\min\left\{1,|\epsilon_2-\epsilon_1||m|\right\}.
\]
Hence, for $\gamma\in(0,1)$,
\begin{align*}
&\left|
\zeta(\epsilon_2m_1)\zeta(\epsilon_2m_2)
-\zeta(\epsilon_1m_1)\zeta(\epsilon_1m_2)
\right|^2\\
&\lesssim
|\epsilon_2-\epsilon_1|^\gamma
\left(
(1+\lambda_1)^{\gamma/4}
+
(1+\lambda_2)^{\gamma/4}
\right)\\
&\lesssim
|\epsilon_2-\epsilon_1|^\gamma
(1+\lambda_1)^{\gamma/4}
(1+\lambda_2)^{\gamma/4}.
\end{align*}

We now consider the first time integral. Since
\[
r_2-\alpha_q
=
h+r_1-\alpha_q,
\]
we have
\begin{align*}
&\prod_{q=1}^2
\mathrm{e}^{-\lambda_q(r_2-\alpha_q)}
-
\prod_{q=1}^2
\mathrm{e}^{-\lambda_q(r_1-\alpha_q)}\\
&=
\mathrm{e}^{-\lambda_1(r_1-\alpha_1)}
\mathrm{e}^{-\lambda_2(r_1-\alpha_2)}
\left(
\mathrm{e}^{-(\lambda_1+\lambda_2)h}-1
\right).
\end{align*}
Therefore,
\begin{align*}
&\int_{[0,r_1]^2}
\left(
\prod_{q=1}^2
\mathrm{e}^{-\lambda_q(r_2-\alpha_q)}
-
\prod_{q=1}^2
\mathrm{e}^{-\lambda_q(r_1-\alpha_q)}
\right)^2
\,\mathrm{d}\alpha_1\,\mathrm{d}\alpha_2\\
&=
\left(
1-\mathrm{e}^{-(\lambda_1+\lambda_2)h}
\right)^2
\prod_{q=1}^2
\int_0^{r_1}
\mathrm{e}^{-2\lambda_q(r_1-\alpha_q)}
\,\mathrm{d}\alpha_q.
\end{align*}
Using
\[
1-\mathrm{e}^{-x}\lesssim \min\{1,x\},
\]
Lemma~\ref{l33} gives
\[
\left(1-\mathrm{e}^{-x}\right)^2
\lesssim x^{\gamma/2},
\qquad x\geq0.
\]
Moreover,
\[
\int_0^{r_1}
\mathrm{e}^{-2\lambda_q(r_1-\alpha_q)}
\,\mathrm{d}\alpha_q
\lesssim
(1+\lambda_q)^{-1}.
\]
Consequently,
\begin{align*}
&\int_{[0,r_1]^2}
\left(
\prod_{q=1}^2
\mathrm{e}^{-\lambda_q(r_2-\alpha_q)}
-
\prod_{q=1}^2
\mathrm{e}^{-\lambda_q(r_1-\alpha_q)}
\right)^2
\,\mathrm{d}\alpha_1\,\mathrm{d}\alpha_2\\
&\lesssim
h^{\gamma/2}
(\lambda_1+\lambda_2)^{\gamma/2}
(1+\lambda_1)^{-1}
(1+\lambda_2)^{-1}\\
&\lesssim
h^{\gamma/2}
(1+\lambda_1)^{\gamma/2-1}
(1+\lambda_2)^{\gamma/2-1}.
\end{align*}
In the last inequality we used
\[
1+\lambda_1+\lambda_2
\leq
(1+\lambda_1)(1+\lambda_2).
\]

Multiplying this estimate by the mollification increment gives
\begin{align*}
&\left|
\zeta(\epsilon_2m_1)\zeta(\epsilon_2m_2)
-\zeta(\epsilon_1m_1)\zeta(\epsilon_1m_2)
\right|^2\\
&\quad\times
\int_{[0,r_1]^2}
\left(
\prod_{q=1}^2
\mathrm{e}^{-\lambda_q(r_2-\alpha_q)}
-
\prod_{q=1}^2
\mathrm{e}^{-\lambda_q(r_1-\alpha_q)}
\right)^2
\,\mathrm{d}\alpha_1\,\mathrm{d}\alpha_2\\
&\lesssim
|\epsilon_2-\epsilon_1|^\gamma
h^{\gamma/2}
\prod_{q=1}^2
(1+\lambda_q)^{\frac{3\gamma}{4}-1}.
\end{align*}
Since $\gamma>0$,
\[
\frac{3\gamma}{4}-1
\leq
\gamma-1,
\]
and hence
\begin{align*}
&\left|
\zeta(\epsilon_2m_1)\zeta(\epsilon_2m_2)
-\zeta(\epsilon_1m_1)\zeta(\epsilon_1m_2)
\right|^2\\
&\quad\times
\int_{[0,r_1]^2}
\left(
\prod_{q=1}^2
\mathrm{e}^{-\lambda_q(r_2-\alpha_q)}
-
\prod_{q=1}^2
\mathrm{e}^{-\lambda_q(r_1-\alpha_q)}
\right)^2
\,\mathrm{d}\alpha_1\,\mathrm{d}\alpha_2\\
&\lesssim
|\epsilon_2-\epsilon_1|^\gamma
h^{\gamma/2}
\prod_{q=1}^2
(1+\lambda_q)^{\gamma-1}.
\end{align*}

We now estimate the second time integral. We decompose
\[
[0,r_2]^2-[0,r_1]^2
\subset
\bigl([r_1,r_2]\times[0,r_2]\bigr)
\cup
\bigl([0,r_2]\times[r_1,r_2]\bigr).
\]
Therefore,
\begin{align*}
&\int_{[0,r_2]^2-[0,r_1]^2}
\mathrm{e}^{-2\lambda_1(r_2-\alpha_1)}
\mathrm{e}^{-2\lambda_2(r_2-\alpha_2)}
\,\mathrm{d}\alpha_1\,\mathrm{d}\alpha_2\\
&\lesssim
\left(
\int_{r_1}^{r_2}
\mathrm{e}^{-2\lambda_1(r_2-\alpha_1)}
\,\mathrm{d}\alpha_1
\right)
\left(
\int_0^{r_2}
\mathrm{e}^{-2\lambda_2(r_2-\alpha_2)}
\,\mathrm{d}\alpha_2
\right)\\
&\quad+
\left(
\int_0^{r_2}
\mathrm{e}^{-2\lambda_1(r_2-\alpha_1)}
\,\mathrm{d}\alpha_1
\right)
\left(
\int_{r_1}^{r_2}
\mathrm{e}^{-2\lambda_2(r_2-\alpha_2)}
\,\mathrm{d}\alpha_2
\right).
\end{align*}
For $q=1,2$,
\[
\int_0^{r_2}
\mathrm{e}^{-2\lambda_q(r_2-\alpha)}
\,\mathrm{d}\alpha
\lesssim
(1+\lambda_q)^{-1}.
\]
On the other hand,
\[
\int_{r_1}^{r_2}
\mathrm{e}^{-2\lambda_q(r_2-\alpha)}
\,\mathrm{d}\alpha
\lesssim h
\]
and
\[
\int_{r_1}^{r_2}
\mathrm{e}^{-2\lambda_q(r_2-\alpha)}
\,\mathrm{d}\alpha
\lesssim
(1+\lambda_q)^{-1}.
\]
Applying Lemma~\ref{l33} with exponent $\gamma/2$, we obtain
\[
\int_{r_1}^{r_2}
\mathrm{e}^{-2\lambda_q(r_2-\alpha)}
\,\mathrm{d}\alpha
\lesssim
h^{\gamma/2}
(1+\lambda_q)^{\frac{\gamma}{2}-1}.
\]
Consequently,
\begin{align*}
&\int_{[0,r_2]^2-[0,r_1]^2}
\mathrm{e}^{-2\lambda_1(r_2-\alpha_1)}
\mathrm{e}^{-2\lambda_2(r_2-\alpha_2)}
\,\mathrm{d}\alpha_1\,\mathrm{d}\alpha_2\\
&\lesssim
h^{\gamma/2}
\Bigl[
(1+\lambda_1)^{\frac{\gamma}{2}-1}
(1+\lambda_2)^{-1}
+
(1+\lambda_1)^{-1}
(1+\lambda_2)^{\frac{\gamma}{2}-1}
\Bigr].
\end{align*}

Multiplying by the mollification increment, we get
\begin{align*}
&\left|
\zeta(\epsilon_2m_1)\zeta(\epsilon_2m_2)
-\zeta(\epsilon_1m_1)\zeta(\epsilon_1m_2)
\right|^2\\
&\quad\times
\int_{[0,r_2]^2-[0,r_1]^2}
\mathrm{e}^{-2\lambda_1(r_2-\alpha_1)}
\mathrm{e}^{-2\lambda_2(r_2-\alpha_2)}
\,\mathrm{d}\alpha_1\,\mathrm{d}\alpha_2\\
&\lesssim
|\epsilon_2-\epsilon_1|^\gamma
h^{\gamma/2}
(1+\lambda_1)^{\gamma/4}
(1+\lambda_2)^{\gamma/4}\\
&\quad\times
\Bigl[
(1+\lambda_1)^{\frac{\gamma}{2}-1}
(1+\lambda_2)^{-1}
+
(1+\lambda_1)^{-1}
(1+\lambda_2)^{\frac{\gamma}{2}-1}
\Bigr]\\
&=
|\epsilon_2-\epsilon_1|^\gamma
h^{\gamma/2}
\Bigl[
(1+\lambda_1)^{\frac{3\gamma}{4}-1}
(1+\lambda_2)^{\frac{\gamma}{4}-1}\\
&\hspace{4cm}
+
(1+\lambda_1)^{\frac{\gamma}{4}-1}
(1+\lambda_2)^{\frac{3\gamma}{4}-1}
\Bigr].
\end{align*}
Since
\[
\frac{3\gamma}{4}\leq\gamma
\qquad\text{and}\qquad
\frac{\gamma}{4}\leq\gamma,
\]
we conclude that
\begin{align*}
&\left|
\zeta(\epsilon_2m_1)\zeta(\epsilon_2m_2)
-\zeta(\epsilon_1m_1)\zeta(\epsilon_1m_2)
\right|^2\\
&\quad\times
\int_{[0,r_2]^2-[0,r_1]^2}
\mathrm{e}^{-2\lambda_1(r_2-\alpha_1)}
\mathrm{e}^{-2\lambda_2(r_2-\alpha_2)}
\,\mathrm{d}\alpha_1\,\mathrm{d}\alpha_2\\
&\lesssim
|\epsilon_2-\epsilon_1|^\gamma
h^{\gamma/2}
\prod_{q=1}^2
(1+\lambda_q)^{\gamma-1}.
\end{align*}

Combining the estimates for the two time integrals and recalling that
\[
\lambda_q=|2\pi m_q|^4
\asymp |m_q|^4,
\]
we finally obtain
\begin{align*}
&\stackrel{\text{Lemma \ref{l33}}}{\leq}
\mu_3|\epsilon_2-\epsilon_1|^{\gamma}
\left|r_2-r_1\right|^{\frac{\gamma}{2}}
\sum_{\left(m_1,m_2\right)\in\left(\mathbb{Z}^d\right)^2}
\left|\rho_{j-1}\left(m_1+m_2\right)\right|^2
\prod_{q=1}^2
\left(1+\left|m_q\right|^4\right)^{\gamma-1}.
\end{align*}

It remains to estimate the discrete convolution. Setting
\[
n=m_1+m_2
\]
and then renaming $n$ as $m_1$, we obtain
\begin{align*}
&\sum_{\left(m_1,m_2\right) \in \left(\mathbb{Z}^d\right)^2}
\left|\rho_{j-1}\left(m_1+m_2\right)\right|^2
\left(1+\left|m_1\right|^4\right)^{\gamma-1}
\left(1+\left|m_2\right|^4\right)^{\gamma-1}\\
&=
\sum_{m_1 \in \mathbb{Z}^d}
\left|\rho_{j-1}\left(m_1\right)\right|^2
\sum_{m_2 \in \mathbb{Z}^d}
\left(1+\left|m_1-m_2\right|^4\right)^{\gamma-1}
\left(1+\left|m_2\right|^4\right)^{\gamma-1}.
\end{align*}

We now make the application of Lemma~\ref{lemma2} explicit. Since
\[
\gamma-1=-\frac{4(1-\gamma)}{4},
\]
we take
\[
\alpha=\beta=4(1-\gamma)
\]
in Lemma~\ref{lemma2}. The hypothesis
\[
\gamma<1-\frac{d}{8}
\]
implies
\[
\alpha+\beta=8(1-\gamma)>d.
\]
Moreover,
\[
\max(\alpha,\beta)=4(1-\gamma)<d
\]
in the dimensions under consideration. Hence Lemma~\ref{lemma2} gives
\begin{align*}
&\sum_{m_2 \in \mathbb{Z}^d}
\left(1+\left|m_1-m_2\right|^4\right)^{\gamma-1}
\left(1+\left|m_2\right|^4\right)^{\gamma-1}\\
&=
\sum_{m_2 \in \mathbb{Z}^d}
\left(1+\left|m_1-m_2\right|^4\right)^{-\frac{4(1-\gamma)}{4}}
\left(1+\left|m_2\right|^4\right)^{-\frac{4(1-\gamma)}{4}}\\
&\stackrel{\text{Lemma \ref{lemma2}}}{\lesssim}
\left(1+\left|m_1\right|^4\right)^{
\frac{d-4(1-\gamma)-4(1-\gamma)}{4}}\\
&=
\left(1+\left|m_1\right|^4\right)^{
\frac{d-8+8\gamma}{4}}\\
&=
\left(1+\left|m_1\right|^4\right)^{
\frac{d}{4}+2\gamma-2}.
\end{align*}
Consequently,
\begin{align*}
&\stackrel{\text{Lemma \ref{lemma2}}}{\leq}
\mu_4|\epsilon_2-\epsilon_1|^{\gamma}
\left|r_2-r_1\right|^{\frac{\gamma}{2}}
\sum_{m_1 \in \mathbb{Z}^d}
\left|\rho_{j-1}\left(m_1\right)\right|^2
\left(1+\left|m_1\right|^4\right)^{\frac{d}{4}+2\gamma-2}.
\end{align*}

Finally, for $j\geq1$, by the support properties of $\rho_{j-1}$,
$|m_1|\asymp 2^{j-1}$ whenever
$\rho_{j-1}(m_1)\neq0.$
Moreover, the number of lattice points in the corresponding dyadic annulus
is of order $2^{d(j-1)}$. For $j=0$, the support of $\rho_{-1}$ is bounded,
and the corresponding estimate follows directly. Therefore,
\begin{align*}
&\sum_{m_1 \in \mathbb{Z}^d}
\left|\rho_{j-1}\left(m_1\right)\right|^2
\left(1+\left|m_1\right|^4\right)^{\frac{d}{4}+2\gamma-2}\\
&\lesssim
2^{d(j-1)}
2^{4\left(\frac{d}{4}+2\gamma-2\right)(j-1)}\\
&=
2^{d(j-1)}
2^{\left(d+8\gamma-8\right)(j-1)}\\
&=
2^{\left(j-1\right)\left(2d+8\gamma-8\right)}.
\end{align*}
Thus,
\begin{align*}
&\left(\mathds{E}\left[\left|\left(\delta_{j-1}\widetilde{X}_{r_2,\epsilon_2,2}\right)\left(x\right)-\left(\delta_{j-1}\widetilde{X}_{r_2,\epsilon_1,2}\right)\left(x\right)-\left(\delta_{j-1}\widetilde{X}_{r_1,\epsilon_2,2}\right)\left(x\right)+\left(\delta_{j-1}\widetilde{X}_{r_1,\epsilon_1,2}\right)\left(x\right)\right|^\beta\right]\right)^{\frac{2}{\beta}}\\
&\leq
\mu_5|\epsilon_2-\epsilon_1|^{\gamma}
2^{\left(j-1\right)\left(2d+8\gamma-8\right)}
\left|r_2-r_1\right|^{\frac{\gamma}{2}}.
\end{align*}

We deduce from the Garsia--Rodemich--Rumsey lemma \cite{MR267632} that for every
$U\in\mathbb{R}_+$,
\[
\lim_{\epsilon_1\downarrow0,\epsilon_2\downarrow0}
\mathds{E}\left[
\sup_{r\in[0,U]}
\left\|
\widetilde{X}_{r,\epsilon_2,2}
-
\widetilde{X}_{r,\epsilon_1,2}
\right\|_{\mathscr{C}^{\theta}}
\right]
=0.
\]
\end{proof}
\chapter{}
\label{ppe3}
\begin{proof}[Proof of Theorem \ref{theorem3}]
We note that for every $\left(r,\epsilon,\varphi\right) \in \mathbb{R}_+\times \mathbb{R}_+^*\times \mathscr{S}$ and for $\mathds{P}$-almost every $\omega \in \Omega,(\widetilde{X}_{r,\epsilon,3}(\omega))(\varphi)=\left(X_{r,\epsilon,3}\left(\varphi\right)\right)\left(\omega\right).$

We fix $\left(\beta,U,\gamma\right) \in \mathbb{R}_+^*\times\mathbb{R}_+\times \left]0,1-\frac{d}{6}\right[.$

It follows from Theorem \ref{nelson} and Lemma \ref{lemma2} that there exists $\left(\mu_1,\mu_2,\mu_3,\mu_4,\mu_5,\mu_6\right) \in \left(\mathbb{R}_+^*\right)^6$ such that for all $\left(r_1,r_2,\epsilon_1,\epsilon_2,j,x\right) \in\left[0,U\right]^2\times(\mathbb{R}_+^*)^2\times \mathbb{N}\times \mathbb{T}^d,$ if $r_1\leq r_2,$ then

\begin{align*}
&\left(\mathds{E}\left[\left|\left(\delta_{j-1}\widetilde{X}_{r_2,\epsilon_2,3}\right)\left(x\right)-\left(\delta_{j-1}\widetilde{X}_{r_2,\epsilon_1,3}\right)\left(x\right)-\left(\delta_{j-1}\widetilde{X}_{r_1,\epsilon_2,3}\right)\left(x\right)+\left(\delta_{j-1}\widetilde{X}_{r_1,\epsilon_1,3}\right)\left(x\right)\right|^\beta\right]\right)^{\frac{2}{\beta}}\\
&\stackrel{\text{Theorem \ref{nelson}}}{\leq} 
\mu_1\mathds{E}\left[\left|\left(\delta_{j-1}\widetilde{X}_{r_2,\epsilon_2,3}\right)\left(x\right)-\left(\delta_{j-1}\widetilde{X}_{r_2,\epsilon_1,3}\right)\left(x\right)-\left(\delta_{j-1}\widetilde{X}_{r_1,\epsilon_2,3}\right)\left(x\right)+\left(\delta_{j-1}\widetilde{X}_{r_1,\epsilon_1,3}\right)\left(x\right)\right|^2\right]\\
&\stackrel{\text{Theorem \ref{nelson}}}{\leq}
\mu_2\int_{\left[0,r_1\right]^3}
\left(
\int_{\left(\mathbb{T}^d\right)^3}
\left|
\int_{\mathbb{T}^d}
K_{j-1}\left(x-y\right)
\left(
\prod_{q=1}^3(\zeta_{\epsilon_2}*\eta_{r_2-\alpha_q})\left(y-v_q\right)
-
\prod_{q=1}^3(\zeta_{\epsilon_1}*\eta_{r_2-\alpha_q})\left(y-v_q\right)
\right.\right.\right.\\
&\left.\left.\left.
\hspace{4cm}
-
\prod_{m=1}^3(\zeta_{\epsilon_2}*\eta_{r_1-\alpha_m})\left(y-v_m\right)
+
\prod_{q=1}^3(\zeta_{\epsilon_1}*\eta_{r_1-\alpha_q})\left(y-v_q\right)
\right)
\,\mathrm{d}y
\right|^2
\mathrm{d}v_1\,\mathrm{d}v_2\,\mathrm{d}v_3
\right)
\mathrm{d}\alpha_1\,\mathrm{d}\alpha_2\,\mathrm{d}\alpha_3\\
&\phantom{=}\ 
+\mu_2
\int_{\left[0,r_2\right]^3-\left[0,r_1\right]^3}
\left(
\int_{\left(\mathbb{T}^d\right)^3}
\left|
\int_{\mathbb{T}^d}
K_{j-1}\left(x-y\right)
\left(
\prod_{q=1}^3(\zeta_{\epsilon_2}*\eta_{r_2-\alpha_q})\left(y-v_q\right)
-
\prod_{q=1}^3(\zeta_{\epsilon_1}*\eta_{r_2-\alpha_q})\left(y-v_q\right)
\right)
\,\mathrm{d}y
\right|^2
\right.\\
&\left.
\hspace{8cm}
\mathrm{d}v_1\,\mathrm{d}v_2\,\mathrm{d}v_3
\right)
\,\mathrm{d}\alpha_1\,\mathrm{d}\alpha_2\,\mathrm{d}\alpha_3\\
&\stackrel{\text{Fubini}}{=}
\mu_2
\int_{\left[0,r_1\right]^3}
\left(
\int_{\left(\mathbb{T}^d\right)^5}
K_{j-1}\left(y_1\right)
K_{j-1}\left(y_2\right)
\prod_{m=1}^2
\left(
\prod_{q=1}^3(\zeta_{\epsilon_2}*\eta_{r_2-\alpha_q})\left(y_m-v_q\right)
-
\prod_{q=1}^3(\zeta_{\epsilon_1}*\eta_{r_2-\alpha_q})\left(y_m-v_q\right)
\right.\right.\\
&\left.\left.
\hspace{3cm}
-
\prod_{q=1}^3(\zeta_{\epsilon_2}*\eta_{r_1-\alpha_q})\left(y_m-v_q\right)
+
\prod_{q=1}^3(\zeta_{\epsilon_1}*\eta_{r_1-\alpha_q})\left(y_m-v_q\right)
\right)
\,\mathrm{d}v_1\,\mathrm{d}v_2\,\mathrm{d}v_3\,
\mathrm{d}y_1\,\mathrm{d}y_2
\right)
\mathrm{d}\alpha_1\,\mathrm{d}\alpha_2\,\mathrm{d}\alpha_3\\
&\phantom{=}\ 
+\mu_2
\int_{\left[0,r_2\right]^3-\left[0,r_1\right]^3}
\left(
\int_{\left(\mathbb{T}^d\right)^5}
K_{j-1}\left(y_1\right)
K_{j-1}\left(y_2\right)
\prod_{m=1}^2
\left(
\prod_{q=1}^3(\zeta_{\epsilon_2}*\eta_{r_2-\alpha_q})\left(y_m-v_q\right)
-
\prod_{q=1}^3(\zeta_{\epsilon_1}*\eta_{r_2-\alpha_q})\left(y_m-v_q\right)
\right)
\right.\\
&\left.
\hspace{8cm}
\,\mathrm{d}y_1\,\mathrm{d}y_2
\,\mathrm{d}v_1\,\mathrm{d}v_2\,\mathrm{d}v_3
\right)
\,\mathrm{d}\alpha_1\,\mathrm{d}\alpha_2\,\mathrm{d}\alpha_3\\
&=
\mu_2
\int_{\left[0,r_1\right]^3}
\left(
\int_{\left(\mathbb{T}^d\right)^2}
K_{j-1}\left(y_1\right)
K_{j-1}\left(y_2\right)
\left(
\prod_{m=1}^3
(\zeta_{\epsilon_2}*\zeta_{\epsilon_2}*
\eta_{2(r_2-\alpha_m)})
\left(y_1-y_2\right)
\right.\right.\\
&\left.\left.
\quad
-2\prod_{q=1}^3
(\zeta_{\epsilon_1}*\zeta_{\epsilon_2}*
\eta_{2(r_2-\alpha_q)})
\left(y_1-y_2\right)
+
\prod_{n=1}^3
(\zeta_{\epsilon_1}*\zeta_{\epsilon_1}*
\eta_{2(r_2-\alpha_n)})
\left(y_1-y_2\right)
\right.\right.\\
&\left.\left.
\quad
-2\prod_{m=1}^3
(\zeta_{\epsilon_2}*\zeta_{\epsilon_2}*
\eta_{r_1+r_2-2\alpha_m})
\left(y_1-y_2\right)
+
4\prod_{q=1}^3
(\zeta_{\epsilon_1}*\zeta_{\epsilon_2}*
\eta_{r_1+r_2-2\alpha_q})
\left(y_1-y_2\right)
\right.\right.\\
&\left.\left.
\quad
-2\prod_{n=1}^3
(\zeta_{\epsilon_1}*\zeta_{\epsilon_1}*
\eta_{r_1+r_2-2\alpha_n})
\left(y_1-y_2\right)
+
\prod_{m=1}^3
(\zeta_{\epsilon_2}*\zeta_{\epsilon_2}*
\eta_{2(r_1-\alpha_m)})
\left(y_1-y_2\right)
\right.\right.\\
&\left.\left.
\quad
-2\prod_{q=1}^3
(\zeta_{\epsilon_1}*\zeta_{\epsilon_2}*
\eta_{2(r_1-\alpha_q)})
\left(y_1-y_2\right)
+
\prod_{n=1}^3
(\zeta_{\epsilon_1}*\zeta_{\epsilon_1}*
\eta_{2(r_1-\alpha_n)})
\left(y_1-y_2\right)
\right)
\,\mathrm{d}y_1\,\mathrm{d}y_2
\right)
\mathrm{d}\alpha_1\,\mathrm{d}\alpha_2\,\mathrm{d}\alpha_3\\
&\phantom{=}\ 
+\mu_2
\int_{\left[0,r_2\right]^3-\left[0,r_1\right]^3}
\left(
\int_{\left(\mathbb{T}^d\right)^2}
K_{j-1}\left(y_1\right)
K_{j-1}\left(y_2\right)
\left(
\prod_{m=1}^3
(\zeta_{\epsilon_2}*\zeta_{\epsilon_2}*
\eta_{2(r_2-\alpha_m)})
\left(y_1-y_2\right)
\right.\right.\\
&\left.\left.
\quad
-2\prod_{q=1}^3
(\zeta_{\epsilon_1}*\zeta_{\epsilon_2}*
\eta_{2(r_2-\alpha_q)})
\left(y_1-y_2\right)
+
\prod_{n=1}^3
(\zeta_{\epsilon_1}*\zeta_{\epsilon_1}*
\eta_{2(r_2-\alpha_n)})
\left(y_1-y_2\right)
\right)
\,\mathrm{d}y_1\,\mathrm{d}y_2
\right)
\mathrm{d}\alpha_1\,\mathrm{d}\alpha_2\,\mathrm{d}\alpha_3.
\end{align*}

Using the semi-group property
\[
\eta_{q_1}*\eta_{q_2}=\eta_{q_1+q_2},
\]
before applying Parseval, the preceding expression becomes
\begin{align*}
&\stackrel{\text{Parseval}}{=}
\mu_2
\sum_{\left(m_1,m_2,m_3\right)\in\left(\mathbb{Z}^d\right)^3}
\left|
\rho_{j-1}\left(m_1+m_2+m_3\right)
\right|^2\\
&\quad\times
\int_{\left[0,r_1\right]^3}
\left|
\prod_{q=1}^3
\zeta\left(\epsilon_2m_q\right)
-
\prod_{q=1}^3
\zeta\left(\epsilon_1m_q\right)
\right|^2\\
&\quad\times
\left(
\prod_{q=1}^3
\mathrm{e}^{-\left|2\pi m_q\right|^4
\left(r_2-\alpha_q\right)}
-
\prod_{q=1}^3
\mathrm{e}^{-\left|2\pi m_q\right|^4
\left(r_1-\alpha_q\right)}
\right)^2
\,\mathrm{d}\alpha_1\,\mathrm{d}\alpha_2\,\mathrm{d}\alpha_3\\
&\phantom{=}\ 
+\mu_2
\sum_{\left(m_1,m_2,m_3\right)\in\left(\mathbb{Z}^d\right)^3}
\left|
\rho_{j-1}\left(m_1+m_2+m_3\right)
\right|^2\\
&\quad\times
\int_{\left[0,r_2\right]^3-\left[0,r_1\right]^3}
\left|
\prod_{q=1}^3
\zeta\left(\epsilon_2m_q\right)
-
\prod_{q=1}^3
\zeta\left(\epsilon_1m_q\right)
\right|^2
\prod_{q=1}^3
\mathrm{e}^{-2\left|2\pi m_q\right|^4
\left(r_2-\alpha_q\right)}
\,\mathrm{d}\alpha_1\,\mathrm{d}\alpha_2\,\mathrm{d}\alpha_3.
\end{align*}

We now estimate the two terms obtained after Parseval. Set
\[
h:=r_2-r_1
\]
and
\[
\lambda_q:=|2\pi m_q|^4,
\qquad q=1,2,3.
\]

We first estimate the mollification increment. By adding and subtracting
intermediate products, we have
\begin{align*}
&\prod_{q=1}^3\zeta(\epsilon_2m_q)
-
\prod_{q=1}^3\zeta(\epsilon_1m_q)\\
&=
\bigl(\zeta(\epsilon_2m_1)-\zeta(\epsilon_1m_1)\bigr)
\zeta(\epsilon_2m_2)\zeta(\epsilon_2m_3)\\
&\quad+
\zeta(\epsilon_1m_1)
\bigl(\zeta(\epsilon_2m_2)-\zeta(\epsilon_1m_2)\bigr)
\zeta(\epsilon_2m_3)\\
&\quad+
\zeta(\epsilon_1m_1)\zeta(\epsilon_1m_2)
\bigl(\zeta(\epsilon_2m_3)-\zeta(\epsilon_1m_3)\bigr).
\end{align*}
Since $\zeta$ is smooth and bounded,
\[
\left|
\zeta(\epsilon_2m)-\zeta(\epsilon_1m)
\right|
\lesssim
\min\left\{1,|\epsilon_2-\epsilon_1||m|\right\}.
\]
Consequently, for $\gamma\in(0,1)$,
\begin{align*}
&\left|
\prod_{q=1}^3\zeta(\epsilon_2m_q)
-
\prod_{q=1}^3\zeta(\epsilon_1m_q)
\right|^2\\
&\lesssim
|\epsilon_2-\epsilon_1|^\gamma
\sum_{q=1}^3
(1+\lambda_q)^{\gamma/4}\\
&\lesssim
|\epsilon_2-\epsilon_1|^\gamma
\prod_{q=1}^3
(1+\lambda_q)^{\gamma/4}.
\end{align*}

We first consider the integral over $[0,r_1]^3$. Since
\[
r_2-\alpha_q
=
h+r_1-\alpha_q,
\]
we have
\begin{align*}
&\prod_{q=1}^3
\mathrm{e}^{-\lambda_q(r_2-\alpha_q)}
-
\prod_{q=1}^3
\mathrm{e}^{-\lambda_q(r_1-\alpha_q)}\\
&=
\prod_{q=1}^3
\mathrm{e}^{-\lambda_q(r_1-\alpha_q)}
\left(
\mathrm{e}^{-h(\lambda_1+\lambda_2+\lambda_3)}-1
\right).
\end{align*}
Hence
\begin{align*}
&\int_{[0,r_1]^3}
\left(
\prod_{q=1}^3
\mathrm{e}^{-\lambda_q(r_2-\alpha_q)}
-
\prod_{q=1}^3
\mathrm{e}^{-\lambda_q(r_1-\alpha_q)}
\right)^2
\,\mathrm{d}\alpha_1\,\mathrm{d}\alpha_2\,\mathrm{d}\alpha_3\\
&=
\left(
1-\mathrm{e}^{-h(\lambda_1+\lambda_2+\lambda_3)}
\right)^2
\prod_{q=1}^3
\int_0^{r_1}
\mathrm{e}^{-2\lambda_q(r_1-\alpha_q)}
\,\mathrm{d}\alpha_q.
\end{align*}
Using
\[
1-\mathrm{e}^{-x}\lesssim\min\{1,x\},
\]
and Lemma~\ref{l33}, we obtain
\[
\left(1-\mathrm{e}^{-x}\right)^2
\lesssim
x^{\gamma/2}.
\]
Moreover,
\[
\int_0^{r_1}
\mathrm{e}^{-2\lambda_q(r_1-\alpha_q)}
\,\mathrm{d}\alpha_q
\lesssim
(1+\lambda_q)^{-1}.
\]
Therefore,
\begin{align*}
&\int_{[0,r_1]^3}
\left(
\prod_{q=1}^3
\mathrm{e}^{-\lambda_q(r_2-\alpha_q)}
-
\prod_{q=1}^3
\mathrm{e}^{-\lambda_q(r_1-\alpha_q)}
\right)^2
\,\mathrm{d}\alpha_1\,\mathrm{d}\alpha_2\,\mathrm{d}\alpha_3\\
&\lesssim
h^{\gamma/2}
(\lambda_1+\lambda_2+\lambda_3)^{\gamma/2}
\prod_{q=1}^3(1+\lambda_q)^{-1}.
\end{align*}
Since
\[
1+\lambda_1+\lambda_2+\lambda_3
\leq
(1+\lambda_1)(1+\lambda_2)(1+\lambda_3),
\]
we get
\begin{align*}
&\int_{[0,r_1]^3}
\left(
\prod_{q=1}^3
\mathrm{e}^{-\lambda_q(r_2-\alpha_q)}
-
\prod_{q=1}^3
\mathrm{e}^{-\lambda_q(r_1-\alpha_q)}
\right)^2
\,\mathrm{d}\alpha_1\,\mathrm{d}\alpha_2\,\mathrm{d}\alpha_3\\
&\lesssim
h^{\gamma/2}
\prod_{q=1}^3
(1+\lambda_q)^{\gamma/2-1}.
\end{align*}

Multiplying by the mollification increment, we obtain
\begin{align*}
&\left|
\prod_{q=1}^3\zeta(\epsilon_2m_q)
-
\prod_{q=1}^3\zeta(\epsilon_1m_q)
\right|^2\\
&\quad\times
\int_{[0,r_1]^3}
\left(
\prod_{q=1}^3
\mathrm{e}^{-\lambda_q(r_2-\alpha_q)}
-
\prod_{q=1}^3
\mathrm{e}^{-\lambda_q(r_1-\alpha_q)}
\right)^2
\,\mathrm{d}\alpha_1\,\mathrm{d}\alpha_2\,\mathrm{d}\alpha_3\\
&\lesssim
|\epsilon_2-\epsilon_1|^\gamma
h^{\gamma/2}
\prod_{q=1}^3
(1+\lambda_q)^{\frac{3\gamma}{4}-1}.
\end{align*}
Since
\[
\frac{3\gamma}{4}-1
\leq
\gamma-1,
\]
we conclude that
\begin{align*}
&\left|
\prod_{q=1}^3\zeta(\epsilon_2m_q)
-
\prod_{q=1}^3\zeta(\epsilon_1m_q)
\right|^2\\
&\quad\times
\int_{[0,r_1]^3}
\left(
\prod_{q=1}^3
\mathrm{e}^{-\lambda_q(r_2-\alpha_q)}
-
\prod_{q=1}^3
\mathrm{e}^{-\lambda_q(r_1-\alpha_q)}
\right)^2
\,\mathrm{d}\alpha_1\,\mathrm{d}\alpha_2\,\mathrm{d}\alpha_3\\
&\lesssim
|\epsilon_2-\epsilon_1|^\gamma
h^{\gamma/2}
\prod_{q=1}^3
(1+\lambda_q)^{\gamma-1}.
\end{align*}

We now turn to the second time integral. We use
\[
[0,r_2]^3-[0,r_1]^3
\subset
\bigl([r_1,r_2]\times[0,r_2]^2\bigr)
\cup
\bigl([0,r_2]\times[r_1,r_2]\times[0,r_2]\bigr)
\cup
\bigl([0,r_2]^2\times[r_1,r_2]\bigr).
\]
Therefore,
\begin{align*}
&\int_{[0,r_2]^3-[0,r_1]^3}
\prod_{q=1}^3
\mathrm{e}^{-2\lambda_q(r_2-\alpha_q)}
\,\mathrm{d}\alpha_1\,\mathrm{d}\alpha_2\,\mathrm{d}\alpha_3\\
&\lesssim
\sum_{i=1}^3
\left(
\int_{r_1}^{r_2}
\mathrm{e}^{-2\lambda_i(r_2-\alpha_i)}
\,\mathrm{d}\alpha_i
\right)
\prod_{\substack{q=1\\q\neq i}}^3
\left(
\int_0^{r_2}
\mathrm{e}^{-2\lambda_q(r_2-\alpha_q)}
\,\mathrm{d}\alpha_q
\right).
\end{align*}
For every $q\in\{1,2,3\}$,
\[
\int_0^{r_2}
\mathrm{e}^{-2\lambda_q(r_2-\alpha)}
\,\mathrm{d}\alpha
\lesssim
(1+\lambda_q)^{-1}.
\]
Moreover,
\[
\int_{r_1}^{r_2}
\mathrm{e}^{-2\lambda_q(r_2-\alpha)}
\,\mathrm{d}\alpha
\lesssim h
\]
and
\[
\int_{r_1}^{r_2}
\mathrm{e}^{-2\lambda_q(r_2-\alpha)}
\,\mathrm{d}\alpha
\lesssim
(1+\lambda_q)^{-1}.
\]
Applying Lemma~\ref{l33} with exponent $\gamma/2$ yields
\[
\int_{r_1}^{r_2}
\mathrm{e}^{-2\lambda_q(r_2-\alpha)}
\,\mathrm{d}\alpha
\lesssim
h^{\gamma/2}
(1+\lambda_q)^{\frac{\gamma}{2}-1}.
\]
Hence
\begin{align*}
&\int_{[0,r_2]^3-[0,r_1]^3}
\prod_{q=1}^3
\mathrm{e}^{-2\lambda_q(r_2-\alpha_q)}
\,\mathrm{d}\alpha_1\,\mathrm{d}\alpha_2\,\mathrm{d}\alpha_3\\
&\lesssim
h^{\gamma/2}
\sum_{i=1}^3
(1+\lambda_i)^{\frac{\gamma}{2}-1}
\prod_{\substack{q=1\\q\neq i}}^3
(1+\lambda_q)^{-1}.
\end{align*}

Multiplying by the mollification increment gives
\begin{align*}
&\left|
\prod_{q=1}^3\zeta(\epsilon_2m_q)
-
\prod_{q=1}^3\zeta(\epsilon_1m_q)
\right|^2\\
&\quad\times
\int_{[0,r_2]^3-[0,r_1]^3}
\prod_{q=1}^3
\mathrm{e}^{-2\lambda_q(r_2-\alpha_q)}
\,\mathrm{d}\alpha_1\,\mathrm{d}\alpha_2\,\mathrm{d}\alpha_3\\
&\lesssim
|\epsilon_2-\epsilon_1|^\gamma
h^{\gamma/2}
\prod_{q=1}^3(1+\lambda_q)^{\gamma/4}\\
&\quad\times
\sum_{i=1}^3
(1+\lambda_i)^{\frac{\gamma}{2}-1}
\prod_{\substack{q=1\\q\neq i}}^3
(1+\lambda_q)^{-1}.
\end{align*}
Thus,
\begin{align*}
&\left|
\prod_{q=1}^3\zeta(\epsilon_2m_q)
-
\prod_{q=1}^3\zeta(\epsilon_1m_q)
\right|^2\\
&\quad\times
\int_{[0,r_2]^3-[0,r_1]^3}
\prod_{q=1}^3
\mathrm{e}^{-2\lambda_q(r_2-\alpha_q)}
\,\mathrm{d}\alpha_1\,\mathrm{d}\alpha_2\,\mathrm{d}\alpha_3\\
&\lesssim
|\epsilon_2-\epsilon_1|^\gamma
h^{\gamma/2}
\sum_{i=1}^3
(1+\lambda_i)^{\frac{3\gamma}{4}-1}
\prod_{\substack{q=1\\q\neq i}}^3
(1+\lambda_q)^{\frac{\gamma}{4}-1}.
\end{align*}
Since
\[
\frac{3\gamma}{4}\leq\gamma
\qquad\text{and}\qquad
\frac{\gamma}{4}\leq\gamma,
\]
each term in the sum is bounded by
\[
\prod_{q=1}^3
(1+\lambda_q)^{\gamma-1}.
\]
Consequently,
\begin{align*}
&\left|
\prod_{q=1}^3\zeta(\epsilon_2m_q)
-
\prod_{q=1}^3\zeta(\epsilon_1m_q)
\right|^2\\
&\quad\times
\int_{[0,r_2]^3-[0,r_1]^3}
\prod_{q=1}^3
\mathrm{e}^{-2\lambda_q(r_2-\alpha_q)}
\,\mathrm{d}\alpha_1\,\mathrm{d}\alpha_2\,\mathrm{d}\alpha_3\\
&\lesssim
|\epsilon_2-\epsilon_1|^\gamma
h^{\gamma/2}
\prod_{q=1}^3
(1+\lambda_q)^{\gamma-1}.
\end{align*}

Combining the estimates for the two terms and recalling that
\[
\lambda_q=|2\pi m_q|^4,
\]
we obtain
\begin{align*}
&\stackrel{\text{Lemma \ref{l33}}}{\leq}
\mu_{3}
|\epsilon_2-\epsilon_1|^{\gamma}
\left|r_2-r_1\right|^{\frac{\gamma}{2}}
\sum_{\left(m_1,m_2,m_3\right)\in\left(\mathbb{Z}^d\right)^3}
\left|
\rho_{j-1}\left(m_1+m_2+m_3\right)
\right|^2
\prod_{q=1}^3
\left(1+\left|m_q\right|^4\right)^{\gamma-1}.
\end{align*}

We next make explicit the discrete convolution estimates. First, by setting
\[
n=m_2+m_3
\]
and then renaming $n$ as $m_2$, we have
\begin{align*}
&\sum_{\left(m_1,m_2,m_3\right)\in\left(\mathbb{Z}^d\right)^3}
\left|
\rho_{j-1}\left(m_1+m_2+m_3\right)
\right|^2
\prod_{q=1}^3
\left(1+\left|m_q\right|^4\right)^{\gamma-1}\\
&=
\sum_{\left(m_1,m_2\right)\in\left(\mathbb{Z}^d\right)^2}
\left|
\rho_{j-1}\left(m_1+m_2\right)
\right|^2
\left(1+\left|m_1\right|^4\right)^{\gamma-1}\\
&\qquad\times
\sum_{m_3\in\mathbb{Z}^d}
\left(1+\left|m_2-m_3\right|^4\right)^{\gamma-1}
\left(1+\left|m_3\right|^4\right)^{\gamma-1}.
\end{align*}

For the inner sum, write
\[
\left(1+\left|m_2-m_3\right|^4\right)^{\gamma-1}
=
\left(1+\left|m_2-m_3\right|^4\right)^{-\frac{4(1-\gamma)}{4}}
\]
and
\[
\left(1+\left|m_3\right|^4\right)^{\gamma-1}
=
\left(1+\left|m_3\right|^4\right)^{-\frac{4(1-\gamma)}{4}}.
\]
Thus, in Lemma~\ref{lemma2}, we take
\[
\alpha=\beta=4(1-\gamma).
\]
Consequently,
\begin{align*}
&\sum_{m_3\in\mathbb{Z}^d}
\left(1+\left|m_2-m_3\right|^4\right)^{\gamma-1}
\left(1+\left|m_3\right|^4\right)^{\gamma-1}\\
&\stackrel{\text{Lemma \ref{lemma2}}}{\lesssim}
\left(1+\left|m_2\right|^4\right)^{
\frac{d-4(1-\gamma)-4(1-\gamma)}{4}}\\
&=
\left(1+\left|m_2\right|^4\right)^{
\frac{d}{4}+2\gamma-2}.
\end{align*}
Hence
\begin{align*}
&\leq
\mu_4
|\epsilon_2-\epsilon_1|^{\gamma}
\left|r_2-r_1\right|^{\frac{\gamma}{2}}
\sum_{\left(m_1,m_2\right)\in\left(\mathbb{Z}^d\right)^2}
\left|
\rho_{j-1}\left(m_1+m_2\right)
\right|^2
\left(1+\left|m_1\right|^4\right)^{\gamma-1}
\left(1+\left|m_2\right|^4\right)^{\frac{d}{4}+2\gamma-2}.
\end{align*}

We now set
\[
n=m_1+m_2
\]
and rename $n$ as $m_1$. We obtain
\begin{align*}
&=
\mu_4
|\epsilon_2-\epsilon_1|^{\gamma}
\left|r_2-r_1\right|^{\frac{\gamma}{2}}
\sum_{m_1\in\mathbb{Z}^d}
\left|
\rho_{j-1}\left(m_1\right)
\right|^2\\
&\qquad\times
\sum_{m_2\in\mathbb{Z}^d}
\left(1+\left|m_1-m_2\right|^4\right)^{\gamma-1}
\left(1+\left|m_2\right|^4\right)^{\frac{d}{4}+2\gamma-2}.
\end{align*}

For the second application of Lemma~\ref{lemma2}, observe that
\[
\gamma-1
=
-\frac{4-4\gamma}{4}
\]
and
\[
\frac{d}{4}+2\gamma-2
=
-\frac{8-d-8\gamma}{4}.
\]
Therefore, we take
\[
\alpha=4-4\gamma,
\qquad
\beta=8-d-8\gamma.
\]
Since
\[
\alpha+\beta
=
12-d-12\gamma,
\]
the condition
\[
\gamma<1-\frac{d}{6}
\]
implies
\[
\alpha+\beta>d.
\]
Hence Lemma~\ref{lemma2} gives
\begin{align*}
&\sum_{m_2\in\mathbb{Z}^d}
\left(1+\left|m_1-m_2\right|^4\right)^{\gamma-1}
\left(1+\left|m_2\right|^4\right)^{\frac{d}{4}+2\gamma-2}\\
&\stackrel{\text{Lemma \ref{lemma2}}}{\lesssim}
\left(1+\left|m_1\right|^4\right)^{
\frac{d-(4-4\gamma)-(8-d-8\gamma)}{4}}\\
&=
\left(1+\left|m_1\right|^4\right)^{
\frac{d}{2}+3\gamma-3}.
\end{align*}
Consequently,
\begin{align*}
&\stackrel{\text{Lemma \ref{lemma2}}}{\leq}
\mu_{5}
|\epsilon_2-\epsilon_1|^{\gamma}
\left|r_2-r_1\right|^{\frac{\gamma}{2}}
\sum_{m\in\mathbb{Z}^d}
\left|\rho_{j-1}\left(m\right)\right|^2
\left(1+\left|m\right|^4\right)^{\frac{d}{2}+3\gamma-3}.
\end{align*}

Finally, for $j\geq1$, by the support properties of $\rho_{j-1}$,
$|m|\asymp 2^{j-1}$ whenever 
$\rho_{j-1}(m)\neq0.$
Moreover, the number of lattice points in the corresponding dyadic annulus is
of order $2^{d(j-1)}$. For $j=0$, the support of $\rho_{-1}$ is bounded,
and the corresponding estimate follows directly. Therefore,
\begin{align*}
&\sum_{m\in\mathbb{Z}^d}
\left|\rho_{j-1}\left(m\right)\right|^2
\left(1+\left|m\right|^4\right)^{\frac{d}{2}+3\gamma-3}\\
&\lesssim
2^{d(j-1)}
2^{4\left(\frac{d}{2}+3\gamma-3\right)(j-1)}\\
&=
2^{\left(j-1\right)\left(3d+12\gamma-12\right)}.
\end{align*}
Thus,
\begin{align*}
&\left(\mathds{E}\left[\left|\left(\delta_{j-1}\widetilde{X}_{r_2,\epsilon_2,3}\right)\left(x\right)-\left(\delta_{j-1}\widetilde{X}_{r_2,\epsilon_1,3}\right)\left(x\right)-\left(\delta_{j-1}\widetilde{X}_{r_1,\epsilon_2,3}\right)\left(x\right)+\left(\delta_{j-1}\widetilde{X}_{r_1,\epsilon_1,3}\right)\left(x\right)\right|^\beta\right]\right)^{\frac{2}{\beta}}\\
&\leq
\mu_{6}
2^{\left(j-1\right)\left(3d+12\gamma-12\right)}
|\epsilon_2-\epsilon_1|^{\gamma}
\left|r_2-r_1\right|^{\frac{\gamma}{2}}.
\end{align*}

We deduce from the Garsia--Rodemich--Rumsey lemma \cite{MR267632} that for every
$U\in\mathbb{R}_+$,
\[
\lim_{\epsilon_1\downarrow0,\epsilon_2\downarrow0}
\mathds{E}\left[
\sup_{r\in[0,U]}
\left\|
\widetilde{X}_{r,\epsilon_2,3}
-
\widetilde{X}_{r,\epsilon_1,3}
\right\|_{\mathscr{C}^{\theta}}
\right]
=0.
\]
\end{proof}

\chapter{}
\label{ppe4}
\begin{proof}[Proof of Theorem \ref{theorem4}]
As the computations are quite similar, we only show the $L^2(\Omega)$ estimates at a fixed time $r$ and not for an increment. 

Let $\gamma \in \left]0,\frac{1}{6}\right[,$
\begin{align*}
\psi:\left(\mathbb{Z}^d\right)^2&\longrightarrow \left[0,1\right]\\
\left(m_1,m_2\right)&\longmapsto\psi\left(m_1,m_2\right)=\sum_{q \in \mathbb{N}}\sum_{j=0}^{\min\left(q+1,2\right)}\rho_{j+\max\left(q-2,-1\right)}\left(m_1\right)\rho_{q-1}\left(m_2\right)
\end{align*}
The following observation is crucial for our estimation $$\psi(m_1,m_2)\neq 0 \implies \kappa_1 (1+|m_2|)\leq 1+|m_1|\leq \kappa_2 (1+|m_2|)$$ for $0<\kappa_1<\kappa_2.$

We note that $(\delta_{m-1}\Xi_{r,\epsilon,1})(x)$ is not a Wiener-Ito integral. Therefore we have to use Lemma \ref{l3} to express it as a linear combination of Wiener-Ito integrals: $$(\delta_{m-1}\Xi_{r,\epsilon,1})(x)=I_{r,\epsilon,1}+3I_{r,\epsilon,2},$$ where $$I_{r,\epsilon,1}=\int_{\mathbb{T}^d}dyK_{m-1}(x-y)\sum_{q \in \mathbb{N}}\sum_{j=0}^{\min\left(q+1,2\right)}V_{4}(f_{r,\epsilon,j,q,y}\otimes_0g_{r,\epsilon,q,y})=V_{4}(\int_{\mathbb{T}^d}dyK_{m-1}(x-y)\sum_{q \in \mathbb{N}}\sum_{j=0}^{\min\left(q+1,2\right)}f_{r,\epsilon,j,q,y}\otimes_0g_{r,\epsilon,q,y}),$$ $$I_{r,\epsilon,2}=\int_{\mathbb{T}^d}dyK_{m-1}(x-y)\sum_{q \in \mathbb{N}}\sum_{j=0}^{\min\left(q+1,2\right)}V_{2}(f_{r,\epsilon,j,q,y}\otimes_1g_{r,\epsilon,q,y})=V_{2}(\int_{\mathbb{T}^d}dyK_{m-1}(x-y)\sum_{q \in \mathbb{N}}\sum_{j=0}^{\min\left(q+1,2\right)}f_{r,\epsilon,j,q,y}\otimes_1g_{r,\epsilon,q,y}),$$ for $$f_{r,\epsilon,j,q,y}((u_1,x_1),(u_2,x_2),(u_3,x_3))=\int_0^r$$$$(K_{j+\max(q-2,-1)}*\Delta\eta_{r-v}*((\zeta_\epsilon*\eta_{v-u_1})(\cdot-x_1)(\zeta_\epsilon*\eta_{v-u_2})(\cdot-x_2)(\zeta_\epsilon*\eta_{v-u_3})(\cdot-x_3)))(y)\mathds{1}_{[0,v]^3}(u_1,u_2,u_3)dv,$$ $$g_{r,\epsilon,q,y}(u_4,x_4)=(K_{q-1}*\zeta_\epsilon*\eta_{r-u_4}(\cdot-x_4))(y)\mathds{1}_{[0,r]}(u_4).$$  

Using Theorem \ref{nelson}, we obtain

\begin{align*}
&\mathds{E}[|I_{r,\epsilon,1}|^2]\\
&\lesssim \int_{(\mathbb{R}_+\times\mathbb{T}^d)^4}du_1dx_1du_2dx_2du_3dx_3du_4dx_4|\int_{\mathbb{T}^d}dyK_{m-1}(x-y)\\
&\sum_{q \in \mathbb{N}}\sum_{j=0}^{\min\left(q+1,2\right)}f_{r,\epsilon,j,q,y}((u_1,x_1),(u_2,x_2),(u_3,x_3))g_{r,\epsilon,q,y}(u_4,x_4)|^2\\
&\stackrel{\text{Fubini}}{=}\int_{(\mathbb{R}_+\times\mathbb{T}^d)^4}du_1dx_1du_2dx_2du_3dx_3du_4dx_4\int_{(\mathbb{T}^d)^2}dy_1dy_2K_{m-1}(x-y_1)K_{m-1}(x-y_2)\\
&\sum_{(q_1,q_2) \in \mathbb{N}^2}\sum_{j_1=0}^{\min\left(q_1+1,2\right)}\sum_{j_2=0}^{\min(q_2+1,2)}f_{r,\epsilon,j_1,q_1,y_1}((u_1,x_1),(u_2,x_2),(u_3,x_3))g_{r,\epsilon,q_1,y_1}(u_4,x_4)\\
&f_{r,\epsilon,j_2,q_2,y_2}((u_1,x_1),(u_2,x_2),(u_3,x_3))g_{r,\epsilon,q_2,y_2}(u_4,x_4)\\
&=\int_{[0,r]^2}dv_1dv_2\int_{(\mathbb{R}_+\times\mathbb{T}^d)^4}du_1dx_1du_2dx_2du_3dx_3du_4dx_4\int_{(\mathbb{T}^d)^6}dy_1dy_2dy_1'dy_2'dy'_3dy_4'K_{m-1}(x-y_1)K_{m-1}(x-y_2)\\
&\sum_{(q_1,q_2) \in \mathbb{N}^2}\sum_{j_1=0}^{\min\left(q_1+1,2\right)}\sum_{j_2=0}^{\min(q_2+1,2)}(K_{j_1+\max(q_1-2,-1)}*\Delta\eta_{r-v_1})(y_1-y'_1)(K_{j_2+\max(q_2-2,-1)}*\Delta\eta_{r-v_2})(y_2-y'_3)\\
&K_{q_1-1}(y_1-y'_2)K_{q_2-1}(y_2-y'_4)(\zeta_\epsilon*\eta_{v_1-u_1})(y'_1-x_1)(\zeta_\epsilon*\eta_{v_1-u_2})(y_1'-x_2)(\zeta_\epsilon*\eta_{v_1-u_3})(y'_1-x_3)\\&(\zeta_\epsilon*\eta_{r-u_4})(y'_2-x_4)(\zeta_\epsilon*\eta_{v_2-u_1})(y'_3-x_1)
(\zeta_\epsilon*\eta_{v_2-u_2})(y_3'-x_2)(\zeta_\epsilon*\eta_{v_2-u_3})(y'_3-x_3)(\zeta_\epsilon*\eta_{r-u_4})(y_4'-x_4)\\&\mathds{1}_{[0,r]}(u_4)\mathds{1}_{[0,\min(v_1,v_2)]^3}(u_1,u_2,u_3)\\
&=\int_{[0,r]^2}dv_1dv_2\int_{[0,\min(v_1,v_2)]^3\times[0,r]}du_1du_2du_3du_4\int_{(\mathbb{T}^d)^6}dy_1dy_2dy_1'dy_2'dy'_3dy_4'K_{m-1}(y_1)K_{m-1}(y_2)\\
&\sum_{(q_1,q_2) \in \mathbb{N}^2}\sum_{j_1=0}^{\min\left(q_1+1,2\right)}\sum_{j_2=0}^{\min(q_2+1,2)}(K_{j_1+\max(q_1-2,-1)}*\Delta\eta_{r-v_1})(y_1-y'_1)(K_{j_2+\max(q_2-2,-1)}*\Delta\eta_{r-v_2})(y_2-y'_3)\\
&K_{q_1-1}(y_1-y'_2)K_{q_2-1}(y_2-y'_4)(\zeta_\epsilon*\zeta_{\epsilon}*\eta_{v_1+v_2-2u_1})(y'_1-y_3')\\&(\zeta_\epsilon*\zeta_{\epsilon}*\eta_{v_1+v_2-2u_2})(y_1'-y'_3)(\zeta_\epsilon*\zeta_{\epsilon}*\eta_{v_1+v_2-2u_3})(y'_1-y_3')(\zeta_\epsilon*\zeta_{\epsilon}*\eta_{2(r-u_4)})(y'_2-y_4')\\
&=\int_{[0,r]^2}dv_1dv_2\int_{[0,\min(v_1,v_2)]^3\times[0,r]}du_1du_2du_3du_4\int_{(\mathbb{T}^d)^2}dy_1dy_2K_{m-1}(y_1)K_{m-1}(y_2)\\
&\sum_{(q_1,q_2) \in \mathbb{N}^2}\sum_{j_1=0}^{\min\left(q_1+1,2\right)}\sum_{j_2=0}^{\min(q_2+1,2)}(K_{j_1+\max(q_1-2,-1)}*\Delta^2\eta_{2r-v_1-v_2}*K_{j_2+\max(q_2-2,-1)}\\&*((\zeta_\epsilon*\zeta_{\epsilon}*\eta_{v_1+v_2-2u_1})(\zeta_\epsilon*\zeta_{\epsilon}*\eta_{v_1+v_2-2u_2})(\zeta_\epsilon*\zeta_{\epsilon}*\eta_{v_1+v_2-2u_3})))(y_1-y_2)\\
&\times (K_{q_1-1}*K_{q_2-1}*(\zeta_\epsilon*\zeta_{\epsilon}*\eta_{2(r-u_4)}))(y_1-y_2),
\end{align*}
and \begin{align*}
&\mathds{E}[|I_{r,\epsilon,2}|^2]\\
&\lesssim \int_{(\mathbb{R}_+\times\mathbb{T}^d)^2}du_1dx_1du_2dx_2|\int_{\mathbb{T}^d}dyK_{m-1}(x-y)\\
&\sum_{q \in \mathbb{N}}\sum_{j=0}^{\min\left(q+1,2\right)}\int_{\mathbb{R}_+\times \mathbb{T}^d}du_3dx_3f_{r,\epsilon,j,q,y}((u_1,x_1),(u_2,x_2),(u_3,x_3))g_{r,\epsilon,q,y}(u_3,x_3)|^2\\
&\stackrel{\text{Fubini}}{=}\int_{(\mathbb{R}_+\times\mathbb{T}^d)^2}du_1dx_1du_2dx_2\int_{(\mathbb{T}^d)^2}dy_1dy_2K_{m-1}(x-y_1)K_{m-1}(x-y_2)\\
&\sum_{(q_1,q_2) \in \mathbb{N}^2}\sum_{j_1=0}^{\min\left(q_1+1,2\right)}\sum_{j_2=0}^{\min(q_2+1,2)}\int_{(\mathbb{R}_+\times \mathbb{T}^d)^2}du_3dx_3du_4dx_4f_{r,\epsilon,j_1,q_1,y_1}((u_1,x_1),(u_2,x_2),(u_3,x_3))g_{r,\epsilon,q_1,y_1}(u_3,x_3)\\
&f_{r,\epsilon,j_2,q_2,y_2}((u_1,x_1),(u_2,x_2),(u_4,x_4))g_{r,\epsilon,q_2,y_2}(u_4,x_4)\\
&=\int_{[0,r]^2}dv_1dv_2\int_{(\mathbb{R}_+\times\mathbb{T}^d)^4}du_1dx_1du_2dx_2du_3dx_3du_4dx_4\int_{(\mathbb{T}^d)^6}dy_1dy_2dy_1'dy_2'dy'_3dy_4'K_{m-1}(x-y_1)K_{m-1}(x-y_2)\\
&\sum_{(q_1,q_2) \in \mathbb{N}^2}\sum_{j_1=0}^{\min\left(q_1+1,2\right)}\sum_{j_2=0}^{\min(q_2+1,2)}(K_{j_1+\max(q_1-2,-1)}*\Delta\eta_{r-v_1})(y_1-y'_1)(K_{j_2+\max(q_2-2,-1)}*\Delta\eta_{r-v_2})(y_2-y'_3)\\
&K_{q_1-1}(y_1-y'_2)K_{q_2-1}(y_2-y'_4)(\zeta_\epsilon*\eta_{v_1-u_1})(y'_1-x_1)(\zeta_\epsilon*\eta_{v_1-u_2})(y_1'-x_2)(\zeta_\epsilon*\eta_{v_1-u_3})(y'_1-x_3)\\&(\zeta_\epsilon*\eta_{r-u_3})(y'_2-x_3)(\zeta_\epsilon*\eta_{v_2-u_1})(y'_3-x_1)
(\zeta_\epsilon*\eta_{v_2-u_2})(y_3'-x_2)(\zeta_\epsilon*\eta_{v_2-u_4})(y'_3-x_4)(\zeta_\epsilon*\eta_{r-u_4})(y_4'-x_4)\\&\mathds{1}_{[0,\min(v_1,v_2)]^2\times[0,v_1]\times[0,v_2]}(u_1,u_2,u_3,u_4)\\
&=\int_{[0,r]^2}dv_1dv_2\int_{[0,\min(v_1,v_2)]^2\times[0,v_1]\times[0,v_2]}du_1du_2du_3du_4\int_{(\mathbb{T}^d)^6}dy_1dy_2dy_1'dy_2'dy'_3dy_4'K_{m-1}(y_1)K_{m-1}(y_2)\\
&\sum_{(q_1,q_2) \in \mathbb{N}^2}\sum_{j_1=0}^{\min\left(q_1+1,2\right)}\sum_{j_2=0}^{\min(q_2+1,2)}(K_{j_1+\max(q_1-2,-1)}*\Delta\eta_{r-v_1})(y_1-y'_1)(K_{j_2+\max(q_2-2,-1)}*\Delta\eta_{r-v_2})(y_2-y'_3)\\
&K_{q_1-1}(y_1-y'_2)K_{q_2-1}(y_2-y'_4)(\zeta_\epsilon*\zeta_{\epsilon}*\eta_{v_1+v_2-2u_1})(y'_1-y_3')\\&(\zeta_\epsilon*\zeta_{\epsilon}*\eta_{v_1+v_2-2u_2})(y_1'-y'_3)(\zeta_\epsilon*\zeta_{\epsilon}*\eta_{r+v_1-2u_3})(y'_1-y_2')(\zeta_\epsilon*\zeta_{\epsilon}*\eta_{r+v_2-2u_4})(y'_3-y_4')\\
&=\int_{[0,r]^2}dv_1dv_2\int_{[0,\min(v_1,v_2)]^2\times[0,v_1]\times[0,v_2]}du_1du_2du_3du_4\int_{(\mathbb{T}^d)^2}dy_1dy_2K_{m-1}(y_1)K_{m-1}(y_2)\\
&\sum_{(q_1,q_2) \in \mathbb{N}^2}\sum_{j_1=0}^{\min\left(q_1+1,2\right)}\sum_{j_2=0}^{\min(q_2+1,2)}(((K_{j_2+\max(q_2-2,-1)}*\Delta\eta_{r-v_2})\times(K_{q_2-1}*\zeta_\epsilon*\zeta_{\epsilon}*\eta_{r+v_2-2u_4}))\\&*((\zeta_\epsilon*\zeta_{\epsilon}*\eta_{v_1+v_2-2u_1})(\zeta_\epsilon*\zeta_{\epsilon}*\eta_{v_1+v_2-2u_2}))*((K_{j_1+\max(q_1-2,-1)}*\Delta\eta_{r-v_1})\times(K_{q_1-1}*\zeta_\epsilon*\zeta_{\epsilon}*\eta_{r+v_1-2u_3})))(y_1-y_2)
\end{align*}
After expressing the expectation in terms of convolutions and products, we are in a position to apply Parseval. 
Let us first explain the Fourier variables appearing in the computation. Recall that
\[
\widehat{K_{m-1}}(k)=\rho_{m-1}(k),
\qquad
\widehat{\eta_t}(k)=\mathrm{e}^{-t|2\pi k|^4},
\qquad
\widehat{\zeta_\epsilon}(k)=\zeta(\epsilon k).
\]
Moreover, the application of the Laplacian gives the Fourier multiplier
$-|2\pi k|^2$. We associate the frequencies $m_1,m_2,m_3$ with the three
stochastic convolutions appearing in the first factor and the frequency $m_4$
with the remaining stochastic convolution. The spatial integrations impose
conservation of frequency, so that the frequency selected by the
Littlewood--Paley block $K_{m-1}$ is
\[
m_1+m_2+m_3+m_4.
\]
Similarly, summing over the indices appearing in the resonant product gives
\[
\psi(m_1+m_2+m_3,m_4).
\]
Finally, the two Laplacians, one from each copy appearing after taking the
square, produce the factor
\[
|m_1+m_2+m_3|^4.
\]
Therefore, after renaming the time variables, Parseval gives
\begin{align*}
&\mathds{E}\left[\left|I_{r,\epsilon,1}\right|^2\right]\\
&\lesssim
\sum_{\left(m_1,m_2,m_3,m_4\right)\in\left(\mathbb{Z}^d\right)^4}
\left|
\rho_{m-1}\left(m_1+m_2+m_3+m_4\right)
\psi\left(m_1+m_2+m_3,m_4\right)
\prod_{k=1}^4\zeta(\epsilon m_k)
\right|^2
\left|m_1+m_2+m_3\right|^4\\
&\quad\times
\int_{\left[0,r\right]^3}
\mathrm{e}^{-\left|2\pi\left(m_1+m_2+m_3\right)\right|^4
\left(2r-\alpha_1-\alpha_2\right)
-2\left|2\pi m_4\right|^4\left(r-\alpha_3\right)}
\\
&\quad\times
\left(
\int_{\left[0,\min\left(\alpha_1,\alpha_2\right)\right]^3}
\mathrm{e}^{-\left|2\pi m_1\right|^4
\left(\alpha_1+\alpha_2-2v_1\right)
-\left|2\pi m_2\right|^4
\left(\alpha_1+\alpha_2-2v_2\right)
-\left|2\pi m_3\right|^4
\left(\alpha_1+\alpha_2-2v_3\right)}
\,\mathrm{d}v_1\,\mathrm{d}v_2\,\mathrm{d}v_3
\right)
\,\mathrm{d}\alpha_1\,\mathrm{d}\alpha_2\,\mathrm{d}\alpha_3.
\end{align*}

We now estimate the time integrals in detail. Since $\zeta$ is smooth and
compactly supported, its Fourier multipliers are uniformly bounded and can
therefore be absorbed into the implicit constant. Moreover,
\[
0\leq\psi\leq1,
\]
and whenever
\[
\psi(m_1+m_2+m_3,m_4)\neq0,
\]
the frequency localization of the resonant product implies
\[
1+|m_1+m_2+m_3|\asymp1+|m_4|.
\]

Set
\[
M:=m_1+m_2+m_3,
\qquad
A:=|2\pi M|^4,
\qquad
a_k:=|2\pi m_k|^4,\quad k=1,2,3,
\qquad
b:=|2\pi m_4|^4.
\]
We first integrate with respect to $\alpha_3$. We have
\begin{align*}
\int_0^r
\mathrm{e}^{-2b(r-\alpha_3)}
\,\mathrm{d}\alpha_3
&\lesssim
(1+b)^{-1}.
\end{align*}
Indeed, if $b>0$, then
\[
\int_0^r
\mathrm{e}^{-2b(r-\alpha_3)}
\,\mathrm{d}\alpha_3
=
\frac{1-\mathrm{e}^{-2br}}{2b}
\lesssim
b^{-1},
\]
whereas, if $b=0$, the integral is equal to $r$ and is bounded on every
bounded time interval.

We next estimate the integrals with respect to $v_1,v_2,v_3$. Let
\[
s:=\min(\alpha_1,\alpha_2).
\]
For every $k\in\{1,2,3\}$,
\begin{align*}
&\int_0^s
\mathrm{e}^{-a_k(\alpha_1+\alpha_2-2v)}
\,\mathrm{d}v\\
&=
\mathrm{e}^{-a_k|\alpha_1-\alpha_2|}
\int_0^s
\mathrm{e}^{-2a_k(s-v)}
\,\mathrm{d}v\\
&\lesssim
(1+a_k)^{-1}
\mathrm{e}^{-c a_k|\alpha_1-\alpha_2|}
\end{align*}
for some $c>0$. Consequently,
\begin{align*}
&\int_{\left[0,\min(\alpha_1,\alpha_2)\right]^3}
\mathrm{e}^{-\sum_{k=1}^3
a_k(\alpha_1+\alpha_2-2v_k)}
\,\mathrm{d}v_1\,\mathrm{d}v_2\,\mathrm{d}v_3\\
&\lesssim
\prod_{k=1}^3(1+a_k)^{-1}
\mathrm{e}^{-c(a_1+a_2+a_3)|\alpha_1-\alpha_2|}.
\end{align*}

Let
\[
S:=a_1+a_2+a_3.
\]
It follows that the complete time factor is bounded by
\begin{align*}
&|M|^4
\int_{\left[0,r\right]^3}
\mathrm{e}^{-A(2r-\alpha_1-\alpha_2)-2b(r-\alpha_3)}
\\
&\quad\times
\left(
\int_{\left[0,\min(\alpha_1,\alpha_2)\right]^3}
\mathrm{e}^{-\sum_{k=1}^3
a_k(\alpha_1+\alpha_2-2v_k)}
\,\mathrm{d}v_1\,\mathrm{d}v_2\,\mathrm{d}v_3
\right)
\,\mathrm{d}\alpha_1\,\mathrm{d}\alpha_2\,\mathrm{d}\alpha_3\\
&\lesssim
(1+b)^{-1}
\prod_{k=1}^3(1+a_k)^{-1}
A
\int_{[0,r]^2}
\mathrm{e}^{-A(2r-\alpha_1-\alpha_2)}
\mathrm{e}^{-cS|\alpha_1-\alpha_2|}
\,\mathrm{d}\alpha_1\,\mathrm{d}\alpha_2,
\end{align*}
where we used $|M|^4\asymp A$.

We now estimate the remaining double integral. Making the change of variables
\[
u:=r-\alpha_1,
\qquad
v:=r-\alpha_2,
\]
we obtain
\begin{align*}
&A
\int_{[0,r]^2}
\mathrm{e}^{-A(2r-\alpha_1-\alpha_2)}
\mathrm{e}^{-cS|\alpha_1-\alpha_2|}
\,\mathrm{d}\alpha_1\,\mathrm{d}\alpha_2\\
&=
A
\int_{[0,r]^2}
\mathrm{e}^{-A(u+v)}
\mathrm{e}^{-cS|u-v|}
\,\mathrm{d}u\,\mathrm{d}v\\
&\leq
A
\int_{(\mathbb{R}_+)^2}
\mathrm{e}^{-A(u+v)}
\mathrm{e}^{-cS|u-v|}
\,\mathrm{d}u\,\mathrm{d}v.
\end{align*}
By symmetry,
\begin{align*}
&A
\int_{(\mathbb{R}_+)^2}
\mathrm{e}^{-A(u+v)}
\mathrm{e}^{-cS|u-v|}
\,\mathrm{d}u\,\mathrm{d}v\\
&=
2A
\int_{\{0\leq v\leq u\}}
\mathrm{e}^{-A(u+v)}
\mathrm{e}^{-cS(u-v)}
\,\mathrm{d}u\,\mathrm{d}v.
\end{align*}
Setting
\[
z:=u-v,
\qquad
w:=v,
\]
gives
\begin{align*}
&2A
\int_0^\infty\int_0^\infty
\mathrm{e}^{-A(z+2w)}
\mathrm{e}^{-cSz}
\,\mathrm{d}w\,\mathrm{d}z\\
&=
2A
\left(
\int_0^\infty
\mathrm{e}^{-2Aw}\,\mathrm{d}w
\right)
\left(
\int_0^\infty
\mathrm{e}^{-(A+cS)z}\,\mathrm{d}z
\right)\\
&\lesssim
(A+S)^{-1}.
\end{align*}
If $A=0$, then the original expression vanishes because of the factor
$|M|^4$. Since the frequencies are integer-valued, if $A>0$ then $A$ is
bounded away from zero. Hence, in all cases,
\[
A
\int_{[0,r]^2}
\mathrm{e}^{-A(2r-\alpha_1-\alpha_2)}
\mathrm{e}^{-cS|\alpha_1-\alpha_2|}
\,\mathrm{d}\alpha_1\,\mathrm{d}\alpha_2
\lesssim
(1+A+S)^{-1}.
\]
Consequently,
\begin{align*}
&|M|^4
\int_{\left[0,r\right]^3}
\mathrm{e}^{-A(2r-\alpha_1-\alpha_2)-2b(r-\alpha_3)}
\\
&\quad\times
\left(
\int_{\left[0,\min(\alpha_1,\alpha_2)\right]^3}
\mathrm{e}^{-\sum_{k=1}^3
a_k(\alpha_1+\alpha_2-2v_k)}
\,\mathrm{d}v_1\,\mathrm{d}v_2\,\mathrm{d}v_3
\right)
\,\mathrm{d}\alpha_1\,\mathrm{d}\alpha_2\,\mathrm{d}\alpha_3\\
&\lesssim
(1+b)^{-1}
\prod_{k=1}^3(1+a_k)^{-1}
(1+A+S)^{-1}.
\end{align*}

It remains to distribute the last factor among the four Fourier variables.
On the support of $\psi(M,m_4)$,
\[
1+|M|\asymp1+|m_4|,
\]
and therefore
\[
1+A\asymp1+b.
\]
Moreover,
\[
1+a_k\lesssim1+A+S,
\qquad k=1,2,3,
\]
and, by the preceding comparability,
\[
1+b\lesssim1+A+S.
\]
Thus
\[
\left(
(1+a_1)(1+a_2)(1+a_3)(1+b)
\right)^{1/4}
\lesssim
1+A+S.
\]
Since $\gamma\in(0,1/6)$, we have $1-5\gamma>0$. Hence
\begin{align*}
(1+A+S)^{-1}
&\leq
(1+A+S)^{-(1-5\gamma)}\\
&\lesssim
\left(
(1+a_1)(1+a_2)(1+a_3)(1+b)
\right)^{-\frac{1-5\gamma}{4}}.
\end{align*}
Combining this estimate with the factors
$(1+a_1)^{-1}$, $(1+a_2)^{-1}$,
$(1+a_3)^{-1}$ and $(1+b)^{-1}$, we obtain
\begin{align*}
&(1+b)^{-1}
\prod_{k=1}^3(1+a_k)^{-1}
(1+A+S)^{-1}\\
&\lesssim
\prod_{k=1}^3
(1+a_k)^{-1-\frac{1-5\gamma}{4}}
(1+b)^{-1-\frac{1-5\gamma}{4}}\\
&=
\prod_{k=1}^3
(1+a_k)^{\frac{5\gamma}{4}-\frac54}
(1+b)^{\frac{5\gamma}{4}-\frac54}.
\end{align*}
Since
\[
1+|2\pi m_k|^4\asymp1+|m_k|^4,
\]
we conclude that
\begin{align*}
&\left|\psi(m_1+m_2+m_3,m_4)\right|^2
\left|m_1+m_2+m_3\right|^4\\
&\quad\times
\int_{\left[0,r\right]^3}
\mathrm{e}^{-\left|2\pi(m_1+m_2+m_3)\right|^4
(2r-\alpha_1-\alpha_2)
-2|2\pi m_4|^4(r-\alpha_3)}
\\
&\quad\times
\left(
\int_{\left[0,\min(\alpha_1,\alpha_2)\right]^3}
\mathrm{e}^{-\sum_{k=1}^3
|2\pi m_k|^4(\alpha_1+\alpha_2-2v_k)}
\,\mathrm{d}v_1\,\mathrm{d}v_2\,\mathrm{d}v_3
\right)
\,\mathrm{d}\alpha_1\,\mathrm{d}\alpha_2\,\mathrm{d}\alpha_3\\
&\lesssim
\prod_{k=1}^4
\left(1+|m_k|^4\right)^{\frac{5\gamma}{4}-\frac54}.
\end{align*}
Hence,
\begin{align*}
\mathds{E}\left[\left|I_{r,\epsilon,1}\right|^2\right]
&\lesssim
\sum_{\left(m_1,m_2,m_3,m_4\right)\in(\mathbb Z^d)^4}
\left|
\rho_{m-1}\left(m_1+m_2+m_3+m_4\right)
\right|^2\\
&\quad\times
\left(1+\left|m_1\right|^4\right)^{\frac{5\gamma}{4}-\frac54}
\left(1+\left|m_2\right|^4\right)^{\frac{5\gamma}{4}-\frac54}\\
&\quad\times
\left(1+\left|m_3\right|^4\right)^{\frac{5\gamma}{4}-\frac54}
\left(1+\left|m_4\right|^4\right)^{\frac{5\gamma}{4}-\frac54}.
\end{align*}

Let
\[
a:=5(1-\gamma)
\]
and define
\[
w(k)
:=
\left(1+|k|^4\right)^{-a/4}
=
\left(1+|k|^4\right)^{\frac{5\gamma}{4}-\frac54}.
\]
If
\[
p=m_1+m_2+m_3+m_4,
\]
then the preceding sum can be rewritten as
\[
\sum_{p\in\mathbb Z^d}
|\rho_{m-1}(p)|^2\,(w*w*w*w)(p).
\]

We now apply Lemma~\ref{lemma2} successively. Since $d=5$, the first
application gives
\[
(w*w)(p)
\lesssim
\left(1+|p|^4\right)^{-\frac{5-10\gamma}{4}}.
\]
Indeed, Lemma~\ref{lemma2}, with
\[
\alpha=\beta=5(1-\gamma),
\]
gives
\[
\frac{d-\alpha-\beta}{4}
=
\frac{5-10(1-\gamma)}{4}
=
-\frac{5-10\gamma}{4}.
\]

Convolving once more with $w$, Lemma~\ref{lemma2} yields
\[
(w*w*w)(p)
\lesssim
\left(1+|p|^4\right)^{-\frac{5-15\gamma}{4}}.
\]
Indeed, the two decay exponents in Lemma~\ref{lemma2} are now
\[
5-10\gamma
\qquad\text{and}\qquad
5-5\gamma,
\]
and hence
\[
\frac{5-(5-10\gamma)-(5-5\gamma)}{4}
=
-\frac{5-15\gamma}{4}.
\]

A final application gives
\[
(w*w*w*w)(p)
\lesssim
\left(1+|p|^4\right)^{-\frac{5-20\gamma}{4}}
=
\left(1+|p|^4\right)^{5\gamma-\frac54}.
\]
Indeed,
\[
\frac{5-(5-15\gamma)-(5-5\gamma)}{4}
=
-\frac{5-20\gamma}{4}
=
5\gamma-\frac54.
\]
The assumptions of Lemma~\ref{lemma2} are satisfied at every step since
$\gamma\in(0,1/6)$. Consequently,
\begin{align*}
\mathds{E}\left[\left|I_{r,\epsilon,1}\right|^2\right]
&\stackrel{\text{Lemma \ref{lemma2}}}{\lesssim}
\sum_{m_1\in\mathbb Z^d}
\left|\rho_{m-1}(m_1)\right|^2
\left(1+|m_1|^4\right)^{5\gamma-\frac54}.
\end{align*}

Finally, for $m\geq1$, on the support of $\rho_{m-1}$ we have
\[
|m_1|\asymp2^{m-1}.
\]
Since $d=5$, the number of lattice points in the corresponding dyadic annulus
is of order $2^{5(m-1)}$. For $m=0$, the support of $\rho_{-1}$ is bounded,
and the corresponding estimate follows directly. Therefore,
\begin{align*}
\sum_{m_1\in\mathbb Z^d}
\left|\rho_{m-1}(m_1)\right|^2
\left(1+|m_1|^4\right)^{5\gamma-\frac54}
&\lesssim
2^{5(m-1)}
2^{4\left(5\gamma-\frac54\right)(m-1)}\\
&=
2^{20\gamma(m-1)}.
\end{align*}
We conclude that
\[
\boxed{
\mathds{E}\left[\left|I_{r,\epsilon,1}\right|^2\right]
\lesssim
2^{20\gamma(m-1)}.
}
\]
After estimating the first term, we turn to $I_{r,\epsilon,2}$. 
Using Parseval and the frequency constraints coming from the spatial
integrations, we obtain
\begin{align*}
&\mathds{E}\left[\left|I_{r,\epsilon,2}\right|^2\right]\\
&\lesssim 
\sum_{\left(m_1,m_2,m_3,m_4\right) \in \left(\mathbb{Z}^d\right)^4}
\left|\rho_{m-1}\left(m_1\right)\right|^2
\psi\left(m_1-m_2,m_2\right)
\psi\left(m_1-m_3,m_3\right)\\
&\quad\times
\left|m_1-m_2\right|^2
\left|m_1-m_3\right|^2
\left|
\zeta\left(\epsilon m_2\right)
\zeta\left(\epsilon m_3\right)
\zeta\left(\epsilon m_4\right)
\zeta\left(\epsilon\left(m_1-m_4\right)\right)
\right|^2\\
&\quad\times
\int_{\left[0,r\right]^2}
\mathrm{e}^{-\left|2\pi \left(m_1-m_2\right)\right|^4
\left(r-\alpha_1\right)
-\left|2\pi \left(m_1-m_3\right)\right|^4
\left(r-\alpha_2\right)}\\
&\quad\times
\int_{\left[0,\alpha_1\right]\times
\left[0,\alpha_2\right]\times
\left[0,\min\left(\alpha_1,\alpha_2\right)\right]^2}
\mathrm{e}^{-\left|2\pi m_2\right|^4
\left(r+\alpha_1-2v_1\right)
-\left|2\pi m_3\right|^4
\left(r+\alpha_2-2v_2\right)}\\
&\qquad\times
\mathrm{e}^{-\left|2\pi m_4\right|^4
\left(\alpha_1+\alpha_2-2v_3\right)
-\left|2\pi \left(m_1-m_4\right)\right|^4
\left(\alpha_1+\alpha_2-2v_4\right)}
\,\mathrm{d}v_1\,\mathrm{d}v_2\,\mathrm{d}v_3\,\mathrm{d}v_4\,
\mathrm{d}\alpha_1\,\mathrm{d}\alpha_2.
\end{align*}

We now estimate the time integrals. Since $\zeta$ is smooth and compactly
supported, its Fourier multipliers are uniformly bounded. Moreover,
\[
0\leq \psi\leq 1,
\]
and the localization property of $\psi$ gives
\[
\psi(m_1-m_2,m_2)\neq0
\quad\Longrightarrow\quad
1+|m_1-m_2|\asymp 1+|m_2|,
\]
and similarly
\[
\psi(m_1-m_3,m_3)\neq0
\quad\Longrightarrow\quad
1+|m_1-m_3|\asymp 1+|m_3|.
\]

Set
\[
A:=\left|2\pi(m_1-m_2)\right|^4,
\qquad
B:=\left|2\pi(m_1-m_3)\right|^4,
\]
and
\[
a:=|2\pi m_2|^4,
\qquad
b:=|2\pi m_3|^4,
\qquad
c:=|2\pi m_4|^4,
\qquad
d:=\left|2\pi(m_1-m_4)\right|^4.
\]
On the support of the two factors involving $\psi$, we have
\[
1+A\asymp1+a,
\qquad
1+B\asymp1+b.
\]

We first estimate the integral with respect to $v_1$. Since
\[
r+\alpha_1-2v_1
=
r-\alpha_1+2(\alpha_1-v_1),
\]
we obtain
\begin{align*}
\int_0^{\alpha_1}
\mathrm{e}^{-a(r+\alpha_1-2v_1)}
\,\mathrm{d}v_1
&=
\mathrm{e}^{-a(r-\alpha_1)}
\int_0^{\alpha_1}
\mathrm{e}^{-2a(\alpha_1-v_1)}
\,\mathrm{d}v_1\\
&\lesssim
(1+a)^{-1}
\mathrm{e}^{-a(r-\alpha_1)}.
\end{align*}
Similarly,
\begin{align*}
\int_0^{\alpha_2}
\mathrm{e}^{-b(r+\alpha_2-2v_2)}
\,\mathrm{d}v_2
&\lesssim
(1+b)^{-1}
\mathrm{e}^{-b(r-\alpha_2)}.
\end{align*}

Let
\[
s:=\min(\alpha_1,\alpha_2).
\]
Since
\[
\alpha_1+\alpha_2-2s
=
|\alpha_1-\alpha_2|,
\]
we have
\begin{align*}
\int_0^s
\mathrm{e}^{-c(\alpha_1+\alpha_2-2v_3)}
\,\mathrm{d}v_3
&=
\mathrm{e}^{-c|\alpha_1-\alpha_2|}
\int_0^s
\mathrm{e}^{-2c(s-v_3)}
\,\mathrm{d}v_3\\
&\lesssim
(1+c)^{-1}
\mathrm{e}^{-c|\alpha_1-\alpha_2|}.
\end{align*}
Likewise,
\begin{align*}
\int_0^s
\mathrm{e}^{-d(\alpha_1+\alpha_2-2v_4)}
\,\mathrm{d}v_4
&\lesssim
(1+d)^{-1}
\mathrm{e}^{-d|\alpha_1-\alpha_2|}.
\end{align*}
Therefore,
\begin{align*}
&\int_{\left[0,\alpha_1\right]\times
\left[0,\alpha_2\right]\times
\left[0,\min\left(\alpha_1,\alpha_2\right)\right]^2}
\mathrm{e}^{-a(r+\alpha_1-2v_1)-b(r+\alpha_2-2v_2)}\\
&\qquad\times
\mathrm{e}^{-c(\alpha_1+\alpha_2-2v_3)
-d(\alpha_1+\alpha_2-2v_4)}
\,\mathrm{d}v_1\,\mathrm{d}v_2\,\mathrm{d}v_3\,\mathrm{d}v_4\\
&\lesssim
(1+a)^{-1}
(1+b)^{-1}
(1+c)^{-1}
(1+d)^{-1}\\
&\qquad\times
\mathrm{e}^{-a(r-\alpha_1)}
\mathrm{e}^{-b(r-\alpha_2)}
\mathrm{e}^{-(c+d)|\alpha_1-\alpha_2|}.
\end{align*}

Consequently, the complete time integral is bounded by
\begin{align*}
&|m_1-m_2|^2|m_1-m_3|^2
(1+a)^{-1}(1+b)^{-1}(1+c)^{-1}(1+d)^{-1}\\
&\quad\times
\int_{[0,r]^2}
\mathrm{e}^{-(A+a)(r-\alpha_1)}
\mathrm{e}^{-(B+b)(r-\alpha_2)}
\mathrm{e}^{-(c+d)|\alpha_1-\alpha_2|}
\,\mathrm{d}\alpha_1\,\mathrm{d}\alpha_2.
\end{align*}
Since
\[
|m_1-m_2|^2\asymp A^{1/2},
\qquad
|m_1-m_3|^2\asymp B^{1/2},
\]
we estimate the remaining double integral by dropping the last exponential:
\begin{align*}
&A^{1/2}B^{1/2}
\int_{[0,r]^2}
\mathrm{e}^{-(A+a)(r-\alpha_1)}
\mathrm{e}^{-(B+b)(r-\alpha_2)}
\mathrm{e}^{-(c+d)|\alpha_1-\alpha_2|}
\,\mathrm{d}\alpha_1\,\mathrm{d}\alpha_2\\
&\leq
A^{1/2}B^{1/2}
\int_0^r
\mathrm{e}^{-(A+a)(r-\alpha_1)}
\,\mathrm{d}\alpha_1
\int_0^r
\mathrm{e}^{-(B+b)(r-\alpha_2)}
\,\mathrm{d}\alpha_2\\
&\lesssim
\frac{A^{1/2}}{1+A+a}
\frac{B^{1/2}}{1+B+b}.
\end{align*}
On the support of the $\psi$ factors,
\[
1+A\asymp1+a,
\qquad
1+B\asymp1+b,
\]
and therefore
\[
\frac{A^{1/2}}{1+A+a}
\lesssim
(1+A)^{-1/2},
\qquad
\frac{B^{1/2}}{1+B+b}
\lesssim
(1+B)^{-1/2}.
\]
Hence
\begin{align*}
&\left|m_1-m_2\right|^2
\left|m_1-m_3\right|^2\\
&\quad\times
\int_{\left[0,r\right]^2}
\mathrm{e}^{-\left|2\pi \left(m_1-m_2\right)\right|^4
\left(r-\alpha_1\right)
-\left|2\pi \left(m_1-m_3\right)\right|^4
\left(r-\alpha_2\right)}\\
&\quad\times
\int_{\left[0,\alpha_1\right]\times
\left[0,\alpha_2\right]\times
\left[0,\min\left(\alpha_1,\alpha_2\right)\right]^2}
\mathrm{e}^{-\left|2\pi m_2\right|^4
\left(r+\alpha_1-2v_1\right)
-\left|2\pi m_3\right|^4
\left(r+\alpha_2-2v_2\right)}\\
&\qquad\times
\mathrm{e}^{-\left|2\pi m_4\right|^4
\left(\alpha_1+\alpha_2-2v_3\right)
-\left|2\pi \left(m_1-m_4\right)\right|^4
\left(\alpha_1+\alpha_2-2v_4\right)}
\,\mathrm{d}v_1\,\mathrm{d}v_2\,\mathrm{d}v_3\,\mathrm{d}v_4\,
\mathrm{d}\alpha_1\,\mathrm{d}\alpha_2\\
&\lesssim
(1+a)^{-1}
(1+b)^{-1}
(1+c)^{-1}
(1+d)^{-1}
(1+A)^{-1/2}
(1+B)^{-1/2}.
\end{align*}

Since $\gamma>0$, we may weaken these powers according to
\[
(1+a)^{-1}\leq(1+a)^{\gamma-1},
\qquad
(1+b)^{-1}\leq(1+b)^{\gamma-1},
\]
\[
(1+c)^{-1}\leq(1+c)^{\gamma-1},
\qquad
(1+d)^{-1}\leq(1+d)^{\gamma-1},
\]
and
\[
(1+A)^{-1/2}
\leq
(1+A)^{\frac{\gamma}{2}-\frac12},
\qquad
(1+B)^{-1/2}
\leq
(1+B)^{\frac{\gamma}{2}-\frac12}.
\]
Recalling the definitions of $A,B,a,b,c,d$, we obtain
\begin{align*}
&\left|m_1-m_2\right|^2
\left|m_1-m_3\right|^2\\
&\quad\times
\int_{\left[0,r\right]^2}
\mathrm{e}^{-\left|2\pi \left(m_1-m_2\right)\right|^4
\left(r-\alpha_1\right)
-\left|2\pi \left(m_1-m_3\right)\right|^4
\left(r-\alpha_2\right)}\\
&\quad\times
\int_{\left[0,\alpha_1\right]\times
\left[0,\alpha_2\right]\times
\left[0,\min\left(\alpha_1,\alpha_2\right)\right]^2}
\mathrm{e}^{-\left|2\pi m_2\right|^4
\left(r+\alpha_1-2v_1\right)
-\left|2\pi m_3\right|^4
\left(r+\alpha_2-2v_2\right)}\\
&\qquad\times
\mathrm{e}^{-\left|2\pi m_4\right|^4
\left(\alpha_1+\alpha_2-2v_3\right)
-\left|2\pi \left(m_1-m_4\right)\right|^4
\left(\alpha_1+\alpha_2-2v_4\right)}
\,\mathrm{d}v_1\,\mathrm{d}v_2\,\mathrm{d}v_3\,\mathrm{d}v_4\,
\mathrm{d}\alpha_1\,\mathrm{d}\alpha_2\\
&\lesssim
\left(1+\left|m_2\right|^4\right)^{\gamma-1}
\left(1+\left|m_3\right|^4\right)^{\gamma-1}
\left(1+\left|m_4\right|^4\right)^{\gamma-1}\\
&\qquad\times
\left(1+\left|m_1-m_2\right|^4\right)^{\frac{\gamma}{2}-\frac12}
\left(1+\left|m_1-m_3\right|^4\right)^{\frac{\gamma}{2}-\frac12}
\left(1+\left|m_1-m_4\right|^4\right)^{\gamma-1}.
\end{align*}
Consequently,
\begin{align*}
&\mathds{E}\left[\left|I_{r,\epsilon,2}\right|^2\right]\\
&\stackrel{\text{Lemma \ref{l33}}}{\lesssim}
\sum_{\left(m_1,m_2,m_3,m_4\right)\in(\mathbb Z^d)^4}
\left|\rho_{m-1}\left(m_1\right)\right|^2
\left(1+\left|m_2\right|^4\right)^{\gamma-1}
\left(1+\left|m_3\right|^4\right)^{\gamma-1}
\left(1+\left|m_4\right|^4\right)^{\gamma-1}\\
&\qquad\times
\left(1+\left|m_1-m_2\right|^4\right)^{\frac{\gamma}{2}-\frac12}
\left(1+\left|m_1-m_3\right|^4\right)^{\frac{\gamma}{2}-\frac12}
\left(1+\left|m_1-m_4\right|^4\right)^{\gamma-1}.
\end{align*}

We now make explicit the successive applications of
Lemma~\ref{lemma2}. For fixed $m_1$, the sum over $m_2,m_3,m_4$ factorizes as
\begin{align*}
&\left(
\sum_{m_2\in\mathbb Z^d}
\left(1+|m_2|^4\right)^{\gamma-1}
\left(1+|m_1-m_2|^4\right)^{\frac{\gamma}{2}-\frac12}
\right)\\
&\quad\times
\left(
\sum_{m_3\in\mathbb Z^d}
\left(1+|m_3|^4\right)^{\gamma-1}
\left(1+|m_1-m_3|^4\right)^{\frac{\gamma}{2}-\frac12}
\right)\\
&\quad\times
\left(
\sum_{m_4\in\mathbb Z^d}
\left(1+|m_4|^4\right)^{\gamma-1}
\left(1+|m_1-m_4|^4\right)^{\gamma-1}
\right).
\end{align*}

For the first two sums we write
\[
\left(1+|m|^4\right)^{\gamma-1}
=
\left(1+|m|^4\right)^{-\frac{4(1-\gamma)}4}
\]
and
\[
\left(1+|m|^4\right)^{\frac{\gamma}{2}-\frac12}
=
\left(1+|m|^4\right)^{-\frac{2(1-\gamma)}4}.
\]
Since $d=5$ and $\gamma\in(0,1/6)$,
\[
4(1-\gamma)+2(1-\gamma)
=
6(1-\gamma)>5.
\]
Hence Lemma~\ref{lemma2} gives
\begin{align*}
&\sum_{m_2\in\mathbb Z^d}
\left(1+|m_2|^4\right)^{\gamma-1}
\left(1+|m_1-m_2|^4\right)^{\frac{\gamma}{2}-\frac12}\\
&\lesssim
\left(1+|m_1|^4\right)^{
\frac{5-4(1-\gamma)-2(1-\gamma)}4}\\
&=
\left(1+|m_1|^4\right)^{\frac{6\gamma-1}{4}}.
\end{align*}
The same estimate holds for the sum over $m_3$.

For the sum over $m_4$, both weights have exponent $\gamma-1$, and
Lemma~\ref{lemma2} gives
\begin{align*}
&\sum_{m_4\in\mathbb Z^d}
\left(1+|m_4|^4\right)^{\gamma-1}
\left(1+|m_1-m_4|^4\right)^{\gamma-1}\\
&\lesssim
\left(1+|m_1|^4\right)^{
\frac{5-8(1-\gamma)}4}\\
&=
\left(1+|m_1|^4\right)^{2\gamma-\frac34}.
\end{align*}
Combining the three bounds, we obtain
\begin{align*}
&\left(1+|m_1|^4\right)^{\frac{6\gamma-1}{4}}
\left(1+|m_1|^4\right)^{\frac{6\gamma-1}{4}}
\left(1+|m_1|^4\right)^{2\gamma-\frac34}\\
&=
\left(1+|m_1|^4\right)^{
2\left(\frac{6\gamma-1}{4}\right)
+2\gamma-\frac34}\\
&=
\left(1+|m_1|^4\right)^{5\gamma-\frac54}.
\end{align*}
Therefore,
\begin{align*}
\mathds{E}\left[\left|I_{r,\epsilon,2}\right|^2\right]
&\stackrel{\text{Lemma \ref{lemma2}}}{\lesssim}
\sum_{m_1\in\mathbb Z^d}
\left|\rho_{m-1}\left(m_1\right)\right|^2
\left(1+\left|m_1\right|^4\right)^{5\gamma-\frac54}.
\end{align*}

Finally, for $m\geq1$, on the support of $\rho_{m-1}, |m_1|\asymp 2^{m-1},$
and since $d=5$, the number of lattice points in the corresponding dyadic
annulus is of order $2^{5(m-1)}$. For $m=0$, the support of $\rho_{-1}$ is bounded,
and the corresponding estimate follows directly. Hence
\begin{align*}
\sum_{m_1\in\mathbb Z^d}
\left|\rho_{m-1}\left(m_1\right)\right|^2
\left(1+\left|m_1\right|^4\right)^{5\gamma-\frac54}
&\lesssim
2^{5(m-1)}
2^{4\left(5\gamma-\frac54\right)(m-1)}\\
&=
2^{20\gamma(m-1)}.
\end{align*}
We conclude that
\[
\boxed{
\mathds{E}\left[\left|I_{r,\epsilon,2}\right|^2\right]
\lesssim
2^{20\gamma(m-1)}.
}
\]
\end{proof}
\begin{remark}
The derivation of these bounds involves several nested integrals and series. A graphical representation would make the computations shorter and clearer. The drawback is that one must first introduce and become familiar with the corresponding notation. See \cite{MR3746744} for further details on the use of diagrams in the construction of stochastic objects.
\end{remark}

\chapter{}
\label{ppe5}
\begin{proof}[Proof of Theorem \ref{theorem5}]
We only prove the crucial $L^2(\Omega)$ estimate at a fixed time $r$.

Let $\gamma \in \left]0,\frac{1}{11}\right[,$
\begin{align*}
\psi:\left(\mathbb{Z}^d\right)^2&\longrightarrow \left[0,1\right]\\
\left(m_1,m_2\right)&\longmapsto\psi\left(m_1,m_2\right)=\sum_{q \in \mathbb{N}}\sum_{j=0}^{\min\left(q+1,2\right)}\rho_{j+\max\left(q-2,-1\right)}\left(m_1\right)\rho_{q-1}\left(m_2\right)
\end{align*}
The following observation is crucial in our estimation $$\psi(m_1,m_2)\neq 0 \implies \kappa_1 (1+|m_2|)\leq 1+|m_1|\leq \kappa_2 (1+|m_2|)$$ for $0<\kappa_1<\kappa_2.$

It follows from Lemma \ref{l3} that: $$(\delta_{m-1}\Xi_{r,\epsilon,2})(x)=I_{r,\epsilon,1}+4I_{r,\epsilon,2},$$ where $$I_{r,\epsilon,1}=\int_{\mathbb{T}^d}dyK_{m-1}(x-y)\sum_{q \in \mathbb{N}}\sum_{j=0}^{\min\left(q+1,2\right)}V_{4}(f_{r,\epsilon,j,q,y}\otimes_0g_{r,\epsilon,q,y})=V_{4}(\int_{\mathbb{T}^d}dyK_{m-1}(x-y)\sum_{q \in \mathbb{N}}\sum_{j=0}^{\min\left(q+1,2\right)}f_{r,\epsilon,j,q,y}\otimes_0g_{r,\epsilon,q,y}),$$ $$I_{r,\epsilon,2}=\int_{\mathbb{T}^d}dyK_{m-1}(x-y)\sum_{q \in \mathbb{N}}\sum_{j=0}^{\min\left(q+1,2\right)}V_{2}(f_{r,\epsilon,j,q,y}\otimes_1g_{r,\epsilon,q,y})=V_{2}(\int_{\mathbb{T}^d}dyK_{m-1}(x-y)\sum_{q \in \mathbb{N}}\sum_{j=0}^{\min\left(q+1,2\right)}f_{r,\epsilon,j,q,y}\otimes_1g_{r,\epsilon,q,y}),$$ where $$f_{r,\epsilon,j,q,y}((u_1,x_1),(u_2,x_2))=\int_0^rdv(K_{j+\max(q-2,-1)}*\Delta\eta_{r-v}*((\zeta_\epsilon*\eta_{v-u_1})(\cdot-x_1)(\zeta_\epsilon*\eta_{v-u_2})(\cdot-x_2)))(y)$$$$\mathds{1}_{[0,v]^2}(u_1,u_2),$$ $$g_{r,\epsilon,q,y}((u_3,x_3),(u_4,x_4))=(K_{q-1}*((\zeta_\epsilon*\eta_{r-u_3})(\cdot-x_3)(\zeta_\epsilon*\eta_{r-u_4})(\cdot-x_4)))(y)\mathds{1}_{[0,r]^2}(u_3,u_4).$$  
Using Theorem \ref{nelson}, we obtain

\begin{align*}
&\mathds{E}[|I_{r,\epsilon,1}|^2]\\
&\lesssim\int_{(\mathbb{R}_+\times\mathbb{T}^d)^4}du_1dx_1du_2dx_2du_3dx_3du_4dx_4|\int_{\mathbb{T}^d}dyK_{m-1}(x-y)\\
&\sum_{q \in \mathbb{N}}\sum_{j=0}^{\min\left(q+1,2\right)}f_{r,\epsilon,j,q,y}((u_1,x_1),(u_2,x_2))g_{r,\epsilon,q,y}((u_3,x_3),(u_4,x_4))|^2\\
&\stackrel{\text{Fubini}}{=}\int_{(\mathbb{R}_+\times\mathbb{T}^d)^4}du_1dx_1du_2dx_2du_3dx_3du_4dx_4\int_{(\mathbb{T}^d)^2}dy_1dy_2K_{m-1}(x-y_1)K_{m-1}(x-y_2)\\
&\sum_{(q_1,q_2) \in \mathbb{N}^2}\sum_{j_1=0}^{\min\left(q_1+1,2\right)}\sum_{j_2=0}^{\min(q_2+1,2)}f_{r,\epsilon,j_1,q_1,y_1}((u_1,x_1),(u_2,x_2))g_{r,\epsilon,q_1,y_1}((u_3,x_3),(u_4,x_4))\\
&f_{r,\epsilon,j_2,q_2,y_2}((u_1,x_1),(u_2,x_2))g_{r,\epsilon,q_2,y_2}((u_3,x_3),(u_4,x_4))\\
&=\int_{[0,r]^2}dv_1dv_2\int_{(\mathbb{R}_+\times\mathbb{T}^d)^4}du_1dx_1du_2dx_2du_3dx_3du_4dx_4\int_{(\mathbb{T}^d)^6}dy_1dy_2dy_1'dy_2'dy'_3dy_4'K_{m-1}(x-y_1)K_{m-1}(x-y_2)\\
&\sum_{(q_1,q_2) \in \mathbb{N}^2}\sum_{j_1=0}^{\min\left(q_1+1,2\right)}\sum_{j_2=0}^{\min(q_2+1,2)}(K_{j_1+\max(q_1-2,-1)}*\Delta\eta_{r-v_1})(y_1-y'_1)(K_{j_2+\max(q_2-2,-1)}*\Delta\eta_{r-v_2})(y_2-y'_3)\\
&K_{q_1-1}(y_1-y'_2)K_{q_2-1}(y_2-y'_4)(\zeta_\epsilon*\eta_{v_1-u_1})(y'_1-x_1)(\zeta_\epsilon*\eta_{v_1-u_2})(y_1'-x_2)(\zeta_\epsilon*\eta_{r-u_3})(y'_2-x_3)\\
&(\zeta_\epsilon*\eta_{r-u_4})(y'_2-x_4)(\zeta_\epsilon*\eta_{v_2-u_1})(y'_3-x_1)
(\zeta_\epsilon*\eta_{v_2-u_2})(y_3'-x_2)(\zeta_\epsilon*\eta_{r-u_3})(y'_4-x_3)(\zeta_\epsilon*\eta_{r-u_4})(y_4'-x_4)\\
&\mathds{1}_{[0,r]^2}(u_3,u_4)\mathds{1}_{[0,\min(v_1,v_2)]^2}(u_1,u_2)\\
&=\int_{[0,r]^2}dv_1dv_2\int_{[0,\min(v_1,v_2)]^2\times[0,r]^2}du_1du_2du_3du_4\int_{(\mathbb{T}^d)^6}dy_1dy_2dy_1'dy_2'dy'_3dy_4'K_{m-1}(y_1)K_{m-1}(y_2)\\
&\sum_{(q_1,q_2) \in \mathbb{N}^2}\sum_{j_1=0}^{\min\left(q_1+1,2\right)}\sum_{j_2=0}^{\min(q_2+1,2)}(K_{j_1+\max(q_1-2,-1)}*\Delta\eta_{r-v_1})(y_1-y'_1)(K_{j_2+\max(q_2-2,-1)}*\Delta\eta_{r-v_2})(y_2-y'_3)\\
&K_{q_1-1}(y_1-y'_2)K_{q_2-1}(y_2-y'_4)(\zeta_{\epsilon}*\zeta_{\epsilon}*\eta_{v_1+v_2-2u_1})(y'_1-y_3')(\zeta_{\epsilon}*\zeta_{\epsilon}*\eta_{v_1+v_2-2u_2})(y_1'-y'_3)(\zeta_{\epsilon}*\zeta_{\epsilon}*\eta_{2(r-u_3)})(y'_2-y_4')\\
&(\zeta_{\epsilon}*\zeta_{\epsilon}*\eta_{2(r-u_4)})(y'_2-y_4')\\
&=\int_{[0,r]^2}dv_1dv_2\int_{[0,\min(v_1,v_2)]^2\times[0,r]^2}du_1du_2du_3du_4\int_{(\mathbb{T}^d)^2}dy_1dy_2K_{m-1}(y_1)K_{m-1}(y_2)\sum_{(q_1,q_2) \in \mathbb{N}^2}\sum_{j_1=0}^{\min\left(q_1+1,2\right)}\sum_{j_2=0}^{\min(q_2+1,2)}\\
&(K_{j_1+\max(q_1-2,-1)}*\Delta^2\eta_{2r-v_1-v_2}*K_{j_2+\max(q_2-2,-1)}*((\zeta_{\epsilon} * \zeta_{\epsilon} * \eta_{v_1+v_2-2u_1})(\zeta_{\epsilon} * \zeta_{\epsilon} * \eta_{v_1+v_2-2u_2})))(y_1-y_2)\\
&\times (K_{q_1-1}*K_{q_2-1}*((\zeta_{\epsilon}*\zeta_{\epsilon}*\eta_{2(r-u_3)})(\zeta_{\epsilon}*\zeta_{\epsilon}*\eta_{2(r-u_4)})))(y_1-y_2)
\end{align*} and \begin{align*}
&\mathds{E}[|I_{r,\epsilon,2}|^2]\\
&\lesssim\int_{(\mathbb{R}_+\times\mathbb{T}^d)^2}du_1dx_1du_2dx_2|\int_{\mathbb{T}^d}dyK_{m-1}(x-y)\\
&\sum_{q \in \mathbb{N}}\sum_{j=0}^{\min\left(q+1,2\right)}\int_{\mathbb{R}_+\times\mathbb{T}^d}du_3dx_3f_{r,\epsilon,j,q,y}((u_1,x_1),(u_3,x_3))g_{r,\epsilon,q,y}((u_3,x_3),(u_2,x_2))|^2\\
&\stackrel{\text{Fubini}}{=}\int_{(\mathbb{R}_+\times\mathbb{T}^d)^4}du_1dx_1du_2dx_2du_3dx_3du_4dx_4\int_{(\mathbb{T}^d)^2}dy_1dy_2K_{m-1}(x-y_1)K_{m-1}(x-y_2)\\
&\sum_{(q_1,q_2) \in \mathbb{N}^2}\sum_{j_1=0}^{\min\left(q_1+1,2\right)}\sum_{j_2=0}^{\min(q_2+1,2)}f_{r,\epsilon,j_1,q_1,y_1}((u_1,x_1),(u_3,x_3))g_{r,\epsilon,q_1,y_1}((u_3,x_3),(u_2,x_2))\\
&f_{r,\epsilon,j_2,q_2,y_2}((u_1,x_1),(u_4,x_4))g_{r,\epsilon,q_2,y_2}((u_4,x_4),(u_2,x_2))\\
&=\int_{[0,r]^2}dv_1dv_2\int_{(\mathbb{R}_+\times\mathbb{T}^d)^4}du_1dx_1du_2dx_2du_3dx_3du_4dx_4\int_{(\mathbb{T}^d)^6}dy_1dy_2dy_1'dy_2'dy'_3dy_4'K_{m-1}(x-y_1)K_{m-1}(x-y_2)\\
&\sum_{(q_1,q_2) \in \mathbb{N}^2}\sum_{j_1=0}^{\min\left(q_1+1,2\right)}\sum_{j_2=0}^{\min(q_2+1,2)}(K_{j_1+\max(q_1-2,-1)}*\Delta\eta_{r-v_1})(y_1-y'_1)(K_{j_2+\max(q_2-2,-1)}*\Delta\eta_{r-v_2})(y_2-y'_3)\\
&K_{q_1-1}(y_1-y'_2)K_{q_2-1}(y_2-y'_4)(\zeta_\epsilon*\eta_{v_1-u_1})(y'_1-x_1)(\zeta_\epsilon*\eta_{v_1-u_3})(y_1'-x_3)(\zeta_\epsilon*\eta_{r-u_3})(y'_2-x_3)\\
&(\zeta_\epsilon*\eta_{r-u_2})(y'_2-x_2)(\zeta_\epsilon*\eta_{v_2-u_1})(y'_3-x_1)
(\zeta_\epsilon*\eta_{v_2-u_4})(y_3'-x_4)(\zeta_\epsilon*\eta_{r-u_4})(y'_4-x_4)(\zeta_\epsilon*\eta_{r-u_2})(y_4'-x_2)\\
&\mathds{1}_{[0,\min(v_1,v_2)]\times[0,r]\times[0,v_1]\times[0,v_2]}(u_1,u_2,u_3,u_4)\\
&=\int_{[0,r]^2}dv_1dv_2\int_{[0,\min(v_1,v_2)]\times[0,r]\times[0,v_1]\times[0,v_2]}du_1du_2du_3du_4\int_{(\mathbb{T}^d)^6}dy_1dy_2dy_1'dy_2'dy'_3dy_4'K_{m-1}(y_1)K_{m-1}(y_2)\\
&\sum_{(q_1,q_2) \in \mathbb{N}^2}\sum_{j_1=0}^{\min\left(q_1+1,2\right)}\sum_{j_2=0}^{\min(q_2+1,2)}(K_{j_1+\max(q_1-2,-1)}*\Delta\eta_{r-v_1})(y_1-y'_1)(K_{j_2+\max(q_2-2,-1)}*\Delta\eta_{r-v_2})(y_2-y'_3)\\
&K_{q_1-1}(y_1-y'_2)K_{q_2-1}(y_2-y'_4)(\zeta_{\epsilon}*\zeta_{\epsilon}*\eta_{v_1+v_2-2u_1})(y'_1-y_3')(\zeta_{\epsilon}*\zeta_{\epsilon}*\eta_{2(r-u_2)})(y_2'-y'_4)(\zeta_{\epsilon}*\zeta_{\epsilon}*\eta_{r+v_1-2u_3})(y'_1-y_2')\\
&(\zeta_{\epsilon}*\zeta_{\epsilon}*\eta_{r+v_2-2u_4})(y'_3-y_4')\\
&=\int_{[0,r]^2}dv_1dv_2\int_{[0,\min(v_1,v_2)]\times[0,r]\times[0,v_1]\times[0,v_2]}du_1du_2du_3du_4\int_{(\mathbb{T}^d)^6}dy_1dy_2dy_1'dy_2'dy'_3dy_4'K_{m-1}(y_1)K_{m-1}(y_2)\\
&\sum_{(q_1,q_2) \in \mathbb{N}^2}\sum_{j_1=0}^{\min\left(q_1+1,2\right)}\sum_{j_2=0}^{\min(q_2+1,2)}(K_{j_1+\max(q_1-2,-1)}*\Delta\eta_{r-v_1})(y_1-y_2-y'_1)(K_{j_2+\max(q_2-2,-1)}*\Delta\eta_{r-v_2})(y'_3)\\
&K_{q_1-1}(y_1-y_2-y'_2)K_{q_2-1}(y'_4)(\zeta_{\epsilon}*\zeta_{\epsilon}*\eta_{v_1+v_2-2u_1})(y'_1-y_3')(\zeta_{\epsilon}*\zeta_{\epsilon}*\eta_{2(r-u_2)})(y_2'-y'_4)(\zeta_{\epsilon}*\zeta_{\epsilon}*\eta_{r+v_1-2u_3})(y'_1-y_2')\\
&(\zeta_{\epsilon}*\zeta_{\epsilon}*\eta_{r+v_2-2u_4})(y'_3-y_4')
\end{align*}
We are in a position to apply Parseval to obtain the following bound for
$\mathds{E}\left[\left|I_{r,\epsilon,1}\right|^2\right]$:
\begin{align*}
&\mathds{E}\left[\left|I_{r,\epsilon,1}\right|^2\right]\\
&\lesssim
\sum_{\left(m_1,m_2,m_3,m_4\right)\in\left(\mathbb Z^d\right)^4}
\left|
\rho_{m-1}\left(m_1+m_2+m_3+m_4\right)
\psi\left(m_1+m_2,m_3+m_4\right)
\prod_{k=1}^4\zeta(\epsilon m_k)
\right|^2
\left|m_1+m_2\right|^4\\
&\quad\times
\int_{\left[0,r\right]^4}
\mathrm{e}^{-\left|2\pi(m_1+m_2)\right|^4
\left(2r-\alpha_1-\alpha_2\right)
-2\left|2\pi m_3\right|^4\left(r-\alpha_3\right)
-2\left|2\pi m_4\right|^4\left(r-\alpha_4\right)}
\\
&\quad\times
\left(
\int_{\left[0,\min(\alpha_1,\alpha_2)\right]^2}
\mathrm{e}^{-\left|2\pi m_1\right|^4
\left(\alpha_1+\alpha_2-2v_1\right)
-\left|2\pi m_2\right|^4
\left(\alpha_1+\alpha_2-2v_2\right)}
\,\mathrm{d}v_1\,\mathrm{d}v_2
\right)
\,\mathrm{d}\alpha_1\,\mathrm{d}\alpha_2\,
\mathrm{d}\alpha_3\,\mathrm{d}\alpha_4.
\end{align*}

We now estimate the time integrals in detail. Since $\zeta$ is bounded, the
factors $\zeta(\epsilon m_k)$ can be absorbed into the implicit constant.
Moreover,
\[
0\leq\psi\leq1,
\]
and
\[
\psi(m_1+m_2,m_3+m_4)\neq0
\quad\Longrightarrow\quad
1+|m_1+m_2|\asymp1+|m_3+m_4|.
\]

Set
\[
n:=m_1+m_2,
\qquad
q:=m_3+m_4,
\]
and
\[
A:=|2\pi n|^4,
\qquad
a:=|2\pi m_1|^4,
\qquad
b:=|2\pi m_2|^4,
\qquad
c:=|2\pi m_3|^4,
\qquad
d:=|2\pi m_4|^4.
\]

We first integrate with respect to $\alpha_3$ and $\alpha_4$. We have
\[
\int_0^r
\mathrm{e}^{-2c(r-\alpha_3)}
\,\mathrm{d}\alpha_3
\lesssim
(1+c)^{-1},
\]
and similarly
\[
\int_0^r
\mathrm{e}^{-2d(r-\alpha_4)}
\,\mathrm{d}\alpha_4
\lesssim
(1+d)^{-1}.
\]

We next consider the integrals with respect to $v_1$ and $v_2$. Let
\[
s:=\min(\alpha_1,\alpha_2).
\]
Since
\[
\alpha_1+\alpha_2-2s
=
|\alpha_1-\alpha_2|,
\]
we have
\begin{align*}
\int_0^s
\mathrm{e}^{-a(\alpha_1+\alpha_2-2v_1)}
\,\mathrm{d}v_1
&=
\mathrm{e}^{-a|\alpha_1-\alpha_2|}
\int_0^s
\mathrm{e}^{-2a(s-v_1)}
\,\mathrm{d}v_1\\
&\lesssim
(1+a)^{-1}
\mathrm{e}^{-ca|\alpha_1-\alpha_2|}
\end{align*}
for some constant $c>0$, and similarly
\begin{align*}
\int_0^s
\mathrm{e}^{-b(\alpha_1+\alpha_2-2v_2)}
\,\mathrm{d}v_2
&\lesssim
(1+b)^{-1}
\mathrm{e}^{-cb|\alpha_1-\alpha_2|}.
\end{align*}
Consequently,
\begin{align*}
&\int_{\left[0,\min(\alpha_1,\alpha_2)\right]^2}
\mathrm{e}^{-a(\alpha_1+\alpha_2-2v_1)
-b(\alpha_1+\alpha_2-2v_2)}
\,\mathrm{d}v_1\,\mathrm{d}v_2\\
&\lesssim
(1+a)^{-1}(1+b)^{-1}
\mathrm{e}^{-c(a+b)|\alpha_1-\alpha_2|}.
\end{align*}

Therefore the complete time factor is bounded by
\begin{align*}
&(1+a)^{-1}(1+b)^{-1}(1+c)^{-1}(1+d)^{-1}\\
&\quad\times
|n|^4
\int_{[0,r]^2}
\mathrm{e}^{-A(2r-\alpha_1-\alpha_2)}
\mathrm{e}^{-c(a+b)|\alpha_1-\alpha_2|}
\,\mathrm{d}\alpha_1\,\mathrm{d}\alpha_2.
\end{align*}

Making the change of variables
\[
u:=r-\alpha_1,
\qquad
v:=r-\alpha_2,
\]
and using $|n|^4\asymp A$, we obtain
\begin{align*}
&A
\int_{[0,r]^2}
\mathrm{e}^{-A(2r-\alpha_1-\alpha_2)}
\mathrm{e}^{-c(a+b)|\alpha_1-\alpha_2|}
\,\mathrm{d}\alpha_1\,\mathrm{d}\alpha_2\\
&=
A
\int_{[0,r]^2}
\mathrm{e}^{-A(u+v)}
\mathrm{e}^{-c(a+b)|u-v|}
\,\mathrm{d}u\,\mathrm{d}v\\
&\leq
A
\int_{\mathbb R_+^2}
\mathrm{e}^{-A(u+v)}
\mathrm{e}^{-c(a+b)|u-v|}
\,\mathrm{d}u\,\mathrm{d}v.
\end{align*}
By symmetry,
\begin{align*}
&A
\int_{\mathbb R_+^2}
\mathrm{e}^{-A(u+v)}
\mathrm{e}^{-c(a+b)|u-v|}
\,\mathrm{d}u\,\mathrm{d}v\\
&=
2A
\int_{\{0\leq v\leq u\}}
\mathrm{e}^{-A(u+v)}
\mathrm{e}^{-c(a+b)(u-v)}
\,\mathrm{d}u\,\mathrm{d}v.
\end{align*}
Setting
\[
z:=u-v,
\qquad
w:=v,
\]
gives
\begin{align*}
&2A
\int_0^\infty\int_0^\infty
\mathrm{e}^{-A(z+2w)}
\mathrm{e}^{-c(a+b)z}
\,\mathrm{d}w\,\mathrm{d}z\\
&=
2A
\left(
\int_0^\infty
\mathrm{e}^{-2Aw}\,\mathrm{d}w
\right)
\left(
\int_0^\infty
\mathrm{e}^{-(A+c(a+b))z}\,\mathrm{d}z
\right)\\
&\lesssim
(1+A+a+b)^{-1}.
\end{align*}
If $A=0$, the original expression vanishes because of the factor $|n|^4$.

Hence
\begin{align*}
&|n|^4
\int_{\left[0,r\right]^4}
\mathrm{e}^{-A(2r-\alpha_1-\alpha_2)
-2c(r-\alpha_3)-2d(r-\alpha_4)}\\
&\quad\times
\left(
\int_{\left[0,\min(\alpha_1,\alpha_2)\right]^2}
\mathrm{e}^{-a(\alpha_1+\alpha_2-2v_1)
-b(\alpha_1+\alpha_2-2v_2)}
\,\mathrm{d}v_1\,\mathrm{d}v_2
\right)
\,\mathrm{d}\alpha_1\,\mathrm{d}\alpha_2\,
\mathrm{d}\alpha_3\,\mathrm{d}\alpha_4\\
&\lesssim
(1+a)^{-1}(1+b)^{-1}
(1+c)^{-1}(1+d)^{-1}
(1+A+a+b)^{-1}.
\end{align*}

Consequently,
\begin{align*}
&\mathds{E}\left[\left|I_{r,\epsilon,1}\right|^2\right]\\
&\lesssim
\sum_{\left(m_1,m_2,m_3,m_4\right)\in(\mathbb Z^d)^4}
\left|\rho_{m-1}\left(m_1+m_2+m_3+m_4\right)\right|^2
\left|\psi\left(m_1+m_2,m_3+m_4\right)\right|^2\\
&\quad\times
\left(1+|m_1|^4\right)^{-1}
\left(1+|m_2|^4\right)^{-1}
\left(1+|m_3|^4\right)^{-1}
\left(1+|m_4|^4\right)^{-1}\\
&\quad\times
\left(
1+|m_1+m_2|^4+|m_1|^4+|m_2|^4
\right)^{-1}.
\end{align*}

We now sum the internal frequencies. For fixed
\[
n=m_1+m_2,
\]
we first use
\[
\left(
1+|n|^4+|m_1|^4+|m_2|^4
\right)^{-1}
\leq
\left(1+|n|^4\right)^{-1}.
\]
Hence, by Lemma~\ref{lemma2}, with $d=5$ and
\[
\alpha=\beta=4,
\]
we obtain
\begin{align*}
&\sum_{m_1\in\mathbb Z^d}
\left(1+|m_1|^4\right)^{-1}
\left(1+|n-m_1|^4\right)^{-1}
\left(1+|n|^4\right)^{-1}\\
&\lesssim
\left(1+|n|^4\right)^{-\frac34}
\left(1+|n|^4\right)^{-1}\\
&=
\left(1+|n|^4\right)^{-\frac74}.
\end{align*}

Similarly, for fixed
\[
q=m_3+m_4,
\]
Lemma~\ref{lemma2}, again with $\alpha=\beta=4$, gives
\begin{align*}
&\sum_{m_3\in\mathbb Z^d}
\left(1+|m_3|^4\right)^{-1}
\left(1+|q-m_3|^4\right)^{-1}\\
&\lesssim
\left(1+|q|^4\right)^{-\frac34}.
\end{align*}

It follows that
\begin{align*}
\mathds{E}\left[\left|I_{r,\epsilon,1}\right|^2\right]
&\lesssim
\sum_{n,q\in\mathbb Z^d}
\left|\rho_{m-1}(n+q)\right|^2
|\psi(n,q)|^2\\
&\qquad\times
\left(1+|n|^4\right)^{-\frac74}
\left(1+|q|^4\right)^{-\frac34}.
\end{align*}

Now, on the support of $\psi$,
\[
1+|n|\asymp1+|q|,
\]
and therefore
\[
1+|n|^4\asymp1+|q|^4.
\]
Consequently, for every $\gamma>0$,
\begin{align*}
&\left(1+|n|^4\right)^{-\frac74}
\left(1+|q|^4\right)^{-\frac34}\\
&\lesssim
\left(1+|n|^4\right)^{\frac{5\gamma}{2}-\frac54}
\left(1+|q|^4\right)^{\frac{5\gamma}{2}-\frac54}.
\end{align*}
Indeed, because $1+|n|^4\asymp1+|q|^4$, the left-hand side has total
exponent
\[
-\frac74-\frac34=-\frac52,
\]
whereas the right-hand side has total exponent
\[
2\left(\frac{5\gamma}{2}-\frac54\right)
=
5\gamma-\frac52,
\]
which is larger since $\gamma>0$.

Hence
\begin{align*}
\mathds{E}\left[\left|I_{r,\epsilon,1}\right|^2\right]
&\lesssim
\sum_{n,q\in\mathbb Z^d}
\left|\rho_{m-1}(n+q)\right|^2
\left(1+|n|^4\right)^{\frac{5\gamma}{2}-\frac54}
\left(1+|q|^4\right)^{\frac{5\gamma}{2}-\frac54}.
\end{align*}

Applying Lemma~\ref{lemma2} once more and setting
\[
p=n+q,
\]
we obtain
\begin{align*}
&\sum_{n\in\mathbb Z^d}
\left(1+|n|^4\right)^{\frac{5\gamma}{2}-\frac54}
\left(1+|p-n|^4\right)^{\frac{5\gamma}{2}-\frac54}\\
&\lesssim
\left(1+|p|^4\right)^{5\gamma-\frac54}.
\end{align*}
Indeed,
\[
\frac{5\gamma}{2}-\frac54
=
-\frac{5-10\gamma}{4},
\]
so that in Lemma~\ref{lemma2} we take
\[
\alpha=\beta=5-10\gamma.
\]
Since $\gamma\in(0,1/11)$,
\[
0<5-10\gamma<5
\]
and
\[
2(5-10\gamma)>5.
\]
Thus all the assumptions of Lemma~\ref{lemma2} are satisfied, and
\[
\frac{5-2(5-10\gamma)}4
=
5\gamma-\frac54.
\]
Consequently,
\begin{align*}
\mathds{E}\left[\left|I_{r,\epsilon,1}\right|^2\right]
&\stackrel{\text{Lemma \ref{lemma2}}}{\lesssim}
\sum_{p\in\mathbb Z^d}
\left|\rho_{m-1}(p)\right|^2
\left(1+|p|^4\right)^{5\gamma-\frac54}.
\end{align*}

Finally, for $m\geq1$, on the support of $\rho_{m-1},|p|\asymp2^{m-1}.$
Since $d=5$, the number of lattice points in the corresponding dyadic annulus
is of order $2^{5(m-1)}$. For $m=0$, the support of $\rho_{-1}$ is bounded,
and the corresponding estimate follows directly. Therefore,
\begin{align*}
\sum_{p\in\mathbb Z^d}
\left|\rho_{m-1}(p)\right|^2
\left(1+|p|^4\right)^{5\gamma-\frac54}
&\lesssim
2^{5(m-1)}
2^{4\left(5\gamma-\frac54\right)(m-1)}\\
&=
2^{20\gamma(m-1)}.
\end{align*}
We conclude that
\[
\boxed{
\mathds{E}\left[\left|I_{r,\epsilon,1}\right|^2\right]
\lesssim
2^{20\gamma(m-1)}.
}
\]
To bound $\mathds{E}\left[\left|I_{r,\epsilon,2}\right|^2\right]$, we cannot simply use Parseval. Instead, we expand each of the ten factors appearing in the integrand into its Fourier series. For each factor, we introduce a Fourier variable and use the fact that all the kernels involved are even. After imposing the frequency constraints using the orthogonality of the Fourier basis on $\mathbb T^d$ and integrating with respect to
\[
y_1,\ y_2,\ y_1',\ y_2',\ y_3',\ y_4',
\]
we may choose four independent frequencies, denoted by
$m_1,m_2,m_3,m_4\in\mathbb{Z}^d$. With this choice, the frequencies associated with the different factors are
\[
\begin{array}{c|c}
\text{factor} & \text{Fourier frequency}\\
\hline
K_{m-1}(y_1)
& m_1\\[1mm]
K_{m-1}(y_2)
& -m_1\\[1mm]
(K_{j_1+\max(q_1-2,-1)}*\Delta\eta_{r-v_1})(y_1-y_2-y_1')
& -(m_1+m_2)\\[1mm]
(K_{j_2+\max(q_2-2,-1)}*\Delta\eta_{r-v_2})(y_3')
& -(m_1+m_4)\\[1mm]
K_{q_1-1}(y_1-y_2-y_2')
& m_2\\[1mm]
K_{q_2-1}(y_4')
& m_4\\[1mm]
(\zeta_\epsilon*\zeta_\epsilon*
\eta_{v_1+v_2-2u_1})(y_1'-y_3')
& -(m_1+m_3)\\[1mm]
(\zeta_\epsilon*\zeta_\epsilon*
\eta_{2(r-u_2)})(y_2'-y_4')
& m_3\\[1mm]
(\zeta_\epsilon*\zeta_\epsilon*
\eta_{r+v_1-2u_3})(y_1'-y_2')
& m_3-m_2\\[1mm]
(\zeta_\epsilon*\zeta_\epsilon*
\eta_{r+v_2-2u_4})(y_3'-y_4')
& m_4-m_3.
\end{array}
\]
This yields the following bound:
\begin{align*}
&\mathds{E}\left[\left|I_{r,\epsilon,2}\right|^2\right]\\
&\lesssim
\sum_{\left(m_1,m_2,m_3,m_4\right)\in\left(\mathbb Z^d\right)^4}
\left|\rho_{m-1}\left(m_1\right)\right|^2
\psi\left(m_1+m_2,m_2\right)
\psi\left(m_1+m_4,m_4\right)
\left|m_1+m_2\right|^2
\left|m_1+m_4\right|^2\\
&\quad\times
\left|
\zeta\left(\epsilon m_3\right)
\zeta\left(\epsilon\left(m_1+m_3\right)\right)
\zeta\left(\epsilon\left(m_2-m_3\right)\right)
\zeta\left(\epsilon\left(m_3-m_4\right)\right)
\right|^2\\
&\quad\times
\int_{\left[0,r\right]^3}
\mathrm{e}^{-\left|2\pi(m_1+m_4)\right|^4(r-\alpha_1)
-\left|2\pi(m_1+m_2)\right|^4(r-\alpha_2)
-2\left|2\pi m_3\right|^4(r-\alpha_3)}
\\
&\quad\times
\left(
\int_{\left[0,\alpha_1\right]\times
\left[0,\alpha_2\right]\times
\left[0,\min(\alpha_1,\alpha_2)\right]}
\mathrm{e}^{-\left|2\pi(m_3-m_4)\right|^4(r+\alpha_1-2v_1)
-\left|2\pi(m_2-m_3)\right|^4(r+\alpha_2-2v_2)}
\right.\\
&\hspace{7cm}\left.
\times
\mathrm{e}^{-\left|2\pi(m_1+m_3)\right|^4
(\alpha_1+\alpha_2-2v_3)}
\,\mathrm{d}v_1\,\mathrm{d}v_2\,\mathrm{d}v_3
\right)
\,\mathrm{d}\alpha_1\,\mathrm{d}\alpha_2\,\mathrm{d}\alpha_3.
\end{align*}

We now estimate the time integrals in detail. Since $\zeta$ is bounded, all
the factors involving $\zeta$ can be absorbed into the implicit constant.
Moreover, by the localization property of $\psi$,
\[
\psi(m_1+m_2,m_2)\neq0
\quad\Longrightarrow\quad
1+|m_1+m_2|\asymp1+|m_2|,
\]
and
\[
\psi(m_1+m_4,m_4)\neq0
\quad\Longrightarrow\quad
1+|m_1+m_4|\asymp1+|m_4|.
\]

Set
\[
A:=|2\pi(m_1+m_4)|^4,
\qquad
B:=|2\pi(m_1+m_2)|^4,
\qquad
C:=|2\pi m_3|^4,
\]
and
\[
D:=|2\pi(m_3-m_4)|^4,
\qquad
E:=|2\pi(m_2-m_3)|^4,
\qquad
F:=|2\pi(m_1+m_3)|^4.
\]
By the preceding localization properties,
\[
1+A\asymp1+|m_4|^4,
\qquad
1+B\asymp1+|m_2|^4.
\]

We first integrate with respect to $\alpha_3$. We have
\begin{align*}
\int_0^r
\mathrm{e}^{-2C(r-\alpha_3)}
\,\mathrm{d}\alpha_3
\lesssim
(1+C)^{-1}.
\end{align*}

We next integrate with respect to $v_1$. Since
\[
r+\alpha_1-2v_1
=
(r-\alpha_1)+2(\alpha_1-v_1),
\]
we obtain
\begin{align*}
\int_0^{\alpha_1}
\mathrm{e}^{-D(r+\alpha_1-2v_1)}
\,\mathrm{d}v_1
&=
\mathrm{e}^{-D(r-\alpha_1)}
\int_0^{\alpha_1}
\mathrm{e}^{-2D(\alpha_1-v_1)}
\,\mathrm{d}v_1\\
&\lesssim
(1+D)^{-1}
\mathrm{e}^{-D(r-\alpha_1)}.
\end{align*}
Similarly,
\begin{align*}
\int_0^{\alpha_2}
\mathrm{e}^{-E(r+\alpha_2-2v_2)}
\,\mathrm{d}v_2
&\lesssim
(1+E)^{-1}
\mathrm{e}^{-E(r-\alpha_2)}.
\end{align*}

Let
\[
s:=\min(\alpha_1,\alpha_2).
\]
Since
\[
\alpha_1+\alpha_2-2s
=
|\alpha_1-\alpha_2|,
\]
we have
\begin{align*}
\int_0^s
\mathrm{e}^{-F(\alpha_1+\alpha_2-2v_3)}
\,\mathrm{d}v_3
&=
\mathrm{e}^{-F|\alpha_1-\alpha_2|}
\int_0^s
\mathrm{e}^{-2F(s-v_3)}
\,\mathrm{d}v_3\\
&\lesssim
(1+F)^{-1}
\mathrm{e}^{-F|\alpha_1-\alpha_2|}.
\end{align*}

Consequently, after integrating with respect to
$v_1,v_2,v_3,\alpha_3$, the remaining time factor is bounded by
\begin{align*}
&(1+C)^{-1}(1+D)^{-1}(1+E)^{-1}(1+F)^{-1}\\
&\quad\times
|m_1+m_2|^2|m_1+m_4|^2
\int_{[0,r]^2}
\mathrm{e}^{-(A+D)(r-\alpha_1)}
\mathrm{e}^{-(B+E)(r-\alpha_2)}
\mathrm{e}^{-F|\alpha_1-\alpha_2|}
\,\mathrm{d}\alpha_1\,\mathrm{d}\alpha_2.
\end{align*}

Set
\[
R:=
|m_1+m_2|^2|m_1+m_4|^2
\int_{[0,r]^2}
\mathrm{e}^{-(A+D)(r-\alpha_1)}
\mathrm{e}^{-(B+E)(r-\alpha_2)}
\mathrm{e}^{-F|\alpha_1-\alpha_2|}
\,\mathrm{d}\alpha_1\,\mathrm{d}\alpha_2.
\]
Since
\[
|m_1+m_2|^2\asymp B^{1/2},
\qquad
|m_1+m_4|^2\asymp A^{1/2},
\]
we first obtain, by discarding the last exponential,
\begin{align*}
R
&\lesssim
A^{1/2}B^{1/2}
\int_0^r
\mathrm{e}^{-(A+D)(r-\alpha_1)}
\,\mathrm{d}\alpha_1
\int_0^r
\mathrm{e}^{-(B+E)(r-\alpha_2)}
\,\mathrm{d}\alpha_2\\
&\lesssim
(1+A)^{-1/2}(1+B)^{-1/2}.
\end{align*}

We also derive a second estimate which keeps the decay coming from
$F$. Making the change of variables
\[
u:=r-\alpha_1,
\qquad
v:=r-\alpha_2,
\]
and setting
\[
X:=A+D,
\qquad
Y:=B+E,
\]
we have
\begin{align*}
R
&\lesssim
A^{1/2}B^{1/2}
\int_{\mathbb R_+^2}
\mathrm{e}^{-Xu-Yv-F|u-v|}
\,\mathrm{d}u\,\mathrm{d}v.
\end{align*}
Splitting the domain into $\{u\geq v\}$ and $\{v\geq u\}$ gives
\begin{align*}
&\int_{\mathbb R_+^2}
\mathrm{e}^{-Xu-Yv-F|u-v|}
\,\mathrm{d}u\,\mathrm{d}v\\
&=
\frac{1}{X+Y}
\left(
\frac{1}{X+F}
+
\frac{1}{Y+F}
\right).
\end{align*}
Since
\[
A^{1/2}B^{1/2}
\leq
X^{1/2}Y^{1/2}
\]
and
\[
\frac{X^{1/2}Y^{1/2}}{X+Y}
\leq\frac12,
\]
we obtain
\[
R\lesssim(1+F)^{-1}.
\]

Thus
\[
R
\lesssim
(1+A)^{-1/2}(1+B)^{-1/2}
\]
and
\[
R
\lesssim
(1+F)^{-1}.
\]
We now apply Lemma~\ref{l33} to these two bounds. Set
\[
\vartheta:=\frac{3(1-\gamma)}4.
\]
Since $\gamma\in(0,1/11)$, we have $\vartheta\in(0,1)$. Hence
\begin{align*}
R
&\lesssim
\left(
(1+A)^{-1/2}(1+B)^{-1/2}
\right)^{\vartheta}
\left(
(1+F)^{-1}
\right)^{1-\vartheta}\\
&=
(1+A)^{-\frac{3(1-\gamma)}8}
(1+B)^{-\frac{3(1-\gamma)}8}
(1+F)^{-\frac{1+3\gamma}{4}}\\
&=
(1+A)^{\frac{3\gamma}{8}-\frac38}
(1+B)^{\frac{3\gamma}{8}-\frac38}
(1+F)^{-\frac{1+3\gamma}{4}}.
\end{align*}

Multiplying this estimate by the factor $(1+F)^{-1}$ already obtained from
the $v_3$-integration gives
\[
(1+F)^{-1-\frac{1+3\gamma}{4}}
=
(1+F)^{-\frac{5+3\gamma}{4}}.
\]
Since $\gamma>0$,
\[
-\frac{5+3\gamma}{4}
\leq
-\frac{5-5\gamma}{4}
=
\frac{5\gamma}{4}-\frac54,
\]
and therefore
\[
(1+F)^{-\frac{5+3\gamma}{4}}
\leq
(1+F)^{\frac{5\gamma}{4}-\frac54}.
\]
Likewise,
\[
(1+C)^{-1}\leq(1+C)^{\gamma-1},
\]
\[
(1+D)^{-1}\leq(1+D)^{\gamma-1},
\qquad
(1+E)^{-1}\leq(1+E)^{\gamma-1}.
\]
Using also
\[
1+A\asymp1+|m_4|^4,
\qquad
1+B\asymp1+|m_2|^4,
\]
we conclude that
\begin{align*}
&\mathds{E}\left[\left|I_{r,\epsilon,2}\right|^2\right]\\
&\stackrel{\text{Lemma \ref{l33}}}{\lesssim}
\sum_{\left(m_1,m_2,m_3,m_4\right)\in\left(\mathbb Z^d\right)^4}
\left|\rho_{m-1}\left(m_1\right)\right|^2
\left(1+\left|m_3\right|^4\right)^{\gamma-1}
\left(1+\left|m_2-m_3\right|^4\right)^{\gamma-1}
\left(1+\left|m_3-m_4\right|^4\right)^{\gamma-1}\\
&\quad\times
\left(1+\left|m_2\right|^4\right)^{\frac{3\gamma}{8}-\frac38}
\left(1+\left|m_4\right|^4\right)^{\frac{3\gamma}{8}-\frac38}
\left(1+\left|m_1+m_3\right|^4\right)^{\frac{5\gamma}{4}-\frac54}.
\end{align*}

We now make explicit the successive applications of Lemma~\ref{lemma2}. For fixed $m_1$ and $m_3$, consider first the sum over $m_2$. We write
\[
\left(1+|m_2-m_3|^4\right)^{\gamma-1}
=
\left(1+|m_2-m_3|^4\right)^{-\frac{4(1-\gamma)}4}
\]
and
\[
\left(1+|m_2|^4\right)^{\frac{3\gamma}{8}-\frac38}
=
\left(1+|m_2|^4\right)^{-\frac{\frac32(1-\gamma)}4}.
\]
Since $d=5$ and $\gamma\in(0,1/11)$,
\[
4(1-\gamma)+\frac32(1-\gamma)
=
\frac{11}{2}(1-\gamma)>5.
\]
Hence Lemma~\ref{lemma2} gives
\begin{align*}
&\sum_{m_2\in\mathbb Z^d}
\left(1+|m_2-m_3|^4\right)^{\gamma-1}
\left(1+|m_2|^4\right)^{\frac{3\gamma}{8}-\frac38}\\
&\lesssim
\left(1+|m_3|^4\right)^{
\frac{5-4(1-\gamma)-\frac32(1-\gamma)}4}\\
&=
\left(1+|m_3|^4\right)^{\frac{11\gamma}{8}-\frac18}.
\end{align*}
Exactly the same argument applied to the sum over $m_4$ yields
\begin{align*}
&\sum_{m_4\in\mathbb Z^d}
\left(1+|m_3-m_4|^4\right)^{\gamma-1}
\left(1+|m_4|^4\right)^{\frac{3\gamma}{8}-\frac38}\\
&\lesssim
\left(1+|m_3|^4\right)^{\frac{11\gamma}{8}-\frac18}.
\end{align*}
Combining these two estimates with the factor
$\left(1+|m_3|^4\right)^{\gamma-1}$ gives
\begin{align*}
&\left(1+|m_3|^4\right)^{\gamma-1}
\left(1+|m_3|^4\right)^{\frac{11\gamma}{8}-\frac18}
\left(1+|m_3|^4\right)^{\frac{11\gamma}{8}-\frac18}\\
&=
\left(1+|m_3|^4\right)^{
\gamma-1+\frac{11\gamma}{4}-\frac14}\\
&=
\left(1+|m_3|^4\right)^{\frac{15\gamma}{4}-\frac54}.
\end{align*}
We are therefore left with
\[
\sum_{m_3\in\mathbb Z^d}
\left(1+|m_3|^4\right)^{\frac{15\gamma}{4}-\frac54}
\left(1+|m_1+m_3|^4\right)^{\frac{5\gamma}{4}-\frac54}.
\]
Writing
\[
\frac{15\gamma}{4}-\frac54
=
-\frac{5-15\gamma}{4},
\qquad
\frac{5\gamma}{4}-\frac54
=
-\frac{5-5\gamma}{4},
\]
we have
\[
(5-15\gamma)+(5-5\gamma)
=
10-20\gamma>5.
\]
Moreover, since $\gamma>0$,
\[
5-15\gamma<5,
\qquad
5-5\gamma<5.
\]
Thus Lemma~\ref{lemma2} applies once more and yields
\begin{align*}
&\sum_{m_3\in\mathbb Z^d}
\left(1+|m_3|^4\right)^{\frac{15\gamma}{4}-\frac54}
\left(1+|m_1+m_3|^4\right)^{\frac{5\gamma}{4}-\frac54}\\
&\lesssim
\left(1+|m_1|^4\right)^{
\frac{5-(5-15\gamma)-(5-5\gamma)}4}\\
&=
\left(1+|m_1|^4\right)^{5\gamma-\frac54}.
\end{align*}
Consequently,
\begin{align*}
\mathds{E}\left[\left|I_{r,\epsilon,2}\right|^2\right]
&\stackrel{\text{Lemma \ref{lemma2}}}{\lesssim}
\sum_{m_1\in\mathbb Z^d}
\left|\rho_{m-1}\left(m_1\right)\right|^2
\left(1+\left|m_1\right|^4\right)^{5\gamma-\frac54}.
\end{align*}

Finally, for $m\geq1$, on the support of $\rho_{m-1}, |m_1|\asymp 2^{m-1},$
Since $d=5$, the number of lattice points in the corresponding dyadic annulus is of order $2^{5(m-1)}$. For $m=0$, the support of $\rho_{-1}$ is bounded,
and the corresponding estimate follows directly. Therefore,
\begin{align*}
\sum_{m_1\in\mathbb Z^d}
\left|\rho_{m-1}\left(m_1\right)\right|^2
\left(1+\left|m_1\right|^4\right)^{5\gamma-\frac54}
&\lesssim
2^{5(m-1)}
2^{4\left(5\gamma-\frac54\right)(m-1)}\\
&=
2^{20\gamma(m-1)}.
\end{align*}
We conclude that
\[
\mathds{E}\left[\left|I_{r,\epsilon,2}\right|^2\right]
\lesssim
2^{20\gamma(m-1)}.
\]
\end{proof}

\chapter{}
\label{ppe6}
\begin{proof}[Proof of Theorem \ref{theorem6}]
Here again we only give the crucial $L^2(\Omega)$ bound at a fixed time $r$.

Let $\gamma \in \left]0,\frac{1}{11}\right[,$ 
\begin{align*}
\psi:\left(\mathbb{Z}^d\right)^2&\longrightarrow \left[0,1\right]\\
\left(m_1,m_2\right)&\longmapsto\psi\left(m_1,m_2\right)=\sum_{q \in \mathbb{N}}\sum_{j=0}^{\min\left(q+1,2\right)}\rho_{j+\max\left(q-2,-1\right)}\left(m_1\right)\rho_{q-1}\left(m_2\right)
\end{align*}
The following observation is crucial for our estimation $$\psi(m_1,m_2)\neq 0 \implies \kappa_1 (1+|m_2|)\leq 1+|m_1|\leq \kappa_2 (1+|m_2|)$$ for $0<\kappa_1<\kappa_2.$
  
It follows from Lemma \ref{l3} that: $$(\delta_{m-1}\Xi_{r,\epsilon,3})(x)=I_{r,\epsilon,1}+6I_{r,\epsilon,2}+6(I_{r,\epsilon,3}-\psi_{\epsilon,4}(r)X_{r,\epsilon,1}),$$ where $$I_{r,\epsilon,1}=\int_{\mathbb{T}^d}dyK_{m-1}(x-y)\sum_{q \in \mathbb{N}}\sum_{j=0}^{\min\left(q+1,2\right)}V_{5}(f_{r,\epsilon,j,q,y}\otimes_0g_{r,\epsilon,q,y})=V_{5}(\int_{\mathbb{T}^d}dyK_{m-1}(x-y)\sum_{q \in \mathbb{N}}\sum_{j=0}^{\min\left(q+1,2\right)}f_{r,\epsilon,j,q,y}\otimes_0g_{r,\epsilon,q,y}),$$ $$I_{r,\epsilon,2}=\int_{\mathbb{T}^d}dyK_{m-1}(x-y)\sum_{q \in \mathbb{N}}\sum_{j=0}^{\min\left(q+1,2\right)}V_{3}(f_{r,\epsilon,j,q,y}\otimes_1g_{r,\epsilon,q,y})=V_{3}(\int_{\mathbb{T}^d}dyK_{m-1}(x-y)\sum_{q \in \mathbb{N}}\sum_{j=0}^{\min\left(q+1,2\right)}f_{r,\epsilon,j,q,y}\otimes_1g_{r,\epsilon,q,y}),$$ $$I_{r,\epsilon,3}=\int_{\mathbb{T}^d}dyK_{m-1}(x-y)\sum_{q \in \mathbb{N}}\sum_{j=0}^{\min\left(q+1,2\right)}V_{1}(f_{r,\epsilon,j,q,y}\otimes_2g_{r,\epsilon,q,y})=V_{1}(\int_{\mathbb{T}^d}dyK_{m-1}(x-y)\sum_{q \in \mathbb{N}}\sum_{j=0}^{\min\left(q+1,2\right)}f_{r,\epsilon,j,q,y}\otimes_2g_{r,\epsilon,q,y}),$$ where $$f_{r,\epsilon,j,q,y}((u_1,x_1),(u_2,x_2),(u_3,x_3))=\int_0^r$$$$(K_{j+\max(q-2,-1)}*\Delta\eta_{r-v}*((\zeta_\epsilon*\eta_{v-u_1})(\cdot-x_1)(\zeta_\epsilon*\eta_{v-u_2})(\cdot-x_2)(\zeta_\epsilon*\eta_{v-u_3})(\cdot-x_3)))(y)\mathds{1}_{[0,v]^3}(u_1,u_2,u_3)dv,$$ $$g_{r,\epsilon,q,y}((u_4,x_4),(u_5,x_5))=(K_{q-1}*((\zeta_\epsilon*\eta_{r-u_4}(\cdot-x_4)(\zeta_\epsilon*\eta_{r-u_5})(\cdot-x_5)))(y)\mathds{1}_{[0,r]^2}(u_4,u_5).$$
Using Theorem \ref{nelson}, we obtain

\begin{align*}
&\mathds{E}[|I_{r,\epsilon,1}|^2]\\
&\lesssim \int_{(\mathbb{R}_+\times\mathbb{T}^d)^5}du_1dx_1du_2dx_2du_3dx_3du_4dx_4du_5dx_5|\int_{\mathbb{T}^d}dyK_{m-1}(x-y)\\
&\sum_{q \in \mathbb{N}}\sum_{j=0}^{\min\left(q+1,2\right)}f_{r,\epsilon,j,q,y}((u_1,x_1),(u_2,x_2),(u_3,x_3))g_{r,\epsilon,q,y}((u_4,x_4),(u_5,x_5))|^2\\
&\stackrel{\text{Fubini}}{=}\int_{(\mathbb{R}_+\times\mathbb{T}^d)^5}du_1dx_1du_2dx_2du_3dx_3du_4dx_4du_5dx_5\int_{(\mathbb{T}^d)^2}dy_1dy_2K_{m-1}(x-y_1)K_{m-1}(x-y_2)\\
&\sum_{(q_1,q_2) \in \mathbb{N}^2}\sum_{j_1=0}^{\min\left(q_1+1,2\right)}\sum_{j_2=0}^{\min(q_2+1,2)}f_{r,\epsilon,j_1,q_1,y_1}((u_1,x_1),(u_2,x_2),(u_3,x_3))g_{r,\epsilon,q_1,y_1}((u_4,x_4),(u_5,x_5))\\
&f_{r,\epsilon,j_2,q_2,y_2}((u_1,x_1),(u_2,x_2),(u_3,x_3))g_{r,\epsilon,q_2,y_2}((u_4,x_4),(u_5,x_5))\\
&=\int_{[0,r]^2}dv_1dv_2\int_{(\mathbb{R}_+\times\mathbb{T}^d)^5}du_1dx_1du_2dx_2du_3dx_3du_4dx_4du_5dx_5\int_{(\mathbb{T}^d)^6}dy_1dy_2dy_1'dy_2'dy'_3dy_4'K_{m-1}(x-y_1)\\
&K_{m-1}(x-y_2)\sum_{(q_1,q_2) \in \mathbb{N}^2}\sum_{j_1=0}^{\min\left(q_1+1,2\right)}\sum_{j_2=0}^{\min(q_2+1,2)}(K_{j_1+\max(q_1-2,-1)}*\Delta\eta_{r-v_1})(y_1-y'_1)\\
&(K_{j_2+\max(q_2-2,-1)}*\Delta\eta_{r-v_2})(y_2-y'_3)K_{q_1-1}(y_1-y'_2)K_{q_2-1}(y_2-y'_4)(\zeta_\epsilon*\eta_{v_1-u_1})(y'_1-x_1)(\zeta_\epsilon*\eta_{v_1-u_2})(y_1'-x_2)\\
&(\zeta_\epsilon*\eta_{v_1-u_3})(y'_1-x_3)(\zeta_\epsilon*\eta_{r-u_4})(y'_2-x_4)(\zeta_\epsilon*\eta_{r-u_5})(y'_2-x_5)(\zeta_\epsilon*\eta_{v_2-u_1})(y'_3-x_1)(\zeta_\epsilon*\eta_{v_2-u_2})(y_3'-x_2)\\
&(\zeta_\epsilon*\eta_{v_2-u_3})(y_3'-x_3)(\zeta_\epsilon*\eta_{r-u_4})(y'_4-x_4)(\zeta_\epsilon*\eta_{r-u_5})(y_4'-x_5)\mathds{1}_{[0,r]^2}(u_4,u_5)\mathds{1}_{[0,\min(v_1,v_2)]^3}(u_1,u_2,u_3)\\
&=\int_{[0,r]^2}dv_1dv_2\int_{[0,\min(v_1,v_2)]^3\times[0,r]^2}du_1du_2du_3du_4du_5\int_{(\mathbb{T}^d)^6}dy_1dy_2dy_1'dy_2'dy'_3dy_4'K_{m-1}(x-y_1)K_{m-1}(x-y_2)\\
&\sum_{(q_1,q_2) \in \mathbb{N}^2}\sum_{j_1=0}^{\min\left(q_1+1,2\right)}\sum_{j_2=0}^{\min(q_2+1,2)}(K_{j_1+\max(q_1-2,-1)}*\Delta\eta_{r-v_1})(y_1-y'_1)(K_{j_2+\max(q_2-2,-1)}*\Delta\eta_{r-v_2})(y_2-y'_3)\\
&K_{q_1-1}(y_1-y'_2)K_{q_2-1}(y_2-y'_4)(\zeta_\epsilon*\zeta_{\epsilon}*\eta_{v_1+v_2-2u_1})(y'_1-y_3')(\zeta_\epsilon*\zeta_{\epsilon}*\eta_{v_1+v_2-2u_2})(y_1'-y'_3)\\
&(\zeta_\epsilon*\zeta_{\epsilon}*\eta_{v_1+v_2-2u_3})(y_1'-y_3')(\zeta_\epsilon*\zeta_{\epsilon}*\eta_{2(r-u_4)})(y'_2-y_4')(\zeta_\epsilon*\zeta_{\epsilon}*\eta_{2(r-u_5)})(y'_2-y_4')\\
&=\int_{[0,r]^2}dv_1dv_2\int_{[0,\min(v_1,v_2)]^3\times[0,r]^2}du_1du_2du_3du_4du_5\int_{(\mathbb{T}^d)^2}dy_1dy_2K_{m-1}(y_1)K_{m-1}(y_2)\\
&\sum_{(q_1,q_2) \in \mathbb{N}^2}\sum_{j_1=0}^{\min\left(q_1+1,2\right)}\sum_{j_2=0}^{\min(q_2+1,2)}(K_{j_1+\max(q_1-2,-1)}*\Delta^2\eta_{2r-v_1-v_2}*K_{j_2+\max(q_2-2,-1)}\\
&*((\zeta_\epsilon*\zeta_{\epsilon}*\eta_{v_1+v_2-2u_1})(\zeta_\epsilon*\zeta_{\epsilon}*\eta_{v_1+v_2-2u_2})(\zeta_\epsilon*\zeta_{\epsilon}*\eta_{v_1+v_2-2u_3})))(y_1-y_2)\\
&\times (K_{q_1-1}*K_{q_2-1}*((\zeta_\epsilon*\zeta_{\epsilon}*\eta_{2(r-u_4)})(\zeta_\epsilon*\zeta_{\epsilon}*\eta_{2(r-u_5)})))(y_1-y_2)
\end{align*} \begin{align*}
&\mathds{E}[|I_{r,\epsilon,2}|^2]\\
&\lesssim \int_{(\mathbb{R}_+\times\mathbb{T}^d)^3}du_1dx_1du_2dx_2du_3dx_3|\int_{\mathbb{T}^d}dyK_{m-1}(x-y)\\
&\sum_{q \in \mathbb{N}}\sum_{j=0}^{\min\left(q+1,2\right)}\int_{\mathbb{R}_+\times\mathbb{T}^d}du_4dx_4f_{r,\epsilon,j,q,y}((u_1,x_1),(u_2,x_2),(u_4,x_4))g_{r,\epsilon,q,y}((u_4,x_4),(u_3,x_3))|^2\\
&\stackrel{\text{Fubini}}{=}\int_{(\mathbb{R}_+\times\mathbb{T}^d)^5}du_1dx_1du_2dx_2du_3dx_3du_4dx_4du_5dx_5\int_{(\mathbb{T}^d)^2}dy_1dy_2K_{m-1}(x-y_1)K_{m-1}(x-y_2)\\
&\sum_{(q_1,q_2) \in \mathbb{N}^2}\sum_{j_1=0}^{\min\left(q_1+1,2\right)}\sum_{j_2=0}^{\min(q_2+1,2)}f_{r,\epsilon,j_1,q_1,y_1}((u_1,x_1),(u_2,x_2),(u_4,x_4))g_{r,\epsilon,q_1,y_1}((u_4,x_4),(u_3,x_3))\\
&f_{r,\epsilon,j_2,q_2,y_2}((u_1,x_1),(u_2,x_2),(u_5,x_5))g_{r,\epsilon,q_2,y_2}((u_5,x_5),(u_3,x_3))\\
&=\int_{[0,r]^2}dv_1dv_2\int_{(\mathbb{R}_+\times\mathbb{T}^d)^5}du_1dx_1du_2dx_2du_3dx_3du_4dx_4du_5dx_5\int_{(\mathbb{T}^d)^6}dy_1dy_2dy_1'dy_2'dy'_3dy_4'K_{m-1}(x-y_1)\\
&K_{m-1}(x-y_2)\sum_{(q_1,q_2) \in \mathbb{N}^2}\sum_{j_1=0}^{\min\left(q_1+1,2\right)}\sum_{j_2=0}^{\min(q_2+1,2)}(K_{j_1+\max(q_1-2,-1)}*\Delta\eta_{r-v_1})(y_1-y'_1)\\
&(K_{j_2+\max(q_2-2,-1)}*\Delta\eta_{r-v_2})(y_2-y'_3)K_{q_1-1}(y_1-y'_2)K_{q_2-1}(y_2-y'_4)(\zeta_\epsilon*\eta_{v_1-u_1})(y'_1-x_1)(\zeta_\epsilon*\eta_{v_1-u_2})(y_1'-x_2)\\
&(\zeta_\epsilon*\eta_{v_1-u_4})(y'_1-x_4)(\zeta_\epsilon*\eta_{r-u_4})(y'_2-x_4)(\zeta_\epsilon*\eta_{r-u_3})(y'_2-x_3)(\zeta_\epsilon*\eta_{v_2-u_1})(y'_3-x_1)(\zeta_\epsilon*\eta_{v_2-u_2})(y_3'-x_2)\\
&(\zeta_\epsilon*\eta_{v_2-u_5})(y_3'-x_5)(\zeta_\epsilon*\eta_{r-u_5})(y'_4-x_5)(\zeta_\epsilon*\eta_{r-u_3})(y_4'-x_3)\mathds{1}_{[0,r]}(u_3)\mathds{1}_{[0,\min(v_1,v_2)]^2\times[0,v_1]\times[0,v_2]}(u_1,u_2,u_4,u_5)\\
&=\int_{[0,r]^2}dv_1dv_2\int_{[0,\min(v_1,v_2)]^2\times[0,r]\times[0,v_1]\times [0,v_2]}du_1du_2du_3du_4du_5\int_{(\mathbb{T}^d)^6}dy_1dy_2dy_1'dy_2'dy'_3dy_4'K_{m-1}(y_1)K_{m-1}(y_2)\\
&\sum_{(q_1,q_2) \in \mathbb{N}^2}\sum_{j_1=0}^{\min\left(q_1+1,2\right)}\sum_{j_2=0}^{\min(q_2+1,2)}(K_{j_1+\max(q_1-2,-1)}*\Delta\eta_{r-v_1})(y_1-y'_1)(K_{j_2+\max(q_2-2,-1)}*\Delta\eta_{r-v_2})(y_2-y'_3)\\
&K_{q_1-1}(y_1-y'_2)K_{q_2-1}(y_2-y'_4)(\zeta_\epsilon*\zeta_{\epsilon}*\eta_{v_1+v_2-2u_1})(y'_1-y_3')(\zeta_\epsilon*\zeta_{\epsilon}*\eta_{v_1+v_2-2u_2})(y_1'-y'_3)\\
&(\zeta_\epsilon*\zeta_{\epsilon}*\eta_{2(r-u_3)})(y_2'-y_4')(\zeta_\epsilon*\zeta_{\epsilon}*\eta_{r+v_1-2u_4})(y'_1-y_2')(\zeta_\epsilon*\zeta_{\epsilon}*\eta_{r+v_2-2u_5})(y'_3-y_4')\\
&=\int_{[0,r]^2}dv_1dv_2\int_{[0,\min(v_1,v_2)]^2\times[0,r]\times[0,v_1]\times [0,v_2]}du_1du_2du_3du_4du_5\int_{(\mathbb{T}^d)^6}dy_1dy_2dy_1'dy_2'dy'_3dy_4'K_{m-1}(y_1)K_{m-1}(y_2)\\
&\sum_{(q_1,q_2) \in \mathbb{N}^2}\sum_{j_1=0}^{\min\left(q_1+1,2\right)}\sum_{j_2=0}^{\min(q_2+1,2)}(K_{j_1+\max(q_1-2,-1)}*\Delta\eta_{r-v_1})(y_1-y_2-y'_1)(K_{j_2+\max(q_2-2,-1)}*\Delta\eta_{r-v_2})(y'_3)\\
&K_{q_1-1}(y_1-y_2-y'_2)K_{q_2-1}(y'_4)(\zeta_\epsilon*\zeta_{\epsilon}*\eta_{v_1+v_2-2u_1})(y'_1-y_3')(\zeta_\epsilon*\zeta_{\epsilon}*\eta_{v_1+v_2-2u_2})(y_1'-y'_3)\\
&(\zeta_\epsilon*\zeta_{\epsilon}*\eta_{2(r-u_3)})(y_2'-y_4')(\zeta_\epsilon*\zeta_{\epsilon}*\eta_{r+v_1-2u_4})(y'_1-y_2')(\zeta_\epsilon*\zeta_{\epsilon}*\eta_{r+v_2-2u_5})(y'_3-y_4')\\
\end{align*}
\begin{align*}
&\mathds{E}\left[\left|I_{r,\epsilon,3}
-\psi_{\epsilon,4}(r)\left(\delta_{m-1}\widetilde{X}_{r,\epsilon,1}\right)(x)\right|^2\right]\\
&\lesssim \int_{\mathbb{R}_+\times\mathbb{T}^d}du_1dx_1|\int_{\mathbb{T}^d}dyK_{m-1}(x-y)(-\psi_{\epsilon,4}(r)(\zeta_\epsilon*\eta_{r-u_1})(y-x_1)\mathds{1}_{[0,r]}(u_1)\\
&+\sum_{q \in \mathbb{N}}\sum_{j=0}^{\min\left(q+1,2\right)}\int_{(\mathbb{R}_+\times\mathbb{T}^d)^2}du_2dx_2du_3dx_3f_{r,\epsilon,j,q,y}((u_1,x_1),(u_2,x_2),(u_3,x_3))g_{r,\epsilon,q,y}((u_2,x_2),(u_3,x_3)))|^2\\
&=\int_{\mathbb{R}_+\times\mathbb{T}^d}du_1dx_1|\int_{\mathbb{T}^d}dyK_{m-1}(x-y)(-\psi_{\epsilon,4}(r)(\zeta_\epsilon*\eta_{r-u_1})(y-x_1)\mathds{1}_{[0,r]}(u_1)\\
&+\sum_{q \in \mathbb{N}}\sum_{j=0}^{\min\left(q+1,2\right)}\int_{(\mathbb{R}_+\times\mathbb{T}^d)^2}du_2dx_2du_3dx_3\int_{[0,r]\times(\mathbb{T}^d)^2}dvdy'_1dy'_2(K_{j+\max(q-2,-1)}*\Delta\eta_{r-v})(y-y'_1)\\
&K_{q-1}(y-y'_2)(\zeta_\epsilon*\eta_{v-u_1})(y'_1-x_1)(\zeta_\epsilon*\eta_{v-u_2})(y'_1-x_2)(\zeta_\epsilon*\eta_{v-u_3})(y'_1-x_3)(\zeta_\epsilon*\eta_{r-u_2})(y'_2-x_2)\\
&(\zeta_\epsilon*\eta_{r-u_3})(y'_2-x_3)\mathds{1}_{[0,v]^3}(u_1,u_2,u_3))|^2\\
&=\int_{[0,r]\times\mathbb{T}^d}du_1dx_1|\int_{\mathbb{T}^d}dyK_{m-1}(x-y)(-\psi_{\epsilon,4}(r)(\zeta_\epsilon*\eta_{r-u_1})(y-x_1)\\
&+\sum_{q \in \mathbb{N}}\sum_{j=0}^{\min\left(q+1,2\right)}\int_{[u_1,r]\times(\mathbb{T}^d)^2}dvdy'_1dy'_2\int_{[0,v]^2}du_2du_3(K_{j+\max(q-2,-1)}*\Delta\eta_{r-v})(y-y'_1)K_{q-1}(y-y'_2)\\
&(\zeta_\epsilon*\eta_{v-u_1})(y'_1-x_1)(\zeta_\epsilon *\zeta_\epsilon*\eta_{r+v-2u_2})(y'_1-y'_2)(\zeta_\epsilon*\zeta_\epsilon*\eta_{r+v-2u_3})(y'_1-y'_2))|^2\\
&=\int_{[0,r]\times\mathbb{T}^d}du_1dx_1|(-\psi_{\epsilon,4}(r)(K_{m-1}*\zeta_\epsilon*\eta_{r-u_1})(x_1)\\
&+\sum_{q \in \mathbb{N}}\sum_{j=0}^{\min\left(q+1,2\right)}\int_{[u_1,r]}dv\int_{[0,v]^2}du_2du_3(K_{m-1}*\zeta_\epsilon*\eta_{v-u_1}*((K_{j+\max(q-2,-1)}*\Delta\eta_{r-v})\\
&\times(K_{q-1}*((\zeta_\epsilon *\zeta_\epsilon*\eta_{r+v-2u_2})(\zeta_\epsilon*\zeta_\epsilon*\eta_{r+v-2u_3})))))(x_1))|^2
\end{align*}

\begin{enumerate}
\item \textbf{Estimation of $\mathds{E}\left[\left|I_{r,\epsilon,1}\right|^2\right]$:} We are in a position to apply Parseval to obtain the following bound:
\begin{align*}
&\mathds{E}\left[\left|I_{r,\epsilon,1}\right|^2\right]\\
&\lesssim
\sum_{\left(m_1,m_2,m_3,m_4,m_5\right)\in\left(\mathbb Z^d\right)^5}
\left|
\rho_{m-1}\left(m_1+m_2+m_3+m_4+m_5\right)
\psi\left(m_1+m_2+m_3,m_4+m_5\right)
\prod_{k=1}^5\zeta(\epsilon m_k)
\right|^2\\
&\quad\times
\left|m_1+m_2+m_3\right|^4
\int_{\left[0,r\right]^4}
\mathrm{e}^{-\left|2\pi(m_1+m_2+m_3)\right|^4
\left(2r-\alpha_1-\alpha_2\right)
-2\left|2\pi m_4\right|^4\left(r-\alpha_3\right)
-2\left|2\pi m_5\right|^4\left(r-\alpha_4\right)}
\\
&\quad\times
\left(
\int_{\left[0,\min\left(\alpha_1,\alpha_2\right)\right]^3}
\mathrm{e}^{-\left|2\pi m_1\right|^4
\left(\alpha_1+\alpha_2-2v_1\right)
-\left|2\pi m_2\right|^4
\left(\alpha_1+\alpha_2-2v_2\right)
-\left|2\pi m_3\right|^4
\left(\alpha_1+\alpha_2-2v_3\right)}
\,\mathrm{d}v_1\,\mathrm{d}v_2\,\mathrm{d}v_3
\right)\\
&\quad\times
\mathrm{d}\alpha_1\,\mathrm{d}\alpha_2\,
\mathrm{d}\alpha_3\,\mathrm{d}\alpha_4.
\end{align*}

We now estimate the time integrals in detail. Since $\zeta$ is bounded, the
factors involving $\zeta$ can be absorbed into the implicit constant.
Moreover,
\[
0\leq\psi\leq1,
\]
and the localization property of $\psi$ gives
\[
\psi(m_1+m_2+m_3,m_4+m_5)\neq0
\quad\Longrightarrow\quad
1+|m_1+m_2+m_3|\asymp1+|m_4+m_5|.
\]

Set
\[
n:=m_1+m_2,
\qquad
p:=m_1+m_2+m_3=n+m_3,
\qquad
q:=m_4+m_5,
\]
and
\[
A:=|2\pi p|^4,
\qquad
a_k:=|2\pi m_k|^4,\quad k=1,\ldots,5.
\]
We first integrate with respect to $\alpha_3$ and $\alpha_4$. We have
\[
\int_0^r
\mathrm{e}^{-2a_4(r-\alpha_3)}
\,\mathrm{d}\alpha_3
\lesssim
(1+a_4)^{-1},
\]
and similarly
\[
\int_0^r
\mathrm{e}^{-2a_5(r-\alpha_4)}
\,\mathrm{d}\alpha_4
\lesssim
(1+a_5)^{-1}.
\]

We next estimate the integrals with respect to $v_1,v_2,v_3$. Let
\[
s:=\min(\alpha_1,\alpha_2).
\]
For each $k\in\{1,2,3\}$,
\begin{align*}
\int_0^s
\mathrm{e}^{-a_k(\alpha_1+\alpha_2-2v)}
\,\mathrm{d}v
&=
\mathrm{e}^{-a_k|\alpha_1-\alpha_2|}
\int_0^s
\mathrm{e}^{-2a_k(s-v)}
\,\mathrm{d}v\\
&\lesssim
(1+a_k)^{-1}
\mathrm{e}^{-c a_k|\alpha_1-\alpha_2|}
\end{align*}
for some $c>0$. Consequently,
\begin{align*}
&\int_{\left[0,\min(\alpha_1,\alpha_2)\right]^3}
\mathrm{e}^{-\sum_{k=1}^3
a_k(\alpha_1+\alpha_2-2v_k)}
\,\mathrm{d}v_1\,\mathrm{d}v_2\,\mathrm{d}v_3\\
&\lesssim
\prod_{k=1}^3(1+a_k)^{-1}
\mathrm{e}^{-c(a_1+a_2+a_3)|\alpha_1-\alpha_2|}.
\end{align*}

Let
\[
S:=a_1+a_2+a_3.
\]
Using $|p|^4\asymp A$, the remaining $\alpha_1,\alpha_2$ integral is bounded by
\begin{align*}
&A
\int_{[0,r]^2}
\mathrm{e}^{-A(2r-\alpha_1-\alpha_2)}
\mathrm{e}^{-cS|\alpha_1-\alpha_2|}
\,\mathrm{d}\alpha_1\,\mathrm{d}\alpha_2.
\end{align*}
Making the change of variables
\[
u:=r-\alpha_1,
\qquad
v:=r-\alpha_2,
\]
we obtain
\begin{align*}
&A
\int_{[0,r]^2}
\mathrm{e}^{-A(2r-\alpha_1-\alpha_2)}
\mathrm{e}^{-cS|\alpha_1-\alpha_2|}
\,\mathrm{d}\alpha_1\,\mathrm{d}\alpha_2\\
&=
A
\int_{[0,r]^2}
\mathrm{e}^{-A(u+v)}
\mathrm{e}^{-cS|u-v|}
\,\mathrm{d}u\,\mathrm{d}v\\
&\leq
A
\int_{\mathbb R_+^2}
\mathrm{e}^{-A(u+v)}
\mathrm{e}^{-cS|u-v|}
\,\mathrm{d}u\,\mathrm{d}v.
\end{align*}
By symmetry,
\begin{align*}
&A
\int_{\mathbb R_+^2}
\mathrm{e}^{-A(u+v)}
\mathrm{e}^{-cS|u-v|}
\,\mathrm{d}u\,\mathrm{d}v\\
&=
2A
\int_{\{0\leq v\leq u\}}
\mathrm{e}^{-A(u+v)}
\mathrm{e}^{-cS(u-v)}
\,\mathrm{d}u\,\mathrm{d}v.
\end{align*}
Setting
\[
z:=u-v,
\qquad
w:=v,
\]
we get
\begin{align*}
&2A
\int_0^\infty\int_0^\infty
\mathrm{e}^{-A(z+2w)}
\mathrm{e}^{-cSz}
\,\mathrm{d}w\,\mathrm{d}z\\
&=
2A
\left(
\int_0^\infty
\mathrm{e}^{-2Aw}\,\mathrm{d}w
\right)
\left(
\int_0^\infty
\mathrm{e}^{-(A+cS)z}\,\mathrm{d}z
\right)\\
&\lesssim
(1+A+S)^{-1}.
\end{align*}
If $A=0$, the original expression vanishes because of the factor $|p|^4$.

Consequently, the complete time factor is bounded by
\begin{align*}
&\prod_{k=1}^5(1+a_k)^{-1}
(1+A+S)^{-1}.
\end{align*}

We now distribute part of the last factor among the frequencies. Since
\[
|n|^4
=
|m_1+m_2|^4
\lesssim
|m_1|^4+|m_2|^4,
\]
we have
\[
1+|n|^4\lesssim1+A+S.
\]
Clearly,
\[
1+|m_3|^4\lesssim1+A+S,
\qquad
1+|p|^4\lesssim1+A+S.
\]
Set
\[
\vartheta_1:=\frac{1-\gamma}{4},
\qquad
\vartheta_2:=\frac{1-\gamma}{4},
\qquad
\vartheta_3:=\frac{1-5\gamma}{2}.
\]
Since $\gamma\in(0,1/11)$,
\[
\vartheta_1+\vartheta_2+\vartheta_3
=
1-3\gamma<1.
\]
Therefore,
\begin{align*}
(1+A+S)^{-1}
&\leq
(1+A+S)^{-(1-3\gamma)}\\
&\lesssim
\left(1+|n|^4\right)^{-\frac{1-\gamma}{4}}
\left(1+|m_3|^4\right)^{-\frac{1-\gamma}{4}}
\left(1+|p|^4\right)^{-\frac{1-5\gamma}{2}}.
\end{align*}
Moreover,
\[
(1+|m_1|^4)^{-1}
\leq
(1+|m_1|^4)^{\gamma-1},
\]
\[
(1+|m_2|^4)^{-1}
\leq
(1+|m_2|^4)^{\gamma-1},
\]
and
\begin{align*}
&(1+|m_3|^4)^{-1}
(1+|m_3|^4)^{-\frac{1-\gamma}{4}}\\
&=
(1+|m_3|^4)^{-\frac54+\frac{\gamma}{4}}\\
&\leq
(1+|m_3|^4)^{\frac{5\gamma}{4}-\frac54}.
\end{align*}
Thus,
\begin{align*}
&\mathds{E}\left[\left|I_{r,\epsilon,1}\right|^2\right]\\
&\lesssim
\sum_{\left(m_1,m_2,m_3,m_4,m_5\right)\in\left(\mathbb Z^d\right)^5}
\left|
\rho_{m-1}\left(m_1+m_2+m_3+m_4+m_5\right)
\right|^2\\
&\quad\times
\left(1+\left|m_1\right|^4\right)^{\gamma-1}
\left(1+\left|m_2\right|^4\right)^{\gamma-1}
\left(1+\left|m_1+m_2\right|^4\right)^{\frac{\gamma}{4}-\frac14}\\
&\quad\times
\left(1+\left|m_3\right|^4\right)^{\frac{5\gamma}{4}-\frac54}
\left(1+\left|m_4\right|^4\right)^{-1}
\left(1+\left|m_5\right|^4\right)^{-1}\\
&\quad\times
\left(1+\left|m_1+m_2+m_3\right|^4\right)^{-\frac{1-5\gamma}{2}}
\left|
\psi\left(m_1+m_2+m_3,m_4+m_5\right)
\right|^2.
\end{align*}

We now make explicit the successive applications of
Lemma~\ref{lemma2}. Set
\[
n:=m_1+m_2.
\]
For fixed $n$, we first sum with respect to $m_1$. Since $d=5$,
\begin{align*}
&\sum_{m_1\in\mathbb Z^d}
\left(1+|m_1|^4\right)^{\gamma-1}
\left(1+|n-m_1|^4\right)^{\gamma-1}\\
&=
\sum_{m_1\in\mathbb Z^d}
\left(1+|m_1|^4\right)^{-\frac{4(1-\gamma)}4}
\left(1+|n-m_1|^4\right)^{-\frac{4(1-\gamma)}4}\\
&\stackrel{\text{Lemma \ref{lemma2}}}{\lesssim}
\left(1+|n|^4\right)^{
\frac{5-8(1-\gamma)}4}\\
&=
\left(1+|n|^4\right)^{2\gamma-\frac34}.
\end{align*}
Multiplying by the remaining factor depending on $n$ gives
\begin{align*}
\left(1+|n|^4\right)^{2\gamma-\frac34}
\left(1+|n|^4\right)^{\frac{\gamma}{4}-\frac14}
&=
\left(1+|n|^4\right)^{\frac{9\gamma}{4}-1}.
\end{align*}

We next sum this weight against the factor involving $m_3$. Let
\[
p:=n+m_3=m_1+m_2+m_3.
\]
We have
\[
\frac{9\gamma}{4}-1
=
-\frac{4-9\gamma}{4},
\qquad
\frac{5\gamma}{4}-\frac54
=
-\frac{5-5\gamma}{4}.
\]
Hence Lemma~\ref{lemma2} gives
\begin{align*}
&\sum_{n\in\mathbb Z^d}
\left(1+|n|^4\right)^{\frac{9\gamma}{4}-1}
\left(1+|p-n|^4\right)^{\frac{5\gamma}{4}-\frac54}\\
&\lesssim
\left(1+|p|^4\right)^{
\frac{5-(4-9\gamma)-(5-5\gamma)}4}\\
&=
\left(1+|p|^4\right)^{\frac{7\gamma}{2}-1}.
\end{align*}
Thus the contribution of the frequencies $m_1,m_2,m_3$ is bounded by
\[
\left(1+|p|^4\right)^{\frac{7\gamma}{2}-1}.
\]

We now treat $m_4$ and $m_5$. Let
\[
q:=m_4+m_5.
\]
Lemma~\ref{lemma2}, applied with $\alpha=\beta=4$ and $d=5$, gives
\begin{align*}
&\sum_{m_4\in\mathbb Z^d}
\left(1+|m_4|^4\right)^{-1}
\left(1+|q-m_4|^4\right)^{-1}\\
&\lesssim
\left(1+|q|^4\right)^{-\frac34}.
\end{align*}
Moreover, on the support of $\psi(p,q)$,
\[
1+|p|\asymp1+|q|,
\]
and hence
\[
1+|p|^4\asymp1+|q|^4.
\]
Therefore,
\begin{align*}
&\left(1+|p|^4\right)^{-\frac{1-5\gamma}{2}}
\left(1+|q|^4\right)^{-\frac34}\\
&\lesssim
\left(1+|q|^4\right)^{
-\frac{1-5\gamma}{2}-\frac34}\\
&=
\left(1+|q|^4\right)^{\frac{5\gamma}{2}-\frac54}.
\end{align*}

Consequently, writing
\[
p=m_1+m_2+m_3,
\qquad
q=m_4+m_5,
\]
the fivefold sum is bounded by
\[
\sum_{p,q\in\mathbb Z^d}
\left|\rho_{m-1}(p+q)\right|^2
\left(1+|p|^4\right)^{\frac{7\gamma}{2}-1}
\left(1+|q|^4\right)^{\frac{5\gamma}{2}-\frac54}.
\]

Finally, we apply Lemma~\ref{lemma2} once more. Since
\[
\frac{7\gamma}{2}-1
=
-\frac{4-14\gamma}{4},
\qquad
\frac{5\gamma}{2}-\frac54
=
-\frac{5-10\gamma}{4},
\]
we obtain, for
\[
\ell:=p+q,
\]
that
\begin{align*}
&\sum_{p\in\mathbb Z^d}
\left(1+|p|^4\right)^{\frac{7\gamma}{2}-1}
\left(1+|\ell-p|^4\right)^{\frac{5\gamma}{2}-\frac54}\\
&\lesssim
\left(1+|\ell|^4\right)^{
\frac{5-(4-14\gamma)-(5-10\gamma)}4}\\
&=
\left(1+|\ell|^4\right)^{6\gamma-1}.
\end{align*}
All the assumptions of Lemma~\ref{lemma2} in the preceding applications are
satisfied since $\gamma\in(0,1/11)$. Therefore,
\begin{align*}
\mathds{E}\left[\left|I_{r,\epsilon,1}\right|^2\right]
&\stackrel{\text{Lemma \ref{lemma2}}}{\lesssim}
\sum_{\ell\in\mathbb Z^d}
\left|\rho_{m-1}\left(\ell\right)\right|^2
\left(1+\left|\ell\right|^4\right)^{6\gamma-1}.
\end{align*}

Finally, for $m\geq1$, on the support of $\rho_{m-1}$,
\[
|\ell|\asymp2^{m-1}.
\]
Since $d=5$, the number of lattice points in the corresponding dyadic annulus
is of order $2^{5(m-1)}$. For $m=0$, the support of $\rho_{-1}$ is bounded,
and the corresponding estimate follows directly. Hence
\begin{align*}
\sum_{\ell\in\mathbb Z^d}
\left|\rho_{m-1}\left(\ell\right)\right|^2
\left(1+\left|\ell\right|^4\right)^{6\gamma-1}
&\lesssim
2^{5(m-1)}
2^{4(6\gamma-1)(m-1)}\\
&=
2^{(24\gamma+1)(m-1)}.
\end{align*}
We conclude that
\[
\boxed{
\mathds{E}\left[\left|I_{r,\epsilon,1}\right|^2\right]
\lesssim
2^{(24\gamma+1)(m-1)}.
}
\]
\item \textbf{Estimation of $\mathds{E}\left[\left|I_{r,\epsilon,2}\right|^2\right]$:} To bound $\mathds{E}\left[\left|I_{r,\epsilon,2}\right|^2\right]$, we cannot simply use Parseval. Instead, we expand each of the eleven factors appearing in the integrand into its Fourier series. For each factor, we introduce a Fourier variable and use the fact that all the kernels involved are even. After imposing the frequency constraints using the orthogonality of the Fourier basis on $\mathbb T^d$ and integrating with respect to
\[
y_1,\ y_2,\ y_1',\ y_2',\ y_3',\ y_4',
\]
we may choose five independent frequencies, denoted by
$m_1,m_2,m_3,m_4,m_5\in\mathbb{Z}^d$. With this choice, the frequencies associated with the different factors are
\[
\begin{array}{c|c}
\text{factor} & \text{Fourier frequency}\\
\hline
K_{m-1}(y_1)
& m_1\\[1mm]
K_{m-1}(y_2)
& -m_1\\[1mm]
(K_{j_1+\max(q_1-2,-1)}*\Delta\eta_{r-v_1})(y_1-y_2-y_1')
& -(m_1+m_2)\\[1mm]
(K_{j_2+\max(q_2-2,-1)}*\Delta\eta_{r-v_2})(y_3')
& -(m_1+m_4)\\[1mm]
K_{q_1-1}(y_1-y_2-y_2')
& m_2\\[1mm]
K_{q_2-1}(y_4')
& m_4\\[1mm]
(\zeta_\epsilon*\zeta_\epsilon*
\eta_{v_1+v_2-2u_1})(y_1'-y_3')
& -m_5\\[1mm]
(\zeta_\epsilon*\zeta_\epsilon*
\eta_{v_1+v_2-2u_2})(y_1'-y_3')
& -(m_1+m_3-m_5)\\[1mm]
(\zeta_\epsilon*\zeta_\epsilon*
\eta_{2(r-u_3)})(y_2'-y_4')
& m_3\\[1mm]
(\zeta_\epsilon*\zeta_\epsilon*
\eta_{r+v_1-2u_4})(y_1'-y_2')
& m_3-m_2\\[1mm]
(\zeta_\epsilon*\zeta_\epsilon*
\eta_{r+v_2-2u_5})(y_3'-y_4')
& m_4-m_3.
\end{array}
\]
This yields the following bound:
\begin{align*}
&\mathds{E}\left[\left|I_{r,\epsilon,2}\right|^2\right]\\
&\lesssim
\sum_{\left(m_1,m_2,m_3,m_4,m_5\right)\in\left(\mathbb Z^d\right)^5}
\left|\rho_{m-1}\left(m_1\right)\right|^2
\psi\left(m_1+m_2,m_2\right)
\psi\left(m_1+m_4,m_4\right)
\left|m_1+m_2\right|^2
\left|m_1+m_4\right|^2\\
&\quad\times
\left|
\zeta\left(\epsilon m_3\right)
\zeta\left(\epsilon m_5\right)
\zeta\left(\epsilon(m_2-m_3)\right)
\zeta\left(\epsilon(m_3-m_4)\right)
\zeta\left(\epsilon(m_1+m_3-m_5)\right)
\right|^2\\
&\quad\times
\int_{\left[0,r\right]^3}
\mathrm{e}^{-\left|2\pi(m_1+m_4)\right|^4(r-\alpha_1)
-\left|2\pi(m_1+m_2)\right|^4(r-\alpha_2)
-2\left|2\pi m_3\right|^4(r-\alpha_3)}
\\
&\quad\times
\left(
\int_{\left[0,\alpha_1\right]\times
\left[0,\alpha_2\right]\times
\left[0,\min(\alpha_1,\alpha_2)\right]^2}
\mathrm{e}^{-\left|2\pi(m_3-m_4)\right|^4(r+\alpha_1-2v_1)
-\left|2\pi(m_2-m_3)\right|^4(r+\alpha_2-2v_2)}
\right.\\
&\hspace{5.5cm}\left.
\times
\mathrm{e}^{-\left|2\pi m_5\right|^4(\alpha_1+\alpha_2-2v_3)
-\left|2\pi(m_1+m_3-m_5)\right|^4(\alpha_1+\alpha_2-2v_4)}
\,\mathrm{d}v_1\,\mathrm{d}v_2\,\mathrm{d}v_3\,\mathrm{d}v_4
\right)\\
&\quad\times
\mathrm{d}\alpha_1\,\mathrm{d}\alpha_2\,\mathrm{d}\alpha_3.
\end{align*}

We now estimate the time integrals in detail. Since $\zeta$ is bounded, all
the factors involving $\zeta$ can be absorbed into the implicit constant.
Moreover, by the localization property of $\psi$,
\[
\psi(m_1+m_2,m_2)\neq0
\quad\Longrightarrow\quad
1+|m_1+m_2|\asymp 1+|m_2|,
\]
and
\[
\psi(m_1+m_4,m_4)\neq0
\quad\Longrightarrow\quad
1+|m_1+m_4|\asymp 1+|m_4|.
\]

Set
\[
A:=|2\pi(m_1+m_4)|^4,
\qquad
B:=|2\pi(m_1+m_2)|^4,
\qquad
C:=|2\pi m_3|^4,
\]
and
\[
D:=|2\pi(m_3-m_4)|^4,
\qquad
E:=|2\pi(m_2-m_3)|^4,
\]
\[
F:=|2\pi m_5|^4,
\qquad
G:=|2\pi(m_1+m_3-m_5)|^4.
\]
On the support of the two factors involving $\psi$,
\[
1+A\asymp1+|m_4|^4,
\qquad
1+B\asymp1+|m_2|^4.
\]

We first integrate with respect to $\alpha_3$. We have
\begin{align*}
\int_0^r
\mathrm{e}^{-2C(r-\alpha_3)}
\,\mathrm{d}\alpha_3
\lesssim
(1+C)^{-1}.
\end{align*}

We next integrate with respect to $v_1$ and $v_2$. Since
\[
r+\alpha_1-2v_1
=
(r-\alpha_1)+2(\alpha_1-v_1),
\]
we obtain
\begin{align*}
\int_0^{\alpha_1}
\mathrm{e}^{-D(r+\alpha_1-2v_1)}
\,\mathrm{d}v_1
&=
\mathrm{e}^{-D(r-\alpha_1)}
\int_0^{\alpha_1}
\mathrm{e}^{-2D(\alpha_1-v_1)}
\,\mathrm{d}v_1\\
&\lesssim
(1+D)^{-1}
\mathrm{e}^{-D(r-\alpha_1)}.
\end{align*}
Similarly,
\begin{align*}
\int_0^{\alpha_2}
\mathrm{e}^{-E(r+\alpha_2-2v_2)}
\,\mathrm{d}v_2
&\lesssim
(1+E)^{-1}
\mathrm{e}^{-E(r-\alpha_2)}.
\end{align*}

Let
\[
s:=\min(\alpha_1,\alpha_2).
\]
Since
\[
\alpha_1+\alpha_2-2s
=
|\alpha_1-\alpha_2|,
\]
we obtain
\begin{align*}
\int_0^s
\mathrm{e}^{-F(\alpha_1+\alpha_2-2v_3)}
\,\mathrm{d}v_3
&=
\mathrm{e}^{-F|\alpha_1-\alpha_2|}
\int_0^s
\mathrm{e}^{-2F(s-v_3)}
\,\mathrm{d}v_3\\
&\lesssim
(1+F)^{-1}
\mathrm{e}^{-F|\alpha_1-\alpha_2|},
\end{align*}
and likewise
\begin{align*}
\int_0^s
\mathrm{e}^{-G(\alpha_1+\alpha_2-2v_4)}
\,\mathrm{d}v_4
&\lesssim
(1+G)^{-1}
\mathrm{e}^{-G|\alpha_1-\alpha_2|}.
\end{align*}

Consequently, after integrating with respect to
$v_1,v_2,v_3,v_4$ and $\alpha_3$, the remaining time factor is bounded by
\begin{align*}
&(1+C)^{-1}(1+D)^{-1}(1+E)^{-1}
(1+F)^{-1}(1+G)^{-1}\\
&\quad\times
|m_1+m_2|^2|m_1+m_4|^2
\int_{[0,r]^2}
\mathrm{e}^{-(A+D)(r-\alpha_1)}
\mathrm{e}^{-(B+E)(r-\alpha_2)}
\mathrm{e}^{-(F+G)|\alpha_1-\alpha_2|}
\,\mathrm{d}\alpha_1\,\mathrm{d}\alpha_2.
\end{align*}

Set
\[
R:=
|m_1+m_2|^2|m_1+m_4|^2
\int_{[0,r]^2}
\mathrm{e}^{-(A+D)(r-\alpha_1)}
\mathrm{e}^{-(B+E)(r-\alpha_2)}
\mathrm{e}^{-(F+G)|\alpha_1-\alpha_2|}
\,\mathrm{d}\alpha_1\,\mathrm{d}\alpha_2.
\]

We derive two bounds for $R$. First, by discarding the last exponential,
\begin{align*}
R
&\lesssim
A^{1/2}B^{1/2}
\int_0^r
\mathrm{e}^{-(A+D)(r-\alpha_1)}
\,\mathrm{d}\alpha_1
\int_0^r
\mathrm{e}^{-(B+E)(r-\alpha_2)}
\,\mathrm{d}\alpha_2\\
&\lesssim
(1+A)^{-1/2}(1+B)^{-1/2}.
\end{align*}

We also retain the decay coming from $F+G$. Making the change of variables
\[
u:=r-\alpha_1,
\qquad
v:=r-\alpha_2,
\]
and setting
\[
X:=A+D,
\qquad
Y:=B+E,
\qquad
H:=F+G,
\]
we obtain
\[
R
\lesssim
A^{1/2}B^{1/2}
\int_{\mathbb R_+^2}
\mathrm{e}^{-Xu-Yv-H|u-v|}
\,\mathrm{d}u\,\mathrm{d}v.
\]
Splitting the domain into $\{u\geq v\}$ and $\{v\geq u\}$ gives
\begin{align*}
&\int_{\mathbb R_+^2}
\mathrm{e}^{-Xu-Yv-H|u-v|}
\,\mathrm{d}u\,\mathrm{d}v\\
&=
\frac{1}{X+Y}
\left(
\frac{1}{X+H}
+
\frac{1}{Y+H}
\right).
\end{align*}
Since
\[
A^{1/2}B^{1/2}
\leq X^{1/2}Y^{1/2}
\]
and
\[
\frac{X^{1/2}Y^{1/2}}{X+Y}
\leq\frac12,
\]
we conclude that
\[
R\lesssim(1+F+G)^{-1}.
\]

Thus,
\[
R
\lesssim
(1+A)^{-1/2}(1+B)^{-1/2},
\]
and
\[
R
\lesssim
(1+F+G)^{-1}.
\]
We now apply Lemma~\ref{l33} to these two estimates. Set
\[
\vartheta:=\frac{3(1-\gamma)}4.
\]
Since $\gamma\in(0,1/11)$, we have $\vartheta\in(0,1)$. Hence
\begin{align*}
R
&\lesssim
\left(
(1+A)^{-1/2}(1+B)^{-1/2}
\right)^{\vartheta}
\left(
(1+F+G)^{-1}
\right)^{1-\vartheta}\\
&=
(1+A)^{-\frac{3(1-\gamma)}8}
(1+B)^{-\frac{3(1-\gamma)}8}
(1+F+G)^{-\frac{1+3\gamma}{4}}\\
&=
(1+A)^{\frac{3\gamma}{8}-\frac38}
(1+B)^{\frac{3\gamma}{8}-\frac38}
(1+F+G)^{-\frac{1+3\gamma}{4}}.
\end{align*}

Since
\[
1+F+G\geq1+G,
\]
we have
\[
(1+F+G)^{-\frac{1+3\gamma}{4}}
\leq
(1+G)^{-\frac{1+3\gamma}{4}}.
\]
Multiplying by the factor $(1+G)^{-1}$ already obtained from the
$v_4$-integration gives
\begin{align*}
(1+G)^{-1}
(1+G)^{-\frac{1+3\gamma}{4}}
&=
(1+G)^{-\frac{5+3\gamma}{4}}.
\end{align*}
Since $\gamma>0$,
\[
-\frac{5+3\gamma}{4}
\leq
-\frac{5-5\gamma}{4}
=
\frac{5\gamma}{4}-\frac54,
\]
and therefore
\[
(1+G)^{-\frac{5+3\gamma}{4}}
\leq
(1+G)^{\frac{5\gamma}{4}-\frac54}.
\]

For the remaining factors we simply weaken the decay:
\[
(1+C)^{-1}
\leq
(1+C)^{\gamma-1},
\]
\[
(1+D)^{-1}
\leq
(1+D)^{\gamma-1},
\qquad
(1+E)^{-1}
\leq
(1+E)^{\gamma-1},
\]
and
\[
(1+F)^{-1}
\leq
(1+F)^{\gamma-1}.
\]
Recalling
\[
1+A\asymp1+|m_4|^4,
\qquad
1+B\asymp1+|m_2|^4,
\]
we obtain
\begin{align*}
&\mathds{E}\left[\left|I_{r,\epsilon,2}\right|^2\right]\\
&\stackrel{\text{Lemma \ref{l33}}}{\lesssim}
\sum_{\left(m_1,m_2,m_3,m_4,m_5\right)\in\left(\mathbb Z^d\right)^5}
\left|\rho_{m-1}\left(m_1\right)\right|^2
\left(1+\left|m_3\right|^4\right)^{\gamma-1}
\left(1+\left|m_5\right|^4\right)^{\gamma-1}\\
&\quad\times
\left(1+\left|m_2-m_3\right|^4\right)^{\gamma-1}
\left(1+\left|m_3-m_4\right|^4\right)^{\gamma-1}
\left(1+\left|m_2\right|^4\right)^{\frac{3\gamma}{8}-\frac38}\\
&\quad\times
\left(1+\left|m_4\right|^4\right)^{\frac{3\gamma}{8}-\frac38}
\left(1+\left|m_1+m_3-m_5\right|^4\right)^{\frac{5\gamma}{4}-\frac54}.
\end{align*}

We now make explicit the successive applications of Lemma~\ref{lemma2}. For fixed $m_1$, $m_3$, and $m_5$, we first sum with respect to $m_2$. Since
\[
\gamma-1=-\frac{4(1-\gamma)}{4}
\]
and
\[
\frac{3\gamma}{8}-\frac38
=
-\frac{\frac32(1-\gamma)}{4},
\]
Lemma~\ref{lemma2} gives
\begin{align*}
&\sum_{m_2\in\mathbb Z^d}
\left(1+|m_2-m_3|^4\right)^{\gamma-1}
\left(1+|m_2|^4\right)^{\frac{3\gamma}{8}-\frac38}\\
&\lesssim
\left(1+|m_3|^4\right)^{
\frac{5-4(1-\gamma)-\frac32(1-\gamma)}4}\\
&=
\left(1+|m_3|^4\right)^{\frac{11\gamma}{8}-\frac18}.
\end{align*}
Indeed,
\[
4(1-\gamma)+\frac32(1-\gamma)
=
\frac{11}{2}(1-\gamma)>5
\]
since $\gamma\in(0,1/11)$.

Exactly the same argument for $m_4$ gives
\begin{align*}
&\sum_{m_4\in\mathbb Z^d}
\left(1+|m_3-m_4|^4\right)^{\gamma-1}
\left(1+|m_4|^4\right)^{\frac{3\gamma}{8}-\frac38}\\
&\lesssim
\left(1+|m_3|^4\right)^{\frac{11\gamma}{8}-\frac18}.
\end{align*}
Combining these two bounds with the original factor
$\left(1+|m_3|^4\right)^{\gamma-1}$, we obtain
\begin{align*}
&\left(1+|m_3|^4\right)^{\gamma-1}
\left(1+|m_3|^4\right)^{\frac{11\gamma}{8}-\frac18}
\left(1+|m_3|^4\right)^{\frac{11\gamma}{8}-\frac18}\\
&=
\left(1+|m_3|^4\right)^{
\gamma-1+\frac{11\gamma}{4}-\frac14}\\
&=
\left(1+|m_3|^4\right)^{\frac{15\gamma}{4}-\frac54}.
\end{align*}
Thus it remains to estimate
\begin{align*}
&\sum_{m_3,m_5\in\mathbb Z^d}
\left(1+|m_3|^4\right)^{\frac{15\gamma}{4}-\frac54}
\left(1+|m_5|^4\right)^{\gamma-1}\\
&\hspace{4cm}\times
\left(1+|m_1+m_3-m_5|^4\right)^{\frac{5\gamma}{4}-\frac54}.
\end{align*}

We first perform the sum over $m_5$. Since
\[
\gamma-1=-\frac{4-4\gamma}{4},
\qquad
\frac{5\gamma}{4}-\frac54
=
-\frac{5-5\gamma}{4},
\]
Lemma~\ref{lemma2} yields
\begin{align*}
&\sum_{m_5\in\mathbb Z^d}
\left(1+|m_5|^4\right)^{\gamma-1}
\left(1+|m_1+m_3-m_5|^4\right)^{\frac{5\gamma}{4}-\frac54}\\
&\lesssim
\left(1+|m_1+m_3|^4\right)^{
\frac{5-(4-4\gamma)-(5-5\gamma)}4}\\
&=
\left(1+|m_1+m_3|^4\right)^{\frac{9\gamma}{4}-1}.
\end{align*}
Consequently, we are left with
\[
\sum_{m_3\in\mathbb Z^d}
\left(1+|m_3|^4\right)^{\frac{15\gamma}{4}-\frac54}
\left(1+|m_1+m_3|^4\right)^{\frac{9\gamma}{4}-1}.
\]
Now
\[
\frac{15\gamma}{4}-\frac54
=
-\frac{5-15\gamma}{4},
\qquad
\frac{9\gamma}{4}-1
=
-\frac{4-9\gamma}{4}.
\]
Since
\[
(5-15\gamma)+(4-9\gamma)
=
9-24\gamma>5
\]
for $\gamma\in(0,1/11)$, Lemma~\ref{lemma2} applies once more and gives
\begin{align*}
&\sum_{m_3\in\mathbb Z^d}
\left(1+|m_3|^4\right)^{\frac{15\gamma}{4}-\frac54}
\left(1+|m_1+m_3|^4\right)^{\frac{9\gamma}{4}-1}\\
&\lesssim
\left(1+|m_1|^4\right)^{
\frac{5-(5-15\gamma)-(4-9\gamma)}4}\\
&=
\left(1+|m_1|^4\right)^{6\gamma-1}.
\end{align*}
Therefore,
\begin{align*}
\mathds{E}\left[\left|I_{r,\epsilon,2}\right|^2\right]
&\stackrel{\text{Lemma \ref{lemma2}}}{\lesssim}
\sum_{m_1\in\mathbb Z^d}
\left|\rho_{m-1}\left(m_1\right)\right|^2
\left(1+\left|m_1\right|^4\right)^{6\gamma-1}.
\end{align*}

Finally, for $m\geq1$, on the support of $\rho_{m-1}, |m_1|\asymp 2^{m-1},$
Since $d=5$, the number of lattice points in the corresponding dyadic annulus is of order $2^{5(m-1)}$. For $m=0$, the support of $\rho_{-1}$ is bounded,
and the corresponding estimate follows directly. Hence
\begin{align*}
\sum_{m_1\in\mathbb Z^d}
\left|\rho_{m-1}\left(m_1\right)\right|^2
\left(1+\left|m_1\right|^4\right)^{6\gamma-1}
&\lesssim
2^{5(m-1)}
2^{4(6\gamma-1)(m-1)}\\
&=
2^{(24\gamma+1)(m-1)}.
\end{align*}
We conclude that
\[
\mathds{E}\left[\left|I_{r,\epsilon,2}\right|^2\right]
\lesssim
2^{(24\gamma+1)(m-1)}.
\]
\item \textbf{Estimation of $\mathds{E}\left[
\left|
I_{r,\epsilon,3}
-\psi_{\epsilon,4}(r)
\left(\delta_{m-1}\widetilde X_{r,\epsilon,1}\right)(x)
\right|^2
\right]$:} We are in a position to apply Parseval to obtain the following bound.
\begin{align*}
& \mathds{E}\left[
\left|
I_{r,\epsilon,3}
-\psi_{\epsilon,4}(r)
\left(\delta_{m-1}\widetilde X_{r,\epsilon,1}\right)(x)
\right|^2
\right]\\&\lesssim\sum_{m_1\in\mathbb{Z}^d}\int_0^rdu_1|-\rho_{m-1}(m_1)\psi_{\epsilon,4}(r)\zeta(\epsilon m_1)\mathrm{e}^{-|2\pi m_1|^4(r-u_1)}\\&+\rho_{m-1}(m_1)\zeta(\epsilon m_1)\sum_{\left(m_2,m_3\right) \in \left(\mathbb{Z}^5\right)^2}\left|2\pi\left(m_1-m_2\right)\right|^2\psi\left(m_1-m_2,m_2\right)\left|\zeta\left(\epsilon \left(m_2-m_3\right)\right)\right|^2\left|\zeta\left(\epsilon m_3\right)\right|^2\\&\int_{0}^r\mathds{1}_{[0,\alpha]}(u_1)\mathrm{e}^{-|2\pi m_1|^4(\alpha-u_1)}\mathrm{e}^{-\left|2\pi\left(m_1-m_2\right)\right|^4\left(r-\alpha\right)}\left(\int_{\left[0,\alpha\right]^2}\mathrm{e}^{-\left|2\pi \left(m_2-m_3\right)\right|^4\left(r+\alpha-2v_1\right)-\left|2\pi m_3\right|^4 \left(r+\alpha-2v_2\right)}\,\mathrm{d}v_1\,\mathrm{d}v_2\right)\,\mathrm{d}\alpha|^2.
\end{align*} To estimate the preceding term, for $\epsilon>0$ we introduce
\[
\phi_\epsilon:\mathbb{R}_+\times\mathbb{Z}^5\longrightarrow\mathbb{R}
\]
defined by
\begin{align*}
\phi_\epsilon(q,n)
:=
\sum_{(m_2,m_3)\in(\mathbb{Z}^5)^2}
&|2\pi(n-m_2)|^2
\psi(n-m_2,m_2)
|\zeta(\epsilon(m_2-m_3))|^2
|\zeta(\epsilon m_3)|^2\\
&\times
\int_0^q
\mathrm{e}^{-|2\pi(n-m_2)|^4(q-\alpha)}
\left(
\int_{[0,\alpha]^2}
\mathrm{e}^{-|2\pi(m_2-m_3)|^4(q+\alpha-2v_1)
-|2\pi m_3|^4(q+\alpha-2v_2)}
\,dv_1\,dv_2
\right)
\,d\alpha .
\end{align*}
In particular,
\[
\phi_\epsilon(r,0)=\psi_{\epsilon,4}(r).
\]

\begin{lemma}\label{lem:phi-renormalized}
For every $\gamma>0$,
\[
\left|
\phi_\epsilon(r,n)-\psi_{\epsilon,4}(r)
\right|
\lesssim
(1+|n|^4)^{\gamma/2},
\]
uniformly in $\epsilon>0$ and $r$ in a bounded interval.
\end{lemma}

\begin{proof}
Since
\[
\phi_\epsilon(r,0)=\psi_{\epsilon,4}(r),
\]
it is enough to estimate
\[
\phi_\epsilon(r,n)-\phi_\epsilon(r,0).
\]
For $p,s\in\mathbb{Z}^5$, set
\begin{align*}
H_{\epsilon,r}(p,s)
:={}&
|2\pi p|^2\psi(p,s)
\sum_{\ell\in\mathbb{Z}^5}
|\zeta(\epsilon(s-\ell))|^2
|\zeta(\epsilon\ell)|^2\\
&\times
\int_0^r
\mathrm{e}^{-|2\pi p|^4(r-\alpha)}
J_{s-\ell}(r,\alpha)
J_\ell(r,\alpha)
\,d\alpha,
\end{align*}
where
\[
J_k(r,\alpha)
:=
\int_0^\alpha
\mathrm{e}^{-|2\pi k|^4(r+\alpha-2v)}
\,dv.
\]
Then
\[
\phi_\epsilon(r,n)
=
\sum_{s\in\mathbb{Z}^5}
H_{\epsilon,r}(n-s,s)
\]
and
\[
\phi_\epsilon(r,0)
=
\sum_{s\in\mathbb{Z}^5}
H_{\epsilon,r}(-s,s).
\]
Consequently,
\begin{equation}\label{eq:phi-difference}
\phi_\epsilon(r,n)-\phi_\epsilon(r,0)
=
\sum_{s\in\mathbb{Z}^5}
\left(
H_{\epsilon,r}(n-s,s)
-
H_{\epsilon,r}(-s,s)
\right).
\end{equation}

For $0\leq\alpha\leq r$, we have
\begin{equation}\label{eq:J-bound}
J_k(r,\alpha)
\lesssim
\frac{
\mathrm{e}^{-c|k|^4(r-\alpha)}
}{
1+|k|^4
}.
\end{equation}
Indeed, for $k\neq0$,
\begin{align*}
J_k(r,\alpha)
&=
\mathrm{e}^{-|2\pi k|^4(r-\alpha)}
\int_0^\alpha
\mathrm{e}^{-2|2\pi k|^4(\alpha-v)}
\,dv\\
&=
\mathrm{e}^{-|2\pi k|^4(r-\alpha)}
\frac{
1-\mathrm{e}^{-2|2\pi k|^4\alpha}
}{
2|2\pi k|^4
},
\end{align*}
while, for $k=0$, the estimate follows from
\[
J_0(r,\alpha)=\alpha.
\]

Using \eqref{eq:J-bound}, the boundedness of $\zeta$, and integrating with
respect to $\alpha$, we obtain
\begin{align}
|H_{\epsilon,r}(p,s)|
\lesssim
|\psi(p,s)|
\sum_{\ell\in\mathbb{Z}^5}
\frac{|p|^2}{
(1+|s-\ell|^4)(1+|\ell|^4)
}
\frac{1}{
1+|p|^4+|s-\ell|^4+|\ell|^4
}.
\label{eq:H-first-bound}
\end{align}

By the support property of $\psi$,
\[
\psi(p,s)\neq0
\quad\Longrightarrow\quad
1+|p|\asymp 1+|s|.
\]
Hence
\[
\frac{|p|^2}{
1+|p|^4+|s-\ell|^4+|\ell|^4
}
\lesssim
(1+|s|^4)^{-1/2}.
\]
Moreover, Lemma~\ref{lemma2}, applied with $\alpha=\beta=4$ and $d=5$,
yields
\[
\sum_{\ell\in\mathbb{Z}^5}
(1+|s-\ell|^4)^{-1}
(1+|\ell|^4)^{-1}
\lesssim
(1+|s|^4)^{-3/4}.
\]
Combining the preceding estimates gives
\begin{equation}\label{eq:H-critical-bound}
|H_{\epsilon,r}(p,s)|
\lesssim
(1+|s|^4)^{-5/4}.
\end{equation}

We next estimate the variation of $H_{\epsilon,r}$ with respect to its
first frequency. We use the smooth extension of the Littlewood--Paley
multipliers to $\mathbb{R}^5$. The derivative estimates for the
Littlewood--Paley functions give, whenever $1+|p|\asymp1+|s|$,
\[
|\nabla_p\psi(p,s)|
\lesssim
(1+|s|)^{-1}.
\]
Moreover, for every $t\geq0$,
\[
\left|
\nabla_p
\left(
|p|^2\mathrm{e}^{-c|p|^4t}
\right)
\right|
\lesssim
|p|\mathrm{e}^{-c'|p|^4t}.
\]
Repeating the argument leading to \eqref{eq:H-critical-bound}, with one
derivative with respect to $p$, gives
\begin{equation}\label{eq:H-derivative-bound}
|\nabla_pH_{\epsilon,r}(p,s)|
\lesssim
(1+|s|)^{-1}
(1+|s|^4)^{-5/4},
\qquad 1+|p|\asymp1+|s|.
\end{equation}

We now split the sum in \eqref{eq:phi-difference} into the two regions
\[
|s|\leq2|n|
\qquad\text{and}\qquad
|s|>2|n|.
\]
For the first region, by \eqref{eq:H-critical-bound},
\begin{align*}
&\sum_{|s|\leq2|n|}
\left|
H_{\epsilon,r}(n-s,s)
-
H_{\epsilon,r}(-s,s)
\right|\\
&\qquad\lesssim
\sum_{|s|\leq2|n|}
(1+|s|^4)^{-5/4}.
\end{align*}
Since $d=5$, a dyadic shell
\[
2^j\leq|s|<2^{j+1}
\]
contains at most $C2^{5j}$ lattice points, whereas
\[
(1+|s|^4)^{-5/4}
\lesssim2^{-5j}.
\]
Thus every dyadic shell contributes a bounded quantity and
\begin{equation}\label{eq:phi-low-frequencies}
\sum_{|s|\leq2|n|}
\left|
H_{\epsilon,r}(n-s,s)
-
H_{\epsilon,r}(-s,s)
\right|
\lesssim
1+\log(2+|n|).
\end{equation}

Consider now $|s|>2|n|$. For $\theta\in[0,1]$, set
\[
p_\theta=-s+\theta n.
\]
Then
\[
\frac12|s|
\leq
|p_\theta|
\leq
\frac32|s|,
\]
so that $|p_\theta|\asymp|s|$. By the mean value theorem and
\eqref{eq:H-derivative-bound},
\begin{align*}
&
\left|
H_{\epsilon,r}(n-s,s)
-
H_{\epsilon,r}(-s,s)
\right|\\
&\qquad\leq
|n|
\sup_{\theta\in[0,1]}
|\nabla_pH_{\epsilon,r}(p_\theta,s)|\\
&\qquad\lesssim
|n|
(1+|s|)^{-1}
(1+|s|^4)^{-5/4}.
\end{align*}
Therefore
\begin{align*}
&\sum_{|s|>2|n|}
\left|
H_{\epsilon,r}(n-s,s)
-
H_{\epsilon,r}(-s,s)
\right|\\
&\qquad\lesssim
|n|
\sum_{|s|>2|n|}
(1+|s|)^{-1}
(1+|s|^4)^{-5/4}.
\end{align*}
On a dyadic shell of size $2^j$, the last expression is bounded by
\[
C|n|2^{5j}2^{-j}2^{-5j}
=
C|n|2^{-j}.
\]
Summing over $2^j\gtrsim|n|$, we obtain
\begin{equation}\label{eq:phi-high-frequencies}
\sum_{|s|>2|n|}
\left|
H_{\epsilon,r}(n-s,s)
-
H_{\epsilon,r}(-s,s)
\right|
\lesssim1.
\end{equation}

Combining \eqref{eq:phi-low-frequencies} and
\eqref{eq:phi-high-frequencies}, we conclude that
\[
\left|
\phi_\epsilon(r,n)-\phi_\epsilon(r,0)
\right|
\lesssim
1+\log(2+|n|).
\]
Since
\[
\phi_\epsilon(r,0)=\psi_{\epsilon,4}(r),
\]
we obtain
\[
\left|
\phi_\epsilon(r,n)-\psi_{\epsilon,4}(r)
\right|
\lesssim
1+\log(2+|n|).
\]
Finally, for every $\gamma>0$,
\[
1+\log(2+|n|)
\lesssim
(1+|n|^4)^{\gamma/2},
\]
and therefore
\[
\left|
\phi_\epsilon(r,n)-\psi_{\epsilon,4}(r)
\right|
\lesssim
(1+|n|^4)^{\gamma/2}.
\]
\end{proof}

For $n\in\mathbb{Z}^5$ and $\alpha\in[0,r]$, set
\begin{align*}
F_{\epsilon}(n,\alpha)
:=
\sum_{(m_2,m_3)\in(\mathbb{Z}^5)^2}
&|2\pi(n-m_2)|^2
\psi(n-m_2,m_2)
|\zeta(\epsilon(m_2-m_3))|^2
|\zeta(\epsilon m_3)|^2\\
&\times
\mathrm{e}^{-|2\pi(n-m_2)|^4(r-\alpha)}
\left(
\int_{[0,\alpha]^2}
\mathrm{e}^{-|2\pi(m_2-m_3)|^4(r+\alpha-2v_1)
-|2\pi m_3|^4(r+\alpha-2v_2)}
\,dv_1\,dv_2
\right).
\end{align*}
Then
\[
\phi_\epsilon(r,n)=\int_0^r F_\epsilon(n,\alpha)\,d\alpha.
\]
We have
\begin{align*}
&-\psi_{\epsilon,4}(r)
\mathrm{e}^{-|2\pi n|^4(r-u_1)}
+
\int_0^r
\mathds{1}_{[0,\alpha]}(u_1)
\mathrm{e}^{-|2\pi n|^4(\alpha-u_1)}
F_\epsilon(n,\alpha)\,d\alpha\\
&=
\mathrm{e}^{-|2\pi n|^4(r-u_1)}
\left(
\phi_\epsilon(r,n)-\psi_{\epsilon,4}(r)
\right)\\
&\quad+
\int_0^r
\left(
\mathds{1}_{[0,\alpha]}(u_1)
\mathrm{e}^{-|2\pi n|^4(\alpha-u_1)}
-
\mathrm{e}^{-|2\pi n|^4(r-u_1)}
\right)
F_\epsilon(n,\alpha)\,d\alpha.
\end{align*}
Therefore, by the triangle inequality and Minkowski's inequality,
\begin{align}
&\mathds{E}\left[
\left|
I_{r,\epsilon,3}
-\psi_{\epsilon,4}(r)
\left(\delta_{m-1}\widetilde X_{r,\epsilon,1}\right)(x)
\right|^2
\right]\nonumber\\
&\lesssim
\sum_{n\in\mathbb{Z}^5}
|\rho_{m-1}(n)\zeta(\epsilon n)|^2
|\phi_\epsilon(r,n)-\psi_{\epsilon,4}(r)|^2
\int_0^r
\mathrm{e}^{-2|2\pi n|^4(r-u_1)}
\,du_1\nonumber\\
&\quad+
\sum_{n\in\mathbb{Z}^5}
|\rho_{m-1}(n)\zeta(\epsilon n)|^2
\left(
\int_0^r
F_\epsilon(n,\alpha)
D_n(r-\alpha)\,d\alpha
\right)^2,
\label{eq:I3-decomposition}
\end{align}
where
\[
D_n(r-\alpha)
:=
\left(
\int_0^r
\left(
\mathrm{e}^{-|2\pi n|^4(r-u_1)}
-
\mathds{1}_{[0,\alpha]}(u_1)
\mathrm{e}^{-|2\pi n|^4(\alpha-u_1)}
\right)^2
\,du_1
\right)^{1/2}.
\]

We first estimate the first term in \eqref{eq:I3-decomposition}. Using Lemma \ref{lem:phi-renormalized}\[
|\phi_\epsilon(r,n)-\psi_{\epsilon,4}(r)|
\lesssim
(1+|n|^4)^{\gamma/2}
\]
and
\[
\int_0^r
\mathrm{e}^{-2|2\pi n|^4(r-u_1)}
\,du_1
\lesssim
(1+|n|^4)^{-1},
\]
we obtain
\begin{align}
&|\phi_\epsilon(r,n)-\psi_{\epsilon,4}(r)|^2
\int_0^r
\mathrm{e}^{-2|2\pi n|^4(r-u_1)}
\,du_1\nonumber\\
&\qquad\lesssim
(1+|n|^4)^{\gamma-1}.
\label{eq:I3-first-term}
\end{align}

It remains to estimate the second term in
\eqref{eq:I3-decomposition}. Let
\[
t=r-\alpha.
\]
We have \begin{align}
\forall \alpha \in[0,r] \int_0^\alpha
\mathrm{e}^{-|2\pi k|^4(r+\alpha-2v)}
\,dv
&\lesssim
\frac{\mathrm{e}^{-c|k|^4(r-\alpha)}}
{1+|k|^4},c>0.
\label{eq:heat-time-bound}
\end{align}

Using \eqref{eq:heat-time-bound} twice and the boundedness of $\zeta$,
we obtain
\begin{align*}
F_\epsilon(n,\alpha)
\lesssim
\sum_{m_2,m_3\in\mathbb{Z}^5}
&|n-m_2|^2
\psi(n-m_2,m_2)
\mathrm{e}^{-c|n-m_2|^4t}\\
&\times
\frac{\mathrm{e}^{-c|m_2-m_3|^4t}}
{1+|m_2-m_3|^4}
\frac{\mathrm{e}^{-c|m_3|^4t}}
{1+|m_3|^4}.
\end{align*}
Since
\[
|m_2-m_3|^4+|m_3|^4
\gtrsim |m_2|^4,
\]
Lemma~\ref{lemma2}, applied with $\alpha=\beta=4$ and $d=5$, yields
\begin{align*}
&\sum_{m_3\in\mathbb{Z}^5}
\frac{\mathrm{e}^{-c|m_2-m_3|^4t}}
{1+|m_2-m_3|^4}
\frac{\mathrm{e}^{-c|m_3|^4t}}
{1+|m_3|^4}\\
&\qquad\lesssim
\mathrm{e}^{-c'|m_2|^4t}
(1+|m_2|^4)^{-3/4}.
\end{align*}
Moreover, by the support property of $\psi$,
\[
\psi(n-m_2,m_2)\neq0
\quad\Longrightarrow\quad
1+|n-m_2|\lesssim 1+|m_2|.
\]
Consequently,
\begin{align*}
F_\epsilon(n,\alpha)
&\lesssim
\sum_{m_2\in\mathbb{Z}^5}
\mathrm{e}^{-c|m_2|^4t}
(1+|m_2|^4)^{-1/4}.
\end{align*}
Since $d=5$,
\begin{align*}
\sum_{m_2\in\mathbb{Z}^5}
\mathrm{e}^{-c|m_2|^4t}
(1+|m_2|^4)^{-1/4}
&\lesssim t^{-1},
\end{align*}
and hence
\begin{equation}
F_\epsilon(n,\alpha)
\lesssim
(r-\alpha)^{-1}.
\label{eq:F-bound}
\end{equation}

We next estimate $D_n$. For $n\neq0$, setting
\[
\lambda_n=|2\pi n|^4,
\]
a direct computation gives
\begin{align*}
D_n(t)^2
&=
(1-\mathrm{e}^{-\lambda_n t})^2
\frac{1-\mathrm{e}^{-2\lambda_n\alpha}}
{2\lambda_n}
+
\frac{1-\mathrm{e}^{-2\lambda_n t}}
{2\lambda_n}\\
&\lesssim
\lambda_n^{-1}\min\{1,\lambda_n t\}.
\end{align*}
Therefore
\begin{equation}
D_n(t)
\lesssim
|n|^{-2}
\min\{1,|n|^2t^{1/2}\}.
\label{eq:D-bound}
\end{equation}

Combining \eqref{eq:F-bound} and \eqref{eq:D-bound}, we obtain
\begin{align*}
\int_0^r
F_\epsilon(n,\alpha)D_n(r-\alpha)\,d\alpha
&\lesssim
\int_0^r t^{-1}D_n(t)\,dt.
\end{align*}
For $n\neq0$, we split the last integral at $t=|n|^{-4}$. On
$0<t\leq |n|^{-4}$,
\[
D_n(t)\lesssim t^{1/2},
\]
and therefore
\[
\int_0^{|n|^{-4}}
t^{-1}D_n(t)\,dt
\lesssim |n|^{-2}.
\]
On the other hand, for $t\geq |n|^{-4}$,
\[
D_n(t)\lesssim |n|^{-2},
\]
so that
\[
\int_{|n|^{-4}}^r
t^{-1}D_n(t)\,dt
\lesssim
|n|^{-2}\left(1+\log(2+|n|)\right).
\]
Thus
\begin{equation}
\int_0^r
F_\epsilon(n,\alpha)D_n(r-\alpha)\,d\alpha
\lesssim
(1+|n|^4)^{-1/2}
\left(1+\log(2+|n|)\right).
\label{eq:FD-bound}
\end{equation}
For $n=0$, one has
\[
D_0(t)=t^{1/2},
\]
and the same bound follows directly from \eqref{eq:F-bound}.

Since, for every $\gamma>0$,
\[
\left(1+\log(2+|n|)\right)^2
\lesssim
(1+|n|^4)^\gamma,
\]
it follows from \eqref{eq:FD-bound} that
\begin{equation}
\left(
\int_0^r
F_\epsilon(n,\alpha)D_n(r-\alpha)\,d\alpha
\right)^2
\lesssim
(1+|n|^4)^{\gamma-1}.
\label{eq:I3-second-term}
\end{equation}

Combining \eqref{eq:I3-first-term} and
\eqref{eq:I3-second-term} in \eqref{eq:I3-decomposition}, and using the
boundedness of $\zeta$, we deduce that
\begin{align*}
&\mathds{E}\left[
\left|
I_{r,\epsilon,3}
-\psi_{\epsilon,4}(r)
\left(\delta_{m-1}\widetilde X_{r,\epsilon,1}\right)(x)
\right|^2
\right]\\
&\qquad\lesssim
\sum_{n\in\mathbb{Z}^5}
|\rho_{m-1}(n)|^2
(1+|n|^4)^{\gamma-1}.
\end{align*}
Finally
\begin{align*}
\sum_{n\in\mathbb{Z}^5}
|\rho_{m-1}(n)|^2
(1+|n|^4)^{\gamma-1}
&\lesssim
2^{5(m-1)}
2^{4(\gamma-1)(m-1)}\\
&=
2^{(4\gamma+1)(m-1)}.
\end{align*}
Hence
\[
\mathds{E}\left[
\left|
I_{r,\epsilon,3}
-\psi_{\epsilon,4}(r)
\left(\delta_{m-1}\widetilde X_{r,\epsilon,1}\right)(x)
\right|^2
\right]
\lesssim
2^{(4\gamma+1)(m-1)}.
\]\end{enumerate}
\end{proof}


\end{document}